\documentclass[letterpaper]{amsart}

\usepackage[letterpaper]{geometry}
\usepackage[mathscr]{euscript}
\usepackage{amsmath, amsthm, amsfonts, amsbsy, thmtools, amssymb, hyperref, stmaryrd, cleveref, tikz}
\usetikzlibrary{arrows}

\numberwithin{equation}{section}

\newtheorem{thm}{Theorem}[section]
\newtheorem{prop}[thm]{Proposition}
\newtheorem{lem}[thm]{Lemma}
\newtheorem{cor}[thm]{Corollary}

\theoremstyle{remark}
\newtheorem{rem}[thm]{Remark}
\theoremstyle{definition}
\newtheorem{definition}[thm]{Definition}

\DeclareMathOperator{\Vol}{Vol}
\DeclareMathOperator{\Aut}{Aut}

\title[Large Genus Asymptotics for Intersection Numbers and Principal Strata Volumes]{Large Genus Asymptotics for Intersection Numbers and Principal Strata Volumes of Quadratic Differentials}

\author{Amol Aggarwal}

\begin{document}

	\begin{abstract}
		
		In this paper we analyze the large genus asymptotics for intersection numbers between $\psi$-classes, also called correlators, on the moduli space of stable curves. Our proofs proceed through a combinatorial analysis of the recursive relations (Virasoro constraints) that uniquely determine these correlators, together with a comparison between the coefficients in these relations with the jump probabilities of a certain asymmetric simple random walk. As an application of this result, we provide the large genus limits for Masur--Veech volumes and area Siegel--Veech constants associated with principal strata in the moduli space of quadratic differentials. These confirm predictions of Delecroix--Goujard--Zograf--Zorich from 2019.

	\end{abstract}
	
	\maketitle
	
	\tableofcontents

	\section{Introduction} 

	\label{Introduction} 
	
	In this paper we analyze the large genus asymptotics for intersection numbers of $\psi$-classes on the moduli space of stable curves, and for principal strata volumes in the moduli space of quadratic differentials. After describing the context for these results in \Cref{IntersectionLimit}, \Cref{VolumesConstants}, and \Cref{Limit1}, we state them more precisely in \Cref{Asymptotic}. 
	
	\subsection{Asymptotics for Intersection Numbers}
	
	\label{IntersectionLimit}
	
	Fix integers $g, n \ge 0$ such that $2g + n \ge 3$; let $\mathcal{M}_{g, n}$ denote the moduli space of smooth, complex, genus $g$ curves with $n$ marked points; and let $\overline{\mathcal{M}}_{g, n}$ denote its Deligne--Mumford compactification. Equivalently, $\overline{\mathcal{M}}_{g, n}$ is the moduli space of tuples $(C; x_1, x_2, \ldots , x_n)$, where $C$ is a stable curve of genus $g$ and $(x_1, x_2, \ldots , x_n)$ is an ordered collection of nonsingular points on $C$. For each index $i \in \{ 1, 2, \ldots , n \}$, let $\mathcal{L}_i$ denote the line bundle on $\overline{\mathcal{M}}_{g, n}$ whose fiber over any point $(C; x_1, x_2, \ldots , x_n) \in \overline{\mathcal{M}}_{g, n}$ is the cotangent space $T_{x_i}^* C$, and let $\psi_i = c_1 (\mathcal{L}_i)$ denote the first Chern class of this bundle; these are referred to as $\psi$-classes on $\overline{\mathcal{M}}_{g, n}$. 
	
	For any $n$-tuple $\textbf{d} = (d_1, d_2, \ldots,  d_n) \in \mathbb{Z}_{\ge 0}^n$, define the \emph{correlator} $\big\langle \prod_{i = 1}^n \tau_{d_i} \big\rangle$ as the intersection of $\psi$-classes given by 
	\begin{flalign}
	\label{intersection2}
	\Bigg\langle \displaystyle\prod_{i = 1}^n \tau_{d_i} \Bigg\rangle = \displaystyle\int_{\overline{\mathcal{M}}_{g, n}} \displaystyle\prod_{i = 1}^n \psi_i^{d_i}.
	\end{flalign}

	\noindent This quantity is nonzero only if $|\textbf{d}| = 3g + n - 3$, where we have denoted $|\textbf{d}| = \sum_{i = 1}^n d_i$.

	Over the past several decades, it has been understood that numerous fundamental invariants in physics and geometric topology can be expressed in terms of the intersection numbers \eqref{intersection2}. Perhaps the earliest result in this direction can be attributed to Witten \cite{TGITMS}, who realized them as correlation functions for a model of two-dimensional quantum gravity. Later appearances of these correlators include the identity of Mirzakhani \cite{VITMSC} for Weil--Petersson volumes of moduli spaces of bordered Riemann surfaces and the formulas of Mirzakhani \cite{GNSCGHS} and Delecroix--Goujard--Zograf--Zorich \cite{VFGINMSC} for frequencies of geodesic multi-curves in hyperbolic and random flat surfaces, respectively. In an analogous but slightly different direction, the works of Chen, M\"{o}ller, Sauvaget, and Zagier \cite{QLGL,VSI,VSCC,VITPSQD} write Masur--Veech volumes and Siegel--Veech constants for strata of holomorphic and quadratic differentials also through intersection numbers, but between certain classes on compactified projectivizations of Hodge bundles \cite{CCSD}, instead of on $\overline{\mathcal{M}}_{g, n}$. 
	
	It is therefore of interest to evaluate the intersection numbers \eqref{intersection2}, which can be done through the Kontsevich--Witten theorem \cite{ITMSCF,TGITMS}; the latter amounts to a family of recursions, widely referred to as \emph{Virasoro constraints} after the works of Dijkgraaf \cite{ITIHTF} and Verlinde--Verlinde \cite{TTQG}, that explicitly determines the correlators. A different set of recursions for these quantities was later found by Liu--Xu in \cite{RIN}. Alternative, non-recursive routes towards evaluating these intersection numbers include exact formulas for their generating series by Okounkov \cite{GFINMSC} and Zhou \cite{F}. 
	
	Although the methods above all in principle enable one to explicitly evaluate intersection numbers \eqref{intersection2}, it is generally understood that these quantities are quite intricate. Except in certain exceptional cases, they are not believed to admit simple closed form expressions. 
	
	However, one might still hope that they are accessible in an asymptotic sense, in the large genus limit as $g$ tends to $\infty$. Indeed, it was predicted by Delecroix--Goujard--Zograf--Zorich as Conjecture E.6 of \cite{VFGINMSC} that the correlators \eqref{intersection2} simplify considerably under this limit, namely, that
	\begin{flalign}
	\label{intersection1}
	\Bigg\langle \displaystyle\prod_{i = 1}^n \tau_{d_i} \Bigg\rangle = \displaystyle\frac{(6g + 2n - 5)!!}{24^g g! \prod_{i = 1}^n (2d_i + 1)!!} \big( 1 + o (1) \big), \quad \text{as $g$ tends to $\infty$, uniformly in $\textbf{d}$, if $n \le 2 \log g$}.
	\end{flalign}
	
	The asymptotic \eqref{intersection1} was shown to hold by Liu--Xu in \cite{RA} in the special case when the numbers $n$ and $d_2, d_3, \ldots, d_n$ are all uniformly bounded. Furthermore, the predicted lower bound in \eqref{intersection1} was proven by Delecroix--Goujard--Zograf--Zorich in \cite{LBINC} in the case when $d_1 + d_2 = 3g - o(g)$ (that is, when $\textbf{d}$ is ``mostly concentrated'' on $d_1$ and $d_2$) and $n = o(g)$. Along a similar direction, the large genus limits for a different, but related, family of intersection numbers on $\overline{\mathcal{M}}_{g, n}$ that arise in the context of Weil--Petersson volumes were predicted by Zograf in \cite{LGAV} and later analyzed by Mirzakhani in \cite{GVRHSLG} and Mirzakhani--Zograf in \cite{LGAIMSC}. The proofs of these results were all primarily based on combinatorial analyses of the Virasoro recursions that determine the corresponding intersection numbers.  
	
	In this paper we establish as \Cref{limitd} below the limit \eqref{intersection1} under the slightly weaker constraint that $n = o(g^{1/2})$. As we explain in \Cref{gn3g} (and was observed earlier, in Appendix E of \cite{VFGINMSC}), this condition is essentially optimal. 
	
	To establish this result, we also proceed through a study of the Virasoro constraints. However, in addition to implementing a combinatorial analysis, our framework also involves a probabilistic aspect that does not seem to have been present in the previous works \cite{GVRHSLG,LGAIMSC,RA,LBINC} on large genus asymptotics. In particular, by comparing the coefficients in the Virasoro constraints with the jump probabilities of an asymmetric simple random walk, we show that the correlators \eqref{intersection2} can be bounded above and below by certain functionals associated with this walk; see \Cref{destimatef} and \Cref{destimatefupper} below. The latter can be analyzed directly, leading to the asymptotic \eqref{intersection1}.

	\subsection{Masur--Veech Volumes and Siegel--Veech Constants} 
	
	\label{VolumesConstants}

	After establishing \eqref{intersection1}, we proceed to an application concerning large genus asymptotics for Masur--Veech volumes and Siegel--Veech constants associated with principal strata of moduli spaces of quadratic differentials. So, in this section we recall the definitions of these quantities. 
	
	As previously, fix integers $g, n \ge 0$ with $2g + n \ge 3$. Let $\mathcal{Q}_{g, n}$ denote the moduli space of pairs $(X, q)$, where $X$ is a Riemann surface of genus $g$ and $q$ is a meromorphic quadratic differential on $X$ with $n$ poles, each of which is simple. Under this notation, $q$ has $4g + n - 4$ zeroes (counted with multiplicity) on $X$. The moduli space $\mathcal{Q}_{g, n}$ can be decomposed as a disjoint union of orbifolds called \emph{strata}, prescribing how many distinct zeroes $q$ has, along with their multiplicities. In this paper we will only be concerned with the \emph{principal stratum} $\mathcal{Q} (1^{4g + n - 4}, - 1^n) \subseteq \mathcal{Q}_{g, n}$, consisting of those quadratic differentials $(X, q) \in \mathcal{Q}_{g, n}$  with $4g + n - 4$ distinct simple zeroes and $n$ simple poles. This stratum is both open and dense in $\mathcal{Q}_{g, n}$, and its complement has positive codimension. 
	
	The moduli space $\mathcal{Q}_{g, n}$ and each of its strata admit an $\text{SL}_2 (\mathbb{R})$-action, under which an element $\textbf{A} \in \text{SL}_2 (\mathbb{R})$ acts on a quadratic differential $(X, q) \in \mathcal{Q}_{g, n}$ by composing the local coordinate charts on $X$ with $\textbf{A}$. This action is closely related to billiard flow on rational polygons; dynamics on translation surfaces; the theory of interval exchange maps; enumeration of square-tiled surfaces; and Teichm\"{u}ller geodesic flow. We will not explain these topics further here and instead refer to the surveys of Masur--Tabachnikov \cite{RBF}, Wright \cite{TSOC}, and Zorich \cite{FS} for more information.
	
	In any case, there exists a measure on $\mathcal{Q}_{g, n}$ (or equivalently, on each of its strata) that is invariant with respect to this $\text{SL}_2 (\mathbb{R})$-action. We recall the definition of this measure in the case of the principal stratum here, although it is entirely analogous for the remaining strata. 
	
	For any quadratic differential $(X, q) \in \mathcal{Q} (1^{4g + n - 4}, -1^n)$, there exists a double cover $f = f_{X, q}: \widetilde{X} \rightarrow X$ such that the pullback $f^* q = \omega^2$ is the square of some holomorphic one-form $\omega = \omega_{X, q}$ on $\widetilde{X}$. In this way, $X$ is the quotient of $\widetilde{X}$ by the involution $\sigma = \sigma_{X, q}: \widetilde{X} \rightarrow \widetilde{X}$ that interchanges the two preimages above any regular point of $f$ in $X$. Then since $\sigma$ fixes each zero of $\omega$, the set of which we denote by $Z = Z_{\omega} = \{ z_1, z_2, \ldots , z_m \} \subset \widetilde{X}$, it induces an involution $\sigma_*$ on the relative homology group $H = H_{X, q} = H_1 \big( \widetilde{X}, \{ z_1, z_2, \ldots , z_m \}, \mathbb{Z} \big)$. Let $H^+ = H_{X, q}^+ \subseteq H$ and $H^- = H_{X, q}^- \subseteq H$ denote the subspaces of $H$ that are invariant and anti-invariant with respect to $\sigma_*$, respectively (that is, they are the eigenspaces of $H$ whose eigenvalues with respect to $\sigma_*$ are $1$ and $-1$, respectively). Further let $\gamma_1, \gamma_2, \ldots , \gamma_k$ denote a basis of the anti-invariant subspace $H^- \subseteq H$. 
	
	Define the \emph{period map} $\Phi: \mathcal{Q} (1^{4g + n - 4}, -1^n) \rightarrow \mathbb{C}^k$ by setting $\Phi (X, q) = \big( \int_{\gamma_1} \omega, \int_{\gamma_2} \omega, \ldots , \int_{\gamma_k} \omega \big)$, for any quadratic differential $(X, q) \in \mathcal{Q} (1^{4g + n - 4}, -1^n)$. It can be shown that the map $\Phi$ defines a local coordinate chart, called \emph{period coordinates}, for the stratum $\mathcal{Q} (1^{4g + n - 4}, -1^n)$. Pulling back the Lebesgue measure on $\mathbb{C}^k$ with respect to $\Phi$ yields a measure $\nu$ on $\mathcal{Q} (1^{4g + n - 4}, -1^n)$, which is quickly verified to be independent of the basis $\{ \gamma_i \}$ and invariant under the action of $\text{SL}_2 (\mathbb{R})$.
	
	As stated, the volume $\nu \big( \mathcal{Q} (1^{4g + n - 4}, -1^n) \big)$ will be infinite since $(X, cq) \in \mathcal{Q} (1^{4g + n - 4}, -1^n)$ for any $(X, \omega) \in \mathcal{Q} (1^{4g + n - 4}, -1^n)$ and constant $c \in \mathbb{C}$. To remedy this issue, let $\mathcal{Q}_0 (1^{4g + n - 4}, -1^n) \subset \mathcal{Q} (1^{4g + n - 4}, -1^n)$ denote the moduli space of pairs $(X, q) \in \mathcal{Q} (1^{4g + n - 4}, -1^n)$ such that $\int_X |q| = \frac{1}{2}$; this is the hypersurface of the stratum $\mathcal{Q} (1^{4g + n - 4}, -1^n)$ consisting of those differentials $(X, q)$, for which $q$ has area $\frac{1}{2}$. Let $\nu_0$ denote the measure induced by $\nu$ on $\mathcal{Q}_0 (1^{4g + n - 4}, -1^n)$. 
	
	It was shown independently by Masur \cite{ETMF} and Veech \cite{MTSIEM} that $\nu_0$ is ergodic on $\mathcal{Q}_0 (1^{4g + n - 4}, -1^n)$ under the action of $\text{SL}_2 (\mathbb{R})$, and that the volume $\nu_0 \big( \mathcal{Q}_0 (1^{4g + n - 4}, -1^n) \big)$ is finite. The latter quantity is called the \emph{Masur--Veech volume} of the principal stratum $\mathcal{Q} (1^{4g + n - 4}, -1^n)$. Since the complement of this stratum in $\mathcal{Q}_{g, n}$ has positive codimension, we also write $\Vol \mathcal{Q}_{g, n} = \nu_0 \big( \mathcal{Q}_0 (1^{4g + n - 4}, -1^n) \big)$.
	
	Next we recall the definition of the Siegel--Veech constant that will be relevant to us in this paper. For any $(X, q) \in \mathcal{Q}_{g, n}$, $|q|$ defines metric on $X$ that is flat away from a finite set of conical singularities, also called \emph{saddles}, which constitute the zeros and poles of $q$; this makes $(X, q)$ into a flat surface. A \emph{saddle connection} on $(X, q)$ is a geodesic on $X$ connecting two saddles with no saddle in its interior, and a \emph{maximal cylinder} on $(X, q)$ is a Euclidean cylinder isometrically embedded in $X$ whose two boundaries are both unions of saddle connections. 
	
	The area Siegel--Veech constant concerns the enumeration of maximal cylinders, weighted by area, on a typical flat surface in $\mathcal{Q}_{g, n}$. More specifically, for any real number $L > 0$, let 
	\begin{flalign*}
	\mathcal{N}_{\text{area}} (L) = \mathcal{N}_{\text{area}} \big( L; (X, q) \big) = \displaystyle\sum_{w (C) \le L} A (C),
	\end{flalign*}
	
	\noindent where $C$ ranges over all maximal cylinders of $(X, q)$ of circumfrence at most $L$, and $A(X)$ and $A(C)$ denote the areas of $X$ and $C$, respectively. Viewing maximal cylinders as ``thickenings'' of closed geodesics, one might interpret $\mathcal{N}_{\text{area}} (L)$ as a count for closed geodesics weighted by ``thickness.''
	
	It was shown by Eskin--Masur in \cite{AFS} that, for a typical (by which we mean full measure subset with respect to the Masur--Veech volume) flat surface $(X, q) \in \mathcal{Q}_{g, n}$, the quantity $\mathcal{N}_{\text{area}} (L)$ grows quadratically in $L$ with asymptotics
	\begin{flalign}
	\label{c1} 
	c_{\text{area}} (\mathcal{Q}_{g, n}) = \displaystyle\lim_{L \rightarrow \infty} \displaystyle\frac{\mathcal{N}_{\text{area}} \big( L; (X, q) \big)}{\pi L^2},
	\end{flalign}
	
	\noindent where the constant $c_{\text{area}} (\mathcal{Q}_{g, n})$ is independent of the choice typical flat surface $(X, q) \in \mathcal{Q}_{g, n}$. This constant falls into a class of quantities known as \emph{Siegel--Veech constants}, which had been studied in the earlier work \cite{M} of Veech; $c_{\text{area}}$ is specifically known as an \emph{area Siegel--Veech constant}. 
	
	In addition to enumerating geometric phenomena, area Siegel--Veech constants also contain information about dynamics on the moduli space $\mathcal{Q}_{g, n}$. One example is through the \emph{Teichm\"{u}ller geodesic flow} on $\mathcal{Q}_{g, n}$, which is the action of the diagonal one-parameter subgroup $\left[ \begin{smallmatrix} e^t & 0 \\ 0 & e^{-t} \end{smallmatrix} \right] \subset \text{SL}_2 (\mathbb{R})$ on this moduli space. This flow lifts to the Hodge bundle over $\mathcal{Q} (1^{4g + n - 4}, -1^n)$ and, by Oseledets theorem, one can associate this flow with $2g$ \emph{Lyapunov exponents}, denoted by $\lambda_1 (\mathcal{Q}_{g, n}) \ge \lambda_2 (\mathcal{Q}_{g, n}) \ge \cdots \ge \lambda_{2g} (\mathcal{Q}_{g, n})$. These exponents are symmetric with respect to $0$, that is, $\lambda_i + \lambda_{2g - i + 1} = 0$ for each integer $i \in [1, 2g]$. 
	
	It was shown as part (a) of Theorem 2 in the work \cite{ERG} of Eskin--Kontsevich--Zorich that the sum of the first $g$ (namely, the nonnegative) Lyapunov exponents associated with the principal stratum can be expressed explicitly in terms of the area Siegel--Veech constant $c_{\text{area}} (\mathcal{Q}_{g, n})$ through  
	\begin{flalign}
	\label{lambdasum} 
	\displaystyle\sum_{i = 1}^g \lambda_i (\mathcal{Q}_{g, n}) = \displaystyle\frac{1}{24} \left( \displaystyle\frac{20g}{3} - \displaystyle\frac{4n}{3} - \displaystyle\frac{20}{3} \right) + \displaystyle\frac{\pi^2}{3} c_{\text{area}} (\mathcal{Q}_{g, n}).
	\end{flalign} 

	\noindent Thus, knowledge of the area Siegel--Veech constant $c_{\text{area}} (\mathcal{Q}_{g, n})$ enables one to evaluate the sum of the associated positive Lyapunov exponents of the Teichm\"{u}ller geodesic flow. An analog of this result holds for all connected components of any stratum of $\mathcal{Q}_{g, n}$ \cite{ERG}.

	\subsection{Volume Asymptotics} 
	
	\label{Limit1} 
	
	Although the finiteness of the Masur--Veech volumes was established in 1982 \cite{ETMF,MTSIEM}, it was nearly two decades until mathematicians produced general ways of determining them explicitly. In the apparently simpler setting of moduli spaces of holomorphic differentials, one of the earlier exact volume evaluations was due to Zorich \cite{SVMS}, who found them for certain strata of low genus. Later, Eskin--Okounkov \cite{ANBCTV} and Eskin--Okounkov--Pandharipande \cite{TCBC} proposed a way of evaluating the volume of any (connected component of a) stratum in the moduli space of holomorphic differentials, through an algorithm based on asymptotic Hurwitz theory. The more recent work of Chen--M\"{o}ller--Sauvaget--Zagier \cite{VSCC} provides an alternative way to access these strata volumes through a recursion.
	
	Similarly, although the work of Eskin--Masur \cite{AFS} showed that the limits \eqref{c1} defining the area Siegel--Veech constants exist, it did not indicate how to determine them. Again in the case of holomorphic differentials, this was done by Eskin--Masur--Zorich \cite{PBC} (combined with a result of Vorobets \cite{PGTS}), who expressed these constants as combinatorial sums involving the Masur--Veech strata volumes. 
	
	As with the intersection numbers \eqref{intersection2}, it is widely believed that these volumes and constants are quite intricate, and that they do not in general admit simple closed form expressions. Still, based on numerical data tabulated from implementing the above results on explicit strata, Eskin--Zorich \cite{VSDCLG} posed precise predictions for how the Masur--Veech volumes and Siegel--Veech constants associated with arbitrary strata of holomorphic differentials should simplify in the large genus limit. These were first established for the principal and minimal strata by Chen--M\"{o}ller--Zagier \cite{QLGL} and Sauvaget \cite{VSI}, respectively. They were then confirmed in general in \cite{LGAVSD,LGAC} through a combinatorial analysis of the Eskin--Okounkov algorithm \cite{ANBCTV} and of the expressions of Eskin--Masur--Zorich \cite{PBC}. Soon later, an independent and very different, algebro-geometric proof of these predictions was posed by Chen--M\"{o}ller--Sauvaget--Zagier in \cite{VSCC}. More recently, through both combinatorial and algebro-geometric methods, Sauvaget \cite{LGAEV} provided an all-order genus expansion of the Masur--Veech volumes of any stratum of holomorphic differentials. 
	
	Our understanding of these limiting phenomena in the context of quadratic differentials, which will be of interest to us in this paper, is more limited. In this setting, Goujard \cite{CSMSQD} provided expressions for Siegel--Veech constants in terms of strata volumes of quadratic diferentials. Earlier work by Eskin--Okounkov \cite{QF} also proposed an algorithm (again based on asymptotic Hurwitz theory) that finds the volume of any such stratum, but this algorithm is considerably more intricate than its counterpart for holomorphic differentials. Still, it was effectively implemented by Goujard \cite{VSMSQDV} to obtain numerical data for volumes of strata of dimension at most $11$.
	
	For the principal stratum, other (and sometimes more efficient) methods of volume evaluation exist. These include through expressions in terms of correlators by Mirzakhani \cite{ETE,VITMSC}; lattice point enumeration of Strebel--Jenkins differentials by Athreya--Eskin--Zorich \cite{CD} and Delecroix--Goujard--Zograf--Zorich \cite{VFGINMSC}; topological recursions of Andersen--Borot--Charbonnier--Delecroix--Giacchetto--Lewa\'{n}ski--Wheeler \cite{TRV}; and intersection numbers on compactified Hodge bundles by Chen--M\"{o}ller--Sauvaget \cite{VITPSQD}. The latter framework through intersection numbers was reformulated as an explicit, bivariate recursion for these principal strata volumes by Kazarian \cite{RV}. 
	
	Based on empirical data obtained by implementing these results, as well as more elaborate geometrical, analytical, and dynamical considerations, precise predictions for the large genus asymptotic behavior of the Masur--Veech volume and Siegel--Veech constant associated with any stratum of quadratic differentials were proposed recently in \cite{LGAVCSQD}. Until the present paper, these predictions had not been established for any stratum. Let us state them in the case of the principal stratum. 
	
	For the volumes, the specialization of Conjecture 1 of \cite{LGAVCSQD} to the principal stratum (see also Conjecture 1.10 of \cite{VFGINMSC} for the $n = 0$ case) states  
	\begin{flalign}
	\label{qgnlimit}
	\Vol \mathcal{Q}_{g, n} = \pi^{-1} 2^{n + 2} \left( \displaystyle\frac{8}{3} \right)^{4g + n - 4} \big(1 + o(1) \big), \quad \text{as $g$ tends to $\infty$, if $n \le \log g$}.
	\end{flalign} 
	
	\noindent For the Siegel--Veech constant, the specialization of Conjecture 2 of \cite{LGAVCSQD} to the principal stratum states that 
	\begin{flalign}
	\label{constant1} 
	c_{\text{area}} (\mathcal{Q}_{g, n}) = \displaystyle\frac{1}{4} + o (1), \quad \text{as $g$ tends to $\infty$, if $n \le \log g$}. 
	\end{flalign}
	
	As an application of \eqref{intersection1}, we establish \eqref{qgnlimit} for $20n \le \log g$ as \Cref{limitvolume} below and \eqref{constant1} for $n$ fixed as \Cref{constantlimit} below; with further effort, these constraints on $n$ can likely be improved through similar methods to allow $n = g^c$ for some constant $c > 0$, but we decided not to pursue this here. This essentially confirms the predictions of \cite{LGAVCSQD} in the case of the principal stratum. 
	
	To establish the volume asymptotic \eqref{qgnlimit}, we begin with an expression of Delecroix--Goujard--Zograf--Zorich \cite{VFGINMSC} for $\Vol \mathcal{Q}_{g, n}$ as a sum involving the correlators \eqref{intersection2}. Although \eqref{intersection1} enables one to approximate many of these intersection numbers explicitly, the sum remains quite intricate; it is indexed by stable graphs of genus $g$ with $n$ marked points, the number of which grows exponentially in $g$. It was predicted in \cite{VFGINMSC} that the dominant contribution to this sum arises from graphs with one vertex. A similar phenomenon was observed in \cite{LGAC,LGAVSD} for Masur--Veech volumes and Siegel--Veech constants in strata of holomorphic differentials, where analogously large sums were dominated by a single term. 
	
	However, the situation here appears to be considerably more elaborate than in the holomorphic setting. Indeed, even the leading order contribution to $\Vol \mathcal{Q}_{g, n}$ is not quite immediate to evaluate, as it involves an infinite sum of special functions (namely, particular deformations of multi-variate harmonic sums given by \Cref{hkzk} below). A detailed, but heuristic, analysis of this sum was performed in Appendix D of \cite{VFGINMSC}, leading to an exact prediction for its value (and to the constant prefactor on the right side of \eqref{qgnlimit}). Thus, we must first establish this prediction, which we do in \Cref{Sumhkzk} and \Cref{Asymptotichknzkn} (see \Cref{sumkn2} and \Cref{anasymptotic} below) using the complex analytic saddle point method on certain generating series. This yields the leading order contribution to $\Vol \mathcal{Q}_{g, n}$ coming from stable graphs with one vertex. 
	
	We then show that graphs on two or more vertices do not asymptotically contribute to this volume, which is partly facilitated by the combinatorial analysis implemented in \cite{LGAC,LGAVSD} to establish similar phenomena in the holomorphic setting; this leads to the proof of \eqref{qgnlimit}. Given this, \eqref{constant1} is deduced using an identity of Goujard \cite{CSMSQD} for $c_{\text{area}} (\mathcal{Q}_{g, n})$ in terms of principal strata volumes. 
	
	Let us conclude this section by mentioning that volume asymptotics of moduli spaces can often be used to deduce quantitative geometric properties for random surfaces of high genus. For instance, in the context of random hyperbolic surfaces, large genus asymptotics for Weil--Petersson volumes were used by Mirzakhani \cite{GVRHSLG} to estimate Cheeger constants, diameters, and systole lengths; by Mirzakhani--Petri \cite{LCGRSLG} to establish Poisson limiting results for extrema of the length spectrum; and by Gilmore--Le Masson--Sahlsten--Thomas \cite{SGELGS} and Thomas \cite{DELGRS} to prove delocalization results for Laplacian eigenfunctions. Similarly, large genus asymptotics for Masur--Veech strata volumes of holomorphic differentials were used by Masur--Rafi--Randecker \cite{TSGTMS} to bound the covering radius of a typical translation surface. 
	
	In our context, each summand in the expression of $\Vol \mathcal{Q}_{g, n}$ as a weighted sum over stable graphs can be interpreted geometrically, as a multiple of the probability of a random flat surface exhibiting a certain cylinder decomposition or alternatively as a count of closed geodesics on a typical hyperbolic surface with given multi-curve type (see the works of Delecroix--Goujard--Zograf--Zorich \cite{VFGINMSC} and Arana-Herrera \cite{SSRIF} for independent and different proofs of this equivalence). So, the forthcoming work of Delecroix--Goujard--Zograf--Zorich \cite{AGSSSMLG} will combine our results with detailed geometric and analytic considerations to establish precise probabilistic limit theorems for multi-curve statistics in random flat and hyperbolic surfaces. These results were presented in Section 1.10 and Appendix F as conditional statements assuming the results of the present paper.

	\subsection{Results}
	
	\label{Asymptotic}
	
	We begin by stating our results on the intersection numbers \eqref{intersection2}. For any integer $n \ge 1$ and $n$-tuple $\textbf{d} = (d_1, d_2, \ldots , d_n) \in \mathbb{Z}_{\ge 0}^n$ of nonnegative integers, recall that we set $|\textbf{d}| = \sum_{i = 1}^n d_i$. The following definition provides a normalization for the correlators \eqref{intersection2} according to \eqref{intersection1}. 
	
	\begin{definition} 
	
	\label{dpsi} 
	
	Let $g \ge 0$ and $n \ge 1$ denote integers. For any $\textbf{d} = (d_1, d_2, \ldots , d_n) \in \mathbb{Z}_{\ge 0}^n$ such that $|\textbf{d}| = 3g + n - 3$, define the \emph{normalized intersection number} $\langle \textbf{d} \rangle = \langle \textbf{d} \rangle_{g, n} = \langle d_1, d_2, \ldots , d_n \rangle = \langle d_1, d_2, \ldots , d_n \rangle_{g, n}$ by
	\begin{flalign}
	\label{definitiond}
	\langle \textbf{d} \rangle = \displaystyle\frac{24^g g! \prod_{i = 1}^n (2d_i + 1)!!}{\big( 2 |\textbf{d}| + 1 \big)!!} \displaystyle\int_{\overline{\mathcal{M}}_{g, n}} \displaystyle\prod_{i = 1}^n \psi_i^{d_i}.
	\end{flalign}
	
	\end{definition} 

	Before providing an asymptotic result for these normalized intersection numbers, we first state the following exponential bound on $\langle \textbf{d} \rangle_{g, n}$ that holds uniformly in the genus $g$; in particular, it also holds for small values of $g$. This bound will be established in \Cref{Exponential} below.
	
	\begin{prop}
		
		\label{destimateexponential}

		Let $n \in \mathbb{Z}_{\ge 1}$ and $\textbf{\emph{d}} \in \mathbb{Z}_{\ge 0}^n$ satisfy $|\textbf{\emph{d}}| = 3g + n - 3$, for some $g \in \mathbb{Z}_{\ge 0}$. Then, 
		\begin{flalign} 
		\label{destimaten} 
		\langle \textbf{\emph{d}} \rangle_{g, n} \le \left( \frac{3}{2} \right)^{n - 1}.
		\end{flalign}
		
	\end{prop}
	
	\begin{rem} 
		
		Up to a factor of $n^{-1/2}$ (which we have not attempted to optimize), the exponential estimate \eqref{destimaten} is sharp. Indeed, letting $(1^{n - 3}, 0^3)$ denote the $n$-tuple consisting of $n - 3$ parts equal to one and $3$ parts equal to zero, it can quickly be seen as a consequence of \eqref{d1recursion}, \eqref{aab}, and \eqref{limitk} below that 
		\begin{flalign*}
		\langle 1^{n - 3}, 0^3 \rangle_{0, n} = 3^{n - 3} \displaystyle\frac{(n - 3)!}{(2n - 5)!!} = \displaystyle\frac{2 \pi^{1/2}}{9 n^{1/2}} \left( \displaystyle\frac{3}{2} \right)^{n - 1} \big( 1 + o (1) \big), \quad \text{as $n$ tends to $\infty$}. 
		\end{flalign*}
		
	\end{rem}

	Although \Cref{destimateexponential} will follow from the Virasoro constraints by a reasonably direct induction, it has to the best of our knowledge not appeared before in the literature. In addition to being useful in our proof of the asymptotic \eqref{intersection1}, \Cref{destimateexponential} provides a general estimate on intersection numbers in situations where \eqref{intersection1} is no longer valid. The latter point will be beneficial in our application to the asymptotic analysis of $\Vol \mathcal{Q}_{g, n}$ Indeed, although the dominant contribution to this volume comes from intersection numbers satisfying \eqref{intersection1}, the exact expression for it also involves correlators outside this regime. The uniform bound given by \Cref{destimateexponential} is helpful in showing that such terms are indeed negligible in the large genus limit.
	
	Next, we state an asymptotic result for the intersection numbers \eqref{intersection2}, which is essentially \eqref{destimateexponential} under the restriction $n = o(g^{1/2})$. The latter is made more precise through the following definition.  

	\begin{definition}
		
		\label{deltag}
		
		For any real number $\varepsilon > 0$ and integer $g > \varepsilon^{-2}$, define 
		\begin{flalign*} 
		\Delta (g; \varepsilon) = \big\{ \textbf{d} = (d_1, d_2, \ldots , d_n) \in \mathbb{Z}_{\ge 0}^n: |\textbf{d}| = 3g + n - 3, n < \varepsilon g^{1/2} \big\}.
		\end{flalign*}  
		
	\end{definition}

	The following theorem, which will be established in \Cref{LimitIntersection} below, implies \eqref{intersection1} and essentially states $\langle \textbf{d} \rangle_{g, n} \approx 1$, as $g$ tends to $\infty$, if $n = o (g^{1/2})$.

	\begin{thm}
		
	\label{limitd} 
	
	We have that 
	\begin{flalign*}
	\displaystyle\lim_{\varepsilon \rightarrow 0} \left(\displaystyle\lim_{g \rightarrow \infty} \displaystyle\max_{\textbf{\emph{d}} \in \Delta (g; \varepsilon)} \big| \langle \textbf{\emph{d}} \rangle - 1\big| \right) = 0.
	\end{flalign*} 
	\end{thm}

	\begin{rem}
		
		\label{gn3g} 
		
		The approximation $\langle \textbf{d} \rangle_{g, n} \approx 1$ is no longer true for large $g$ if $n = \mathcal{O} (g^{1/2})$. Indeed, letting $(3g - 2, 1^{n - 1})$ denote the $n$-tuple consisting of one part equal to $3g - 2$ and $n - 1$ parts equal to $1$, it can be quickly seen as a consequence of \eqref{d1recursion} and \Cref{n1} below that 
		\begin{flalign*}
		\langle 3g - 2, 1^{n - 1} \rangle_{g, n} = \exp \left( \displaystyle\frac{n^2}{12g} \right) \big( 1 + o (1) \big), \qquad \text{if $n = \mathcal{O} (g^{1/2})$ and $g$ tends to $\infty$.}
		\end{flalign*}
		
	\end{rem}

	We next describe our results concerning the large genus limits for Masur--Veech volumes and Siegel--Veech constants in the principal stratum of quadratic differentials. The following theorem, which will be established in \Cref{Volume} below, provides the volumes asymptotics. In particular, it establishes \eqref{qgnlimit} under the slightly stronger hypothesis that $20n \le \log g$; it is likely that one can further improve this dependence between $n$ and $g$, but we will not pursue this here.

	\begin{thm} 
		
		\label{limitvolume} 
		
		We have that  
		\begin{flalign*} 
		\displaystyle\lim_{g \rightarrow \infty} \Bigg( \displaystyle\max_{20n \le \log g} \bigg| 2^{-n}  \Big( \displaystyle\frac{8}{3} \Big)^{4 - 4g - n} \Vol \mathcal{Q}_{g, n} - \displaystyle\frac{4}{\pi} \bigg| \Bigg) = 0.
		\end{flalign*} 
	\end{thm} 
	
	The next theorem, which will be established in \Cref{ProofConstant} below, verifies the large genus asymptotic \eqref{constant1} for the Siegel--Veech constant $c_{\text{area}} (\mathcal{Q}_{g, n})$, assuming that $n$ is fixed as $g$ tends to $\infty$. As before, it is likely that one can strengthen this statement using similar methods to allow $n$ to grow slowly with $g$, but we will not address this improvement here. 
	
	\begin{thm}
		
		\label{constantlimit}
		
		For any fixed integer $n \ge 0$, we have that 
		\begin{flalign*} 
		\displaystyle\lim_{g \rightarrow \infty} c_{\text{\emph{area}}} (\mathcal{Q}_{g, n}) = \displaystyle\frac{1}{4}.  
		\end{flalign*} 
		
	\end{thm} 

	We conclude with the following corollary that approximates for the sum of Lyapunov exponents of the Teichm\"{u}ller geodesic flow on $\mathcal{Q}_{g, n}$; it follows as a consequence of \Cref{constantlimit} and \eqref{lambdasum} (see also equation (3) of \cite{LGAVCSQD}). 
	
	\begin{cor} 
		
		\label{exponentsum} 
		
		For any fixed integer $n \ge 0$, we have that  
		\begin{flalign*}
		\displaystyle\lim_{g \rightarrow \infty} \left( \displaystyle\sum_{i = 1}^g \lambda_i (\mathcal{Q}_{g, n}) - \displaystyle\frac{5g}{18} \right) = \displaystyle\frac{\pi^2}{12} - \displaystyle\frac{n}{18} - \displaystyle\frac{5}{18}. 
		\end{flalign*}
		
	\end{cor}

	The remainder of this paper is organized as follows. After recalling several preliminary results in \Cref{Estimates1}, we establish general bound on intersection numbers given by \Cref{destimateexponential} in \Cref{Exponential}. Next, we establish the limiting result \Cref{limitd} for these intersection numbers in \Cref{Coefficients} and \Cref{Upperd}. In \Cref{Sumhkzk} and \Cref{Asymptotichknzkn}, we introduce and provide asymptotics for certain types of multi-variate harmonic sums that appear in the Masur--Veech volumes of principal strata. We then analyze the large genus asymptotics for these volumes, proving \Cref{limitvolume} in \Cref{VolumePrincipal}, \Cref{Estimatev2}, and \Cref{Estimateve1}. We conclude with the proof of the Siegel--Veech asymptotic result given by \Cref{constantlimit} in \Cref{AsymptoticConstant}.

	\subsection*{Acknowledgments}
	
	The author heartily thanks Anton Zorich for numerous enlightening discussions, valuable encouragements, and helpful comments on earlier versions of this manuscript. The author would also like to express his gratitude toward Vincent Delecroix, \'{E}lise Goujard, Peter Zograf, and Anton Zorich for kindly sharing early versions of their papers \cite{AGSSSMLG,LBINC}.  The author further thanks Alexei Borodin and Kasra Rafi for useful conversations. This work was partially supported by NSF grant NSF DMS-1664619, the NSF Graduate Research Fellowship under grant DGE-1144152, and a Harvard Merit/Graduate Society Term-time Research Fellowship.

	\section{Miscellaneous Preliminaries} 
	
	\label{Estimates1}
	
	In this section we state several (mostly known) results that will be used throughout this paper. We begin in \Cref{Estimate} with several combinatorial estimates, and next we will recall certain recursions and bounds on intersection numbers in \Cref{RecursionIntersection}. Then, we recall from \cite{VFGINMSC} an expression for $\Vol \mathcal{Q}_{g, n}$ through a weighted sum over stable graphs.  

	\subsection{Estimates}	
	
	\label{Estimate} 
	
	In this section we recall some notation, identities, and estimates that will be useful later in this paper. Throughout this article, for any two functions $F_1, F_2: \mathbb{Z} \rightarrow \mathbb{R}$ such that $F_2 (k)$ is nonzero for sufficiently large $k$, we write $F_1 \sim F_2$ if $\lim_{k \rightarrow \infty} F_1 (k) F_2 (k)^{-1} = 1$.
	
	First recall for any integers $A, B \ge 0$ that 
	\begin{flalign}
	\label{aab} 
	\begin{aligned} 
	& (2A + 1)!! = \displaystyle\frac{(2A + 1)!}{2^A A!}; \qquad \displaystyle\sum_{k = 0}^A \binom{2A + 2}{2k + 1} = 2^{2A + 1}; \\
	& \displaystyle\frac{(2A + 2B + 1)!}{(2A + 1)! (2B + 1)!} = \displaystyle\frac{1}{2 (A + B + 1)} \binom{2A + 2B + 2}{2A + 1}.
	\end{aligned} 
	\end{flalign}
	
	Next, for any integers $m \ge 1$ and $N \ge 0$, let $\mathcal{C}_m (N)$ denote the set of \emph{compositions of $N$ of length $k$}, that is, the set of $m$-tuples $\textbf{a} = (a_1, a_2, \ldots , a_m) \in \mathbb{Z}_{\ge 1}^m$ of positive integers such that $|\textbf{a}| = N$. Further let $\mathcal{K}_m (N)$ denote the set of \emph{nonnegative compositions of $N$ of length $k$}, that is, the set of $m$-tuples $\textbf{a} = (a_1, a_2, \ldots , a_m) \in \mathbb{Z}_{\ge 0}^m$ of nonnegative integers such that $|\textbf{a}| = N$. Observe in particular that 
	\begin{flalign}
	\label{cmnkmn} 
	\big| \mathcal{C}_m (N) \big| = \binom{N - 1}{m - 1}; \qquad \big| \mathcal{K}_m (N) \big| = \binom{N + m - 1}{m - 1}.
	\end{flalign}

	The following estimate on products of factorials will be used in \Cref{EstimateLambdagvst} below. 
	
	\begin{lem}
		
		\label{sumaibici2}
		
		Fix integers $m \ge 1$ and $A, B, C \ge 0$. Then,
		\begin{flalign}
		\label{sumaibiciestimateabc}
		\displaystyle\sum \displaystyle\prod_{i = 1}^m (a_i + b_i + c_i - 2)! \le 2^{12m + 9} (A + B + C - 3m + 1)!, 
		\end{flalign}
		
		\noindent where on the left side of \eqref{sumaibiciestimateabc}, we sum over all triples $(\textbf{\emph{a}}, \textbf{\emph{b}}, \textbf{\emph{c}})$ of nonnegative compositions $\textbf{\emph{a}} = (a_1, a_2, \ldots,  a_m) \in \mathcal{K}_m (A)$, $\textbf{\emph{b}} = (b_1, b_2, \ldots,  b_m) \in \mathcal{K}_m (B)$, and $\textbf{\emph{c}} = (c_1, c_2, \ldots,  c_m) \in \mathcal{K}_m (C)$, such that $a_i + b_i + c_i \ge 3$, for each $i \in [1, m]$.
		
	\end{lem}

	Observe here that the factor $(A + B + C - 3m + 1)!$ appearing on the right side of \eqref{sumaibiciestimateabc} occurs if $a_1 + b_1 + c_1 = A + B + C - 3m + 3$ and $a_i + b_i + c_i = 3$, for each $i \in [2, m]$. Thus, \Cref{sumaibici2} indicates that the sum over all compositions of $\prod_{i = 1}^m (a_i + b_i + c_i - 2)!$ bounded by this single summand, up to a factor $2^{12m + 9}$. 
	
	To establish \Cref{sumaibici2}, we will use the following lemma from \cite{LGAC} that estimates a similar quantity as on the let side of \eqref{sumaibiciestimateabc}, but without the constraint that each $a_i + b_i + c_i \ge 3$. 
	
	\begin{lem}[{\cite[Lemma 2.7]{LGAC}}]
		
		\label{sumaibici}
		
		Fix integers $m \ge 1$ and $A, B, C \ge 0$. Then,
		\begin{flalign*}
		\displaystyle\sum_{\textbf{\emph{a}} \in \mathcal{K}_m (A)} \displaystyle\sum_{\textbf{\emph{a}} \in \mathcal{K}_m (A)} \displaystyle\sum_{\textbf{\emph{a}} \in \mathcal{K}_m (A)} \displaystyle\prod_{i = 1}^m (a_i + b_i + c_i + 1)! \le 2^{8m + 9} (A + B + C + 1)!, 
		\end{flalign*}
		
		\noindent where we have denoted $\textbf{\emph{a}} = (a_1, a_2, \ldots,  a_m)$, $\textbf{\emph{b}} = (b_1, b_2, \ldots,  b_m)$, and $\textbf{\emph{c}} = (c_1, c_2, \ldots,  c_m)$.
	\end{lem}

	Given \Cref{sumaibici}, we can now establish \Cref{sumaibici2}. 

	\begin{proof}[Proof of \Cref{sumaibici2}]
		
		For any integer $n \ge 1$ and sequences $\textbf{y} = (y_1, y_2, \ldots , y_n) = \mathbb{Z}^n$ and $\textbf{z} = (z_1, z_2, \ldots , z_n) \in \mathbb{Z}^n$, we write $\textbf{y} \ge \textbf{z}$ if $y_i \ge z_i$ for each $i \in [1, n]$. Then, since each $a_i + b_i + c_i \ge 3$ on the left side of \eqref{sumaibiciestimateabc}, there must exist for each $i \in [1, m]$ some nonnegative composition $\textbf{x} (i) \in \mathcal{K}_3 (3)$ such that $(a_i, b_i, c_i) \ge \textbf{x} (i)$. For any $m$-tuple of nonnegative compositions $\textbf{X} = \big( \textbf{x} (1), \textbf{x} (2) \ldots , \textbf{x} (m) \big) \in \mathcal{K}_3 (3)^m$, let $\mathfrak{Y} (\textbf{X}) = \mathfrak{Y}_{A, B, C} (\textbf{X}) \subseteq  \mathcal{K}_m (A) \times \mathcal{K}_m (B) \times \mathcal{K}_m (C)$ denote the subset of triples $(\textbf{a}, \textbf{b}, \textbf{c}) \in \mathcal{K}_m (A) \times \mathcal{K}_m (B) \times \mathcal{K}_m (C)$ such that $(a_i, b_i, c_i) \ge \textbf{x} (i)$, for each $i \in [1, m]$. 
		
		Under this notation, the left side of \eqref{sumaibiciestimateabc} is bounded above by 
		\begin{flalign}
		\label{xabc} 
		\displaystyle\sum_{\textbf{X} \in \mathcal{K}_3 (3)^m } \displaystyle\sum_{(\textbf{a}, \textbf{b}, \textbf{c}) \in \mathfrak{Y} (\textbf{X})} \displaystyle\prod_{i = 1}^m (a_i + b_i + c_i - 2)! \le 2^{4m} \displaystyle\max_{\textbf{X} \in \mathcal{K}_3 (3)^m } \displaystyle\sum_{(\textbf{a}, \textbf{b}, \textbf{c}) \in \mathfrak{Y} (\textbf{X})} \displaystyle\prod_{i = 1}^m (a_i + b_i + c_i - 2)!,
		\end{flalign} 
		
		\noindent where in the last inequality we used the fact that $\big| \mathcal{K}_3 (3) \big| = \binom{5}{2} = 10 < 2^4$. 
		
		So, fix some $\textbf{X} = \big( \textbf{x} (1), \textbf{x} (2) \ldots , \textbf{x} (m) \big) \in \mathcal{K}_3 (3)^m$; denote $\textbf{x} (j) = \big( x_1 (j), x_2 (j), x_3 (j) \big)$ for each integer $j \in [1, m]$; and set $X_i = \sum_{j = 1}^m x_i (j)$, for each $i \in \{ 1, 2, 3 \}$. Then, defining
		\begin{flalign*}
		a_i' = a_i - x_1 (i); \qquad b_i' = b_i - x_2 (i); \qquad c_i' = c_i - x_3 (i),
		\end{flalign*} 
		
		\noindent and the nonnegative compositions $\textbf{a}' = (a_1', a_2', \ldots , a_m') \in \mathcal{K}_m (A - X_1)$, $\textbf{b}' = (b_1', b_2', \ldots , b_m') \in \mathcal{K}_m (B - X_2)$, and $\textbf{c}' = (c_1', c_2', \ldots , c_m') \in \mathcal{K}_m (C - X_3)$, we obtain 
		\begin{flalign}
		\label{sumabcyx}
		& \displaystyle\sum_{(\textbf{a}, \textbf{b}, \textbf{c}) \in \mathfrak{Y} (\textbf{X})} \displaystyle\prod_{i = 1}^m (a_i + b_i + c_i - 2)! = \quad \displaystyle\sum_{\textbf{a}' \in \mathcal{K}_m (A - X_1)} \displaystyle\sum_{\textbf{b}' \in \mathcal{K}_m (B - X_2)} \displaystyle\sum_{\textbf{c}' \in \mathcal{K}_m (C - X_3)} \displaystyle\prod_{i = 1}^m (a_i' + b_i' + c_i' + 1)!.
		\end{flalign}
		
		\noindent By \Cref{sumaibici} and the fact that $X_1 + X_2 + X_3 = 3m$, we deduce that
		\begin{flalign*}
		\displaystyle\sum_{\textbf{a}' \in \mathcal{K}_m (A - X_1)} \displaystyle\sum_{\textbf{b}' \in \mathcal{K}_m (B - X_2)} \displaystyle\sum_{\textbf{c}' \in \mathcal{K}_m (C - X_3)} \displaystyle\prod_{i = 1}^m (a_i' + b_i' + c_i' + 1)! \le 2^{8m + 9} (A + B + C - 3m + 1)!,
		\end{flalign*}
		
		\noindent which together with \eqref{xabc} and \eqref{sumabcyx} implies the lemma.
	\end{proof}

	We will also require the following estimate on products of multinomial coefficients. This result is known, but we will provide its proof for completeness. 
	
	\begin{lem}
		
		\label{sumaijaiestimate}
		
		Let $n$ and $r$ be positive integers; also let $\{ A_i \}$ and $\{ A_{i, j} \}$, for $1 \le i \le n$ and $1 \le j \le r$, be sets of nonnegative integers such that $\sum_{j = 1}^r A_{i, j} = A_i$ for each $i$. Then,
		\begin{flalign}
		\label{aijai}
		\displaystyle\prod_{i = 1}^n \binom{A_i}{A_{i, 1}, A_{i, 2}, \ldots , A_{i, r}} \le \binom{\sum_{j = 1}^n A_i}{\sum_{i = 1}^n A_{i, 1}, \sum_{i = 1}^n A_{i, 2}, \ldots , \sum_{i = 1}^n A_{i, r}}.
		\end{flalign}
		
	\end{lem}

	\begin{proof}
		
		For each $0 \le k \le n$, let $T_k = \sum_{j = 1}^k A_j$ (with $T_0 = 0$). Define the sets $\mathcal{S} = \{ 1, 2, \ldots , T_n \}$ and $\mathcal{S}_i = \big\{ T_{i - 1} + 1, T_{i - 1}  + 2, \ldots , T_i \big\}$ for each $1 \le i \le n$. Then, the left side of \eqref{aijai} counts the number of ways to partition each of the $\mathcal{S}_i$ into $r$ mutually disjoint subsets $\mathcal{S}_{i, 1},  \mathcal{S}_{i, 2}, \ldots \mathcal{S}_{i, r}$ consisting of $A_{i, 1}, A_{i, 2}, \ldots , A_{i, r}$ elements, respectively. Similarly, the right side of \eqref{aijai} counts the number of ways to partition $\mathcal{S}$ into $r$ mutually disjoint subsets $\mathcal{S}^{(1)}, \mathcal{S}^{(2)}, \ldots , \mathcal{S}^{(r)}$ consisting of $\sum_{i = 1}^n A_{i, 1}, \sum_{j = 1}^n A_{i, 2}, \ldots,  \sum_{j = 1}^n A_{i, r}$ elements, respectively.
		
		Any partition of the former type injectively gives rise to a partition of the latter type by setting $\mathcal{S}^{(k)} = \bigcup_{i = 1}^n \mathcal{S}_{i, k}$ for each $1 \le k \le r$. Thus, the left side of \eqref{aijai} is at most equal to the right side of \eqref{aijai}. 
	\end{proof}

	Next we state and establish estimates on exponentials and factorials that will later be useful to us. First, we have that the limit
	\begin{flalign}
	\label{limitk}
	k! \sim (2 \pi k)^{1/2} \left( \displaystyle\frac{k}{e} \right)^k, \qquad \text{as $k$ tends to $\infty$},
	\end{flalign}
	
	\noindent and also the finite $k \ge 1$ bound 
	\begin{flalign} 
	\label{kestimate1}
	2 e^{-k} k^{k + 1/2} \le k! \le 3 e^{-k} k^{k + 1/2}.
	\end{flalign}
	
	We also have the following lemma, which estimates the error in a truncation for the Taylor expansion of $e^R$.

	\begin{lem} 
		
		\label{exponentialm}
		
		Let $R, \delta > 0$ denote two real numbers, and let $K \in \mathbb{Z}$ denote an integer such that $K > (1 + 2 \delta) R$. Then, 
		\begin{flalign*}
		\Bigg| e^R - \displaystyle\sum_{j = 0}^K \displaystyle\frac{R^j}{j!} \Bigg| < \delta^{-1} (1 + \delta)^{- \delta R} e^R. 
		\end{flalign*}
	\end{lem}
	
	\begin{proof}

		By a Taylor expansion and the fact that $j! \ge \big( \frac{j}{e} \big)^j$ (by \eqref{kestimate1}), we have
		\begin{flalign}
		\label{exponentialestimater}
		\Bigg| e^R - \displaystyle\sum_{j = 0}^K \displaystyle\frac{R^j}{j!} \Bigg| = \displaystyle\sum_{j = K + 1}^{\infty} \displaystyle\frac{R^j}{j!} \le \displaystyle\sum_{j = K + 1}^{\infty} \left( \displaystyle\frac{eR}{j} \right)^j.
		\end{flalign} 
		
		\noindent Next, observe that $\big( \frac{eR}{j + 1} \big)^{j + 1} \le (1 + \delta)^{-1} \big( \frac{eR}{j} \big)^j$ for $j \ge (1 + \delta) R$, since $\big(1 - \frac{1}{j + 1} \big)^{j + 1} e \le 1$. Thus, since $\big( \frac{eR}{j} \big)^j \le e^R$ for any $j > 0$ and $K > (1 + 2 \delta) R$, it follows that 
		\begin{flalign*} 
		\left( \displaystyle\frac{eR}{j} \right)^j & \le \left( \displaystyle\frac{eR}{j - R (1 + \delta)} \right)^{j - R (1 + \delta)} (1 + \delta)^{R (1 + \delta) - j} \\
		& \le e^R (1 + \delta)^{R (1 + \delta) - j} \le e^R (1 + \delta)^{K - j - \delta R}, \qquad \text{whenever $j \ge K$}.
		\end{flalign*} 
		
		\noindent Hence, 
		\begin{flalign*}
		\displaystyle\sum_{j = K + 1}^{\infty} \left( \displaystyle\frac{eR}{j} \right)^j \le e^R (1 + \delta)^{-\delta R} \displaystyle\sum_{j = K + 1}^{\infty} (1 + \delta)^{K - j} \le \delta^{-1} (1 + \delta)^{-\delta R} e^R,
		\end{flalign*} 
		
		\noindent which with \eqref{exponentialestimater} yields the lemma. 
	\end{proof}

	\subsection{Recursions and Identities for Intersection Numbers}
	
	\label{RecursionIntersection} 

	The proofs of \Cref{destimateexponential} and \Cref{limitd} are based on an analysis of recursive relations (resulting from the Witten--Kontsevich theorem \cite{ITMSCF,TGITMS}) determining the $\langle \textbf{d} \rangle$. These are summarized through the following lemma.
	
	\begin{lem}[\cite{ITIHTF,TTQG,TGITMS}]
		
		\label{recursionsd}
		
		Fix integers $g \ge 0$ and $n \ge 1$ such that $3g + n \ge 3$, and an $n$-tuple $\textbf{\emph{d}} = (d_1, d_2, \ldots , d_n) \in \mathbb{Z}_{\ge 0}^n$.
			
		\begin{enumerate} 
		
		\item{\emph{Initial data}: We have $\langle 0, 0, 0 \rangle_{0, 3} = 1 = \langle 1 \rangle_{1, 1}$.}
		
		\item{\emph{String equation}: If $|\textbf{\emph{d}}| = 3g + n - 3$, then 
			\begin{flalign}
			\label{d0recursion}
			\langle \textbf{\emph{d}}, 0 \rangle_{g, n + 1} = \displaystyle\frac{1}{6g + 2n - 3} \displaystyle\sum_{j = 1}^n (2 d_j + 1) \big\langle d_j - 1, \textbf{\emph{d}} \setminus \{ d_j \} \big\rangle_{g, n}.
			\end{flalign}}
		
		\item{\emph{Dilation equation}: If $|\textbf{\emph{d}}| = 3g + n - 3$, then 
			\begin{flalign}
			\label{d1recursion} 
			\langle \textbf{\emph{d}}, 1 \rangle_{g, n + 1} = \displaystyle\frac{6g + 3n - 6}{6g + 2n - 3} \langle \textbf{\emph{d}} \rangle_{g, n}.
			\end{flalign}}

		\item{\emph{Virasoro constraints}: Fix an integer $k \ge -1$. If $|\textbf{\emph{d}}| = 3g + n - k - 3$, then 
		\begin{flalign}
		\label{recursionintersection2}
		\begin{aligned}
		\langle k + 1, \textbf{\emph{d}} \rangle_{g, n + 1} & = \displaystyle\frac{1}{6g + 2n - 3} \displaystyle\sum_{j = 1}^n (2d_j + 1) \big\langle d_j + k, \textbf{\emph{d}} \setminus \{ d_j \} \big\rangle_{g, n}  \\
		& \quad + \displaystyle\frac{12g}{(6g + 2n - 3) (6g + 2n - 5)} \displaystyle\sum_{\substack{r + s = k - 1 \\ r, s \ge 0}} \langle r, s, \textbf{\emph{d}} \rangle_{g - 1, n + 2} \\
		& \quad + \displaystyle\frac{1}{2} \displaystyle\sum_{\substack{r + s = k - 1 \\ r, s \ge 0}} \displaystyle\sum_{\substack{I \cup J = \{ 1, 2, \ldots , n \} \\ |I \cap J| = 0}}  \displaystyle\frac{g!}{g'! g''!} \displaystyle\frac{(6g' + 2n' - 3)!! (6g'' + 2n'' - 3)!!}{(6g + 2n - 3)!!} \\
		& \qquad \qquad \qquad \qquad \qquad \qquad \times \langle r, \textbf{\emph{d}} |_I  \rangle_{g', n' + 1} \langle s, \textbf{\emph{d}} |_J \rangle_{g'', n'' + 1},
		\end{aligned}
		\end{flalign}
		
		\noindent where in \eqref{recursionintersection2} we have set $|I| = n'$, $|J| = n''$, and
		\begin{flalign}
		\label{dsdidj} 
		\textbf{\emph{d}}_S = (d_s)_{s \in S}; \qquad |\textbf{\emph{d}}_I| + r = 3g' + n' - 2; \qquad |\textbf{\emph{d}}_J| + s = 3g'' + n'' - 2,
		\end{flalign}
		
		\noindent which in particular satisfy $n' + n'' = n$ and $g' + g'' = g$.} 
	
		\end{enumerate}

		\noindent Moreover, all normalized intersection numbers $\big\langle \textbf{\emph{d}} \big\rangle$ are determined by the above four properties. 
		
	\end{lem}

	\begin{rem}
		
		It is more common to phrase the results of \Cref{recursionsd} in terms of the correlators from \eqref{intersection2}. Under this notation, the initial data takes the form $\langle \tau_0^3 \rangle = 1$ and $\langle \tau_1 \rangle = \frac{1}{24}$; see equations (2.42) and (2.46) of \cite{TGITMS} for the former and latter, respectively. The analog of \eqref{recursionintersection2}, sometimes referred to as the Dijkgraaf--Verlinde--Verlinde formulation of the Virasoro constraints given originally as equation (7.13) of \cite{TTQG} (and, in our form, as equation (7.27) of \cite{ITIHTF}), states
		\begin{flalign*}
		\Bigg\langle \tau_{k + 1} \displaystyle\prod_{i = 1}^n \tau_{d_i} \Bigg\rangle & = \displaystyle\frac{1}{(2k + 3)!!} \Bigg( \displaystyle\sum_{j = 1}^n \displaystyle\frac{(2k + 2d_j + 1)!!}{(2d_j - 1)!!} \bigg\langle \tau_{d_j + k} \displaystyle\prod_{\substack{1 \le i \le n \\ i \ne j}} \tau_{d_i} \bigg\rangle  \\
		& \qquad \quad + \displaystyle\frac{1}{2} \displaystyle\sum_{\substack{r + s = k - 1 \\ r, s \ge 0}} (2r + 1)!! (2s + 1)!! \bigg\langle \tau_r \tau_s \displaystyle\prod_{i = 1}^n \tau_{d_i} \bigg\rangle \\
		& \qquad \quad + \displaystyle\frac{1}{2} \displaystyle\sum_{\substack{r + s = k - 1 \\ r, s \ge 0}} (2r + 1)!! (2s + 1)!! \displaystyle\sum_{\substack{I \cup J = \{ 1, 2, \ldots , n \} \\ |I \cap J| = 0}} \bigg\langle \tau_r \displaystyle\prod_{i \in I} \tau_{d_i} \bigg\rangle \bigg\langle \tau_s \displaystyle\prod_{j \in J} \tau_{d_j} \bigg\rangle  \Bigg).
		\end{flalign*}
		
		The string and dilation equations (which are also the $k = -1$ and $k = 0$ special cases of the Virasoro constraints, respectively) are given by  
		\begin{flalign*}
		\Bigg\langle \tau_0 \displaystyle\prod_{i = 1}^n \tau_{d_i} \Bigg\rangle = \displaystyle\sum_{j = 1}^n \Bigg\langle \tau_{d_j - 1} \displaystyle\prod_{\substack{1 \le i \le n \\ i \ne j}} \tau_{d_i} \Bigg\rangle; \qquad \Bigg\langle \tau_1 \displaystyle\prod_{i = 1}^n \tau_{d_i} \Bigg\rangle = (2g + n - 2) \Bigg\langle \displaystyle\prod_{i = 1}^n \tau_{d_i} \Bigg\rangle.
		\end{flalign*}
		
		\noindent We refer to equations (2.40) and (2.41) of \cite{TGITMS} for the former and latter, respectively.
		 
	\end{rem}

	 Next we recall certain identities and asymptotics for the intersection numbers $\langle \textbf{d} \rangle$ in the cases $n \in \{ 1, 2 \}$. In the case $n = 1$, an exact identity for these numbers was first predicted below equation (2.26) in \cite{TGITMS} and then proven as result of Theorem 1.2 of \cite{ITMSCF}.
	
	\begin{lem}[{\cite[Theorem 1.2]{ITMSCF}}]
		
	\label{n1}
	
	For any integer $g \ge 1$, we have $\langle 3g - 2 \rangle_{g, 1} = 1$. 
		
	\end{lem}
	
	\noindent In the case $n = 2$, the large genus asymptotics for the intersection numbers $\langle k, 3g - k - 1 \rangle_{g, 2}$ was provided by Proposition 4.1 of \cite{VFGINMSC} (based on an exact identity given below equation (8) of \cite{C}).
	
	\begin{lem}[{\cite[Proposition 4.1]{VFGINMSC}}] 
		
	\label{n2} 
	
	For any integers $g \ge 1$ and $0 \le k \le 3g - 1$, we have	
	\begin{flalign*}
	\displaystyle\frac{6g - 3}{6g - 1} \le \langle k, 3g - k - 1\rangle_{g, 2} \le 1.
	\end{flalign*}
	
	\end{lem}

			\subsection{Stable Graphs and Volumes}
	
	\label{VolumesStable}

	For the proof of \Cref{limitvolume}, we will require an expression from \cite{VFGINMSC} for $\Vol \mathcal{Q}_{g, n}$ in terms of the normalized intersection numbers $\langle \textbf{d} \rangle$ from \Cref{dpsi}, through a weighted sum over stable graphs. So, we begin by recalling the latter.
	
	\begin{definition} 
		
		\label{graph}
		
		Fix integers $g, n \ge 0$. A \emph{graph $\Gamma$ with $n$ legs and a genus decoration} consists of the following. 
		
		\begin{enumerate} 
			
			\item A finite, nonempty set $\mathfrak{V} = \mathfrak{V}(\Gamma)$ of \emph{vertices}
			\item A finite set $\mathfrak{H} = \mathfrak{H}(\Gamma)$ of \emph{half-edges}
			\item A map $\alpha: \mathfrak{H} \rightarrow \mathfrak{V}$ associating each half-edge with a vertex
			\item An $n$-element subset $\mathfrak{L} = \mathfrak{L} (\Gamma) \subseteq \mathfrak{H}(\Gamma)$ of \emph{legs}
			\item A bijection $\lambda: \mathfrak{L} \rightarrow \{ 1, 2, \ldots , n \}$ that labels each leg of $\Gamma$  
			\item An involution $\mathfrak{i}: \mathfrak{H} \rightarrow \mathfrak{H}$, which fixes each element in $\mathfrak{L}$ but none in $\mathfrak{H} \setminus \mathfrak{L}$
			\item A \emph{genus decoration}, which is a set $\textbf{g} = (g_v)_{v \in \mathfrak{V}}$ of nonnegative integers, one associated with each vertex $v \in \mathfrak{V}$
		\end{enumerate}  
		
		We view any pair $e = (h, h') \in (\mathfrak{H} \setminus \mathfrak{L}) \times (\mathfrak{H} \setminus \mathfrak{L})$ of distinct half-edges (which are not legs) such that $\mathfrak{i} (h) = h'$ as an \emph{edge} of $\Gamma$ connecting $\alpha (h)$ with $\alpha (h')$. We say that $e$ is a \emph{self-edge} if $\alpha (h) = \alpha (h')$, and we say that $e$ is a \emph{simple edge} if $\alpha (h) \ne \alpha (h')$. Let $\mathfrak{E} = \mathfrak{E} (\Gamma)$ denote the set of edges of $\Gamma$, where we identify $(h, h')$ and $(h', h)$. Observe that $|\mathfrak{H}| = 2 |\mathfrak{E}| + n$, since $\mathfrak{i}: \mathfrak{H} \rightarrow \mathfrak{H}$ is an involution, whose $n$ fixed points are given by the elements of $\mathfrak{L}$. 
		
		We say that $\Gamma$ is a \emph{stable graph of genus $g$} if the following three conditions are satisfied.
		
		\begin{enumerate}
			\item \emph{Connectivity}: The graph $\Gamma$ is connected, meaning that for any distinct vertices $u, v \in \mathfrak{V}$, there exists a sequence of vertices $u = v_0, v_1, \ldots , v_k = v \in \mathfrak{V}$ such that $v_j$ and $v_{j + 1}$ are connected by an edge $(h_j, h_j') \in \mathfrak{E}$, for each $j \in [0, k - 1]$. 
			
			\item \emph{Genus condition}: Denoting $V = V_{\Gamma} = \big| \mathfrak{V} (\Gamma) \big|$ and $E = E_{\Gamma} = \big| \mathfrak{E} (\Gamma) \big|$, we have  
			\begin{flalign}
			\label{gvgve1}
			\displaystyle\sum_{v \in \mathfrak{V}} g_v = g - E + V - 1.
			\end{flalign}
			
			\item \emph{Stability condition}: Letting $m_v = \big| \alpha^{-1} (v)\big|$ denote the number of half-edges incident to any vertex $v \in \mathfrak{V} (\Gamma)$, we have  
			\begin{flalign}
			\label{gvnv3}
			2g_v + m_v \ge 3.
			\end{flalign}
		\end{enumerate}
		
		An \emph{isomorphism} between two graphs with genus decoration, $\Gamma = (\mathfrak{V}, \mathfrak{H}, \alpha, \mathfrak{L}, \lambda, \mathfrak{i}, \textbf{g})$ and $\Gamma' = (\mathfrak{V}', \mathfrak{H}', \alpha', \mathfrak{L}', \lambda', \mathfrak{i}', \textbf{g}')$, consists of two bijections $\mu: \mathfrak{V} \rightarrow \mathfrak{V}$ and $\nu: \mathfrak{H} \rightarrow \mathfrak{H}'$ such that $\alpha' \big( \nu (h) \big) = \mu \big( \alpha (h) \big)$ and $\mathfrak{i}' \big( \nu (h) \big) = \nu \big( \mathfrak{i} (h) \big)$, for each half-edge $h \in \mathfrak{H}$; such that $\lambda' \big( \nu (h) \big) = \lambda (h)$, for each leg $h \in \mathfrak{L}$; and such that $g_{\mu (v)}' = g_v$, for each vertex $v \in \mathfrak{V}$. 
		
		For any integer $g \ge 0$, we further denote by $\mathcal{G}_{g, n}$ the set of stable graphs with $n$ legs of genus $g$, up to isomorphism (so, in particular, the vertices and edges of a stable graph in $\mathcal{G}_{g, n}$ are unlabeled, but the legs are). For any $\Gamma \in \mathcal{G}_{g, n}$, an isomorphism $(\mu, \nu)$ from $\Gamma$ to itself is an called an \emph{automorphism} of $\Gamma$; let $\Aut (\Gamma)$ denote the set of automorphisms of $\Gamma$. 
		
	\end{definition} 
	
	\begin{rem}
		
		\label{vestable}
		
		Fix $g \in \mathbb{Z}_{\ge 0}$ and $\Gamma \in \mathcal{G}_{g, n}$, and adopt the notation from \Cref{graph}. Then, 
		\begin{flalign}
		\label{nvsum} 
		\displaystyle\sum_{v \in \mathfrak{V}} m_v = 2E + n.
		\end{flalign}
		
		\noindent Moreover, by summing \eqref{gvnv3} over all $v \in \mathfrak{V} (\Gamma)$ and using \eqref{gvgve1} and \eqref{nvsum} yields 
		\begin{flalign}
		\label{gv2} 
		2g + n \ge V + 2.
		\end{flalign}
		
		\noindent Since $\Gamma$ is connected, we must also have $E \ge V - 1$. Furthermore, \eqref{gvgve1}, the fact that each $g_v \ge 0$, and \eqref{gv2} together imply
		\begin{flalign}
		\label{3ge} 
		E \le g + V - 1 \le 3g + n - 3.
		\end{flalign}
		
	\end{rem}
	
	Following equation (1.12) of \cite{VFGINMSC}, we will next define a polynomial associated with any stable graph, which will be used to evaluate $\Vol \mathcal{Q}_{g, n}$. However, we first require the following multi-variate polynomial, whose coefficients are given in terms of the normalized intersection numbers from \Cref{dpsi}, and linear map that was originally introduced as equation (1.11) of \cite{VFGINMSC}. In the below, for any integer $k > 1$, $\zeta (k) = \sum_{j = 1}^{\infty} j^{-k}$ denotes the Riemann zeta function. 
	
	\begin{definition}
		
		\label{ngnb}
		
		For any integers $g \ge 0$ and $n \ge 1$, and set of $n$ variables $\textbf{b} = (b_1, b_2, \ldots,  b_n)$, define the polynomial 
		\begin{flalign}
		\label{ngnb1} 
		N_{g, n} (\textbf{b}) = \displaystyle\frac{(6g + 2n - 5)!!}{2^{5g + n - 3} 3^g g!} \displaystyle\sum_{\textbf{d} \in \mathcal{K}_n (3g + n - 3)} \langle \textbf{d} \rangle_{g, n} \displaystyle\prod_{j = 1}^n \displaystyle\frac{b_j^{2d_j}}{(2 d_j + 1)!},
		\end{flalign}
		
		\noindent where we have recalled the normalized intersection number $\langle \textbf{d} \rangle_{g, n}$ from \Cref{dpsi} and denoted $\textbf{d} = (d_1, d_2, \ldots , d_n)$. We further define the map $\mathcal{Z}: \mathbb{C}[\textbf{b}] \rightarrow \mathbb{C}$ by setting it on monomials to be 
		\begin{flalign*}
		\mathcal{Z} \Bigg( \displaystyle\prod_{j = 1}^k b_{i_j}^{r_j} \Bigg) = \displaystyle\prod_{j = 1}^k r_j! \zeta (r_j + 1), 
		\end{flalign*}
		
		\noindent for any mutually distinct indices $i_1, i_2, \ldots , i_k \in [1, n]$ and integers $r_1, r_2, \ldots , r_k \ge 1$, and then extending it by linearity to all of $\mathbb{C} [\textbf{b}]$. 
	\end{definition}
	
	Now we can define the following polynomial associated with any stable graph $\Gamma$, given by equation (1.12) of \cite{VFGINMSC}.
	
	\begin{definition}
		
		\label{gammap} 
		
		Fix integers $g, n \ge 0$ and a stable graph $\Gamma \in \mathcal{G}_{g, n}$; adopt the notation for $\Gamma = (\mathfrak{V}, \mathfrak{H}, \alpha, \mathfrak{L}, \lambda, \mathfrak{i}, \textbf{g})$ from \Cref{graph}. For each edge $e \in \mathfrak{E} (\Gamma)$, let $b_e$ denote a variable; for each half-edge $h \in \mathfrak{H} (\Gamma) \setminus \mathfrak{L} (\Gamma)$ that is not a leg, set $b_h = b_e$, where $e = (h, h')$ is the edge containing $h$; and for each leg $h \in \mathfrak{L}(\Gamma)$, set $b_h = 0$.  For each vertex $v \in \mathfrak{V} (\Gamma)$, let $\textbf{b}^{(v)} = (b_h)$ denote the set of variables $b_h$, where $h$ ranges over all half-edges that are incident to $v$. In particular if $e$ is a simple edge incident to $v$, then $b_e$ is counted with multiplicity one in $\textbf{b}^{(v)}$; if $e$ is a self-edge incident to $v$, then $b_e$ is counted with multiplicity two in $\textbf{b}^{(v)}$; and $\textbf{b}^{(v)}$ also contains $n = \big|\mathfrak{L} (\Gamma) \big|$ entries equal to $0$. Therefore, $\textbf{b}^{(v)}$ consists of $\big| \alpha^{-1} (v) \big| = m_v$ elements. Define the polynomial $\mathbb{P} (\Gamma) \in \mathbb{C} \big[ (b_e)_{e \in \mathfrak{E} (\Gamma)} \big]$ by setting
		\begin{flalign*}
		P (\Gamma) = 2^{6g + 2n - 4} \displaystyle\frac{(4g + n - 4)!}{(6g + 2n - 7)!} \displaystyle\frac{1}{2^{|\mathfrak{V} (\Gamma)|} \big| \Aut (\Gamma) \big|} \displaystyle\prod_{e \in \mathfrak{E} (\Gamma)} b_e \displaystyle\prod_{v \in \mathfrak{V} (\Gamma)} N_{g_v, m_v} \big( \textbf{b}^{(v)} \big).
		\end{flalign*}
		
	\end{definition}
	
	Now we can state the following expression for $\Vol \mathcal{Q}_{g, n}$ from \cite{VFGINMSC}.
	
	\begin{prop}[{\cite[Theorem 1.6]{VFGINMSC}}]
		
		\label{volumesumgraph}
		
		For any integers $g, n \ge 0$, we have 
		\begin{flalign}
		\label{sumgraphsq}
		\Vol \mathcal{Q}_{g, n} = \displaystyle\sum_{\Gamma \in \mathcal{G}_{g, n}} \mathcal{Z} \big( P (\Gamma) \big). 
		\end{flalign}
		
	\end{prop}

	\section{Exponential Upper Bound on \texorpdfstring{$\langle \textbf{d} \rangle_{g, n}$}{}}
	
	\label{Exponential} 
	
	In this section we establish \Cref{destimateexponential}. We begin in \Cref{CoefficientSum} by bounding the sum of coefficients in the last term of \eqref{recursionintersection2}, and then we prove \Cref{destimateexponential} in \Cref{ProofExponential}.
	
	\subsection{Estimating a Sum of Coefficients}
	
	\label{CoefficientSum}
	
	In this section we establish the following lemma.

	\begin{lem}
		
		\label{sumgn1} 
		
		Fix integers $k \ge 1$ and $n \ge 2$, and an $n$-tuple $\textbf{\emph{d}} = (d_1, d_2, \ldots , d_n) \in \mathbb{Z}_{\ge 0}^n$ of nonnegative integers. Assume that $d_j \ge k + 1$ for each $j \in [1, n]$ and that there exists an integer $g \ge 0$ such that $|\textbf{\emph{d}}| + k = 3g + n - 3$. Then,
		\begin{flalign}
		\label{sestimate}
		S = \displaystyle\frac{1}{2} \displaystyle\sum_{\substack{r + s = k - 1 \\ r, s \ge 0}} \displaystyle\sum_{\substack{I \cup J = \{ 1, 2, \ldots , n \} \\ |I \cap J| = 0}} \displaystyle\frac{g!}{g'! g''!} \displaystyle\frac{(6g' + 2n' - 3)!! (6g'' + 2n'' - 3)!!}{(6g + 2n - 3)!!} \le \displaystyle\frac{n + 1}{8 g^2},
		\end{flalign}
		
		\noindent where we have set $|I| = n'$, $|J| = n''$, and have adopted the notation from \eqref{dsdidj}.
		
	\end{lem} 		
	
	To that end, we begin by bounding summands appearing in the quantity $S$ from \Cref{sumgn1}. 

	\begin{lem}
		
		\label{estimateproduct} 
		
		Let $g', g'' \ge 1$ and $n', n'' \ge 0$ be integers, and set $g = g' + g''$ and $n = n' + n''$. Then, 
		\begin{flalign*}
		\binom{n}{n'} \binom{g}{g'} \displaystyle\frac{(6g' + 2n' - 3)!! (6g'' + 2n'' - 3)!!}{(6g + 2n - 3)!!} \le \displaystyle\frac{1}{6g + 2n - 3} \binom{2g - 2}{2g' - 1}^{-1}.
		\end{flalign*}
		
	\end{lem} 
	
	\begin{proof} 
		
		Since, $n' + n'' = n$ and $g' + g'' = g$, we have 
		\begin{flalign*}
		(6g + 2n - 3)!! & = (6g + 2n - 3) \displaystyle\frac{(6g + 2n - 4)!}{2^{3g + n - 2} (3g + n - 2)!}; \\
		(6g' + 2n' - 3)!! (6g'' + 2n'' - 3)!! & = \displaystyle\frac{(6g' + 2n' - 2)! (6g'' + 2n'' - 2)!}{2^{3g + n - 2} (3g' + n' - 1)! (3g'' + n'' - 1)!},
		\end{flalign*}
		
		\noindent and so
		\begin{flalign}
		\label{gngn2}
		\displaystyle\frac{(6g' + 2n' - 3)!! (6g'' + 2n'' - 3)!!}{(6g + 2n - 3)!!} = \displaystyle\frac{1}{6g + 2n - 3} \binom{3g + n - 2}{3g' + n' - 1} \binom{6g + 2n - 4}{6g' + 2n' - 2}^{-1}.
		\end{flalign} 
		
		\noindent Moreover, since $g', g'' \ge 1$, we have  
		\begin{flalign}
		\label{gngn3} 
		\begin{aligned}
		\binom{n}{n'} \binom{g}{g'} & \binom{3g + n - 2}{3g' + n' - 1} \binom{6g + 2n - 4}{6g' + 2n' - 2}^{-1} \\
		& = \binom{6g' + 2n' - 2}{g', n', 3g' + n' - 1, 2g' - 1}  \binom{6g'' + 2n'' - 2}{g'', n'', 3g'' + n'' - 1, 2g'' - 1} \\
		& \qquad \times \binom{6g + 2n - 4}{g, n, 3g + n - 2, 2g - 2}^{-1} \binom{2g - 2}{2g' - 1}^{-1}.
		\end{aligned} 
		\end{flalign}
		
		Now applying the $(n, r) = (2, 4)$ case of \Cref{sumaijaiestimate} (with the facts that $g' + g'' = g$ and $n' + n'' = n$), we obtain that
		\begin{flalign}
		\label{gngn4}
		\binom{6g' + 2n' - 2}{g', n', 3g' + n' - 1, 2g' - 1}  \binom{6g'' + 2n'' - 2}{g'', n'', 3g'' + n'' - 1, 2g'' - 1} \le \binom{6g + 2n - 4}{g, n, 3g + n - 2, 2g - 2}.
		\end{flalign}
		
		\noindent Inserting \eqref{gngn4} into \eqref{gngn3}, we obtain that
		\begin{flalign*}
		\binom{n}{n'} \binom{g}{g'} & \binom{3g + n - 2}{3g' + n' - 1} \binom{6g + 2n - 4}{6g' + 2n' - 2}^{-1} \le \binom{2g - 2}{2g' - 1}^{-1},
		\end{flalign*}
		
		\noindent which upon insertion into \eqref{gngn2} yields the lemma.
	\end{proof}
	
	Now we can establish \Cref{sumgn1}.
	
	\begin{proof}[Proof of \Cref{sumgn1}]
		
		First observe that, since $d_j \ge k + 1 \ge 2$ for each $j \in [1, n]$, we have $2n' + r \le |\textbf{d}_I| + r = 3g' + n' - 2$; so, $3g' - n' - 2 \ge r \ge 0$, and similarly $3g'' - n'' - 2 \ge 0$. Hence, $g', g'' \ge 1$. Since $g' + g'' = g$, it follows that $g', g'' \in [1, g - 1]$. 
		
		Now, for a fixed pair of nonnegative integers $(n', n'')$ with $n' + n'' = n$, there are $\binom{n}{n'}$ choices for a pair of nonempty, disjoint sets $(I, J)$ such that $I \cup J = \{ 1, 2, \ldots , n \}$; $|I| = n'$; and $|J| = n''$. Fixing these two subsets and a pair of nonnegative integers $(r, s)$ with $r + s = k - 1$ also fixes $g', g'' \in [1, g - 1]$ through \eqref{dsdidj}. Thus,	
		\begin{flalign*}
		S \le \displaystyle\frac{1}{2} \displaystyle\sum_{g' = 1}^{g - 1} \displaystyle\sum_{n' = 0}^n \binom{n}{n'} \binom{g}{g'} \displaystyle\frac{(6g' + 2n' - 3)!! (6g'' + 2n'' - 3)!!}{(6g + 2n - 3)!!},
		\end{flalign*}
		
		\noindent which by \Cref{estimateproduct} yields 
		\begin{flalign}
		\label{sestimate2} 
		S & \le \displaystyle\frac{1}{2 (6g + 2n - 3)} \displaystyle\sum_{g' = 1}^{g - 1} (n + 1) \binom{2g - 2}{2g' - 1}^{-1}.
		\end{flalign}
		
		\noindent Thus, for $g = 2$, we have since $n \ge 2$ that 
		\begin{flalign*} 
		S \le \frac{n + 1}{2 (6g + 2n - 3) (2g - 2)} = \displaystyle\frac{n + 1}{8n + 36} < \displaystyle\frac{n + 1}{32} = \frac{n + 1}{8g^2},
		\end{flalign*} 
		
		\noindent which verifies \eqref{sestimate} if $g = 2$. So, since $g = g' + g'' \ge 2$, we may assume that $g \ge 3$. In this case, since $\binom{2g - 2}{2g' - 1} \ge \binom{2g - 2}{3}$ for $g' \in [2, g - 2]$, it follows from \eqref{sestimate2} and the bounds $n \ge 2$ and $g \ge 3$ that  
		\begin{flalign*} 
		S \le \displaystyle\frac{n + 1}{12g + 2} \Bigg( \displaystyle\frac{1}{g - 1} + \displaystyle\sum_{g' = 2}^{g - 2} \binom{2g - 2}{2g' - 1}^{-1} \Bigg) \le \displaystyle\frac{n + 1}{12g + 2} \Bigg( \displaystyle\frac{1}{g - 1} + (g - 3) \binom{2g - 2}{3}^{-1} \Bigg) \le \displaystyle\frac{n + 1}{8g^2},
		\end{flalign*}
		
		\noindent from which we deduce the lemma.
	\end{proof}

	\subsection{Proof of \Cref{destimateexponential}}
	
	\label{ProofExponential}
	
	In this section we establish \Cref{destimateexponential}. To that end, we will bound weighted sums of the coefficients appearing on the right side of \eqref{recursionintersection2}. This is provided by the following lemma, where in the below $U$ denotes the sum of the coefficients in the first term there; $V$ denotes the sum of those in the second; and $S$ denotes the sum of those in the third.
	
	\begin{lem}
		
		\label{uvsestimate} 
		
		Adopt the notation and assumptions of \Cref{sumgn1}, and further set 
		\begin{flalign}
		\label{vu} 
		& U = \displaystyle\frac{1}{6g + 2n - 3} \displaystyle\sum_{j = 1}^n (2d_i + 1); \qquad V = \displaystyle\frac{12gk}{(6g + 2n - 3) (6g + 2n - 5)}.
		\end{flalign}
		
		\noindent Then, 
		\begin{flalign}
		\label{sumuvs}
		\displaystyle\frac{2U}{3} + \displaystyle\frac{3V}{2} + S \le 1. 
		\end{flalign}
	\end{lem} 
	
	\begin{proof}
		
		First observe, since $|\textbf{d}| + k = 3g + n - 3$, we have 
		\begin{flalign*} 
		U = \displaystyle\frac{2 |\textbf{d}| + n}{6g + 2n - 3} = 1 + \displaystyle\frac{n - 2k - 3}{6g + 2n - 3}.
		\end{flalign*}	
		
		\noindent Applying this, the definition \eqref{vu} of $V$, and \Cref{sumgn1}, we deduce that  
		\begin{flalign}
		\label{uvestimate2} 
		\begin{aligned}
		\displaystyle\frac{2U}{3} + \displaystyle\frac{3V}{2} + S & \le \displaystyle\frac{2}{3} + \displaystyle\frac{2n - 6}{18g + 6n - 9} - \displaystyle\frac{4k}{3 (6g + 2n - 3)} + \displaystyle\frac{18gk}{(6g + 2n - 3) (6g + 2n - 5)} + \displaystyle\frac{n + 1}{8 g^2} \\
		& =	\displaystyle\frac{2}{3} + \displaystyle\frac{k}{(6g + 2n - 3) (6g + 2n - 5)} \left( 10 g - \displaystyle\frac{8n}{3} + \displaystyle\frac{20}{3} \right)  + \displaystyle\frac{2n - 6}{18g + 6n - 9} + \displaystyle\frac{n + 1}{8 g^2}.
		\end{aligned} 
		\end{flalign} 
		
		\noindent Now, since $d_j \ge k + 1$ for each $j \in [1, n]$, we have $n (k + 1) + k \le |\textbf{d}| + k = 3g + n - 3$, and so 
		\begin{flalign} 
		\label{n1kg} 
		(n + 1) k \le 3g - 3.
		\end{flalign}
		
		\noindent Inserting \eqref{n1kg} into \eqref{uvestimate2}, we obtain that 
		\begin{flalign}
		\label{uvestimate3} 
		\begin{aligned}
		\displaystyle\frac{2U}{3} + \displaystyle\frac{3V}{2} + S & \le \displaystyle\frac{2}{3} + \displaystyle\frac{2n - 6}{18g + 6n - 9} + \displaystyle\frac{n + 1}{8 g^2} \\
		& \qquad + \displaystyle\frac{3g - 3}{(n + 1) (6g + 2n - 3) (6g + 2n - 5)} \left( 10 g - \displaystyle\frac{8n}{3} + \displaystyle\frac{20}{3} \right).
		\end{aligned}
		\end{flalign}
		
		\noindent Since $n \ge 2$, we have $\frac{5}{3} \big( 6g + 2n - 3 \big) \ge 10 g - \frac{8n}{3} + \frac{20}{3}$ and $3g - 3 \le \frac{1}{2} (6g + 2n - 5)$, so
		\begin{flalign*}
		\displaystyle\frac{3g - 3}{(n + 1) (6g + 2n - 3) (6g + 2n - 5)} \left( 10 g - \displaystyle\frac{8n}{3} + \displaystyle\frac{20}{3} \right) \le \displaystyle\frac{5}{6 (n + 1)}. 
		\end{flalign*} 
		
		\noindent Hence, \eqref{uvestimate3} yields
		\begin{flalign} 
		\label{uvsestimaten5}
		\begin{aligned}
		\displaystyle\frac{2U}{3} + \displaystyle\frac{3U}{2} + S & \le \displaystyle\frac{2}{3} + \displaystyle\frac{2n - 6}{18g + 6n - 9} + \displaystyle\frac{5}{6 (n + 1)}  + \displaystyle\frac{n + 1}{8 g^2} \\
		& \le \displaystyle\frac{2}{3} + \displaystyle\frac{2n - 6}{12n + 15} + \displaystyle\frac{5}{6 (n + 1)}  + \displaystyle\frac{9 (n + 1)}{8 (n + 4)^2},
		\end{aligned}
		\end{flalign} 
		
		\noindent where in the second bound we used the fact that $3g \ge n + 4$ (which follows from \eqref{n1kg} and the fact that $k \ge 1$). If $n = 2$, then the right side of \eqref{uvsestimaten5} is equal to $\frac{3695}{3744} < 1$, in which case \eqref{sumuvs} holds.
		
		If instead $n \ge 3$, then $\frac{2n - 6}{12n + 15} \le \frac{n - 3}{6 (n + 1)}$, and so \eqref{uvsestimaten5} implies 
		\begin{flalign}
		\label{uvsestimaten6}
		\displaystyle\frac{2U}{3} + \displaystyle\frac{3U}{2} + S & \le \displaystyle\frac{2}{3} + \displaystyle\frac{n + 2}{6 (n + 1)}  + \displaystyle\frac{9 (n + 1)}{8 (n + 4)^2} = \displaystyle\frac{5}{6} + \displaystyle\frac{1}{6 (n + 1)}  + \displaystyle\frac{9 (n + 1)}{8 (n + 4)^2},
		\end{flalign}
		
		\noindent It is quickly verified that the right side of \eqref{uvsestimaten6} is decreasing for $n \ge 2$ and is equal to $\frac{283}{288} < 1$ for $n = 2$. Thus, we deduce that \eqref{sumuvs} holds for any $n \ge 3$, which by the above verifies the lemma. 
	\end{proof}
	
	Now we can establish \Cref{destimateexponential}.

	\begin{proof}[Proof of \Cref{destimateexponential}]
		
		We induct on $|\textbf{d}| = 3g + n - 3$. The result holds if $3g + n - 3 = 0$, since then $\langle \textbf{d} \rangle_{g, n} = \langle 0, 0, 0 \rangle_{0, 3} = 1$, by \Cref{recursionsd}. Thus, we fix some integer $B \ge 1$ and assume that \eqref{destimaten} holds whenever $|\textbf{d}| \le B - 1$; we will show the same bound holds for any $\textbf{d}$ with $|\textbf{d}| = B$. To that end, fix some $\textbf{d} = (d_1, d_2, \ldots , d_n)$ with $|\textbf{d}| = B$. In view of \Cref{n1} and \Cref{n2}, \eqref{destimaten} holds when $n \in \{ 1, 2 \}$. Thus, we will assume in what follows that $n \ge 3$. 
		
		Let $k \in \mathbb{Z}_{\ge -1}$ be such that $k + 1 = \min_{j \in [1, n + 1]} d_j$. Set $\textbf{d}' = \textbf{d} \setminus \{ k + 1 \} = (d_1', d_2', \ldots , d_{n - 1}') \in \mathbb{Z}_{\ge 0}^{n - 1}$, so that $\textbf{d} = (k + 1, \textbf{d}')$. We consider the three cases $k = -1$, $k = 0$, and $k \ge 1$ separately.
		
		So, let us first assume that $k = -1$, in which case \eqref{d0recursion} implies 
		\begin{flalign}
		\label{d0recursion2}
		\langle \textbf{d} \rangle_{g, n} = \langle \textbf{d}', 0 \rangle_{g, n} = \displaystyle\frac{1}{6g + 2n - 5} \displaystyle\sum_{j = 1}^{n - 1} (2 d_j' + 1) \big\langle d_j' - 1, \textbf{d}' \setminus \{ d_j' \} \big\rangle_{g, n - 1}.
		\end{flalign}
		
		\noindent Applying \eqref{destimaten} on each summand on the right side of \eqref{d0recursion2} gives
		\begin{flalign}
		\label{g1estimate}
		\langle \textbf{d} \rangle_{g, n} = \langle \textbf{d}', 0 \rangle_{g, n} \le \displaystyle\frac{1}{6g + 2n - 5} \left( \displaystyle\frac{3}{2} \right)^{n - 2} \displaystyle\sum_{j = 1}^{n - 1} (2 d_j' + 1) & = \displaystyle\frac{2 |\textbf{d}'| + n - 1}{6g + 2n - 5} \left( \displaystyle\frac{3}{2} \right)^{n - 2}.
		\end{flalign}
		
		\noindent Since $k = -1$ implies that $|\textbf{d}'| = |\textbf{d}| = 3g + n - 3$, we have for $g \ge 1$ that
		\begin{flalign}
		\label{g1estimate2}
		\displaystyle\frac{2 |\textbf{d}'| + n - 1}{6g + 2n - 5} & = \displaystyle\frac{6g + 3n - 7}{6g + 2n - 5} \le \displaystyle\frac{3}{2}.
		\end{flalign}
		
		\noindent Together, \eqref{g1estimate} and \eqref{g1estimate2} yield \eqref{destimaten} when $k = -1$ and $g \ge 1$. To address the case $(k, g) = (-1, 0)$, observe that $g = 0$ implies $|\textbf{d}'| = n - 3$, so $n \ge 3$ and at least two of the $d_j'$ are equal to $0$. Then, at least two summands on the right side of \eqref{d0recursion2} are equal to $0$, and so
		\begin{flalign*}
		\langle \textbf{d} \rangle_{g, n} = \langle \textbf{d}', 0 \rangle_{g, n} & \le \displaystyle\frac{1}{2n - 5} \Bigg( \bigg( \displaystyle\sum_{j = 1}^{n - 1} (2 d_j' + 1) \bigg) - 2 \Bigg) \displaystyle\max_{j \in [1, n - 1]} \big\langle d_j' - 1, \textbf{d}' \setminus \{ d_j \} \big\rangle_{g, n - 1} \\
		& = \displaystyle\frac{3n - 9}{2n - 5} \displaystyle\max_{j \in [1, n - 1]} \big\langle d_j' - 1, \textbf{d}' \setminus \{ d_j \} \big\rangle_{g, n - 1} \le \displaystyle\frac{3n - 9}{2n - 5} \left( \displaystyle\frac{3}{2} \right)^{n - 2}  \le \left( \displaystyle\frac{3}{2} \right)^{n - 1},
		\end{flalign*}
		
		\noindent where in the second statement we used the fact that $2 |\textbf{d}'| + n - 1 = 3n - 7$ when $g = 0$, in the third statement we used \eqref{destimaten}, and in the fourth statement we used the bound $\frac{3n - 9}{2n - 5} \le \frac{3}{2}$ that holds for each $n \ge 3$. This verifies \eqref{destimaten} if $k = -1$. 
		
		Next we assume that $k = 0$. In this case, \eqref{d1recursion} and \eqref{destimaten} together give
		\begin{flalign*}
		\langle \textbf{d} \rangle_{g, n} = \langle \textbf{d}', 1 \rangle_{g, n} = \displaystyle\frac{6g + 3n - 9}{6g + 2n - 5} \langle \textbf{d}' \rangle_{g, n - 1} \le \displaystyle\frac{6g + 3n - 9}{6g + 2n - 5} \left( \displaystyle\frac{3}{2} \right)^{n - 2} \le \left( \displaystyle\frac{3}{2} \right)^{n - 1},
		\end{flalign*}		
		
		\noindent where in the last bound we used the fact that $\frac{6g + 3n - 9}{6g + 2n - 5} \le \frac{3}{2}$ holds whenever $g \ge 0$ and $6g + 2n - 5 \ge 0$. This verifies \eqref{destimaten} if $k = 0$.  
		
		Now let us assume that $k \ge 1$. Then, the assumptions of \Cref{sumgn1} and \Cref{uvsestimate} apply. Adopting the notation from those statements, applying \eqref{recursionintersection2} (with the $n$ there equal to $n - 1$ here), using \eqref{destimaten} on each summand there, and using the fact that $n' + n''$ in \eqref{recursionintersection2} is equal to $n - 1$ here then gives
		\begin{flalign}
		\label{dgn1estimate}
		\langle \textbf{d} \rangle_{g, n} = \langle k + 1, \textbf{d}' \rangle_{g, n} \le \left( \displaystyle\frac{3}{2} \right)^{n - 2} U + \left( \displaystyle\frac{3}{2} \right)^n V + \left( \displaystyle\frac{3}{2} \right)^{n - 1} S = \left( \displaystyle\frac{2U}{3} + \displaystyle\frac{3V}{2} + S \right) \left( \displaystyle\frac{3}{2} \right)^{n - 1}.
		\end{flalign}
		
		\noindent Then, \eqref{destimaten} follows from \Cref{uvsestimate} and \eqref{dgn1estimate}; this establishes the proposition.
	\end{proof}

	\section{Asymptotics for \texorpdfstring{$\langle \textbf{d} \rangle_{g, n}$}{}} 
	
	\label{Coefficients}
	
	In \Cref{LimitIntersection} we prove \Cref{limitd}, assuming an asymptotic lower and upper bound, given by \Cref{dlower} and \Cref{duppern}, respectively. The remainder of this section is devoted to the proof of the lower bound \Cref{dlower}. We begin in \Cref{Recursion1} by providing a recursive lower bound for the normalized intersection number $\big\langle \textbf{d} \rangle_{g, n}$ (given by \Cref{thetagnestimate}). By interpreting this estimate through a random walk, we then in \Cref{Thetaw} bound $\big\langle \textbf{d} \rangle_{g, n}$ by a certain probability for this random walk (given by \Cref{destimatef}), which we analyze in \Cref{LowerdProof} to establish \Cref{dlower}.

	\subsection{Proof of \Cref{limitd}}
	
	\label{LimitIntersection}

	The proof of \Cref{limitd} will consist in a lower and upper bound on $\langle \textbf{d} \rangle_{g, n}$, given by the following two propositions. The first will be established in \Cref{LowerdProof} and the second in \Cref{ProofUpperd}. In the below, we recall $\mathcal{K}_m (N)$ from \Cref{Estimate}.

	\begin{prop}
		
		\label{dlower}
		
		Let $g > 2^{15}$ and $n \ge 1$ be integers such that $g > 30 n$, and let $\textbf{\emph{d}} \in \mathcal{K}_n (3g + n - 3)$ be a nonnegative composition. Then, we have 
		\begin{flalign*}
		\langle \textbf{\emph{d}} \rangle_{g, n} \ge 1 - 20 g^{-1} (n + 4 \log g).
		\end{flalign*}
		
	\end{prop}
	
	\begin{prop} 
		
		\label{duppern}

		Let $g > 2^{30}$ and $n \ge 1$ be integers such that $g > 800 n^2$, and let $\textbf{\emph{d}} \in \mathcal{K}_n (3g + n - 3)$ be a nonnegative composition. Then, we have 
		\begin{flalign*}
		\langle \textbf{\emph{d}} \rangle_{g, n} \le \exp \big( 625 g^{-1} (n + 2 \log g)^2 \big).
		\end{flalign*}
		
	\end{prop} 
	
	Assuming \Cref{dlower} and \Cref{duppern}, we can now quickly establish \Cref{limitd}.
	
	\begin{proof}[Proof of \Cref{limitd} Assuming \Cref{dlower} and \Cref{duppern}]
		
		Since for sufficiently large $g$ and any $\textbf{d} = (d_1, d_2, \ldots , d_n) \in \Delta (g; \varepsilon)$ we have $n + 4 \log g < 2 \varepsilon g^{1 / 2}$, \Cref{dlower} and \Cref{duppern} together imply
		\begin{flalign*}
		\displaystyle\lim_{\varepsilon \rightarrow 0} \left(\displaystyle\lim_{g \rightarrow \infty} \displaystyle\max_{\textbf{d} \in \Delta (g; \varepsilon)} \big| \langle \textbf{d} \rangle - 1\big| \right) \le \displaystyle\lim_{\varepsilon \rightarrow 0} \left(\displaystyle\lim_{g \rightarrow \infty} \displaystyle\max \big\{ 40 \varepsilon g^{-1/2}, \exp (2500 \varepsilon^2) - 1 \big\} \right) = 0, 
		\end{flalign*}
		
		\noindent from which we deduce the theorem.
	\end{proof}	
	
	Before proceeding, it will be useful to introduce the following notation. 
	
	\begin{definition} 
		
		\label{omegagn} 
		
		For any integers $g \ge 0$ and $n \ge 2$, let 
		\begin{flalign*}
		\Omega (g, n) = \bigcup_{G = g}^{\infty} \bigcup_{m = 2}^n \mathcal{K}_m (3G + m - 3),  
		\end{flalign*}
		
		\noindent and set 
		\begin{flalign*}
		\theta_{g, n} = \inf_{\textbf{d} \in \Omega_{g, n}} \langle \textbf{d} \rangle_{g, n}; \qquad \Theta_{g, n} = \displaystyle\sup_{\textbf{d} \in \Omega_{g, n}} \langle \textbf{d} \rangle_{g, n}.
		\end{flalign*}
		
	\end{definition} 
	
	\begin{rem}
		
		Observe in particular that 
		\begin{flalign}
		\label{thetatheta} 
		\theta_{g, n} \ge \theta_{g', n'}, \quad \text{and} \quad \Theta_{g, n} \le \Theta_{g', n'}, \quad \text{if $g \ge g'$ and $n \le n'$}, 
		\end{flalign} 
		
		\noindent since then $\Omega (g, n) \subseteq \Omega (g', n')$. 
		
	\end{rem} 
	
	As such, if $n \ge 2$ (the case $n = 1$ is addressed separately by \Cref{n1}), \Cref{dlower} corresponds to a lower bound on $\theta_{g, n}$ and \Cref{duppern} to an upper bound on $\Theta_{g, n}$.

	\subsection{Recursive Estimate for \texorpdfstring{$\theta_{g, n}$}{}} 
	
	\label{Recursion1}
	
	Let us briefly (and heuristically) explain how we will use the linear recursion \eqref{recursionintersection2}. One might initially hope that, after applying \eqref{recursionintersection2} to some $\langle k + 1, \textbf{d} \rangle_{g, n + 1}$, all resulting terms will correspond to intersection numbers $\langle \textbf{d}' \rangle$ of length at most $n$. If this were true, then by using \eqref{recursionintersection2} $n - 1$ times, we would obtain an expression for $\langle \textbf{d} \rangle_{g, n}$ as a linear combination of two-point intersection numbers $\langle\textbf{d}'\rangle_{g', 2}$.	 Since the large genus asymptotics for the latter are known by \Cref{n2}, this would yield those for $\langle\textbf{d}\rangle_{g, n}$. 
	
	Unfortunately, it is not quite true that the length of $\textbf{d}$ decreases after applying \eqref{recursionintersection2}. Indeed, although this is indeed the case for the first term on the right side of that equation, it is not true for the second term. However, we will show that the length of $\textbf{d}$ has a ``decreasing drift,'' through the following lemma. In the below, we recall the quantity $\theta_{g, n}$ from \Cref{omegagn}. 
		
	\begin{lem} 
		
	\label{thetagnestimate} 

	For any integers $g \ge 1$ and $n \ge 2$, we have
	\begin{flalign}
	\label{thetagn1}
	\theta_{g, n + 1} \ge \left( \displaystyle\frac{2\theta_{g, n}}{3} + \displaystyle\frac{\theta_{g - 1, n + 2}}{3} \right) \left(1 - \displaystyle\frac{1}{4g} \right).
	\end{flalign}
	
	\end{lem} 

	\begin{proof}
		
		Fix some $\textbf{d} = (d_1, d_2, \ldots , d_{m + 1}) \in \Omega (g, n + 1)$, for some integer $m \in [2, n]$, and define $G \in \mathbb{Z}_{\ge 0}$ so that $|\textbf{d}| = 3G + m - 2$. Let $k \ge - 1$ be such that $k + 1 = \min_{j \in [1, m + 1]} d_j$, and let $\textbf{d} = (k + 1, \textbf{d}')$, for some $\textbf{d}' = (d_1', d_2', \ldots ,d_m') \in \mathbb{Z}_{\ge 0}^m$. We consider three cases.
		
		The first is if $k = -1$, in which case \eqref{d0recursion} yields 
		\begin{flalign}
		\label{dgm1}
		\begin{aligned}
		\langle \textbf{d} \rangle_{G, m + 1} = \langle \textbf{d}', 0 \rangle_{G, m + 1} & = \displaystyle\frac{1}{6G + 2m - 3} \displaystyle\sum_{j = 1}^m (2d_j' + 1) \big\langle d_j' - 1, \textbf{d}' \setminus \{ d_j' \} \big\rangle_{G, m} \\
		& \ge \left( \displaystyle\frac{2 |\textbf{d}'| + 1}{6G + 2m - 3} \right) \theta_{g, n} = \theta_{g, n},
		\end{aligned}
		\end{flalign}
		
		\noindent where in the last equality we used the identity $|\textbf{d}'| = |\textbf{d}| = 3G + m - 2$. Since \eqref{thetatheta} gives $\theta_{g, n} \ge \theta_{g - 1, n + 2}$, by ranging over $\textbf{d} \in \Omega (g, n + 1)$ in \eqref{dgm1} we deduce 
		\begin{flalign*}
		\theta_{g, n + 1} \ge \theta_{g, n} \ge \displaystyle\frac{2 \theta_{g, n}}{3} + \displaystyle\frac{\theta_{g - 1, n + 2}}{3},
		\end{flalign*}
		
		\noindent which implies \eqref{thetagn1}.  
		
		The second is if $k = 0$, in which case \eqref{d1recursion} yields (using the fact that $m \ge 2$)
		\begin{flalign}
		\label{dgmk0}
		\begin{aligned}
		\langle \textbf{d} \rangle_{G, m + 1} = \langle \textbf{d}', 1 \rangle_{G, m + 1} = \left(\displaystyle\frac{6G + 3m - 6}{6G + 2m - 3} \right) \langle \textbf{d}' \rangle_{G, m}  & \ge \left( 1 - \displaystyle\frac{1}{6G}\right) \theta_{g, n}\\
		&  \ge \left( \displaystyle\frac{2\theta_{g, n}}{3} + \displaystyle\frac{\theta_{g - 1, n + 2}}{3} \right) \left(1 - \displaystyle\frac{1}{6g} \right),
		\end{aligned}
		\end{flalign}
		
		\noindent where in the last inequality we used the fact that $G \ge g$ and again the bound $\theta_{g, n} \ge \theta_{g - 1, n + 2}$. Ranging over $\textbf{d} \in \Omega (g, n + 1)$ in \eqref{dgmk0} then gives \eqref{thetagn1}. 
		
		The third is if $k \ge 1$, in which case \eqref{recursionintersection2} yields
		\begin{flalign*}
		\langle \textbf{d} \rangle_{G, m + 1} = \langle k + 1, \textbf{d}' \rangle_{G, m + 1} \ge \displaystyle\frac{\theta_{g, n}}{6G + 2m - 3} \displaystyle\sum_{j = 1}^m (2d_j' + 1) + \displaystyle\frac{12G k \theta_{g - 1, n + 2}}{(6G + 2m - 3) (6G + 2m - 5)},
		\end{flalign*}
		
		\noindent where we have used the nonnegativity of the intersection numbers \eqref{definitiond} to omit the last term on the right side of \eqref{recursionintersection2}. Therefore, using the fact that $|\textbf{d}'| = |\textbf{d}| - k - 1 = 3G + m - k - 3$, we obtain
		\begin{flalign}
		\label{dk2estimate} 
		\begin{aligned} 
		\langle \textbf{d} \rangle_{G, m + 1} & \ge \displaystyle\frac{\big( 2 |\textbf{d}'| + m \big) \theta_{g, n}}{6G + 2m - 3} + \displaystyle\frac{12G k \theta_{g - 1, n + 2}}{(6G + 2m - 3) (6G + 2m - 5)} \\
		& =\displaystyle\frac{(6G + 3m - 2k - 6) \theta_{g, n}}{6G + 2m - 3} + \displaystyle\frac{12G k \theta_{g - 1, n + 2}}{(6G + 2m - 3) (6G + 2m - 5)}.
		\end{aligned} 
		\end{flalign}
		
		\noindent Now, letting 
		\begin{flalign*}
		U = \displaystyle\frac{6G + 3m - 2k - 6}{6G + 2m - 3}; \qquad V = \displaystyle\frac{12G k}{(6G + 2m - 3) (6G + 2m - 5)},
		\end{flalign*} 
		
		\noindent we claim that 
		\begin{flalign}
		\label{uvestimate}
		U + V \ge 1 - \displaystyle\frac{1}{6g}; \qquad V \le \displaystyle\frac{1}{3}.
		\end{flalign}
		
		\noindent Assuming \eqref{uvestimate}, it follows by ranging over $\textbf{d} \in \Omega (g, n + 1)$ in \eqref{dk2estimate} and the bound $\theta_{g, n} \ge \theta_{g - 1, n + 2}$ that 
		\begin{flalign*}
		\theta_{g, n + 1} & \ge U \theta_{g, n} + V \theta_{g - 1, n + 2} \ge \left( U + V - \displaystyle\frac{1}{3} \right) \theta_{g, n} + \displaystyle\frac{\theta_{g - 1, n + 2}}{3} \ge \left( \displaystyle\frac{2}{3} - \displaystyle\frac{1}{6g} \right) \theta_{g, n} + \displaystyle\frac{\theta_{g - 1, n + 2}}{3},
		\end{flalign*}
		
		\noindent which implies \eqref{thetagn1}. 
		
		So, it remains to verify both bounds in \eqref{uvestimate}. We begin with the latter. To that end, observe that since $k + 1 = \min_{j \in [1, m + 1]} d_j$, we have $(m + 1) (k + 1) \le |\textbf{d}| = 3G + m - 2$. Thus, 
		\begin{flalign}
		\label{g1k} 
		k \le \displaystyle\frac{3G - 3}{m + 1} \le G - 1,
		\end{flalign}
		
		\noindent where in the last bound we used the fact that $m \ge 2$. Hence \eqref{g1k}; the fact that $G \ge k + 1 \ge 2$, which follows from \eqref{g1k} and the bound $k \ge 1$; and the fact that $m \ge 2$ together imply 
		\begin{flalign*}
		V \le \displaystyle\frac{12 Gk}{(6G + 1) (6G - 1)} \le  \displaystyle\frac{12 G (G - 1)}{36G^2 - 1} \le \displaystyle\frac{1}{3}.
		\end{flalign*} 
		
		\noindent This establishes the second estimate in \eqref{uvestimate}. To verify the first, observe that 
		\begin{flalign}
		\label{uv1}
		U + V - 1 & = \displaystyle\frac{(6G + 2m - 5)(m - 3) - 2k (2m - 5)}{(6G + 2m - 3) (6G + 2m - 5)}.
		\end{flalign}
		
		\noindent If $m \ge 4$, then \eqref{g1k} implies 
		\begin{flalign*} 
		(6G + 2m - 5)(m - 3) - 2k (2m - 5) & \ge 6G (m - 3) - 2k (2m - 5) \ge (6G - 6k) (m - 3) \ge 0, 
		\end{flalign*} 
		
		\noindent and so by \eqref{uv1} the first bound in \eqref{uvestimate} holds. If $m = 3$, then \eqref{uv1} and \eqref{g1k} together yield 
		\begin{flalign*}
		U + V - 1 & = \displaystyle\frac{-2k}{(6G + 3) (6G + 1)} \ge \displaystyle\frac{2 - 2G}{(6G + 3) (6G + 1)} \ge - \displaystyle\frac{1}{18G} \ge -\displaystyle\frac{1}{18 g},
		\end{flalign*}
		
		\noindent which again verifies the first bound in \eqref{uvestimate}. 
		 
		If $m = 2$, then \eqref{uv1} implies that
		\begin{flalign*}
		U + V - 1 = \displaystyle\frac{1 - 6G + 2k}{36G^2 - 1} \ge - \displaystyle\frac{1}{6G} \ge - \displaystyle\frac{1}{6g},
		\end{flalign*}
		
		\noindent which verifies \eqref{uvestimate} and therefore \eqref{thetagn1} in the case when $k \ge 1$. This confirms \eqref{thetagn1} in all cases, thereby establishing the lemma. 	
	\end{proof}

	\subsection{Comparison to a Random Walk}

	\label{Thetaw} 
	
	In view of \eqref{thetagn1} (and omitting the factor of $1 - \frac{1}{4g}$, which should tend to $1$ as $g$ tends to $\infty$), one might view the $n$-parameter in $\theta_{g, n}$ as performing an asymmetric simple random walk with left jump probability $\frac{2}{3}$ and right jump probability $\frac{1}{3}$. The following provides notation for this walk, that starts at some integer $n \ge 3$ is absorbed at $2$.

	\begin{definition}
	
	\label{wt} 
	
	Fix an integer $n \ge 3$, and define the random function $w_n (t): \mathbb{Z}_{\ge 0} \rightarrow \mathbb{Z}_{\ge 2}$ as follows. Set $w_n (0) = n$ and, for $t \ge 1$, define $w_n (t)$ through the following recursive procedure. 
	
	\begin{enumerate} 
		\item If $w_n (t - 1) > 2$, then set $w_n (t) = w_n (t - 1) - 1$ with probability $\frac{2}{3}$ and $w (t) = w_n (t - 1) + 1$ with probability $\frac{1}{3}$. 
		\item If $w_n (t - 1) = 2$, then set $w_n (t) = 2$. 
	\end{enumerate} 

	\noindent We further let $\mathbb{P} = \mathbb{P}_n$ and $\mathbb{E} = \mathbb{E}_n$ denote the probability measure and expectation with respect to $w_n$, respectively.

	\end{definition} 

	Next we define a quantity associated with this walk $w_n (t)$ that we will use to lower bound $\theta_{g, n}$ in \Cref{destimatef} below. 

	\begin{definition}
		\label{d2} 
		
		For any integers $g > t \ge 0$ and $n \ge 2$, define $f_{g, n} (t)$ as follows. 
		
		\begin{enumerate}
			
			\item If $n = 2$, then set $f_{g, n} (t) = \frac{6g - 3}{6g - 1}$.
			
			\item If $n \ge 3$, then define the absorbing random walk $w_n (t)$ as in \Cref{wt} and set
			\begin{flalign*}
			f_{g, n} (t) = \bigg( 1 - \displaystyle\frac{1}{4g - 4t} \bigg)^t \bigg( 1 - \displaystyle\frac{2}{6g - 6t - 1} \bigg) \mathbb{P} \big[ w_n (t) = 2 \big].
			\end{flalign*}
			
	\end{enumerate}
	
	\end{definition} 

	\begin{prop}
	
	\label{destimatef}
	
	For any integers $g > t \ge 0$ and $n \ge 2$, we have $\theta_{g, n} \ge f_{g, n} (t)$. 
	\end{prop} 

	\begin{proof}
		
		If $n = 2$, then \Cref{n2} implies that $\theta_{g, n} \ge \frac{6g - 3}{6g - 1} \ge f_{g, n} (t)$, so we may assume $n \ge 3$. 
		
		In this case, we induct on $t$. If $t = 0$, then since $n \ge 3$ we have $f_{g, n} (0) = 0$ (as $w_n (0) = n \ne 2$), and so $\theta_{g, n} \ge f_{g, n} (t)$ by the nonnegativity of the intersection numbers \eqref{definitiond}. Thus, let us suppose for some integer $T \ge 1$ that $\theta_{g, n} \ge f_{g, n} (t)$ holds for any $t \in [0, T - 1]$ when $g > t$, and we will show that $\theta_{g, n} \ge f_{g, n} (T)$ when $g > T$.
		
		To that end, observe by \Cref{thetagnestimate} (with the $n + 1$ there equal to $n$ here) that 
		\begin{flalign*}
		\theta_{g, n} & \ge \left( \displaystyle\frac{2 \theta_{g, n - 1}}{3} + \displaystyle\frac{\theta_{g - 1, n + 1}}{3} \right) \left( 1 - \displaystyle\frac{1}{4 g} \right) \ge  \left( \displaystyle\frac{2 f_{g, n - 1} (T - 1)}{3} + \displaystyle\frac{f_{g - 1, n + 1} (T - 1)}{3} \right) \left( 1 - \displaystyle\frac{1}{4g} \right).
		\end{flalign*}
		
		\noindent Hence, by \Cref{d2}, 
		\begin{flalign}
		\label{thetagn2}
		\begin{aligned} 
		\theta_{g, n + 1} & \ge \displaystyle\frac{2}{3} \left( 1 - \displaystyle\frac{1}{4g} \right) \left( 1 - \displaystyle\frac{1}{4g - 4T + 4} \right)^{T - 1} \left( 1 - \displaystyle\frac{2}{6g - 6T + 5} \right) \mathbb{P} \big[ w_{n - 1} (T - 1) = 2 \big]  \\
		& \qquad + \displaystyle\frac{1}{3} \left( 1 - \displaystyle\frac{1}{4g} \right) \left( 1 - \displaystyle\frac{1}{4g - 4T} \right)^{T - 1} \left( 1 - \displaystyle\frac{2}{6g - 6T - 1} \right) \mathbb{P} \big[ w_{n + 1} (T - 1) = 2 \big]  \\
		& \ge \left( 1 - \displaystyle\frac{1}{4g} \right) \left( 1 - \displaystyle\frac{1}{4g - 4T} \right)^{T - 1} \left( 1 - \displaystyle\frac{2}{6g - 6T - 1} \right) \\
		& \qquad \times  \left( \displaystyle\frac{2}{3} \mathbb{P} \big[ w_{n - 1} (T - 1) = 2 \big] + \displaystyle\frac{1}{3} \mathbb{P} \big[ w_{n + 1} (T - 1) = 2 \big] \right).
		\end{aligned}
		\end{flalign} 
		
		\noindent Now, observe that 
		\begin{flalign*}
		\mathbb{P} \big[ w_n (T) = 2 \big] & = \mathbb{P} \big[ w_n (1) = n - 1 \big] \mathbb{P} \big[ w_{n - 1} (T - 1) = 2 \big] + \mathbb{P} \big[ w_n (1) = n + 1 \big] \mathbb{P} \big[ w_{n + 1} (T - 1) = 2 \big] \\
		& = \displaystyle\frac{2}{3} \mathbb{P} \big[ w_{n - 1} (T - 1) = 2 \big] + \displaystyle\frac{1}{3} \mathbb{P} \big[ w_{n + 1} (T - 1) = 2 \big],
		\end{flalign*}
		
		\noindent which upon insertion into \eqref{thetagn2} yields 
		\begin{flalign*}
		\theta_{g, n + 1} & \ge  \left( 1 - \displaystyle\frac{1}{4g} \right) \left( 1 - \displaystyle\frac{1}{4g - 4T} \right)^{T - 1} \left( 1 - \displaystyle\frac{2}{6g - 6T - 1} \right) \mathbb{P} \big[ w_n (T) = 2 \big] \\
		& \ge \left( 1 - \displaystyle\frac{1}{4g - 4T} \right)^T \left( 1 - \displaystyle\frac{2}{6g - 6T - 1} \right) \mathbb{P} \big[ w_n (T) = 2 \big] = f_{g, n} (T),
		\end{flalign*} 
		
		\noindent from which we deduce the proposition.
	\end{proof}

	\subsection{Proof of \Cref{dlower}}
	
	\label{LowerdProof}  
	
	By \Cref{destimatef}, to show \Cref{dlower} it remains to lower bound $f_{g, n} \ge 1 - o (1)$, to which end by \Cref{d2} it suffices to show that $\mathbb{P} \big[ w_n (t) = 2 \big] \approx 1$ for some choice of $t = o (g)$. Since $w_n (t)$ has a drift of $\frac{1}{3}$ to the left, by the law of large numbers, one expects for $w_n (t) = 2$ to hold for $n$ substantially larger than $3n$. Although this can be quickly made precise, we will proceed slightly differently through an exponential weighting estimate given by the following lemma, as this will also be useful in the proof of the upper bound \Cref{duppern} in \Cref{ProofUpperd} below.

	\begin{lem}
		
	\label{wt1} 
	
	Recalling $w_n (t)$ from \Cref{wt}, we have for any integers $n \ge 3$ and $t \ge 0$ that 
	\begin{flalign*} 
	\mathbb{E} \Bigg[ \bigg( \displaystyle\frac{3}{2} \bigg)^{w_n (t)} \textbf{\emph{1}}_{w_n (t) > 2} \Bigg] \le \left( \displaystyle\frac{2}{3} \right)^{t / 10 - n}. 
	\end{flalign*} 	
	
	\end{lem}

	\begin{proof}

	For any integers $n \ge 3$ and $t \ge 0$, define  
	\begin{flalign*} 
	\mathcal{W}_n (t) = \mathbb{E} \Bigg[ \bigg( \displaystyle\frac{3}{2} \bigg)^{w_n (t)} \textbf{1}_{w_n (t) > 2} \Bigg].
	\end{flalign*}
	
	\noindent We claim for any integer $t \ge 1$ that 
	\begin{flalign} 
	\label{twntwn1} 
	\mathcal{W}_n (t) \le \displaystyle\frac{17}{18} \mathcal{W}_n (t - 1).
	\end{flalign} 
	
	\noindent Given \eqref{twntwn1}, we can quickly establish the lemma. Indeed, by induction on $t$ and the fact that $\mathcal{W}_n (0) = \big( \frac{3}{2} \big)^n$ (since $w_n (0) = n$ and $n \ge 3$), we have 
	\begin{flalign*} 
	\mathcal{W}_n (t) \le \left( \displaystyle\frac{17}{18} \right)^t \mathcal{W}_n (0) \le \left( \displaystyle\frac{17}{18} \right)^t \left( \displaystyle\frac{3}{2} \right)^n,
	\end{flalign*} 
	
	\noindent from which the lemma follows since $\frac{17}{18} \le \big( \frac{2}{3} \big)^{1/10}$.
	
	It therefore remains to establish \eqref{twntwn1}. To that end, since $\mathbb{P} \big[ w_n (t) = w_n (t - 1) - 1 \big] = \frac{2}{3}$ and $\mathbb{P} \big[ w_n (t) = w_n (t - 1) + 1 \big] = \frac{1}{3}$ whenever $w_n (t - 1) > 2$, we have 
	\begin{flalign*}
	\mathcal{W}_n (t) = \mathbb{E} \Bigg[ \bigg( \displaystyle\frac{3}{2} \bigg)^{w_n (t)} \textbf{1}_{w_n (t) > 2} \Bigg] & \le \displaystyle\frac{2}{3} \mathbb{E} \Bigg[ \bigg( \displaystyle\frac{3}{2} \bigg)^{w_n (t - 1)} \textbf{1}_{w_n (t - 1) > 2} \Bigg] \mathbb{P} \big[ w_n (t) = w_n (t - 1) - 1  \big] \\
	& \qquad + \displaystyle\frac{3}{2} \mathbb{E} \Bigg[ \bigg( \displaystyle\frac{3}{2} \bigg)^{w_n (t - 1)} \textbf{1}_{w_n (t - 1) > 2} \Bigg] \mathbb{P} \big[ w_n (t) = w_n (t - 1) + 1  \big] \\
	& = \displaystyle\frac{17}{18} \mathbb{E} \Bigg[ \bigg( \displaystyle\frac{3}{2} \bigg)^{w_n (t - 1)} \textbf{1}_{w_n (t - 1) > 2} \Bigg] = \displaystyle\frac{17}{18} \mathcal{W}_n (t - 1),
	\end{flalign*}	
	
	\noindent which verifies \eqref{twntwn1} and therefore establishes the lemma. 
	\end{proof}
 
 	We can then deduce as a corollary that $w_n (t) = 2$ likely holds for $t$ much larger than $10n$. 

	\begin{cor}
		
	\label{2wt} 
	
	For any integers $n \ge 3$ and $t \ge 0$, we have 
	\begin{flalign*}
	\mathbb{P} \big[ w_n (t) \ne 2 \big] < \left( \displaystyle\frac{2}{3} \right)^{t / 10 - n}.
	\end{flalign*} 
	\end{cor}
	
	\begin{proof}
		
		Since $w_n (t) \ge 2$ and $\big( \frac{3}{2} \big)^{w_n (t)} \ge 1$ whenever $w_n (t) \ge 3$, we have  
		\begin{flalign*}
		\mathbb{P} \big[ w_n (t) \ne 2 \big] = \mathbb{P} \big[ w_n (t) > 2 \big] \le \mathbb{E} \Bigg[ \bigg( \displaystyle\frac{3}{2} \bigg)^{w_n (t)} \textbf{1}_{w_n (t) > 2} \Bigg] \le \left( \displaystyle\frac{3}{2} \right)^{t / 10 - n}, 
		\end{flalign*}
		
		\noindent where in the last inequality we applied \Cref{wt1}; this yields the corollary.
	\end{proof}

	Now we can quickly establish \Cref{dlower}. 

	\begin{proof}[Proof of \Cref{dlower}]
		
		If $n = 1$, then the proposition follows from \Cref{n1}, so we may assume that $n \ge 2$. Throughout the remainder of this proof, we set $t = 10n + 30 \lceil \log g \rceil$. 
		
		Since $10n < \frac{g}{3}$, we have for $g > 2^{15}$ that $\frac{g}{2} > t$. Thus, for $g > 2^{15}$, \Cref{destimatef} and \Cref{d2} together give
		\begin{flalign}
		\label{thetagn3} 
		\begin{aligned} 
		\theta_{g, n} & \ge \left( 1 - \displaystyle\frac{1}{4g - 4t} \right)^t \left( 1 - \displaystyle\frac{2}{6g - 6t - 1} \right) \mathbb{P} \big[ w_n (t) = 2 \big] \\
		& \ge \left( 1 - \displaystyle\frac{1}{2g} \right)^{t + 3} \Big( 1 - \mathbb{P} \big[ w_n (t) > 2 \big] \Big) \ge \left( 1 - \displaystyle\frac{t}{g} \right) \Big( 1 - \mathbb{P} \big[ w_n (t) \ne 2 \big] \Big).
		\end{aligned}
		\end{flalign}
		
		\noindent By \Cref{2wt}, and the facts that $\frac{t}{10} - n \ge 3 \log g$ and $\big( \frac{2}{3} \big)^3 < e^{-1}$, we have 
		\begin{flalign*}
		\mathbb{P} \big[ w_n (t) \ne 2 \big] \le \bigg( \frac{2}{3} \bigg)^{t / 10 - n} \le \bigg( \displaystyle\frac{2}{3} \bigg)^{3 \log g} < \displaystyle\frac{1}{g}.
		\end{flalign*} 
		
		\noindent Hence, it follows from \eqref{thetagn3} that 
		\begin{flalign*}
		\langle \textbf{d} \rangle_{g, n} \ge \theta_{g, n} \ge \left( 1 - \displaystyle\frac{t}{g} \right) \left( 1 - \displaystyle\frac{1}{g} \right) \ge 1 - \displaystyle\frac{2t}{g},
		\end{flalign*} 
		
		\noindent for any $\textbf{d} = (d_1, d_2, \ldots , d_n) \in \mathbb{Z}_{\ge 0}^n$ with $n \ge 2$ and $|\textbf{d}| = 3g + n - 3$. Together with the fact that $t = 10 n + 30 \lceil \log g \rceil \le 10 n + 40 \log g$ for $g > 2^{15}$, this yields the proposition.
	\end{proof}

		\section{Upper Bound on \texorpdfstring{$\langle \textbf{d} \rangle$}{}} 
		
		\label{Upperd}
		
		In this section we establish \Cref{duppern}. As in the proof of \Cref{dlower}, we begin with a recursive upper bound on $\big\langle \textbf{d} \rangle_{g, n}$ in \Cref{Recursion2}. Then, in \Cref{ProofUpperd}, we interpret this recursion through a random walk, which we then analyze to prove \Cref{duppern}.
		
		\subsection{Recursive Estimate for \texorpdfstring{$\Theta_{g, n}$}{}}
		
		\label{Recursion2} 
		
		To establish \Cref{duppern}, we proceed as in the proof of \Cref{dlower}. More specifically, we first establish a recursive estimate for the quantity $\Theta_{g, n}$ from \Cref{omegagn} and then bound it by a function associated with the random walk $w_n (t)$ from \Cref{wt}. The following lemma implements the former task.  
				
		\begin{lem} 
			
		\label{thetagnupper}
		
		For any integers $g, n \ge 2$ with $g \ge 2n$, we have 
		\begin{flalign}
		\label{gthetagn2}
		\Theta_{g, n + 1} \le \left( \displaystyle\frac{2 \Theta_{g, n} }{3} + \displaystyle\frac{\Theta_{g - n, n + 2}}{3} \right) \left( 1 + \displaystyle\frac{n + 10}{4g} \right). 
		\end{flalign}

		\end{lem} 
	
		Unlike in the proof of \Cref{thetagnestimate}, providing a recursive estimate for the lower bound $\theta_{g, n}$, to establish \Cref{thetagnupper} we must now estimate the third term on the right side of \eqref{recursionintersection2}. This bound is given by \Cref{sumtheta} below but, to prove it, we first require the following binomial coefficient estimate. 
		
		\begin{lem} 
			
			\label{gestimaten} 
			
			Let $g, g' \ge 1$ denote integers such that $g \ge 2g'$. Then, 
			\begin{flalign*}
			\binom{2g - 2}{2g' - 1} \ge 2^{2g' - 1} (g - 1).
			\end{flalign*}
		\end{lem} 	
	
		\begin{proof}
			
		This follows from the fact that
		\begin{flalign*}
		\binom{2g - 2}{2g' - 1} = (2g - 2) \displaystyle\prod_{j = 1}^{2g' - 2} \displaystyle\frac{2g - j - 2}{j + 1} \le (2g - 2) 2^{2g' - 2},
		\end{flalign*}
		
		\noindent where the last bound holds since $2g \ge 2j + 3$ for each $j \in [1, 2g' - 2]$ (as $g \ge 2g'$). 
		\end{proof}

		\begin{lem}
			
		\label{sumtheta} 
		
		Fix integers $g \ge 0$, $k \ge 1$, and $n \ge 2$ such that  $g \ge 2n$; let $\textbf{\emph{d}} = (d_1, d_2, \ldots , d_n) \in \mathbb{Z}_{\ge 0}^n$ denote an $n$-tuple of nonnegative integers such that $|\textbf{\emph{d}}| = 3g + n - k - 3$; and assume that $k + 1 \le \min_{j \in [1, n]} d_j$. Define $T = T(k + 1; \textbf{\emph{d}})$ by 
		\begin{flalign*}
		T = \displaystyle\frac{1}{2} \displaystyle\sum_{\substack{r + s = k - 1 \\ r, s \ge 0}} \displaystyle\sum_{\substack{I \cup J = \{ 1, 2, \ldots , n \} \\ |I \cap J| = 0}}  \displaystyle\frac{g!}{g'! g''!} \displaystyle\frac{(6g' + 2n' - 3)!! (6g'' + 2n'' - 3)!!}{(6g + 2n - 3)!!} & \langle r, \textbf{\emph{d}} |_I  \rangle_{g', n' + 1} \langle s, \textbf{\emph{d}} |_J \rangle_{g'', n'' + 1},
		\end{flalign*}
		
		\noindent where we have set $|I| = n'$, $|J| = n''$, and have adopted the notation from \eqref{dsdidj}. Then, we have
		\begin{flalign*}
		T \le \displaystyle\frac{\Theta_{g - n, n + 1}}{2 g} + \displaystyle\frac{1}{2g}.
		\end{flalign*}
		\end{lem}
	
		\begin{proof}
			
			First observe that \eqref{dsdidj} and the fact that $r + s = k - 1$ together imply that $g' + g'' = g$ and $n' + n'' = n$. Moreover, \eqref{dsdidj} and the fact that $d_j \ge k + 1 \ge 2$ for each $i \in [1, n]$ together imply that $3g' \ge |\textbf{d}_I| - n' + 2 \ge n' + 2$ and $3g'' \ge |\textbf{d}_J| - n'' + 2 \ge n'' + 2$. Defining
			\begin{flalign*}
			g_0 (m) = \left\lceil \displaystyle\frac{m + 2}{3} \right\rceil, \quad \text{for any integer $m \ge 0$}, 
			\end{flalign*}
			
			\noindent these imply that 
			\begin{flalign*}
			g' \ge g_0 (n') \ge 1; \qquad g'' \ge g_0 (n'') \ge 1; \qquad g = g' + g'' \ge \left\lceil \displaystyle\frac{n + 4}{3} \right\rceil \ge 2.
			\end{flalign*}
			
			Next, observe (as in the proof of \Cref{sumgn1}) that for a fixed pair of nonnegative integers $(n', n'')$ with $n' + n'' = n$, there are $\binom{n}{n'}$ choices for a pair of nonempty, disjoint sets $(I, J)$ such that $I \cup J = \{ 1, 2, \ldots , n \}$; $|I| = n'$; and $|J| = n''$. Fixing these two subsets and a pair of nonnegative integers $(r, s)$ with $r + s = k - 1$ also fixes $g' \in \big[g_0 (n'), g - g_0 (n'') \big]$ and $g'' \in \big[ g_0 (n''), g - g_0 (n') \big]$ through \eqref{dsdidj}. Thus,	
			\begin{flalign}
			\label{estimatet1}
			\begin{aligned} 	
			T & \le \displaystyle\frac{1}{2} \displaystyle\sum_{g' = g_0 (n')}^{g - g_0 (n'')} \displaystyle\sum_{n' = 0}^n \binom{n}{n'} \binom{g}{g'} \displaystyle\frac{(6g' + 2n' - 3)!! (6g'' + 2n'' - 3)!!}{(6g + 2n - 3)!!} \langle r, \textbf{d} |_I  \rangle_{g', n' + 1} \langle s, \textbf{d} |_J \rangle_{g'', n'' + 1} \\
			& \le \displaystyle\frac{1}{2 (6g + 2n - 3)} \displaystyle\sum_{n' = 0}^n \displaystyle\sum_{g' = g_0 (n')}^{g - g_0 (n'')} \binom{2g - 2}{2g' - 1}^{-1} \Theta_{g', n' + 1} \Theta_{g - g', n - n' + 1},
			\end{aligned}
			\end{flalign}
			
			\noindent where in the last inequality we used \Cref{estimateproduct} and the facts that $\langle r, \textbf{d} |_I  \rangle_{g', n' + 1} \le \Theta_{g', n' + 1}$ and $\langle s, \textbf{d} |_J \rangle_{g'', n'' + 1} \le \Theta_{g'', n'' + 1} = \Theta_{g - g', n - n' + 1}$. By the symmetry on the right side of \eqref{estimatet1} between $(n', g')$ and $(n - n', g - g')$, it follows that
			\begin{flalign}
			\label{2estimatet} 
			\begin{aligned}
			T & \le \displaystyle\frac{1}{6g + 2n - 3} \displaystyle\sum_{n' = 0}^n \displaystyle\sum_{g' = g_0 (n')}^{\lfloor g / 2 \rfloor} \binom{2g - 2}{2g' - 1}^{-1} \Theta_{g', n' + 1} \Theta_{g - g', n - n' + 1} \\
			& =  \displaystyle\frac{1}{6g + 2n - 3} \displaystyle\sum_{n' = 0}^n \big( X (n') + Y (n') \big),
			\end{aligned}
			\end{flalign}
			
			\noindent where $X(n') = X(n'; g, n)$ and $Y(n') = Y(n'; g, n)$ are defined by
			\begin{flalign*}
			& X (n') =  \displaystyle\sum_{g' = g_0 (n')}^n \binom{2g - 2}{2g' - 1}^{-1} \Theta_{g', n' + 1} \Theta_{g - g', n - n' + 1}; \\
			& Y (n') = \displaystyle\sum_{g' =  n + 1}^{\lfloor g / 2 \rfloor} \binom{2g - 2}{2g' - 1}^{-1} \Theta_{g', n' + 1} \Theta_{g - g', n - n' + 1}.
			\end{flalign*}
			
			\noindent We claim that
			\begin{flalign}
			\label{xy} 
			X(n') \le \displaystyle\frac{3 \Theta_{g - n, n + 1}}{g - 1}; \qquad Y(n') \le \displaystyle\frac{1}{2 (g - 1)}. 
			\end{flalign}
			
			\noindent To establish these estimates, first observe by the bound $g \ge 2n$, \Cref{gestimaten}, and \Cref{destimateexponential} that
			\begin{flalign}
			\label{gestimaten2} 
			\begin{aligned}
			\binom{2g - 2}{2g' - 1}^{-1} \Theta_{g', n' + 1} \Theta_{g - g', n - n' + 1} & \le \displaystyle\frac{1}{2^{2g' - 1} (g - 1)} \Theta_{g', n' + 1} \Theta_{g - g', n - n' + 1} \\
			& \le \displaystyle\frac{1}{2^{2g' - 1} (g - 1)} \left( \displaystyle\frac{3}{2} \right)^{n'} \Theta_{g - g', n - n' + 1}. 
			\end{aligned}
			\end{flalign} 
			
			\noindent Now, to verify the first estimate in \eqref{xy}, observe by \eqref{gestimaten2} that 
			\begin{flalign*}
			X (n') =  \displaystyle\sum_{g' = g_0 (n')}^n \binom{2g - 2}{2g' - 1}^{-1} \Theta_{g', n' + 1} \Theta_{g - g', n - n' + 1} & \le \displaystyle\frac{2}{g - 1} \left( \displaystyle\frac{3}{2} \right)^{n'} \displaystyle\sum_{g' = g_0 (n')}^n 2^{-2g'} \Theta_{g - g', n - n' + 1} \\
			& \le \displaystyle\frac{2}{g - 1} \left( \displaystyle\frac{3}{2} \right)^{n'} 2^{-2n' / 3} \displaystyle\sum_{j = 0}^n 2^{-2j} \Theta_{g - j, n - n' + 1},
			\end{flalign*}
			
			\noindent where in the last bound we changed variables $g' = g_0 (n') + j$ and used the fact that $g_0 (n') \ge \frac{n'}{3}$. Since $2^{-2/3} < \frac{2}{3}$, it follows that 
			\begin{flalign*}
			X (n') & \le \displaystyle\frac{2}{g - 1} \displaystyle\sum_{j = 0}^n 2^{-2j} \Theta_{g - j, n - n' + 1} \le \displaystyle\frac{2 \Theta_{g - n, n + 1}}{g - 1} \displaystyle\sum_{j = 0}^{\infty} 2^{-2j} \le \displaystyle\frac{3 \Theta_{g - n, n + 1}}{g - 1},
			\end{flalign*}
			
			\noindent where in the second inequality we used the fact that $\Theta_{g - n, n + 1} \ge \Theta_{g - j, n - n' + 1}$ for $j \le n$ (which holds by \eqref{thetatheta}). This yields the first bound in \eqref{xy}.  
			
			To prove the second, we again apply \eqref{gestimaten2} and \Cref{destimateexponential} to obtain that
			 \begin{flalign*}
			Y (n') \le \displaystyle\frac{2}{g - 1} \displaystyle\sum_{g' = n + 1}^{\lfloor g / 2 \rfloor} 2^{- 2g'} \left( \frac{3}{2} \right)^{n'} \Theta_{g - g', n - n' + 1} & \le \displaystyle\frac{2}{g - 1}\displaystyle\sum_{g' = n + 1}^{\lfloor g / 2 \rfloor} 2^{- 2g'} \left( \frac{3}{2} \right)^n \\
			& = \displaystyle\frac{1}{2 (g - 1)}\left( \frac{3}{4} \right)^n \displaystyle\sum_{j = 0}^{\infty} 2^{- 2j - n} \le \displaystyle\frac{1}{2 (g - 1)},
			 \end{flalign*}
			 
			 \noindent where in the third statement we changed variables $g' = n + j + 1$ and in the fourth we used the fact that $n \ge 1$. This establishes the second bound in \eqref{xy}. 
			 
			 Now, inserting \eqref{xy} into \eqref{2estimatet} yields 
			 \begin{flalign*}
			 T \le \displaystyle\frac{n + 1}{2 (6g + 2n - 3) (g - 1)} (6 \Theta_{g - n, n + 1} + 1) \le \displaystyle\frac{\Theta_{g - n, n + 1}}{2 g} + \displaystyle\frac{1}{2 g},
			 \end{flalign*}
			 
			 \noindent where in the last bound we used the facts that $6g + 2n - 3 \ge 6g$ (since $n \ge 2$) and that $n + 1 \le g - 1$ (since $g \ge 2n \ge 4$). This establishes the lemma.
		\end{proof} 
	
		Now we can establish \Cref{thetagnupper}.
	
		\begin{proof}[Proof of \Cref{thetagnupper}]

			To show \eqref{gthetagn2}, we proceed as in the proof of \Cref{thetagnestimate}. To that end, fix some $(m + 1)$-tuple $\textbf{d} = (d_1, d_2, \ldots , d_{m + 1}) \in \Omega (g, n + 1)$, for some integer $m \in [2, n]$; define $G \in \mathbb{Z}_{\ge 0}$ by $|\textbf{d}| = 3G + m - 2$. Let $k \ge - 1$ be such that $k + 1 = \min_{j \in [1, m + 1]} d_j$, and let $\textbf{d} = (k + 1, \textbf{d}')$, for some $\textbf{d}' = (d_1', d_2', \ldots ,d_m') \in \mathbb{Z}_{\ge 0}^m$. We consider three cases.
			
			The first is if $k = -1$, in which case \eqref{d0recursion} yields 
			\begin{flalign*}
			\langle \textbf{d} \rangle_{G, m + 1} = \langle \textbf{d}', 0 \rangle_{G, m + 1} & = \displaystyle\frac{1}{6G + 2m - 3} \displaystyle\sum_{j = 1}^m (2d_j' + 1) \big\langle d_j' - 1, \textbf{d}' \setminus \{ d_j' \} \big\rangle_{G, m} \\
			& \le \left( \displaystyle\frac{2 |\textbf{d}'| + m}{6G + 2m - 3} \right) \Theta_{g, n} = \left( 1 + \displaystyle\frac{m - 1}{6G + 2m - 3} \right) \Theta_{g, n},
			\end{flalign*}
			
			\noindent where in the last equality we used the identity $|\textbf{d}'| = |\textbf{d}| = 3G + m - 2$. Hence, since \eqref{thetatheta} gives $\Theta_{g, n} \le \Theta_{g - n, n + 2}$, by we obtain 
			\begin{flalign}
			\label{dgm3}
			\langle \textbf{d} \rangle_{G, m + 1} \le \left( \displaystyle\frac{2 \Theta_{g, n}}{3} + \displaystyle\frac{\Theta_{g - n, n + 2}}{3}  \right) \left( 1 + \displaystyle\frac{m - 1}{6G + 2m - 3} \right) \le \left( \displaystyle\frac{2 \Theta_{g, n}}{3} + \displaystyle\frac{\Theta_{g - n, n + 2}}{3}  \right) \left( 1 + \displaystyle\frac{n}{6g} \right),
			\end{flalign}
			
			\noindent where in the last bound we used the facts that $m \in [2, n]$ and $G \ge g$. Now \eqref{gthetagn2} follows from ranging over $\textbf{d} \in \Omega (g, n + 1)$ in \eqref{dgm3}.	
			
			The second is if $k = 0$, in which case \eqref{d1recursion} yields 
			\begin{flalign}
			\label{dgmk3}
			\begin{aligned}
			\langle \textbf{d} \rangle_{G, m + 1} = \langle \textbf{d}', 1 \rangle_{G, m + 1} & = \left( 1 + \displaystyle\frac{m - 3}{6G + 2m - 3} \right) \langle \textbf{d}' \rangle_{G, m} \\
			& \le \left( 1 + \displaystyle\frac{n}{6g} \right) \Theta_{g, n} \le \left( \displaystyle\frac{2 \Theta_{g, n}}{3} + \displaystyle\frac{\Theta_{g - n, n + 2}}{3}  \right) \left( 1 + \displaystyle\frac{n}{6g} \right),
			\end{aligned}
			\end{flalign}
			
			\noindent where in the third statement we used the facts that $m \in [2,  n]$ and $G \ge g$ and last inequality we again used the fact that $\Theta_{g, n} \le \Theta_{g - n, n + 2}$. Ranging over $\textbf{d} \in \Omega (g, n + 1)$ in \eqref{dgmk3} then gives \eqref{gthetagn2}. 
			
			The third is if $k \ge 1$, in which case \eqref{recursionintersection2} yields (recalling $T = T(k + 1; \textbf{d})$ from \Cref{sumtheta})
			\begin{flalign*}
			\langle \textbf{d} \rangle_{G, m + 1} = \langle k + 1, \textbf{d}' \rangle_{G, m + 1} & \le \displaystyle\frac{\Theta_{g, n}}{6G + 2m - 3} \displaystyle\sum_{j = 1}^m (2d_j' + 1) + \displaystyle\frac{12G k \Theta_{g - 1, n + 2}}{(6G + 2m - 3) (6G + 2m - 5)} + T.
			\end{flalign*}
			
			\noindent  Therefore, using the fact that $|\textbf{d}'| = |\textbf{d}| - k - 1 = 3G + m - k - 3$ and \Cref{sumtheta}, we obtain
			\begin{flalign}
			\label{dk2estimate1} 
			\begin{aligned} 
			\langle \textbf{d} \rangle_{G, m + 1} & \le \displaystyle\frac{(6G + 3m - 2k - 6) \Theta_{g, n}}{6G + 2m - 3} + \displaystyle\frac{12G k \Theta_{g - 1, n + 2}}{(6G + 2m - 3) (6G + 2m - 5)} + T \\
			& \le U \Theta_{g, n} + V \Theta_{g - 1, n + 2} + \displaystyle\frac{\Theta_{g - n, n + 1}}{2g} + \displaystyle\frac{1}{2 g},
			\end{aligned} 
			\end{flalign}
			
			\noindent where we have set 
			\begin{flalign*}
			U = \displaystyle\frac{6G + 3m - 2k - 6}{6G + 2m - 3}; \qquad V = \displaystyle\frac{12Gk}{(6G + 2m - 3) (6G + 2m - 5)}. 
			\end{flalign*}
			
			\noindent Now, similarly to \eqref{uvestimate}, we claim that 
			\begin{flalign} 
			\label{2uvestimate} 
			U + V \le 1 + \displaystyle\frac{n}{6g}; \qquad V \le \displaystyle\frac{1}{3}.
			\end{flalign}
			
			Assuming \eqref{2uvestimate}, we can quickly establish \eqref{gthetagn2}. Indeed, \eqref{dk2estimate1}, \eqref{2uvestimate}, and the bounds $\Theta_{g, n} \le \Theta_{g - 1, n + 2} \le \Theta_{g - n, n + 2}$ and $\Theta_{g - n, n + 1} \le \Theta_{g - n, n + 2}$ (which follow from \eqref{thetatheta}) together imply that 
			\begin{flalign}
			\label{dgm1k1} 
			\begin{aligned}
			\langle \textbf{d} \rangle_{G, m + 1} & \le U \Theta_{g, n} + V \Theta_{g - 1, n + 2} + \displaystyle\frac{\Theta_{g - n, n + 1}}{2 g} + \displaystyle\frac{1}{2 g} \\
			& \le \left( U + V - \displaystyle\frac{1}{3} \right) \Theta_{g, n} + \displaystyle\frac{\Theta_{g - 1, n + 2}}{3} + \displaystyle\frac{\Theta_{g - n, n + 1}}{2 g} + \displaystyle\frac{1}{2 g} \\
			& \le  \bigg( 1 + \displaystyle\frac{n + 1}{4g} \bigg) \displaystyle\frac{2 \Theta_{g, n}}{3} + \left( 1 + \displaystyle\frac{3}{2g} \right) \displaystyle\frac{\Theta_{g - n, n + 2}}{3} + \displaystyle\frac{1}{2 g} \\
			& \le \left( \displaystyle\frac{2\Theta_{g, n}}{3} + \displaystyle\frac{\Theta_{g - n, n + 2}}{3} \right) \left( 1 + \displaystyle\frac{n + 10}{4g} \right),
			\end{aligned}
			\end{flalign} 			
			
			\noindent where the last bound follows from the fact that $\Theta_{g, n} \ge \frac{1}{2}$ (which is a consequence of \Cref{n2}). So, \eqref{gthetagn2} follows from again ranging \eqref{dgm1k1} over $\textbf{d} \in \Omega (g, n + 1)$. 
			
			Hence, it remains to prove \eqref{2uvestimate}. The second estimate there was shown in \eqref{uvestimate}, so we must verify the first. To that end, observe that 
			\begin{flalign*}
			U + V & = 1 + \displaystyle\frac{(6G + 2m - 5) (m - 3) - 2k (2m - 5)}{(6G + 2m - 3) (6G + 2m - 5)} \\
			& \le 1 + \displaystyle\frac{m - 3}{6G + 2m - 3} + \displaystyle\frac{2k}{(6G + 2m - 3) (6G + 2m - 5)} \\
			& \le 1 + \displaystyle\frac{m - 2}{6G + 2m - 3} \le 1 + \displaystyle\frac{m}{6G} \le 1 + \displaystyle\frac{n}{6g}.
			\end{flalign*} 
			
			\noindent Here, in the second statement we used the fact that $2m - 5 \ge - 1$ (since $m \ge 2$); in the third we used the fact that $2k \le 2 \big( |\textbf{d}| - 1 \big) = 6 G + 2m - 6 < 6 G + 2m - 5$; in the fourth we used the fact that $6G + 2m - 3 \ge 6G$ (since $m \ge 2$); and in the fifth we used the facts that $m \le n$ and that $G \ge g$. This verifies the first bound in \eqref{2uvestimate} and therefore establishes \eqref{gthetagn2} in all cases. 
		\end{proof}

		\subsection{Proof of \Cref{duppern}}
		
		\label{ProofUpperd}	
		
		Analogously to \Cref{d2}, the following definition provides a quantity associated with the random walk $w_n (t)$ from \Cref{wt} that we will compare to $\Theta_{g, n}$ in \Cref{destimatefupper} below.

		\begin{definition}
			
		\label{fgnt}

		For any integers $t \ge 0$ and $g, n \ge 2$ such that $g > tn + t^2$, define $F_{g, n} (t)$ as follows. 
		
		\begin{enumerate}
			
			\item If $n = 2$, then set $F_{g, n} (t) = 1$.
			
			\item If $n \ge 3$, then define the absorbing random walk $w_n (t)$ as in \Cref{wt} and set
			\begin{flalign*}
			F_{g, n} (t) = \bigg( 1 + \displaystyle\frac{n + 2t + 9}{4 (g - tn - t^2)} \bigg)^t \Bigg( \mathbb{P} \big[ w_n (t) = 2 \big] + \mathbb{E} \bigg[ \Big( \displaystyle\frac{3}{2} \Big)^{w_n (t)} \textbf{1}_{w_n (t) > 2} \bigg] \Bigg).
			\end{flalign*}
			
		\end{enumerate} 
			
		\end{definition}

		\begin{prop}
			
			\label{destimatefupper}
			
			For any integers $t \ge 0$ and $g, n \ge 2$ such that $g > (t + 2) n + t^2$, we have $\Theta_{g, n} \le F_{g, n} (t)$. 
		\end{prop} 
		
		\begin{proof}
			
			If $n = 2$, then \Cref{n2} implies that $\Theta_{g, n} \le 1 = F_{g, n} (t)$, so we may assume $n \ge 3$. 
			
			In this case, upon setting  
			\begin{flalign}
			\label{hgn}
			H_{g, n} (t) = \bigg( 1 + \displaystyle\frac{n + 2t + 9}{4 (g - tn - t^2)} \bigg)^t \mathbb{E} \Big[ \Theta_{g - tn - t^2, w_n (t)} \Big],
			\end{flalign}
			
			\noindent it suffices to show that 
			\begin{flalign}
			\label{thetah} 
			\Theta_{g, n} \le H_{g, n} (t), \qquad \text{for $g > (t + 2) n + t^2$}. 
			\end{flalign} 
			
			\noindent Indeed, given \eqref{thetah}, the bound $\Theta_{g, n} \le F_{g, n} (t)$ follows from the fact that $F_{g, n} (t) \le H_{g, n} (t)$, which holds since 
			\begin{flalign*}
			\mathbb{E} \Big[ \Theta_{g - tn- t^2, w_n (t)} \Big] & = \mathbb{E} \Big[ \Theta_{g - tn - t^2, w_n (t)} \textbf{1}_{w_n (t) = 2} \Big] + \mathbb{E} \Big[ \Theta_{g - tn - t^2, w_n (t)} \textbf{1}_{w_n (t) > 2} \Big] \\
			& \le \mathbb{E} \big[ \textbf{1}_{w_n (t) = 2} \big] + \mathbb{E} \Bigg[ \bigg( \displaystyle\frac{3}{2} \bigg)^{w_n (t)} \textbf{1}_{w_n (t) > 2} \Bigg] \\
			& = \mathbb{P} \big[ w_n (t) = 2 \big] + \mathbb{E} \Bigg[ \bigg( \displaystyle\frac{3}{2} \bigg)^{w_n (t)} \textbf{1}_{w_n (t) > 2} \Bigg].
			\end{flalign*} 
			
			\noindent Here, the first statement follows from the fact that $w_n (t) \ge 2$, and the second follows from the facts that $\Theta_{g, 2} \le 1$ (by \Cref{n2}) and $\Theta_{g, n} < \big( \frac{3}{2} \big)^n$ (by \Cref{destimateexponential}).
			
			Therefore, it remains to show \eqref{thetah}, to which end we induct on $t$, similarly to as in the proof of \Cref{destimatef}. If $t = 0$, then \eqref{thetah} holds since $H_{g, n} (0) = \Theta_{g, n}$ (as $w_n (0) = n$). Thus, let us suppose for some integer $T \ge 1$ that $\Theta_{g, n} \le H_{g, n} (t)$ holds for any $t \in [0, T - 1]$ when $g > (t + 2) n + t^2$, and we will show that $\Theta_{g, n} \le H_{g, n} (T)$ when $g > (T + 2) n + T^2$.
			
			To that end, observe that
			\begin{flalign*}
			\Theta_{g, n} & \le \left( \displaystyle\frac{2 \Theta_{g, n - 1}}{3} + \displaystyle\frac{\Theta_{g - n + 1, n + 1}}{3} \right) \left( 1 + \displaystyle\frac{n + 9}{4g} \right) \\
			& \le  \left( \displaystyle\frac{2 H_{g, n - 1} (T - 1)}{3} + \displaystyle\frac{H_{g - n + 1, n + 1} (T - 1)}{3} \right) \left( 1 + \displaystyle\frac{n + 9}{4g} \right). 
			\end{flalign*}
			
			\noindent Here, in the first inequality we used \Cref{thetagnupper} (with the $n + 1$ there equal to $n$ here), which applies since $g > (T + 2) n + T^2 \ge 2 (n - 1)$, and in the third we used the facts that $\Theta_{g, n - 1} \le H_{g, n - 1} (T - 1)$ and $\Theta_{g - n + 1, n + 1} \le H_{g - n + 1, n + 1} (T - 1)$, which apply since $g \ge (T + 1) (n - 1) + (T - 1)^2$ and $g - n + 1 \ge (T + 1) (n + 1) + (T - 1)^2$ both hold if $g > (T + 2) n + T^2$ (as $T \ge 1$). Hence, by \eqref{hgn} 
			\begin{flalign*} 
			\Theta_{g, n + 1} \le  \bigg( 1 + \displaystyle\frac{n + 9}{4g} & \bigg) \Bigg( \displaystyle\frac{2}{3} \bigg( 1 + \displaystyle\frac{n + 2T + 6}{4 \big(g - (T - 1) (n - 1) - (T - 1)^2 \big)} \bigg)^{T - 1} \\
			& \qquad \qquad \times \mathbb{E} \Big[ \Theta_{g - (T - 1) (n - 1) - (T - 1)^2, w_{n - 1} (T - 1)} \Big] \\
			& \qquad + \displaystyle\frac{1}{3} \bigg( 1 + \displaystyle\frac{n + 2T + 8}{4 \big(g - n + 1 - (T - 1) (n + 1) - (T - 1)^2 \big)} \bigg)^{T - 1} \\
			& \qquad \qquad \qquad \times \mathbb{E} \Big[ \Theta_{g - n + 1 - (T - 1) (n + 1) - (T - 1)^2, w_{n + 1} (T - 1)}  \Big] \Bigg). 
			\end{flalign*} 
			
			\noindent Using the facts that $g - (T - 1) (n - 1) - (T - 1)^2 \ge g - Tn - T^2$ and $g - n + 1 - (T - 1) (n + 1) - (T - 1)^2 \ge g - Tn - T^2$ and applying \eqref{thetatheta}, we therefore deduce that
			\begin{flalign}
			\label{hthetagn}
			\begin{aligned} 
			\Theta_{g, n + 1} & \le \bigg( 1 + \displaystyle\frac{n + 9}{4g} \bigg) \Bigg( \displaystyle\frac{2}{3} \bigg( 1 + \displaystyle\frac{n + 2T + 6}{4 \big(g - Tn - T^2 \big)} \bigg)^{T - 1} \mathbb{E} \Big[ \Theta_{g - Tn - T^2, w_{n - 1} (T - 1)} \Big] \\
			& \qquad \qquad \qquad \qquad + \displaystyle\frac{1}{3} \bigg( 1 + \displaystyle\frac{n + 2T + 8}{4 \big(g - Tn - T^2 \big)} \bigg)^{T - 1} \mathbb{E} \Big[ \Theta_{g - Tn - T^2, w_{n + 1} (T - 1)}  \Big] \Bigg) \\
			& \le \bigg( 1 + \displaystyle\frac{n + 9}{4g} \bigg)  \bigg( 1 + \displaystyle\frac{n + 2T + 8}{4 \big(g - Tn - T^2 \big)} \bigg)^{T - 1} \\
			& \qquad \times \Bigg( \displaystyle\frac{2}{3} \mathbb{E} \Big[ \Theta_{g - Tn - T^2, w_{n - 1} (T - 1)} \Big] + \displaystyle\frac{1}{3} \mathbb{E} \Big[ \Theta_{g - Tn - T^2, w_{n + 1} (T - 1)}  \Big] \Bigg).
			\end{aligned}
			\end{flalign} 
			
			 Now, observe for any integer $G \ge 0$ that 
			\begin{flalign}
			\label{thetagnwt}
			\begin{aligned}
			\mathbb{E} \Big[ \Theta_{G, w_n (T)} \Big] & = \mathbb{P} \big[ w_n (1) = n - 1 \big] \mathbb{E} \Big[ \Theta_{G, w_{n - 1} (T - 1)} \Big] + \mathbb{P} \big[ w_n (1) = n + 1 \big] \mathbb{E} \Big[ \Theta_{G, w_{n + 1} (T - 1)} \Big] \\
			& = \displaystyle\frac{2}{3} \mathbb{E} \Big[ \Theta_{G, w_{n - 1} (T - 1)} \big] + \displaystyle\frac{1}{3} \mathbb{E} \Big[ \Theta_{G, w_{n + 1} (T - 1)} \Big],
			\end{aligned}
			\end{flalign}
			
			\noindent Inserting the $G = g - Tn - T^2$ case of \eqref{thetagnwt} into \eqref{hthetagn} yields 
			\begin{flalign*}	
			\Theta_{g, n + 1} & \le  \left( 1 + \displaystyle\frac{n + 9}{4g} \right) \left( 1 - \displaystyle\frac{n + 2T + 8}{4 (g - Tn - T^2)} \right)^{T - 1} \mathbb{E} \Big[ \Theta_{g - nT - T^2, w_n (T)} \Big] \\
			& \le \left( 1 + \displaystyle\frac{n + 2T + 9}{4 (g - Tn - T^2)} \right)^T \mathbb{E} \Big[ \Theta_{g - nT - T^2, w_n (T)} \big] = H_{g, n} (T),
			\end{flalign*} 
			
			\noindent from which we deduce \eqref{thetah} and therefore the proposition.
		\end{proof}
		
		Now we can establish \Cref{duppern}.

		\begin{proof}[Proof of \Cref{duppern}] 
			
			If $n \in \{ 1, 2 \}$, then \Cref{n1} and \Cref{n2} imply that $\langle \textbf{d} \rangle_{g, n} \le 1$, from which the proposition follows. So, we may assume in the below that $n \ge 3$. Throughout the remainder of this proof, we set $t = 10n + 30 \lceil \log g \rceil$. Now, observe that $g > (t + 2) n + t^2$, since 
			\begin{flalign}
			\label{t2n} 
			\begin{aligned}
			(t + 2) n + t^2 = 10n^2 + 2n + 30 n \lceil \log g \rceil + \big( 10n + 30 \lceil \log g \rceil \big)^2 & \le 227 n^2 + 1815 \lceil \log g \rceil^2 \\
			& < 400 n^2 + \displaystyle\frac{g}{2} < g,
			\end{aligned} 
			\end{flalign} 
			
			\noindent where in the second statement we used the facts that $30 n \lceil \log g \rceil \le 15n^2 + 15 \lceil \log g \rceil^2$, $n \ge 1$, and $(a + b)^2 \le 2a^2 + 2b^2$ (with $a = 10n$ and $b = 30 \lceil \log g \rceil$); in the third, we used the fact that $g > 2^{30}$; and, in the fourth, we used the fact that $g > 800 n^2$.
			
			So \Cref{destimatefupper} applies, which together with \Cref{fgnt}, yields (recalling the random walk $w_n (t)$ from \Cref{wt}) 
			\begin{flalign*}
			\langle \textbf{d} \rangle_{g, n} \le F_{g, n} (t) \le \left( 1 + \displaystyle\frac{n + 2t + 9}{4 (g - tn - t^2)} \right)^t \Bigg( 1 + \mathbb{E} \bigg[ \Big( \displaystyle\frac{3}{2} \Big)^{w_n (t)} \textbf{1}_{w_n (t) > 2} \bigg] \Bigg).
			\end{flalign*}  
			
			\noindent In view of \Cref{wt1}; the fact that $\frac{t}{10} = n + 3 \lceil \log g \rceil \le n + 4 \log g$ for $g > 2^{30}$; the fact that $g \ge 2 (tn + t^2)$ for $g > 2^{30}$ (following the proof of \eqref{t2n}, using the bound $g \ge 800 n^2$); and the estimates $\big( \frac{2}{3} \big)^3 < e^{-1}$ and $x + 1 \le e^x$ for any $x \ge 0$, it follows that 
			\begin{flalign*}
			\langle \textbf{d} \rangle_{g, n} \le F_{g, n} (t) & \le \left( 1 + \displaystyle\frac{20n + 20 \log g}{g - tn - t^2} \right)^t \Bigg( 1 +  \bigg( \displaystyle\frac{2}{3} \bigg)^{-3 \log g} \Bigg) \\
			& \le \left( 1 + \displaystyle\frac{ 40 n + 40 \log g}{g} \right)^{10n + 40 \log g} \bigg( 1 + \displaystyle\frac{1}{g} \bigg) \\
			& \le \exp \left( \displaystyle\frac{(40 n + 40 \log g) (10 n + 40 \log g)}{g} + \displaystyle\frac{1}{g} \right) \le \exp \big( g^{-1} (20n + 40 \log g)^2 \big),
			\end{flalign*}  
			
			\noindent which implies the proposition.		
		\end{proof}

		\section{Multi-variate Harmonic Sums}
		
		\label{Sumhkzk}

		In this section we provide estimates and asymptotics for certain types of multi-variate harmonic sums that will appear in the proof of \Cref{limitvolume}. After stating these results in \Cref{SumEstimates}, we introduce generating series for these sums in \Cref{IntegralSum}. Using these series, we will establish one of the results (\Cref{zkmestimate2}) from \Cref{SumEstimates} in \Cref{Estimatezk}; the other will be proven in \Cref{Asymptotichknzkn} below.

		\subsection{Estimates on Multi-variate Harmonic Sums} 
		
		\label{SumEstimates}

		Recalling $\mathcal{C}_k (m)$ from \Cref{Estimates1}, the following definition recalls two types of multi-variate harmonic sums from equations (1.45) and (1.46) of \cite{VFGINMSC}. 
		
		\begin{definition} 	
			
			\label{hkzk}
			
			For any integers $m \ge k \ge 1$, define 
			\begin{flalign}
			\label{hkzkdefinition} 
			H_k (m) = \displaystyle\sum_{\textbf{a} \in \mathcal{C}_k (m)} \displaystyle\prod_{i = 1}^k \displaystyle\frac{1}{a_i}; \qquad Z_k (m) = \displaystyle\sum_{\textbf{a} \in \mathcal{C}_k (m)} \displaystyle\prod_{i = 1}^k \displaystyle\frac{\zeta (2 a_i)}{a_i}, 
			\end{flalign}
			
			\noindent where we have denoted $\textbf{a} = (a_1, a_2, \ldots , a_k) \in \mathbb{Z}_{\ge 1}^k$. We also set $Z_0 (0) = 1$; $Z_0 (m) = 0$ for any integer $m \ge 1$ and $Z_k (m) = 0$ if $m < k$. 
			
		\end{definition}
	
		The series $Z_k (m)$ will appear in \Cref{VolumePrincipal}, \Cref{Estimatev2}, and \Cref{Estimateve1} below when analyzing the limiting behavior of $\Vol \mathcal{Q}_{g, n}$. Although the function $H_k (m)$ will not be relevant for the purposes of this paper, it will be used in the forthcoming work \cite{AGSSSMLG} to be established more geometric information about random surfaces in $\mathcal{Q}_{g, n}$. Since the analyses of $Z_k (m)$ and $H_k (m)$ are in any case entirely analogous, we provide the asymptotics for both here. 
		
		We will require several estimates and asymptotics on these sums. The first is given by the following lemma. 
		
		\begin{lem} 
			
			\label{zkproductsum} 
			
			Fix integers $m, j \ge 0$ and $r \ge 1$, and a composition $\textbf{\emph{j}} = (j_1, j_2, \ldots , j_r) \in \mathcal{K}_r (j)$. Then,
			\begin{flalign*}
			\displaystyle\sum_{\textbf{\emph{m}} \in \mathcal{K}_r (m)} \displaystyle\prod_{k = 1}^r Z_{j_k} (m_k) \le Z_j (m),
			\end{flalign*} 
			
			\noindent where we have denoted $\textbf{\emph{m}} = (m_1, m_2, \ldots , m_r) \in \mathcal{K}_r (m)$. 
		\end{lem}
		
		\begin{proof}
			
			Since $Z_0 (0) = 1$ and $Z_0 (m) = 0$ for $m > 0$, it suffices to address the situation when $j_k \ge 1$ for each $k$, namely, that $\textbf{j} \in \mathcal{C}_r (j)$. In this case, we have 
			\begin{flalign}
			\label{zsuma} 
			\displaystyle\sum_{\textbf{m} \in \mathcal{K}_r (m)} \displaystyle\prod_{k = 1}^r Z_{j_k} (m_k) = \displaystyle\sum_{\textbf{m} \in \mathcal{C}_r (m)} \displaystyle\sum_{\textbf{a}^{(1)} \in \mathcal{C}_{j_1} (m_1)} \cdots \displaystyle\sum_{\textbf{a}^{(r)} \in \mathcal{C}_{j_r} (m_r)} \displaystyle\prod_{k = 1}^r \displaystyle\prod_{i = 1}^{j_k} \displaystyle\frac{1}{a_i^{(k)}},
			\end{flalign} 
			
			\noindent where we have denoted the compositions $\textbf{a}^{(k)} = \big( a_1^{(j)}, a_2^{(j)}, \ldots , a_{j_k}^{(k)} \big) \in \mathcal{C}_{j_k} (m_k)$ (and used the fact that $Z_{j_k} (0) = 0$ to restrict the sum over $\textbf{m} \in \mathcal{K}_r (m)$ on the left side of \eqref{zsuma} to one over $\textbf{m} \in \mathcal{C}_r (m)$ on the right side). The facts that $\sum_{k = 1}^r j_k = j$ and $\sum_{k = 1}^r m_k = m$ together imply for any $\textbf{m} \in \mathcal{C}_r (m)$ that the (ordered) union $\textbf{a} = \bigcup_{k = 1}^r \textbf{a}^{(k)} \in \mathcal{K}_j (m)$. Since any $\textbf{a}$ can be obtained in this way from at most one family of compositions $\big( \textbf{m}, \textbf{a}^{(1)}, \textbf{a}^{(2)}, \ldots , \textbf{a}^{(k)} \big)$, \eqref{zsuma} yields
			\begin{flalign*}
			\displaystyle\sum_{\textbf{m} \in \mathcal{K}_r (m)} \displaystyle\prod_{k = 1}^r Z_{k_j} (m_j) \le \displaystyle\sum_{\textbf{a} \in \mathcal{C}_r (m)} \displaystyle\prod_{k = 1}^r \displaystyle\frac{1}{a_j},
			\end{flalign*} 
			
			\noindent from which we deduce the lemma.
		\end{proof}
	
		Next, we have the following bound on $Z_k (N)$ that will be established in \Cref{Estimatezk} below.
		
		\begin{lem}
			
			\label{zkmestimate2} 
			
			For any integers $N \ge k \ge 1$, we have 
			\begin{flalign*}
			Z_k (N) \le \displaystyle\frac{2 k (\log N + 5)^{k - 1}}{N}.
			\end{flalign*}
		\end{lem}

		The following proposition provides asymptotics for certain linear combinations of the series $H_k (N)$ and $Z_k (N)$. It was predicted as Conditional Theorem D.9 of \cite{VFGINMSC} and will be established in \Cref{Sumhz} below. 
		
		\begin{prop} 
			
			\label{sumkn2}
			
			For any fixed real number $\omega > \frac{1}{2}$, we have
			\begin{flalign}
			\label{zkestimate2}
			\displaystyle\lim_{N \rightarrow \infty} N^{1 / 2} \displaystyle\sum_{k = 1}^{\lfloor \omega \log N \rfloor} \displaystyle\frac{H_k (N)}{2^{k - 1} k!} = 2 \pi^{-1 / 2}; \qquad 
			\displaystyle\lim_{N \rightarrow \infty} N^{1 / 2} \displaystyle\sum_{k = 1}^{\lfloor \omega \log N \rfloor} \displaystyle\frac{Z_k (N)}{2^{k - 1} k!} = 2^{3 / 2} \pi^{-1 / 2}.
			\end{flalign}
			
		\end{prop}

		\subsection{Contour Integral Representations} 
		
		\label{IntegralSum}
		
		In this section we provide generating series, which lead to exact contour integral representations, for the sums $H_k (N)$ and $Z_k (N)$ (given by \eqref{ansum} below); we will then asymptotically analyze these integrals later, in \Cref{Asymptotichknzkn}. To that end, we begin with the following lemma. 
		
		In the below, we define $f, g, F_k, G_k: \mathbb{C} \setminus \mathbb{Z} \rightarrow \mathbb{C}$ by 
		\begin{flalign}
		\label{fgfg} 
		f(w) = - \log (1 - w);  \qquad g(w) = - \displaystyle\sum_{j = 1}^{\infty} \log \left( 1 - \displaystyle\frac{w}{j^2} \right); \qquad 
		F_k (w) = f(w)^k; \qquad G_k (w) = g(w)^k,
		\end{flalign}
		
		\noindent where here are using the principal branch of the logarithm. Observe that the series for $g(w)$ converges since $\big| \log \big(1 - \frac{w}{j^2} \big) \big| < \frac{2 |w|}{j^2}$, for sufficiently large $j$, and $\sum_{j = 1}^{\infty} \frac{|w|}{j^2} < \infty$. 
		
		\begin{lem}
			
			\label{fgsum}
			
			For any $w \in \mathbb{C}$ with $|w| < 1$, we have 
			\begin{flalign}
			\label{fwgw}
			F_k (w) = \displaystyle\sum_{m = k}^{\infty} H_k (m) w^m; \qquad G_k (w) = \displaystyle\sum_{m = k}^{\infty} Z_k (m) w^m.
			\end{flalign}
			
		\end{lem}
		
		\begin{proof}
			By \Cref{hkzk}, we have 
			\begin{flalign}
			\label{fg}
			\displaystyle\sum_{m = k}^{\infty} H_k (m) w^m = \left( \displaystyle\sum_{m = 1}^{\infty} \displaystyle\frac{w^m}{m} \right)^k; \qquad \displaystyle\sum_{m = k}^{\infty} Z_k (m) w^m= \left( \displaystyle\sum_{m = 1}^{\infty} \displaystyle\frac{\zeta(2m) w^m}{m} \right)^k.  
			\end{flalign}
			
			\noindent The first statement of \eqref{fwgw} then follows from \eqref{fg} and the fact that $f(w) = \sum_{m = 1}^{\infty} \frac{w^m}{m}$ for $|w| < 1$. To establish the second, observe that
			\begin{flalign*}
			\displaystyle\sum_{m = 1}^{\infty} \displaystyle\frac{\zeta(2m) w^m}{m} = \displaystyle\sum_{m = 1}^{\infty} \displaystyle\sum_{j = 1}^{\infty} \displaystyle\frac{w^m}{j^{2m} m} =  \displaystyle\sum_{j = 1}^{\infty} \displaystyle\sum_{m = 1}^{\infty} \displaystyle\frac{1}{m} \left( \displaystyle\frac{w}{j^2} \right)^m = -\displaystyle\sum_{j = 1}^{\infty} \log \left( 1 - \displaystyle\frac{w}{j^2} \right) = g(w),
			\end{flalign*}
			
			\noindent which by the second identity in \eqref{fg} yields the second statement of \eqref{fwgw}. 
		\end{proof}
		
		By \Cref{fgsum}, we have 
		\begin{flalign}
		\label{hz1}
		H_k (N) = \displaystyle\frac{1}{2 \pi \textbf{i}} \displaystyle\oint_{\mathcal{C}} w^{- N - 1} F_k (w) dw; \qquad Z_k (N) = \displaystyle\frac{1}{2 \pi \textbf{i}} \displaystyle\oint_{\mathcal{C}} w^{- N - 1} G_k (w) dw,
		\end{flalign}
		
		\noindent where the contour $\mathcal{C}$ is a positively oriented circle centered at $0$ with radius $\frac{1}{2}$ (see the left side of \Cref{cgamma}). It will next be useful to deform the contour $\mathcal{C}$ to the contour $\gamma$, defined as follows. 
		
		\begin{definition} 
			
			Fix the real numbers 
			\begin{flalign} 
			\label{rkappa} 
			R = R_N = 1 + \displaystyle\frac{(\log N)^3}{N}; \qquad \kappa = \sqrt{R^2 - N^{-2}},
			\end{flalign}
			
			\noindent and define the subsets $\gamma_0, \gamma_1, \gamma_2, \gamma_3 \subset \mathbb{C}$ by 
			\begin{flalign*}
			& \gamma_0 = \Big\{ z \in \mathbb{C}: |z| = R, |\Im z| > \displaystyle\frac{1}{N} \Big\} \cup \big\{ z \in \mathbb{C}: |z| = R, \Re z < 0 \big\}; \\
			& \gamma_1 = \Big\{ z \in \mathbb{C}: \Re z \in [1, \kappa], \Im z = \displaystyle\frac{1}{N} \Big\}; \qquad \gamma_3 = \Big\{ z \in \mathbb{C}: \Re z \in [1, \kappa], \Im z = - \displaystyle\frac{1}{N} \Big\}; \\
			& \gamma_2 = \Big\{ z \in \mathbb{C}: |z - 1| = \displaystyle\frac{1}{N}, \Re z < 1 \Big\}.
			\end{flalign*}
			
			\noindent Define the contour $\gamma = \gamma_0 \cup \gamma_1 \cup \gamma_2 \cup \gamma_3$, oriented counterclockwise. We refer to the right side of \Cref{cgamma} for a depiction. 
			
		\end{definition}

		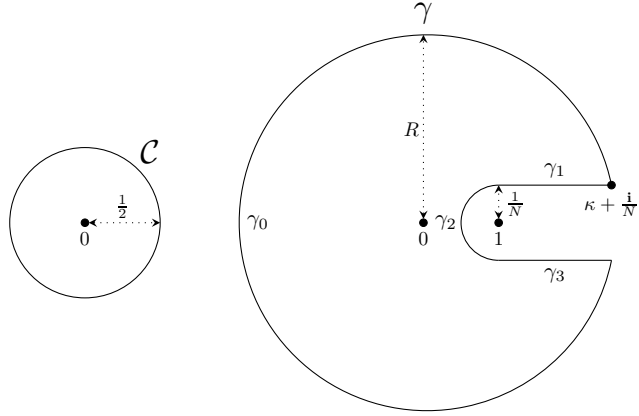
\begin{figure}[t]
			
			\begin{center}
				
				\begin{tikzpicture}[
				>=stealth,
				]

				\draw[] (0, 0) circle[radius = 1];
				\draw[<->, dotted] (.05, 0) -- (1, 0);
				
				\draw[fill=black] (0, 0) circle[radius = .05] node[below = 1, scale = .75]{$0$};
				
				\draw[] (.5, 0) circle[radius=0] node[above = -1, scale = .75]{$\frac{1}{2}$};
				\draw[] (.866, .5) circle[radius=0] node[above = 4, scale = 1.25]{$\mathcal{C}$};
				
				\draw[] (7, .5) arc(11.547:348.453:2.5); 
				\draw[] (5.5, .5) arc(90:270:.5); 
				\draw[] (7, .5) -- (5.5, .5); 
				\draw[] (7, -.5) -- (5.5, -.5);
				\draw[<->,dotted] (4.5, .05) -- (4.5, 2.5);
				\draw[<->,dotted] (5.5, .05) -- (5.5, .5);
				
				\draw[fill=black] (4.5, 0) circle[radius = .05] node[below = 1, scale = .75]{$0$};
				\draw[fill=black] (5.5, 0) circle[radius = .05] node[below = 1, scale = .75]{$1$};
				\draw[] (4.5, 2.5) circle[radius=0] node[above, scale = 1.25]{$\gamma$};
				\draw[] (2.3, 0) circle[radius =0] node[scale = .85]{$\gamma_0$};
				\draw[] (6.25, .7) circle[radius =0] node[scale = .85]{$\gamma_1$};
				\draw[] (4.8, 0) circle[radius =0] node[scale = .85]{$\gamma_2$};
				\draw[] (6.25, -.7) circle[radius =0] node[scale = .85]{$\gamma_3$};
				
				\draw[] (4.5, 1.25) circle[radius=0] node[left = -1, scale = .75]{$R$};
				\draw[] (5.5, .25) circle[radius=0] node[right, scale = .75]{$\frac{1}{N}$};
				
				\draw[fill = black] (7, .5) circle[radius = .05] node[scale = .75, below]{$\kappa + \frac{\textbf{i}}{N}$};

				\end{tikzpicture}
				
			\end{center}	
			
			\caption{\label{cgamma} The contours $\mathcal{C}$ and $\gamma = \gamma_0 \cup \gamma_1 \cup \gamma_2 \cup \gamma_3$ are depicted above.} 
		\end{figure}

		Since $F$ and $G$ are analytic away from $\mathbb{Z}$, we may deform the contour $\mathcal{C}$ in \eqref{hz1} to $\gamma$, if $N > 2^{15}$ (so that $R < \frac{3}{2}$). Thus, \eqref{hz1} yields 
		\begin{flalign}
		\label{ansum} 
		\begin{aligned}
		H_k (N) & = \displaystyle\frac{1}{2 \pi \textbf{i}} \displaystyle\int_{\gamma_0} w^{-N - 1} F_k (w) dw + \displaystyle\frac{1}{2 \pi \textbf{i}} \displaystyle\int_{\gamma_1 \cup \gamma_2 \cup \gamma_3} w^{-N - 1} F_k (w) dw; \\
		Z_k (N) & = \displaystyle\frac{1}{2 \pi \textbf{i}} \displaystyle\int_{\gamma_0} w^{-N - 1} G_k (w) dw + \displaystyle\frac{1}{2 \pi \textbf{i}} \displaystyle\int_{\gamma_1 \cup \gamma_2 \cup \gamma_3} w^{-N - 1} G_k (w) dw.
		\end{aligned}
		\end{flalign}
		
		\subsection{Proof of \Cref{zkmestimate2}}
		
		\label{Estimatezk} 
		
		In this section we establish \Cref{zkmestimate2} using the generating series given by \Cref{fgsum}. To that end, we begin with the following two bounds.
		
		\begin{lem} 
			
			\label{sumestimatezeta} 
			
			For any integer $N \ge 1$, we have 
			\begin{flalign*}
			\displaystyle\sum_{j = 1}^N \displaystyle\frac{\zeta (2j)}{j} \le \log N + 5.
			\end{flalign*}
			
		\end{lem} 
		
		\begin{proof}
			
			Observe for any integer $k \ge 2$ that 
			\begin{flalign*}
			\zeta (k) - 1 = \displaystyle\sum_{m = 2}^{\infty} m^{-k} = 2^{-k} \displaystyle\sum_{m = 2}^{\infty} \left( \displaystyle\frac{2}{m}\right)^k \le 2^{-k} \displaystyle\sum_{m = 2}^{\infty} 4 m^{-2} = 2^{-k} \big(4 \zeta(2) - 4 \big) \le 3 \cdot 2^{- k}.
			\end{flalign*} 
			
			\noindent Thus, 
			\begin{flalign*}
			\displaystyle\sum_{j = 1}^N \displaystyle\frac{\zeta (2j)}{j} \le \displaystyle\sum_{j = 1}^N \displaystyle\frac{1}{j} + 3 \displaystyle\sum_{j = 1}^{\infty} 2^{-2j} \le \log N + 5, 
			\end{flalign*}
			
			\noindent from which we deduce the lemma.
		\end{proof} 
		
		\begin{cor}
			
			\label{zkmestimate} 
			
			For any integers $N \ge k \ge 1$, we have 
			\begin{flalign}
			\label{zkn}
			\displaystyle\sum_{m = 1}^N Z_k (m) \le (\log N + 5)^k.
			\end{flalign}
			
		\end{cor}

		\begin{proof}
			
			First observe that
			\begin{flalign}
			\label{sumzkm} 
			\displaystyle\sum_{m = 1}^N Z_k (m) \le \Bigg( \displaystyle\sum_{j = 1}^N \displaystyle\frac{\zeta (2j)}{j} \Bigg)^k, 
			\end{flalign}
			
			\noindent since the expansion of the right side of \eqref{sumzkm} contains each summand appearing in the definition \eqref{hkzkdefinition} of $Z_k (m)$, for every $m \in [1, N]$. Thus, the corollary follows from \eqref{sumzkm} and \Cref{sumestimatezeta}. 
		\end{proof}
		
		Now we can establish \Cref{zkmestimate2}.

		\begin{proof}[Proof of \Cref{zkmestimate2}]
			
			Since \Cref{fgsum} implies that $Z_k (N)$ is the coefficient of $w^N$ in $G_k (w) = g(w)^k$, it follows that $N Z_k (N)$ is the coefficient of $w^{N - 1}$ in 
			\begin{flalign*}
			\displaystyle\frac{\partial}{\partial w} \big( g(w)^k \big) = k g' (w) g(w)^{k - 1} & = k \Bigg( \displaystyle\sum_{j = 1}^{\infty} \displaystyle\frac{1}{j^2 - w} \Bigg) G_{k - 1} (w) \\
			& = k \Bigg(\displaystyle\sum_{j = 1}^{\infty} \displaystyle\frac{1}{j^2} \displaystyle\sum_{i = 0}^{\infty} \displaystyle\frac{w^i}{j^{2i}} \Bigg) \displaystyle\sum_{m = k - 1}^{\infty} Z_{k - 1} (m) w^m.
			\end{flalign*}
			
			\noindent Hence,
			\begin{flalign*}
			Z_k (N) = \displaystyle\frac{k}{N} \displaystyle\sum_{j = 1}^{\infty} \displaystyle\frac{1}{j^2} \displaystyle\sum_{m = k - 1}^{N - 1} \displaystyle\frac{Z_{k - 1} (m)}{j^{2(N - m - 1)}} \le \displaystyle\frac{k}{N} \displaystyle\sum_{j = 1}^{\infty} \displaystyle\frac{1}{j^2} \displaystyle\sum_{m = k - 1}^{N - 1} Z_{k - 1} (m) & \le \displaystyle\frac{k \zeta (2)}{N} \displaystyle\sum_{m = 1}^N Z_{k - 1} (m) \\ 
			& \le \displaystyle\frac{2k (\log N + 5)^{k - 1}}{N},
			\end{flalign*}
			
			\noindent where in the last estimate we applied \Cref{zkmestimate} and the bound $\zeta (2) \le 2$. 
		\end{proof}

		\section{Asymptotics for \texorpdfstring{$H_k (N)$}{} and \texorpdfstring{$Z_k (N)$}{}}
		
		\label{Asymptotichknzkn}

		In \Cref{Sumhz} we establish \Cref{sumkn2}, assuming an asymptotic expansion (given by \Cref{anasymptotic} below) of the sums $H_k (N)$ and $Z_k (N)$, as $N$ tends to $\infty$. The remainder of this section is directed toward the proof of this expansion, which proceeds through an asymptotic analysis of the contour integral representation given by \eqref{ansum}.

		 \subsection{Proof of \Cref{sumkn2}}
		 
		 \label{Sumhz} 
		 
		 We will establish \Cref{sumkn2} using asymptotic expansions for $H_k (N)$ and $Z_k (N)$ (as $N$ tends to $\infty$). To state this expansion, we first require the following coefficients $\varphi_j$. 
		 
		 \begin{definition}
		 	
		 	\label{psij} 
		 	
		 	For each $j \in \mathbb{Z}_{\ge 0}$ and $s \in \mathbb{C} \setminus \mathbb{Z}_{\ge 1}$, set 
		 	\begin{flalign*} 
		 	\Phi_j (s) = \displaystyle\frac{1}{\pi} \displaystyle\frac{\partial^j}{\partial s^j} \big( \Gamma (1 - s) \sin (\pi s) \big); \qquad \varphi_j = \Phi_j (0).
		 	\end{flalign*}
		 	
		 \end{definition}

		 Now we can state the following proposition, which provides asymptotic expansions for $H_k (N)$ and $Z_k (N)$, with explicit estimates on the respective errors $\varepsilon_k^H (N)$ and $\varepsilon_k^Z (N)$. It was predicted as Conjecture 1.30 of \cite{VFGINMSC} and will be proven in \Cref{Proofhz} below. 
		 
		 \begin{prop} 
		 	
		 	\label{anasymptotic}
		 	
		 	There exists a constant $C > 1$ such that the following holds. Let $k$ and $N$ be positive integers satisfying $k \le (\log N)^2$. Setting
		 	\begin{flalign*}
		 	& \varepsilon_k^H (N) = \bigg| N H_k (N) - \displaystyle\sum_{j = 1}^k \binom{k}{j} \varphi_j (\log N)^{k - j} \bigg|; \\
		 	& \varepsilon_k^Z (N) = \bigg| N Z_k (N) - \displaystyle\sum_{j = 1}^k \binom{k}{j} \varphi_j (\log N + \log 2)^{k - j} \bigg|, 
		 	\end{flalign*}
		 	
		 	\noindent we have 
		 	\begin{flalign}
		 	\label{anapproximate5}
		 	\varepsilon_k^H (N) \le C 2^k k! (\log N)^{14} N^{-1 / 2}; \qquad \varepsilon_k^H (N) \le C 2^k k! (\log N)^{14} N^{-1 / 2}. 
		 	\end{flalign}
		 	
		 \end{prop}
		 
		 Theorem 6.2 of \cite{AC} establishes the same expansion for $H_k (N)$, but does not make the $k$-dependence of the error bound explicit or address the series $Z_k (N)$. Still, the proof of \Cref{anasymptotic} will closely follow that of Theorem 6.2 of \cite{AC}.
		 
		 Assuming \Cref{anasymptotic}, we can now prove \Cref{sumkn2}. To that end, we first require the following lemma that provides a generating series for the coefficients $\varphi_j$ through the function $\Phi_0$.
		 
		 \begin{lem} 
		 	
		 	\label{psisum} 
		 	
		 	For any $z \in \mathbb{C}$ with $|z| < 1$, we have 
		 	\begin{flalign}
		 	\label{psi0}
		 	\displaystyle\sum_{j = 0}^{\infty} \displaystyle\frac{|\varphi_j z^j|}{j!} < \infty; \qquad \Phi_0 (z) = \displaystyle\sum_{j = 0}^{\infty} \displaystyle\frac{\varphi_j z^j}{j!}.
		 	\end{flalign}
		 	
		 \end{lem}  
		 
		 \begin{proof}
		 	
		 	The first statement of \eqref{psi0} follows from the facts that $\Phi_0$ is analytic on $\{ z \in \mathbb{C}: |z| < 1 \}$ and that $\varphi_j = \frac{\partial^j}{\partial z^j} \Phi_0 (0)$. The second follows from these two facts and a Taylor expansion.  
		 \end{proof}
		 
		 Now we can establish \Cref{sumkn2}. 
		 
		 \begin{proof}[Proof of \Cref{sumkn2} Assuming \Cref{anasymptotic}]
		 	
		 	We only establish the second identity in \eqref{zkestimate2}, as the proof of the first is entirely analogous. To that end, set 
		 	\begin{flalign*} 
		 	\delta = \delta_{\omega} = \frac{1}{4} \big( \omega - \frac{1}{2} \big); \qquad \mu = \mu_{\omega} (N) = \lfloor \omega \log N \rfloor; \qquad \nu = \nu_{\omega} (N) = \lfloor \delta \log N \rfloor.
		 	\end{flalign*} 
		 	
		 	\noindent By \Cref{anasymptotic}; the fact that $\mu = \mathcal{O} (\log N)$; and the fact that $\varphi_0 = 0$, we have
		 	\begin{flalign}
		 	\label{sumzn} 
		 	\begin{aligned}
		 	N^{1 / 2} \displaystyle\sum_{k = 1}^{\mu} \displaystyle\frac{Z_k (N)}{2^{k - 1} k!} & = N^{-1 / 2} \displaystyle\sum_{k = 1}^{\mu} \displaystyle\frac{1}{2^{k - 1} k!} \displaystyle\sum_{j = 1}^k \binom{k}{j} \varphi_j (\log N + \log 2)^{k - j} + \mathcal{O}\left( N^{-1 / 2} \displaystyle\sum_{k = 1}^{\mu} \displaystyle\frac{\varepsilon_k^H (N)}{2^k k!} \right) \\
		 	& = N^{-1 / 2} \displaystyle\sum_{k = 1}^{\mu} \displaystyle\sum_{j = 1}^k \displaystyle\frac{(\log N + \log 2)^{k - j}}{2^{k - j} (k - j)!}  \displaystyle\frac{\varphi_j}{2^{j - 1} j!} + \mathcal{O} \Big( \displaystyle\frac{\mu (\log N)^{14}}{N} \Big) \\
		 	& = 2 N^{-1 / 2} \displaystyle\sum_{j = 0}^{\mu} \displaystyle\frac{\varphi_j}{2^j j!} \displaystyle\sum_{r = 0}^{\mu - j} \displaystyle\frac{(\log N + \log 2)^r}{2^r r!}  + \mathcal{O} \Big( \displaystyle\frac{(\log N)^{15}}{N} \Big),
		 	\end{aligned}
		 	\end{flalign} 
		 	
		 	\noindent where we have set $r = k - j$, and the implicit constant is uniform in $N$ (but might depend on $\omega$). 
		 	
		 	Next, let us replace the sum over $j \in [0, \mu]$ in the right side of \eqref{sumzn} with a sum over $j \in [0, \nu]$. To that end, observe that 
		 	\begin{flalign}
		 	\label{sumzn2} 
		 	\begin{aligned}
		 	2 N^{-1 / 2} \Bigg| \displaystyle\sum_{j = 0}^{\mu} \displaystyle\frac{\varphi_j}{2^j j!} & \displaystyle\sum_{r = 0}^{\mu - j} \displaystyle\frac{(\log N + \log 2)^r}{2^r r!} - \displaystyle\sum_{j = 0}^{\nu} \displaystyle\frac{\varphi_j}{2^j j!} \displaystyle\sum_{r = 0}^{\mu - j} \displaystyle\frac{(\log N + \log 2)^r}{2^r r!} \Bigg| \\
		 	& \le 2N^{-1 / 2} \displaystyle\sum_{j = \nu + 1}^{\mu} \displaystyle\frac{|\varphi_j|}{2^j j!} \displaystyle\sum_{r = 0}^{\infty} \displaystyle\frac{(\log N + \log 2)^r}{2^r r!} \\
		 	& = 2N^{-1 / 2} \exp \left( \displaystyle\frac{\log N + \log 2}{2} \right)\displaystyle\sum_{j = \nu + 1}^{\mu} \displaystyle\frac{|\varphi_j|}{2^j j!} \le 2^{3 / 2} \displaystyle\sum_{j = \nu + 1}^{\infty} \displaystyle\frac{|\varphi_j|}{2^j j!} = o(1),
		 	\end{aligned}
		 	\end{flalign} 
		 	
		 	\noindent where for the last estimate we used the first statement of \eqref{psi0} and the fact that $\lim_{N \rightarrow \infty} \nu_{\omega} (N) = \infty$. 
		 	
		 	Inserting \eqref{sumzn2} into \eqref{sumzn} yields
		 	\begin{flalign}
		 	\label{sumzn3}
		 	\begin{aligned}
		 	N^{1 / 2} \displaystyle\sum_{k = 1}^{\mu} \displaystyle\frac{Z_k (N)}{k! 2^{k - 1}}  = 2 N^{-1 / 2} \displaystyle\sum_{j = 0}^{\nu} \displaystyle\frac{\varphi_j}{2^j j!} \displaystyle\sum_{r = 0}^{\mu - j} \displaystyle\frac{(\log N + \log 2)^r}{2^r r!}  + o(1). 
		 	\end{aligned}
		 	\end{flalign} 
		 	
		 	\noindent Now, observe from the $R = \frac{\log (2N)}{2}$ case of \Cref{exponentialm} that
		 	\begin{flalign*}
		 	\Bigg| \displaystyle\sum_{r = 0}^{\mu - j} \displaystyle\frac{(\log N + \log 2)^r}{2^r r!} - (2N)^{1 / 2} \Bigg| \le \delta^{-1} (1 + \delta)^{-\delta \log N / 2} (2N)^{1 / 2} = o(N^{1 / 2}), \qquad \text{for $j \le \nu$}, 
		 	\end{flalign*}
		 	
		 	\noindent which upon insertion into \eqref{sumzn3} (using the first statement of \eqref{psi0}) yields
		 	\begin{flalign*}
		 	N^{1 / 2} \displaystyle\sum_{k = 1}^{\mu} \displaystyle\frac{Z_k (N)}{k! 2^{k - 1}} = 2^{3 / 2} \displaystyle\sum_{j = 0}^{\nu} \displaystyle\frac{\varphi_j}{2^j j!} + o(1). 
		 	\end{flalign*} 
		 	
		 	\noindent Together with the $z = \frac{1}{2}$ case of \Cref{psisum} and the fact that $\lim_{N \rightarrow \infty} \nu_{\omega} (N) = \infty$, this yields 
		 	\begin{flalign*}
		 	N^{1 / 2} \displaystyle\sum_{k = 1}^{\mu} \displaystyle\frac{Z_k (N)}{k! 2^{k - 1}} = 2^{3 / 2} \Phi_0 \left( \displaystyle\frac{1}{2} \right) + o(1),
		 	\end{flalign*}
		 	
		 	\noindent from which we deduce the theorem since $\Phi_0 \big( \frac{1}{2} \big) = \pi^{-1} \Gamma \big( \frac{1}{2} \big) = \pi^{-1 / 2}$.	
		 \end{proof}

		\subsection{Scaling Around \texorpdfstring{$w = 1$}{}}
		
		\label{Scalew1}
		
		In this section we first bound the contributions of the first terms (integrals along $\gamma_0$) on the right sides of \eqref{ansum}; we then change variables in the remaining integrals (over $\gamma_1 \cup \gamma_2 \cup \gamma_3$) in \eqref{ansum} by scaling $w$ by a factor of $N$ around $1$. The following lemma implements the former task.

		Throughout the remainder of this paper, we assume that $N \ge 2^{90}$ and omit the parameter $k$ from our notation, abbreviating $F = F_k$; $G = G_k$; $H (m) = H_k (m)$; and $Z (m) = Z_k (m)$.

		\begin{lem}
			
			\label{gamma0hz}
			
			We have that 
			\begin{flalign}
			\label{gamma0a} 
			\left|\displaystyle\int_{\gamma_0} w^{-N - 1} F(w) dw \right| \le \displaystyle\frac{2 \pi (\log N)^k}{N^2}; \qquad \left|\displaystyle\int_{\gamma_0} w^{-N - 1} G (w) dw \right| \le \displaystyle\frac{2 \pi (\log N + 2)^k}{N^2}. 
			\end{flalign}
		\end{lem} 
		
		\begin{proof}
			
			Recalling the functions $f$ and $g$ from \eqref{fgfg}, observe since $N \ge 3$ that  
			\begin{flalign}
			\label{fgamma0}
			\displaystyle\sup_{w \in \gamma_0} \big| f(w) \big| \le \displaystyle\max \Big\{ \big| \log (R - 1) \big|, \big| \log (R + 1) \big|  \Big\} \le \displaystyle\max \big\{ \log N, \log 3 \big\} = \log N. 
			\end{flalign}
			
			\noindent Together with the facts that $|w| = R$ on $\gamma_0$ and that the length of $\gamma_0$ is less than $2 \pi R$, \eqref{fgamma0} yields 
			\begin{flalign}
			\label{gamma0a1} 
			\left|\displaystyle\int_{\gamma_0} w^{-N - 1} F(w) dw \right| \le 2 \pi (\log N)^k R^{-N}. 
			\end{flalign}
			
			\noindent Now the first estimate in \eqref{gamma0a} follows from \eqref{gamma0a1} and the fact that $R^{-N} < N^{-2}$ (which holds since $R = 1 + \frac{(\log N)^3}{N}$ and $N \ge 2^{90}$). 
			
			To establish the second, observe since $N \ge 2^{90}$ that $R < \frac{3}{2}$. Therefore, 
			\begin{flalign}
			\label{gamma0g}
			\begin{aligned}
			\displaystyle\sup_{w \in \gamma_0} \big| g(w) \big| & \le \displaystyle\sup_{w \in \gamma_0} \big| f(w) \big| + \displaystyle\sum_{j = 2}^{\infty} \displaystyle\sup_{w \in \gamma_0} \Bigg| \log \bigg( 1 - \displaystyle\frac{w^2}{j^2} \bigg) \Bigg| \\
			& \le \log N + \displaystyle\sum_{j = 2}^{\infty} \log \bigg( 1 + \displaystyle\frac{9}{4 j^2} \bigg) \le \log N + \displaystyle\frac{9}{4} \displaystyle\sum_{j = 2}^{\infty} \displaystyle\frac{1}{j^2} < \log N + 2,
			\end{aligned}
			\end{flalign} 
			
			\noindent where we have used the facts that $\log (1 + z) \le z$ for $z \ge 0$ and that $\sum_{j = 2}^{\infty} \frac{1}{j^2} < \frac{2}{3}$. Now the proof of the second estimate in \eqref{gamma0a} is entirely analogous to that of the former.   
		\end{proof}

		The remaining quantities on the right sides of \eqref{ansum} involve integration over $\gamma_1 \cup \gamma_2 \cup \gamma_3$. Since this contour lies in (slightly larger than) a $N^{-1}$-neighborhood of $w = 1$, we change variables 
		\begin{flalign}
		\label{tz} 
		w = 1 + \frac{t}{N}.
		\end{flalign}
		
		\begin{definition} 
			
			\label{xi}  
			
			Define the contours $\Xi_1 = \Xi_1 (N)$, $\Xi_2 = \Xi_2 (N)$, and $\Xi_3 = \Xi_3 (N)$ to be the images of $\gamma_1, \gamma_2, \gamma_3$ under \eqref{tz}; also set $\Xi = \Xi (N) = \Xi_1 \cup \Xi_2 \cup \Xi_3$. We refer to \Cref{gamma} for a depiction, where there $\rho = N (\kappa - 1)$ is the image of $\kappa$ under the change of variables \eqref{tz}. 
			
		\end{definition}
		
		\begin{rem} 
			
			\label{estimaterho}
			
			By \eqref{rkappa} and \eqref{tz}, $\rho$ is quickly seen to satisfy $(\log N)^3 - 1 < \rho < (\log N)^3$.
			
		\end{rem} 
		
		\noindent Denoting the function 
		\begin{flalign} 
		\label{tnsum}
		\Psi (t) = \Psi (t; N) = \log N - \log (-t) - \displaystyle\sum_{j = 2}^{\infty} \log \Big( 1 - \displaystyle\frac{1}{j^2} - \displaystyle\frac{t}{N j^2} \Big),
		\end{flalign} 
		
		\noindent \eqref{ansum} and \eqref{gamma0a} together imply that 
		\begin{flalign}
		\label{ansum2} 
		\begin{aligned}
		H (N) & = \displaystyle\frac{1}{2 \pi \textbf{i} N}  \displaystyle\int_{\Xi} \Big( 1 + \frac{t}{N} \Big)^{-N - 1} \big( \log N - \log (-t) \big)^k dt + \mathcal{O} \Big( \displaystyle\frac{(\log N)^k}{N^2} \Big); \\
		Z(N) & = \displaystyle\frac{1}{2 \pi \textbf{i} N}  \displaystyle\int_{\Xi} \Big( 1 + \frac{t}{N} \Big)^{-N - 1} \Psi (t)^k dt  + \mathcal{O} \Big( \displaystyle\frac{( \log N + 2)^k}{N^2} \Big), \\
		\end{aligned}
		\end{flalign}
		
		\noindent where the implicit constants are uniform in $k$ and $N$. Observe that $-t$ does not intersect the branch cut (which is $\mathbb{R}_{\le 0}$) for the principal logarithm, as $t$ ranges over $\Xi$.

		\begin{figure}[t]
			
			\begin{center}
				
				\begin{tikzpicture}[
				>=stealth,
				]

				\draw[] (0, 2) arc(90:270:2); 
				\draw[] (0, 2) -- (4, 2); 
				\draw[<-] (0, -2) -- (4, -2);
				\draw[<->, dotted] (2, 0) -- (2, 2);
				
				\draw[fill=black] (0, 0) circle[radius = .05] node[below = 1, scale = .75]{$0$};

				\draw[fill=black] (4, 2) circle[radius = .05] node[below, scale = .75]{$\rho + \textbf{i}$};
				\draw[fill=black] (4, -2) circle[radius = .05] node[above, scale = .75]{$\rho - \textbf{i}$};
				
				\draw[] (2, 1) circle[radius=0] node[right, scale = .75]{$1$};
				
				\draw[] (2, 2) circle[radius=0] node[above]{$\Xi_1$};
				\draw[] (-2, 0) circle[radius=0] node[left]{$\Xi_2$};
				\draw[] (2, -2) circle[radius=0] node[below]{$\Xi_3$};

				\draw[->, dashed, very thick] (4, 2) -- (7, 2);
				\draw[->, dashed, very thick] (7, -2) -- (4.05, -2);

				\draw[] (5.5, 2) circle[radius=0] node[above]{$\Omega_1 \setminus \Xi_1$};
				\draw[] (5.5, -2) circle[radius=0] node[below]{$\Omega_3 \setminus \Xi_3$};

				\end{tikzpicture}
				
			\end{center}	
			
			\caption{\label{gamma} The contours $\Xi_1, \Xi_2, \Xi_3$ are depicted above as solid, and the contours $\Omega_1 \setminus \Xi_1$ and $\Omega_3 \setminus \Xi_3$ are depicted above as dashed.} 
		\end{figure}
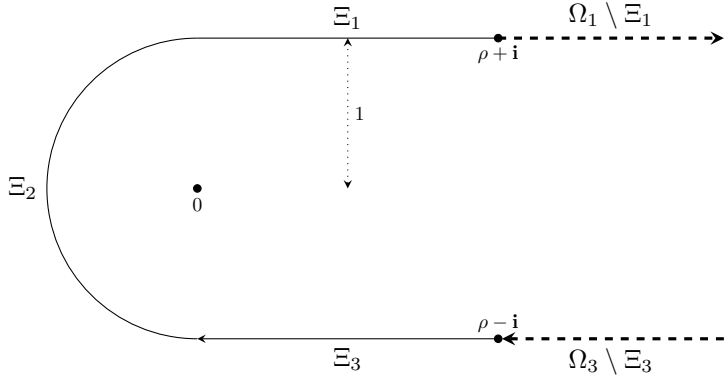 	
		
		\subsection{The Limiting Integrands}
		
		\label{LimitIntegral}
		
		In this section we establish the following proposition, which approximates the integrands on the right sides of \eqref{ansum2} by simpler quantities. 
		
		\begin{prop}
			
			\label{hzexponential} 
			
			We have that 
			\begin{flalign}
			\label{hz3} 
			\begin{aligned}
			H (N) & = \displaystyle\frac{1}{2 \pi \textbf{\emph{i}} N}  \displaystyle\int_{\Xi} e^{-t} \big( \log N - \log (-t) \big)^k dt + \mathcal{O} \left( \displaystyle\frac{(\log N + 4 \log \log N)^{k + 9}}{N^2} \right), \\
			Z(N) & = \displaystyle\frac{1}{2 \pi \textbf{\emph{i}} N}  \displaystyle\int_{\Xi} e^{-t} \big( \log N - \log (-t) + \log 2 \big)^k dt + \mathcal{O} \left( \displaystyle\frac{k (\log N + 5 \log \log N)^{k + 9}}{N^2} \right), \\
			\end{aligned}
			\end{flalign}
			
			\noindent where the implicit constants are uniform in $k$ and $N$.
		\end{prop}
		
		In particular, this approximates $\big( 1 + \frac{t}{N} \big)^{-N - 1} \approx e^{-t}$ and $\Psi (t) \approx \log N - \log (-t) + \log 2$ (recall \eqref{tnsum}) in \eqref{ansum2}. To implement the former, observe from a Taylor expansion of $\log (1 + z)$ and the fact that $|t| \le (\log N)^3 < N^{1 / 3}$ for any $t \in \Xi$ (since $N > 2^{90}$) that 
		\begin{flalign}
		\label{tnapproximate}
		\displaystyle\sup_{t \in \Xi} \bigg| e^t \Big( 1 + \displaystyle\frac{t}{N} \Big)^{-N - 1} - 1 \bigg| = \mathcal{O} \left( \displaystyle\sup_{t \in \Xi} \displaystyle\frac{|t|^2}{N} \right) = \mathcal{O} \left( \displaystyle\frac{(\log N)^6}{N} \right), 
		\end{flalign} 
		
		\noindent where the implicit constants are uniform in $t$ and $N$. 
		
		The following lemma approximates $\Psi (t) \approx \log N - \log (-t) + \log 2$.
		
		\begin{lem}
			
			\label{zgerror}
			
			For any integers $N \ge 2^{90}$ and $k \ge 1$, we have
			\begin{flalign}
			\label{txiestimate}
			\begin{aligned}
			& \displaystyle\sup_{t \in \Xi} \big| \Psi (t) \big| < \log N + 5 \log \log N; \\
			& \displaystyle\sup_{t \in \Xi} \Big| \Psi (t)^k - \big( \log N - \log (-t) + \log 2 \big)^k \Big| \le \displaystyle\frac{2 k (\log N + 5 \log \log N)^{k + 2}}{N}.
			\end{aligned}
			\end{flalign}
			
		\end{lem}
		
		\begin{proof}
			
			By a Taylor expansion of $\log (1 + z)$ and the facts that $N \ge 2^{90}$ and $|t| \le (\log N)^3 < N^{1 / 3}$ whenever $t \in \Xi$, we have 
			\begin{flalign*}
			\displaystyle\sup_{t \in \Xi} \bigg| \log \Big( 1 - \displaystyle\frac{1}{j^2} - \displaystyle\frac{t}{N j^2} \Big) - \log \Big( 1 - \displaystyle\frac{1}{j^2} \Big) \bigg| < \displaystyle\frac{2|t|}{N j^2},
			\end{flalign*} 
			
			\noindent for any $j \ge 2$. Summing over $j$ and using the facts that $\sum_{j = 2}^{\infty} \log \big( 1 - \frac{1}{j^2} \big) = - \log 2$ and $\sum_{j = 2}^{\infty} \frac{1}{j^2} < 1$, we obtain 
			\begin{flalign}
			\label{txiestimate1}
			\displaystyle\sup_{t \in \Xi} \bigg| \displaystyle\sum_{j = 2}^{\infty} \log \Big( 1 - \displaystyle\frac{1}{j^2} - \displaystyle\frac{t}{N j^2} \Big) - \log 2 \bigg| < \displaystyle\frac{2|t|}{N} \le \displaystyle\frac{2 (\rho + 1)}{N}.
			\end{flalign} 
			
			\noindent Thus, the first bound in \eqref{txiestimate} follows from \eqref{tnsum}; \eqref{txiestimate1}; \Cref{estimaterho}; and the facts that $N \ge 2^{90}$ and that $\big| \log (-t) \big| \le \log |t| + 2 \pi \le 4 \log \log N$ (which holds by \Cref{estimaterho}). The second bound in \eqref{txiestimate} follows from the first bound there; \eqref{txiestimate1}; \Cref{estimaterho}; and the facts that $N \ge 2^{90}$ and that $|A^k - B^k| \le k (A - B) \displaystyle\max\{ |A|^{k - 1}, |B|^{k - 1} \}$ for any $A, B \in \mathbb{C}$. 
		\end{proof}

		Now we can establish \Cref{hzexponential}. 
		
		\begin{proof}[Proof of \Cref{hzexponential}]
			
			Let us only establish the second bound in \eqref{hz3}, as the proof of the first is entirely analogous. To that end, \eqref{tnapproximate}; the fact that the length of the contour $\Xi$ is at most $2 \rho + \pi \le 3 (\log N)^3$ (by \Cref{estimaterho}); and the first bound in \eqref{txiestimate} together imply that
			\begin{flalign}
			\label{zn1}
			\begin{aligned}
			\Bigg| \displaystyle\int_{\Xi} \Big( 1 + \displaystyle\frac{t}{N} \Big)^{-N - 1} \Psi (t) dt - \displaystyle\int_{\Xi} e^{-t} \Psi (t) dt \Bigg| & = \mathcal{O} \left( \displaystyle\frac{(\log N)^9}{N} \displaystyle\sup_{t \in \Xi} \big| \Psi (t) \big|^k \right) \\
			& = \mathcal{O} \left(\displaystyle\frac{(\log N + 5 \log \log N)^{k + 9}}{N} \right),
			\end{aligned}
			\end{flalign}
			
			\noindent where the implicit constants are uniform in $N$ and $k$. Moreover, from the second bound of \eqref{txiestimate} (and from the fact that the length of $\Xi$ is at most $3 (\log N)^3$), we deduce that
			\begin{flalign}
			\label{zn2}
			\begin{aligned}
			\Bigg| \displaystyle\int_{\Xi} e^{-t} \Psi (t) dt - \displaystyle\int_{\Xi} e^{-t} \big( \log N - \log (-t) + 2 \big)^k dt \Bigg| & < \displaystyle\frac{6 k (\log N + 5 \log \log N)^{k + 5}}{N} \displaystyle\sup_{t \in \Xi} |e^{-t}| \\
			& = \displaystyle\frac{6 ek (\log N + 5 \log \log N)^{k + 5}}{N},
			\end{aligned}
			\end{flalign}
			
			\noindent where to establish the second estimate we used the fact that $\inf_{t \in \Xi} \Re t = -1$. Now the second bound in \eqref{hz3} follows from the second bound in \eqref{ansum2}, \eqref{zn1}, and \eqref{zn2}.
		\end{proof}

		\subsection{The Limiting Contour}
		
		\label{Contour}

		Observe that the contour $\Xi$ in \eqref{hz3} depends on $N$. In this section we will estimate the error in replacing $\Xi$ by a contour that is independent of $N$, given by the following definition. 
		
		\begin{definition} 
			
			\label{omega} 
			Define the contours
			\begin{flalign*}
			\Omega_1 = \{ t \in \mathbb{C}: \Re t & \ge 0, \Im t = 1 \}; \qquad \Omega_2 = \bigg\{ t = e^{\textbf{i} \theta}: \theta \in \Big[ \frac{\pi}{2}, \frac{3\pi}{2}\Big] \bigg\}; \\
			&  \Omega_3 = \{ t \in \mathbb{C}: \Re t \ge 0, \Im t = -1 \},
			\end{flalign*}
			
			\noindent and set $\Omega = \Omega_1 \cup \Omega_2 \cup \Omega_3$. 
			
		\end{definition} 
		
		Observe that $\Xi_2 = \Omega_2$ and, since $\lim_{N \rightarrow \infty} N (R_N - 1) = \infty$, the limits of the contours $\Xi_1 (N)$ and $\Xi_3 (N)$ are $\Omega_1$ and $\Omega_3$, respectively, as $N$ tends to $\infty$. The difference between $\Xi$ and $\Omega$ is depicted as dashed in \Cref{gamma}. The following proposition estimates the error in replacing the integration over $\Xi$ in \eqref{hz3} with integration over $\Omega$. 
		
		\begin{prop}
			
			\label{hzexponential2} 
			
			For any integers $N > 1$ and $k \le (\log N)^2$, we have
			\begin{flalign}
			\label{integralomegaxi}
			\begin{aligned}
			& \displaystyle\int_{\Omega \setminus \Xi} \Big| e^{-t} \big( \log N - \log (-t) \big)^k \Big| dt < \displaystyle\frac{2}{N^2}; \qquad \displaystyle\int_{\Omega \setminus \Xi} \Big| e^{-t} \big( \log N - \log (-t) + \log 2 \big)^k dt \Big| < \frac{2}{N^2},
			\end{aligned}
			\end{flalign}
			
			\noindent where the implicit constants are independent of $k$ and $N$. 
		\end{prop}
		
		\begin{proof} 
			
			Let us only establish the second estimate in \eqref{integralomegaxi}, as the proof of the first is entirely analogous. By \Cref{estimaterho} and the facts that $N \ge 2^{90}$ and $\big| \log (-t) \big| \le \log \Re t + 2 \pi$ for $t \in (\Omega_1 \setminus \Xi_1) \cup \big( \Omega_3 \setminus \Xi_3 \big) = \Omega \setminus \Xi$, we obtain 
			\begin{flalign}
			\label{integralapproximate}
			\begin{aligned}
			\bigg| \displaystyle\int_{\Omega \setminus \Xi} e^{-t} \big( \log N - \log (-t) + \log 2 \big)^k dt \bigg| < 2 \displaystyle\int_{(\log N)^3 - 1}^{\infty} e^{-s} (\log N + \log s + 7)^k ds.
			\end{aligned}
			\end{flalign}
			
			\noindent Now, since $N \ge 2^{90}$, it is quickly verified that 
			\begin{flalign*}
			(\log N + \log s + 7)^k \le e^{s / 2}, \qquad \text{whenever $k \le (\log N)^2$ and $s \ge (\log N)^3 - 1$},
			\end{flalign*}
			
			\noindent and
			\begin{flalign*}
			\displaystyle\int_{(\log N)^3 - 1}^{\infty} e^{-s / 2} ds < \displaystyle\frac{1}{N^2}.
			\end{flalign*}
			
			\noindent Upon insertion into \eqref{integralapproximate}, these estimates yield the second bound in \eqref{integralomegaxi}. 
		\end{proof}
		
		\subsection{Proof of \Cref{anasymptotic}}
		
		\label{Proofhz}
		
		In this section we establish \Cref{anasymptotic}. However, before doing so, we require the following two lemmas. The first bounds the second terms appearing on the right sides of \eqref{hz3}; the second provides a contour integral representation for the coefficients $\varphi_j$ from \Cref{psij}.

		\begin{lem} 
			
			\label{k2kestimate} 
			
			For each integer $K \ge 1$, we have
			\begin{flalign}
			\label{estimatekn}
			\displaystyle\frac{(\log N + 5 \log \log N)^K}{N} \le \displaystyle\frac{2^K K! (\log N)^{5 / 2}}{N^{1 / 2}}.
			\end{flalign}
			
		\end{lem} 
		
		\begin{proof}
			
			Recall from \eqref{kestimate1} that $K! \ge \big( \frac{K}{e} \big)^K$, and observe for fixed $v > 0$ that the function $\big( \frac{K}{ev} \big)^K$ is minimized at $K = v$. Thus, 
			\begin{flalign*}
			\displaystyle\frac{2^K K!}{(\log N + 5 \log \log N)^K} \ge \left(\displaystyle\frac{2 K}{e (\log N + 5 \log \log N)} \right)^K \ge e^{-(\log N + 5 \log \log N) / 2} = N^{-1 / 2} (\log N)^{- 5 / 2},
			\end{flalign*}
			
			\noindent which verifies \eqref{estimatekn}. 
			
		\end{proof} 
		
		\begin{lem}
			
			\label{integralpsi} 
			
			For each integer $j \ge 0$, we have 
			\begin{flalign}
			\label{integralomega}
			\displaystyle\frac{1}{2 \pi \textbf{\emph{i}}} \displaystyle\int_{\Omega} \big(- \log (-t) \big)^j e^{-t} dt = \varphi_j.
			\end{flalign}
			
		\end{lem} 
		
		\begin{proof}
			
			By Theorem B.1 of \cite{AC}, we have
			\begin{flalign}
			\label{integralomega1}
			\displaystyle\frac{1}{2 \pi \textbf{i}} \displaystyle\int_{\Omega} (-t)^{-s} e^{-t} dt = \pi^{-1} \Gamma (1 - s) \sin (\pi s),
			\end{flalign}
			
			\noindent for each $s \in \mathbb{R}_{< 1}$. Thus, \eqref{integralomega} follows from differentiating \eqref{integralomega1} $j$ times with respect to $s$ and then setting $s = 0$.		
		\end{proof}

		Now we can establish \Cref{anasymptotic}.

		\begin{proof}[Proof of \Cref{anasymptotic}]
			
			Let us only establish the second bound in \eqref{anapproximate5}, as the proof of the first is entirely analogous. Together, \Cref{hzexponential}, \Cref{hzexponential2}, \Cref{k2kestimate}, \Cref{integralpsi}, and the fact that $k \le (\log N)^2$ yield
			\begin{flalign*}
			N Z (N) & = \displaystyle\frac{1}{2 \pi \textbf{i}} \displaystyle\int_{\Omega} e^{-t} \big( \log N - \log (-t) + \log 2 \big)^k dt + \mathcal{O} \bigg( \displaystyle\frac{2^k (k + 1)! (\log N)^{12}}{N^{1 / 2}}\bigg) \\
			& =  \displaystyle\frac{1}{2 \pi \textbf{i}} \displaystyle\sum_{j = 1}^k \binom{k}{j} (\log N + \log 2)^{k - j} \displaystyle\int_{\Omega} \big( -\log (-t) \big)^j e^{-t} dt + \mathcal{O} \bigg( \displaystyle\frac{2^k k! (\log N)^{14}}{N^{1 / 2}}\bigg) \\
			& = \displaystyle\sum_{j = 1}^k \binom{k}{j} \varphi_j (\log N + \log 2)^{k - j}  + \mathcal{O} \bigg( \displaystyle\frac{2^k k! (\log N)^{14}}{N^{1 / 2}}\bigg),
			\end{flalign*}
			
			\noindent where the implicit constants are uniform in $k$ and $N$; this establishes the second bound in \eqref{anapproximate5}.
		\end{proof}

		\section{Volume Asymptotics for the Principal Stratum} 
		
		\label{VolumePrincipal} 
		
		In this section we analyze the large $g$ limit for $\Vol \mathcal{Q}_{g, n}$, by considering the stable graph contributions to this volume coming from \eqref{sumgraphsq}. In \Cref{Volume} we establish \Cref{limitvolume}, assuming three asymptotic estimates for these contributions arising from graphs with one vertex, two vertices, and at least three vertices; these are given by \Cref{sumv1}, \Cref{lambdav}, and \Cref{lambdag3} respectively. Then, in \Cref{Lambda1} we prove the single-vertex asymptotic result (\Cref{sumv1}), which provides the leading order contribution to the volume. The remaining two results are established in \Cref{Estimatev2} and \Cref{Estimateve1}. Throughout the remainder of this paper, we recall the notation from \Cref{VolumesStable}.

		\subsection{Proof of \Cref{limitvolume}}
		
		\label{Volume}
		
		We begin with the following notation for sets and quantities associated with stable graphs in $\mathcal{G}_{g, n}$ with a prescribed number of vertices. 
		
		\begin{definition} 
			
			\label{gnv} 
			
			Fix integers $g, n \ge 0$ with $2g + n \ge 3$, and $V \in [1, 2g + n - 2]$. Let $\mathcal{G}_{g, n} (V) \subseteq \mathcal{G}_{g, n}$ denote the set of stable graphs $\Gamma \in \mathcal{G}_{g, n}$ with $V$ vertices, that is, such that $\big| \mathfrak{V} (\Gamma) \big| = V$. Further let
			\begin{flalign}
			\label{lambdagve} 
			\Upsilon_{g, n}^{(V)} = \displaystyle\sum_{\Gamma \in \mathcal{G}_{g, n} (V)} \mathcal{Z} \big( P (\Gamma) \big).
			\end{flalign}
			
		\end{definition} 
		
		\begin{rem}

			By \Cref{volumesumgraph} and \eqref{gv2}, we have 
			\begin{flalign}
			\label{lambdavolume} 
			\Vol \mathcal{Q}_{g, n} = \displaystyle\sum_{V = 1}^{2g + n - 2} \Upsilon_{g, n}^{(V)}.
			\end{flalign}

		\end{rem}
	
		Now we can state the following three results that provide asymptotic estimates for $\Upsilon_{g, n}^{(V)}$ in the cases $V = 1$, $V = 2$, and $V \ge 3$. We will establish \Cref{sumv1} in \Cref{Lambda1}, \Cref{lambdav} in \Cref{Estimatev22}, and \Cref{lambdag3} in \Cref{Proofv3}. In the below we recall that, for any two functions $F_1, F_2: \mathbb{Z} \rightarrow \mathbb{R}$ such that $F_2 (k)$ is nonzero for sufficiently large $k$, we write $F_1 \sim F_2$ if $\lim_{k \rightarrow \infty} F_1 (k) F_2 (k)^{-1} = 1$.
			
		\begin{prop} 
			
			\label{sumv1}
			
			As $g$ tends to $\infty$, we have for $20n \le \log g$ that
			\begin{flalign*}
			\Upsilon_{g, n}^{(1)} \sim \pi^{-1} 2^{n + 2} \left( \displaystyle\frac{8}{3} \right)^{4g + n - 4}.
			\end{flalign*} 			
		\end{prop}

		\begin{prop}
			
			\label{lambdav} 
			
			Fix integers $g > 2^{120}$ and $n \ge 0$ such that $20n < \log g$. Then, 
			\begin{flalign*}
			2^{-n} \left( \displaystyle\frac{8}{3} \right)^{-4g - n} \Upsilon_{g, n}^{(2)} \le 2^{140} (\log g)^{14} g^{-1 / 4}.
			\end{flalign*}
		\end{prop}

		\begin{prop} 
			
			\label{lambdag3} 
			
			Fix integers $g > 2^{300}$ and $n \ge 0$ such that $20n \le \log g$. Then,  
			\begin{flalign*}
			\displaystyle\sum_{V = 3}^{2g - 2} 2^{-n} \left( \displaystyle\frac{8}{3} \right)^{-4g - n} \Upsilon_{g, n}^{(V)} \le 2^{240} (\log g)^{24} g^{-1/8}.
			\end{flalign*}
		\end{prop} 
		
		Given these three results, we can quickly establish \Cref{limitvolume}.

		\begin{proof}[Proof of \Cref{limitvolume} Assuming \Cref{sumv1}, \Cref{lambdav}, and \Cref{lambdag3}]
			
			Due to the identity \eqref{lambdavolume}, this follows from \Cref{sumv1}, \Cref{lambdav}, and \Cref{lambdag3}.
		\end{proof}

		Before proceeding, it will be useful introduce the following notation for sets and quantities associated with stable graphs with prescribed total numbers of vertices, self-edges, and simple edges. 
		
		\begin{definition} 
			
		\label{gvst}
		
			Fix integers $g, n \ge 0$ with $2g + n \ge 3$; $V \in [1, 2g + n - 2]$; $S \ge 0$; and $T \ge V - 1$. Set $E = S + T$, and assume $E \le 3g + n - 3$. Let $\mathcal{G}_{g, n} (V; S, T) \subseteq \mathcal{G}_{g, n} (V)$ denote the set of stable graphs $\Gamma \in \mathcal{G}_{g, n} (V)$ with $V$ vertices, $S$ self-edges, and $T$ simple edges (we must have $T \ge V - 1$ for $\Gamma$ to be connected). In particular, $\big| \mathfrak{E}(\Gamma) \big| = S + T = E$ holds for any $\Gamma \in \mathcal{G}_{g, n} (V; S, T)$. Analogously to \eqref{lambdagve}, define
		\begin{flalign*}
		\Upsilon_{g, n}^{(V; S, T)} = \displaystyle\sum_{\Gamma \in \mathcal{G}_{g, n} (V; S, T)} \mathcal{Z} \big( P (\Gamma) \big),
		\end{flalign*}
		
		\noindent so that (by \eqref{3ge})
		\begin{flalign}
		\label{lambdagvsum} 
		\Upsilon_{g, n}^{(V)} = \displaystyle\sum_{T = 0}^{3g + n - 3} \displaystyle\sum_{S = 0}^{3g + n - 3 - T} \Upsilon_{g, n}^{(V; S, T)}. 
		\end{flalign}
		
		\end{definition}

		\subsection{Proof of \Cref{sumv1}}
		
		\label{Lambda1}

		Here we analyze the large genus asymptotics for $\Upsilon_{g, n}^{(1)}$ (recall \Cref{gnv}), which provides the contribution to the right side of \eqref{sumgraphsq} over all stable graphs $\Gamma \in \mathcal{G}_{g, n}$ with one vertex. The following definition provides notation for the unique stable genus $g$ graph with $n$ legs, one vertex, and $E$ edges.
	
		\begin{definition} 
			
		\label{gammage} 
		
		For any integers $g, n \ge 0$ and $E \in [0, g]$, we define $\Gamma_{g, n} (E) \in \mathcal{G}_{g, n} (1)$ to be the stable genus $g$ graph with $n$ legs, one vertex, $E$ self-edges, and genus decoration $\textbf{g} = (g - E)$.
		
		\end{definition}
		
		We now have the following lemma that explicitly evaluates $\mathcal{Z} \big( P \big( \Gamma_{g, n} (E) \big) \big)$.

		\begin{lem} 
			
			\label{lambda1identity} 
		
		For any integers $g \ge 2$, $n \ge 0$, and $E \in [0, g]$, we have 
		\begin{flalign*}
		\mathcal{Z} \Big( P \big( \Gamma_{g, n} (E) \big) \Big) & = \displaystyle\frac{12^E}{2^{2g - 1} 3^g}  \displaystyle\frac{(6g + 2n - 2E - 5)!}{(6g + 2n - 7)!}  \displaystyle\frac{(4g + n - 4)!}{(3g + n - E - 3)! (g - E)!} \displaystyle\frac{1}{E!} \\
		& \quad \times \displaystyle\sum_{\textbf{\emph{d}} \in \mathcal{K}_{2E} (3g + n - E - 3)}  \langle \textbf{\emph{d}}, 0^n \rangle_{g - E, 2E + n}  \displaystyle\prod_{j = 1}^E \binom{2d_{2j - 1} + 2d_{2j} + 2}{2d_{2j} + 1}\displaystyle\frac{\zeta (2d_{2j - 1} 2d_{2j} + 2)}{d_{2j - 1} + d_{2j} + 1},
		\end{flalign*}
		
		\noindent where we have denoted $\textbf{\emph{d}} = (d_1, d_2, \ldots , d_{2E})$ and $(\textbf{\emph{d}}, 0^n) = (d_1, d_2, \ldots, d_{2E}, 0, 0, \ldots , 0)$ (containing $n$ zeroes at the end).
		
		\end{lem} 
	
	\begin{proof}
		
		Let $b_1, b_2, \ldots , b_E$ denote $E$ variables, and define the set 
		\begin{flalign*} 
		\textbf{b} = (b_1, b_1, b_2, b_2, \ldots , b_E, b_E, 0, 0, \ldots , 0),
		\end{flalign*} 
		
		\noindent  where each $b_j$ appears with multiplicity two and $0$ appears with multiplicity $n$ (as prescribed by \Cref{gammap}, since all $E$ edges of $\Gamma_{g, n} (E)$ are self-edges and $\Gamma_{g, n} (E)$ has $n$ legs). Then, $\big| \Aut \big( \Gamma_{g, n} (E) \big) \big| = 2^E E!$, since there are $E!$ permutations of the edges in $\Gamma_{g, n} (E)$ and $2$ ways to permute the half-edges in each edge. Thus, \Cref{gammap} implies that 
		\begin{flalign}
		\label{pgammae} 
		P \big( \Gamma_{g, n} (E) \big) = 2^{6g + 2n - E - 5} \displaystyle\frac{(4g + n - 4)!}{(6g + 2n - 7)! E!} N_{g - E, 2E + n} (\textbf{b}) \displaystyle\prod_{i = 1}^E b_i.
		\end{flalign}
		
		\noindent Next, observe from \Cref{ngnb} that
		\begin{flalign}
		\label{nge}
		\begin{aligned}
		N_{g - E, 2E + n} (\textbf{b}) & = \displaystyle\frac{(6g + 2n - 2E - 5)!!}{2^{5g + n - 3E - 3} 3^{g - E} (g - E)!} \\
		& \qquad \times \displaystyle\sum_{\textbf{d} \in \mathcal{K}_{2E} (3g + n - E - 3)} \langle \textbf{d}, 0^n \rangle_{g - E, 2E + n} \displaystyle\prod_{j = 1}^E \displaystyle\frac{b_j^{2d_{2j - 1} + 2d_{2j}}}{(2d_{2j - 1} + 1)! (2d_{2j} + 1)!},
		\end{aligned} 
		\end{flalign}
		
		\noindent since the last $n$ entries of any $\textbf{d}' \in \mathcal{K}_{2E + n} (3g + n - E - 3)$ contributing to $N_{g, n} (\textbf{b})$ on the right side of \eqref{ngnb1} in our setting must be equal to $0$, since the last $n$ entries of $\textbf{b}$ are.
		
		By \eqref{pgammae}, \eqref{nge}, and the first statement of \eqref{aab}, we obtain
		\begin{flalign*}
		\mathcal{Z} \Big( P \big( \Gamma_{g, n} (E) \big) \Big) & = \displaystyle\frac{2^{g + n - 2} (4g + n - 4)!}{3^g (6g + 2n - 7)!}  \displaystyle\frac{12^{E} (6g + 2n - 2E - 5)!!}{ E! (g - E)!} \\
		& \qquad \times \displaystyle\sum_{\textbf{d} \in \mathcal{K}_{2E} (3g + n - E - 3)} \mathcal{Z} \left( \displaystyle\prod_{j = 1}^E \displaystyle\frac{b_j^{2d_{2j - 1} + 2d_{2j} + 1}}{(2d_{2j - 1} + 1)! (2d_{2j} + 1)!} \right) \langle \textbf{d}, 0^n \rangle_{g - E, 2E + n} \\
		& = \displaystyle\frac{24^E}{2^{2g - 1} 3^g} \displaystyle\frac{(6g + 2n - 2E - 5)!}{(6g + 2n - 7)!} \displaystyle\frac{(4g + n - 4)!}{(3g + n - E - 3)! (g - E)!} \displaystyle\frac{1}{E!} \\
		& \qquad \times  \displaystyle\sum_{\textbf{d} \in \mathcal{K}_{2E} (3g + n - E - 3)} \langle \textbf{d}, 0^n \rangle_{g - E, 2E + n}  \\
		& \qquad \qquad \qquad \qquad \quad \times \displaystyle\prod_{j = 1}^E \displaystyle\frac{(2d_{2j - 1} + 2d_{2j} + 1)! \zeta (2d_{2j - 1} + 2d_{2j} + 2)}{(2d_{2j - 1} + 1)! (2d_{2j} + 1)!}.
		\end{flalign*}
		
		\noindent Now the lemma follows from the last statement of \eqref{aab}. 
	\end{proof}

	Now we must analyze the sum of $\mathcal{Z}\big( \mathcal{P} \big(\Gamma_{g, n}(E) \big) \big)$ over $E \in [0, g]$. The following proposition does this in the case $0 \le E \le 9 \log g$, which provides the leading order contribution. It is a minor generalization of Conditional Theorem F.4 of \cite{VFGINMSC} to the case $n \ne 0$, which was established there assuming \Cref{limitd} and \Cref{sumkn2} above.

	\begin{prop}
		
		\label{lambda1sum1} 
		
		As $g$ tends to $\infty$, we have for $20n \le \log g$ that 
		\begin{flalign*}
		 \displaystyle\sum_{E = 0}^{\lfloor 9 \log g \rfloor} \mathcal{Z} \Big( P \big( \Gamma_{g, n} (E) \big) \Big) \sim \pi^{-1} 2^{n + 2} \left( \displaystyle\frac{8}{3} \right)^{4g + n - 4}.
		\end{flalign*}
	\end{prop}

	\begin{proof}

		In view of \Cref{lambda1identity}, \Cref{limitd}, and the fact that $20n \le \log g$, we have for $E \le 9 \log g$ that
		\begin{flalign*}
		\mathcal{Z} \Big( P \big( \Gamma_{g, n} (E) \big) \Big) & \sim \displaystyle\frac{12^E}{2^{2g - 1} 3^g} \displaystyle\frac{(6g + 2n - 2E - 5)!}{(6g + 2n - 7)!}  \displaystyle\frac{(4g + n - 4)!}{(3g + n - E - 3)! (g - E)!} \displaystyle\frac{1}{E!} \\
		&  \qquad \quad \times \displaystyle\sum_{\textbf{d} \in \mathcal{K}_{2E} (3g + n - E - 3)}  \displaystyle\prod_{j = 1}^E \binom{2d_{2j - 1} + 2d_{2j} + 2}{2d_{2j} + 1}\displaystyle\frac{\zeta (2d_{2j - 1} + 2d_{2j} + 2)}{d_{2j - 1} + d_{2j} + 1}.
		\end{flalign*} 
		
		\noindent Denote $d_{2j - 1} + d_{2j} = D_j$ and set $\textbf{D} = (D_1, D_2, \ldots , D_E)$; since $\textbf{d} \in \mathcal{K}_{2E} (3g + n - E - 3)$, we have $\textbf{D} \in \mathcal{K}_E (3g + n - E - 3)$. So, for $E \le 9 \log g$, 
		\begin{flalign*}
		\mathcal{Z} \Big( P \big( \Gamma_{g, n} (E) \big) \Big) & \sim \displaystyle\frac{12^E}{2^{2g - 1} 3^g}  \displaystyle\frac{(6g + 2n - 2E - 5)!}{(6g + 2n - 7)!}  \displaystyle\frac{(4g + n - 4)!}{(3g + n - E - 3)! (g - E)!}   \\
		&  \qquad \qquad \times \displaystyle\frac{1}{E!}  \displaystyle\sum_{\textbf{D} \in \mathcal{K}_E (3g + n - E - 3)}  \displaystyle\frac{\zeta (2 D_j + 2)}{D_j + 1} \displaystyle\prod_{j = 1}^E \displaystyle\sum_{d_{2j} = 0}^{D_j} \binom{2 D_j + 2}{2d_{2j} + 1}.
		\end{flalign*} 
		
		\noindent By the second statement of \eqref{aab} and the identity $\sum_{j = 1}^E (2D_j + 1) = 6g + 2n - E - 6$, we deduce for $E \le 9 \log g$ that 
		\begin{flalign}
		\label{sumlambdag1e}
		\begin{aligned}
		\mathcal{Z} \Big( P \big( \Gamma_{g, n} (E) \big) \Big) & \sim \displaystyle\frac{2^{4g + 2n - 5} 6^E}{3^g} \displaystyle\frac{(6g + 2n - 2E - 5)!}{(6g + 2n - 7)!}  \displaystyle\frac{(4g + n - 4)!}{(3g + n - E - 3)! (g - E)!} \\
		& \qquad \qquad  \times \displaystyle\frac{1}{E!} \displaystyle\sum_{\textbf{D} \in \mathcal{K}_E (3g + n - E - 3)}  \displaystyle\frac{\zeta (2 D_j + 2)}{D_j + 1}.
		\end{aligned}
		\end{flalign}
		
		Next observe since $20n \le \log g$ we have for $E \le 9 \log g$ that 
		\begin{flalign*}
		\displaystyle\frac{(6g + 2n - 2E - 5)!}{(6g + 2n - 7)!} \sim (6g)^{2 - 2E}; \qquad \displaystyle\frac{(3g + n - 3)!}{(3g + n - E - 3)!} \sim (3g)^E; \qquad \displaystyle\frac{(g - 1)!}{(g - E)!} \sim g^{E - 1}.
		\end{flalign*}
		
		\noindent This, together with \eqref{sumlambdag1e}, implies for $E \le 9 \log g$ that
		\begin{flalign}
		\label{lambdag1e2}
		\mathcal{Z} \Big( P \big( \Gamma_{g, n} (E) \big) \Big) & \sim  \displaystyle\frac{2^{4g + 2n - 3} g}{3^{g - 2}} \displaystyle\frac{(4g + n - 4)!}{(3g + n - 3)! (g - 1)!} \displaystyle\frac{1}{2^E E!} \displaystyle\sum_{\textbf{D} \in \mathcal{C}_E (3g + n - 3)}  \displaystyle\frac{\zeta (2 D_j)}{D_j},
		\end{flalign} 
		
		\noindent where we have also replaced the $D_j + 1$ in \eqref{sumlambdag1e} with $D_j$ here. Again using the bound $20n \le \log g$ gives
		\begin{flalign*} 
		(3g + n - 3)! \sim (3g - 3)! (3g - 3)^n; \qquad (4g + n - 4)! \sim (4g - 4)! (4g - 4)^n,
		\end{flalign*} 
		
		\noindent which upon insertion into \eqref{lambdag1e2} yields for $E \le 9 \log g$ that 
		\begin{flalign}
		\label{lambdag1e4}
		 \mathcal{Z} \Big( P \big( \Gamma_{g, n} (E) \big) \Big) & \sim  \displaystyle\frac{2^{4g - 3} g}{3^{g - 2}} \left( \displaystyle\frac{16}{3} \right)^n \displaystyle\frac{(4g - 4)!}{(3g - 3)! (g - 1)!} \displaystyle\frac{1}{2^E E!} \displaystyle\sum_{\textbf{D} \in \mathcal{C}_E (3g + n - 3)}  \displaystyle\frac{\zeta (2 D_j)}{D_j},
		\end{flalign} 
		
		Further using \eqref{limitk} gives
		 \begin{flalign*}
		 \displaystyle\frac{(4g - 4)!}{(3g - 3)! (g - 1)!} \sim \left( \displaystyle\frac{2}{3 \pi g} \right)^{1/2} \displaystyle\frac{4^{4g - 4}}{3^{3g - 3}},
		 \end{flalign*}
		
		\noindent which upon insertion into \eqref{lambdag1e4} and recalling the definition of $Z_k (m)$ from \Cref{hkzk} yields for $E \le 9 \log g$ that 
		\begin{flalign}
		\label{lambdag1e3}
		\mathcal{Z} \Big( P \big( \Gamma_{g, n} (E) \big) \Big) & \sim  \displaystyle\frac{2^{12 g - 12} }{3^{4g - 5}} \left( \displaystyle\frac{16}{3} \right)^n \left( \displaystyle\frac{2 g}{3 \pi} \right)^{1/2} \displaystyle\frac{Z_E (3g + n - 3)}{2^{E - 1} E!}.
		\end{flalign} 
		
		By the second limit in \eqref{zkestimate2}, we have 
		\begin{flalign*}
		\displaystyle\sum_{E = 0}^{\lfloor 9 \log g \rfloor} \displaystyle\frac{Z_E (3g + n - 3)}{2^{E - 1} E!} \sim \left( \displaystyle\frac{8}{\pi (3g + n - 3)} \right)^{1/2},
		\end{flalign*}
		
		\noindent which upon insertion into \eqref{lambdag1e3} (again using the fact that $20n \le \log g$) yields the proposition.
	\end{proof}

		Before providing an asymptotic for the sum of $\mathcal{Z} \big( P \big( \Gamma_{g, n} (E) \big) \big)$ over $E > 9 \log g$, we require the following bound for a term that will appear in these quantities. For the purposes of establishing \Cref{sumv1} in this section, only the $V = 1$ case of this lemma will be necessary. However, the $V \ge 2$ case will be useful for the proofs of \Cref{lambdav} and \Cref{lambdag3} in \Cref{Estimatev2} and \Cref{Estimateve1}, respectively.

		\begin{lem} 
			
			\label{productestimate2} 
			
			For any integers $g \ge 2$, $n \in [0, g]$, $V \ge 1$, and $E \in [V - 1, 3g + n - 3]$, we have
			\begin{flalign*}
			  \displaystyle\frac{(6g + 2n - 2E - 5)!}{(6g + 2n - 7)!}  \displaystyle\frac{(4g + n - 4)!}{(3g + n - E - 3)! (g - E + V - 1)!} \le 128 g^{3 / 2 - V} 3^{2V} 	12^{- E} \left( \displaystyle\frac{4}{3} \right)^n \left( \displaystyle\frac{256}{27} \right)^{g - 1}.
			\end{flalign*} 
			
		\end{lem} 
		
		\begin{proof}

			Observe using the first identity in \eqref{aab} that 
			\begin{flalign*}
			\displaystyle\frac{(6g + 2n - 2E - 5)!}{(3g + n - E - 3)!} \displaystyle\frac{(3g + n - 3)!}{(6g + 2n - 7)!} = 2^{-E} (6g + 2n - 6) (6g + 2n - 5) \displaystyle\frac{(6g + 2n - 2E - 5)!!}{(6g + 2n - 5)!!},
			\end{flalign*}
			
			\noindent and so 
			\begin{flalign*}
			& \displaystyle\frac{(6g + 2n - 2E - 5)!}{(3g + n - E - 3)!} \displaystyle\frac{(3g + n - 3)!}{(6g + 2n - 7)!} \displaystyle\frac{(g - 1)!}{(g - E + V - 1)!} \\
			& \qquad = 2^{-E} g^{-1} (6g + 2n - 6) (6g + 2n - 5) \displaystyle\prod_{k = 0}^{V - 2} \displaystyle\frac{1}{6g + 2n - 2k - 5} \displaystyle\prod_{j = 0}^{E - V} \displaystyle\frac{g - j}{6g + 2n - 2V - 2j - 3}.
			\end{flalign*} 
			
			\noindent Using the facts that $6(g - j) \le 6g + 2n - 2j - 2V - 3$ for $j \ge V + 1$; that $2(g - j) \le 6g + 2n - 2j - 2V - 3$ for any $j \ge 0$ (since $V \le 2g + n - 2$ by \eqref{gv2}); that $6g + 2n - 2V - 1 \ge 2g$ (again since $V \le 2g + n - 2$); and that $(6g + 2n - 6) (6g + 2n - 5) \le 64 g^2$ (since $n \le g$), it follows that 
			\begin{flalign*}
			\displaystyle\frac{(6g + 2n - 2E - 5)!}{(3g + n - E - 3)!} \displaystyle\frac{(3g + n - 3)!}{(6g + 2n - 7)!} \displaystyle\frac{(g - 1)!}{(g - E)!} \le 64 g^{2 - V} 3^{2V} 12^{ - E}.
			\end{flalign*} 
			
			\noindent Hence, 
			\begin{flalign*}
			\displaystyle\frac{(6g + 2n - 2E - 5)!}{(6g + 2n - 7)!}  \displaystyle\frac{(4g + n - 4)!}{(3g + n - E - 3)! (g - E)!} \le 64 g^{2 - V} 3^{2V} 12^{- E} \displaystyle\frac{(4g + n - 4)!}{(3g + n - 3)! (g - 1)!}.
			\end{flalign*} 
			
			\noindent Since
			\begin{flalign*}
			\displaystyle\frac{(4g + n - 4)!}{(3g + n - 3)!} = \displaystyle\frac{(4g - 4)!}{(3g - 3)!} \displaystyle\prod_{j = 1}^n \displaystyle\frac{4g + j - 4}{3g + j - 3} \le \left( \displaystyle\frac{4}{3} \right)^n \displaystyle\frac{(4g - 4)!}{(3g - 3)!},
			\end{flalign*}
			
			\noindent it follows that
			\begin{flalign}
			\label{gne1}
			\displaystyle\frac{(6g + 2n - 2E - 5)!}{(6g + 2n - 7)!}  \displaystyle\frac{(4g + n - 4)!}{(3g + n - E - 3)! (g - E)!} \le 64 g^{2 - V} 3^{2V} 12^{- E} \left( \displaystyle\frac{4}{3} \right)^n \displaystyle\frac{(4g - 4)!}{(3g - 3)! (g - 1)!}.
			\end{flalign} 
			
			\noindent We additionally have by \eqref{kestimate1} that 
			\begin{flalign*}
			\displaystyle\frac{(4g - 4)!}{(3g - 3)! (g - 1)!} \le (g - 1)^{-1/2} 4^{4g - 4} 3^{3 - 3g} \le 2 g^{-1/2} \left( \displaystyle\frac{256}{27}\right)^{g - 1},
			\end{flalign*} 
			
			\noindent which upon insertion into \eqref{gne1} implies the lemma.
		\end{proof}
		
		Now we can bound the sum of $\mathcal{Z} \big( P \big( \Gamma_{g, n} (E) \big) \big)$ over $E > 9 \log g$. 
		
		\begin{prop}
			
			\label{lambda1sum2}
			
			For any integers $g \ge 2^{70}$ and $n \ge 0$ with $20n \le \log g$, we have 	
			\begin{flalign*}
			\displaystyle\sum_{E = \lfloor 9 \log g \rfloor + 1}^{3g + n - 3} \mathcal{Z} \Big( P \big( \Gamma_{g, n} (E) \big) \Big) \le \displaystyle\frac{3^n}{g} \left( \displaystyle\frac{8}{3} \right)^{4g + n}.
			\end{flalign*}
		\end{prop}
		
		\begin{proof}

			In view of \Cref{lambda1identity} and \Cref{destimateexponential}, we have
			\begin{flalign*}
			\displaystyle\sum_{E = \lfloor 9 \log g \rfloor + 1}^{3g + n - 3} \mathcal{Z} \Big( P \big( \Gamma_{g, n} (E) \big) \Big) & \le \displaystyle\sum_{E = \lfloor 9 \log g \rfloor + 1}^{3g + n - 3} \displaystyle\frac{12^E}{2^{2g - 1} 3^g} \left( \displaystyle\frac{3}{2} \right)^{2E + n} \displaystyle\frac{1}{E!}  \\
			&  \quad \times \displaystyle\frac{(6g + 2n - 2E - 5)!}{(6g + 2n - 7)!}\displaystyle\frac{(4g + n - 4)!}{(3g + n - E - 3)! (g - E)!}\\
			& \quad \times \displaystyle\sum_{\textbf{d} \in \mathcal{K}_{2E} (3g + n - E - 3)}  \displaystyle\prod_{j = 1}^E \binom{2d_{2j - 1} + 2d_{2j} + 2}{2d_{2j} + 1}\displaystyle\frac{\zeta (2d_{2j - 1} + 2d_{2j} + 2)}{d_{2j - 1} + d_{2j} + 1}.
			\end{flalign*} 
			
			\noindent As in the proof of \Cref{lambda1sum1}, we set $d_{2j - 1} + d_{2j} = D_j$ and $\textbf{D} = (D_1, D_2, \ldots , D_E)$; use the second statement of \eqref{aab}; and the identity $\sum_{j = 1}^E (2D_j + 1) = 6g + 2n - E - 6$ to find that
			\begin{flalign*}
			\displaystyle\sum_{E = \lfloor 9 \log g \rfloor + 1}^{3g + n - 3} \mathcal{Z} \Big( P \big( \Gamma_{g, n} (E) \big) \Big) & \le \displaystyle\sum_{E = \lfloor 9 \log g \rfloor + 1}^{3g + n - 3}   \displaystyle\frac{(6g + 2n - 2E - 5)!}{(6g + 2n - 7)!}  \displaystyle\frac{(4g + n - 4)!}{(3g + n - E - 3)! (g - E)!}\\
			& \qquad \qquad \qquad \times \displaystyle\frac{2^{4g + 2n - 5} 6^E}{3^g} \left( \displaystyle\frac{3}{2} \right)^{2E + n} \displaystyle\frac{Z_E (3g + n - 3)}{E!}.
			\end{flalign*} 
			
			\noindent This, together with the $V = 1$ case of \eqref{productestimate2} and the bound $3^5 < 2^8$, implies that
			\begin{flalign*}
			\displaystyle\sum_{E = \lfloor 9 \log g \rfloor + 1}^{3g + n - 3} \mathcal{Z} \Big( P \big( \Gamma_{g, n} (E) \big) \Big) & \le 4 g^{1/2} 8^n \left( \displaystyle\frac{8}{3} \right)^{4g} \displaystyle\sum_{E = \lfloor 9 \log g \rfloor + 1}^{3g + n - 3} \left( \displaystyle\frac{9}{8} \right)^E \displaystyle\frac{Z_E (3g + n - 3)}{E!}.
			\end{flalign*}
			
			\noindent By \Cref{zkmestimate2} and the bounds $3g + n - 3 \ge g$ (which holds since $g \ge 2^{70}$) and $\log (3g + n - 3) \le \log g + 2$ (since $n \le \log g \le g$), we find that
			\begin{flalign}
			\label{sumeestimate}
			\displaystyle\sum_{E = \lfloor 9 \log g \rfloor + 1}^{3g + n - 3} \mathcal{Z} \Big( P \big( \Gamma_{g, n} (E) \big) \Big) & \le 9 g^{-1/2} 3^n \left( \displaystyle\frac{8}{3} \right)^{4g + n} \displaystyle\sum_{E = \lfloor 9 \log g \rfloor + 1}^{3g + n - 3} \left( \displaystyle\frac{9}{8} \right)^{E - 1} \displaystyle\frac{(\log g + 7)^{E - 1}}{(E - 1)!}
			\end{flalign} 
			
			\noindent Now set $R = \frac{9}{8} (\log g + 7)$, and apply \Cref{exponentialm}, with the $\delta$ there equal to $3$ here. Since $\lfloor 9 \log g \rfloor + 1 \ge 7 R$ (as $g \ge 2^{70}$), this gives
			\begin{flalign}
			\label{zsume}
			\displaystyle\sum_{E = \lfloor 9 \log g \rfloor + 1}^{\infty} \left( \displaystyle\frac{9}{8} \right)^{E - 1} \displaystyle\frac{(\log g + 7)^{E - 1}}{(E - 1)!} \le e^{-2R} \le \displaystyle\frac{1}{9 g^2},
			\end{flalign} 
			
			\noindent Combining \eqref{sumeestimate} and \eqref{zsume} then yields the proposition.
		\end{proof} 
		
		Given the above bounds, we can quickly establish \Cref{sumv1}.

		\begin{proof}[Proof of \Cref{sumv1}]
		
			This follows from the definition \eqref{gnv} of $\Upsilon_{g, n}^{(1)}$, combined with \Cref{lambda1sum1} and \Cref{lambda1sum2} (using the fact that $\big( \frac{3}{2} \big)^n < g^{1/2}$, since $20n < \log g$). 
		\end{proof}

		\section{Bounds for \texorpdfstring{$\Upsilon_{g, n}^{(2)}$}{}}
		
		\label{Estimatev2} 
		
		In this section we establish \Cref{lambdav}, which bounds the contribution to the right side of \eqref{sumgraphsq} arising from stable graphs with two vertices. Although these graphs will be addressed slightly differently from those with $V \ge 3$ vertices in \Cref{Estimateve1} below (essentially due to the fact that the decaying factor of $g^{3/2 - V}$ on the right side of \eqref{productestimate2} is not yet small enough when $V = 2$ for the method in \Cref{Estimateve1} to directly apply), both cases will still exhibit several common aspects. So, we hope that the analysis presented here will provide some indication for how the contributions coming from graphs with more vertices can be bounded. 
		
		In \Cref{Graph1} we estimate the contribution to the volume coming from a single graph; in \Cref{Estimate2v} we then simplify this bound, which we use to prove \Cref{lambdav} in \Cref{Estimatev22}. 
		
		\subsection{Estimates for an Individual Graph}
		
		\label{Graph1}
		
		We begin with the following notation for a graph on two vertices with a prescribed number of legs, self-edges, and simple edges at each vertex, and also a prescribed genus decoration.
		
		\begin{definition} 
			
		\label{gn2ts} 
		
		Fix integers $g \ge 2$, $n \ge 0$, $T \ge 1$, and $S \ge 0$; set $E = S + T$, and assume that $E \le 3g + n - 3$. Further fix nonnegative compositions $\textbf{n} = (n_1, n_2) \in \mathcal{K}_2 (n)$, $\textbf{g} = (g_1, g_2) \in \mathcal{K}_2 (g - E + 1)$, and $\textbf{s} = (s_1, s_2) \in \mathcal{K}_2 (S)$ such that $2g_i + 2s_i + n_i + T \ge 3$, for each $i \in \{ 1, 2 \}$. Then let $\mathcal{G}_{g, n} (\textbf{n}; \textbf{s}, T; \textbf{g}) \subseteq \mathcal{G}_{g, n} (2; S, T)$ denote the set of stable graphs on two vertices with $T$ simple edges connecting these vertices, such that the following holds. There exists a labeling of its two vertices with a (distinct) integer in $\{ 1, 2 \}$, such that vertex $i \in \{ 1, 2 \}$ has $n_i$ legs, $s_i$ self-edges, and is assigned $g_i$ under its genus decoration. 
		
		Under this notation, $m_i = 2s_i + n_i + T$ half-edges are incident to vertex $i \in [1, 2]$.
		
		\end{definition} 
	
		The following lemma bounds $\mathcal{Z} \big( P (\Gamma) \big)$ for any $\Gamma \in \mathcal{G}_{g, n} (\textbf{n}; \textbf{s}, T; \textbf{g})$. Here, we recall $Z_k (m)$ from \Cref{hkzk}.
		
		\begin{lem} 
			
			\label{lambdag2ts}
			
			Adopt the notation of \Cref{gn2ts}, and set 
			\begin{flalign}
			\label{kappai} 
			\xi_i = \displaystyle\max_{\textbf{\emph{d}} \in \mathcal{K}_{m_i} (3g_i + m_i + 3)} \langle \textbf{\emph{d}} \rangle_{g_i, m_i}, \qquad \text{for each $i \in \{ 1, 2 \}$}.  
			\end{flalign} 
			
			\noindent Then, for any $\Gamma \in \mathcal{G}_{g, n} (\textbf{\emph{n}}; \textbf{\emph{s}}, T; \textbf{\emph{g}})$, we have that
			\begin{flalign}
			\label{2zpgamma}
			\begin{aligned}
			\mathcal{Z} \big( P (\Gamma) \big)  & \le \displaystyle\frac{2^{4g + 2n + 2E - S - 8}}{3^{g - E + 1}} \displaystyle\frac{(4g + n - 4)!}{(6g + 2n - 7)!} \displaystyle\frac{Z_E (3g + n - 3)}{T!} \displaystyle\prod_{i = 1}^2 \xi_i \displaystyle\frac{(6g_i + 4s_i + 2n_i + 2T - 5)! }{(3g_i + 2s_i + n_i + T - 3)! g_i! s_i!} .
			\end{aligned} 
			\end{flalign}

		\end{lem}

			\begin{proof}
				
				 We first assign variables to each edge of $\Gamma$, following \Cref{gammap}. To that end, label each vertex of $\Gamma$ with an index in $\{ 1, 2 \}$, such that the conditions in \Cref{gn2ts} hold. For any $i \in \{ 1, 2 \}$, let $b_1^{(i)}, b_2^{(i)}, \ldots , b_{s_i}^{(i)}$ be variables (associated with the $s_i$ self-edges at vertex $i$), and let $b_1^{(1, 2)}, b_2^{(1, 2)}, \ldots , b_T^{(1, 2)}$ be variables (associated with the $T$ simple edges between the two vertices of $\Gamma$). For each $i \in \{ 1, 2 \}$, define the variable set 
			\begin{flalign*}
			\textbf{b}^{(i)} = \big( b_1^{(i)}, b_1^{(i)}, b_2^{(i)}, b_2^{(i)}, \ldots , b_{s_i}^{(i)}, b_{s_i}^{(i)} \big) \cup \big( b_1^{(1, 2)}, b_2^{(1, 2)}, \ldots , b_T^{(1, 2)} \big) \cup \{ 0, 0, \ldots , 0 \},
			\end{flalign*} 
			
			\noindent As prescribed in \Cref{gammap}, here the variables $b_k^{(i)}$ associated with self-edges appear with multiplicity two; the variables $b_k^{(1, 2)}$ associated with simple edges appear with multiplicity one; and zero appears with multiplicity $n_i$. 
			
			Recalling the polynomial $N_{g, n}$ from \Cref{ngnb}, we then have that 
			\begin{flalign}
			\label{pgammavsttg2}
			\mathcal{Z} \big( P (\Gamma) \big) = 2^{6g + 2n - 6} \displaystyle\frac{(4g + n - 4)!}{(6g + 2n - 7)!} \displaystyle\frac{1}{\big| \Aut (\Gamma) \big|} \mathcal{Z} \left( \displaystyle\prod_{k = 1}^T b_k^{(1, 2)} \displaystyle\prod_{i = 1}^2 \displaystyle\prod_{k = 1}^{s_i} b_k^{(i)}   \displaystyle\prod_{i = 1}^2 N_{g_i, m_i} \big( \textbf{b}^{(i)} \big) \right).
			\end{flalign}
			
			\noindent Now, observe using the first identity in \eqref{aab} that 
			\begin{flalign}
			\label{gini2}
			\begin{aligned}
			N_{g_i, m_i} \big(\textbf{b}^{(i)} \big) & = \displaystyle\sum_{\textbf{d}^{(i)} \in \mathcal{K}_{2s_i + T} (3g_i + m_i - 3)} \displaystyle\frac{(6g_i + 2m_i - 5)!}{2^{8g_i + 2 m_i - 6} 3^{g_i} (3g_i + m_i - 3)! g_i!} \big\langle \textbf{d}^{(i)}, 0^{n_i} \big\rangle_{g_i, m_i} \\
			& \qquad \qquad \qquad \qquad \qquad \times \displaystyle\prod_{k = 1}^T \displaystyle\frac{\big( b_k^{(1, 2)} \big)^{2 d_{2s_i + k}^{(i)}}}{\big( 2d_{2s_i + k}^{(i)} + 1 \big)!}  \displaystyle\prod_{k = 1}^{s_i} \displaystyle\frac{ \big(b_k^{(i)} \big)^{2d_{2k - 1}^{(i)} + 2 d_{2k}^{(i)}}}{\big( 2 d_{2j - 1}^{(i)} + 1 \big)! \big( 2 d_{2j}^{(i)} + 1 \big)!},
			\end{aligned}
			\end{flalign}
			
			\noindent where for each $i \in \{ 1, 2 \}$ we have denoted 
			\begin{flalign*} 
			\textbf{d}^{(i)} = \big( d_1^{(i)}, d_2^{(i)}, \ldots , d_{2s_i + T}^{(i)} \big); \qquad \big( \textbf{d}^{(i)}, 0^{n_i} \big) = \big( d_1^{(i)}, d_2^{(i)}, \ldots , d_{2s_i + T}^{(i)}, 0, 0, \ldots , 0 \big), 
			\end{flalign*} 
			
			\noindent where $\big( \textbf{d}^{(i)}, 0^{n_i} \big)$ contains $n_i$ zeroes at the end. Indeed, this is due to the fact that the last $n_i$ entries of any $\textbf{d}' \in \mathcal{K}_{m_i} (3g_i + m_i - 3)$ contributing to $N_{g_i, m_i} \big( \textbf{b}^{(i)} \big)$ in the right side of \eqref{ngnb1} in our setting must be equal to $0$, since the last $n_i$ entries of $\textbf{b}^{(i)}$ are.
			
			Together \eqref{gvgve1}, the fact that 
			\begin{flalign*}
			(8g_1 + 2m_1 - 6) + (8g_2 + 2m_2 - 6) = 8(g - E + 1) + 4E + 2n - 12 = 8g + 2n - 4E - 4,
			\end{flalign*}
			
			\noindent and inserting \eqref{gini2} into \eqref{pgammavsttg2} yield
			\begin{flalign}
			\label{zpgamma2estimate} 
			\begin{aligned}
			\mathcal{Z} \big( P (\Gamma) \big) & = \displaystyle\frac{2^{4E - 2g - 2}}{3^{g - E + 1}} \displaystyle\frac{(4g + n - 4)!}{(6g + 2n - 7)!} \displaystyle\frac{1}{\big| \Aut (\Gamma) \big|}  \\
			& \quad \times  \displaystyle\sum_{\textbf{d}^{(1)} \in \mathcal{K}_{2s_1 + T} (3g_1 + m_1 - 3)}  \displaystyle\sum_{\textbf{d}^{(2)} \in \mathcal{K}_{2s_2 + T} (3g_2 + m_2 - 3)} \displaystyle\prod_{i = 1}^2 \big\langle \textbf{d}^{(i)}, 0^{n_i} \big\rangle_{g_i, m_i} \displaystyle\frac{(6g_i + 2m_i - 5)!}{(3g_i + m_i - 3)! g_i!}  \\
			& \quad \qquad \times \mathcal{Z} \left( \displaystyle\prod_{i = 1}^2 \displaystyle\prod_{k = 1}^{s_i} \displaystyle\frac{ \big(b_k^{(i)} \big)^{2d_{2k - 1}^{(i)} + 2 d_{2k}^{(i)} + 1}}{\big( 2 d_{2j - 1}^{(i)} + 1 \big)! \big( 2 d_{2j}^{(i)} + 1 \big)!} \displaystyle\prod_{k = 1}^T \displaystyle\frac{\big( b_k^{(1, 2)} \big)^{2 d_{2s_1 + k}^{(1)} + 2 d_{2s_2 + k}^{(2)} + 1}}{\big( 2d_{2s_1 + k}^{(1)} + 1 \big)! \big( 2d_{2s_2 + k}^{(2)} + 1 \big)!}  \right).
			\end{aligned} 
			\end{flalign}
			
			Next, let us bound $\big|\Aut (\Gamma) \big|$. At each vertex $i \in \{ 1, 2 \}$, there are $s_i$ permutations of the self-edges at vertex $i$ and $2$ ways to permute the half-edges in constituting them edge; moreover, there are $T!$ permutations of the $T$ simple edges connecting vertices $1$ and $2$. Thus, $\big|\Aut (\Gamma) \big| \ge 2^{s_1 + s_2} s_1! s_2! T! = 2^S s_1! s_2! T!$. Inserting this, the definition of $\mathcal{Z}$ (from \Cref{ngnb}), the fact that $m_i = 2s_i + n_i + T$, and the definition \eqref{kappai} of the $\xi_i$ into \eqref{zpgamma2estimate} we obtain
			\begin{flalign}
			\label{zpgammav21} 
			\begin{aligned}
			\mathcal{Z} \big( P (\Gamma) \big) & \le \displaystyle\frac{2^{4E - 2g - S - 2}}{3^{g - E + 1}} \displaystyle\frac{(4g + n - 4)!}{(6g + 2n - 7)!} \displaystyle\frac{1}{T!} \displaystyle\prod_{i = 1}^2 \displaystyle\frac{(6g_i + 4s_i + 2n_i + 2T - 5)!}{(3g_i + 2s_i + n_i + T - 3)! g_i! s_i!}  \\
			& \qquad \times  \displaystyle\sum_{\textbf{d}^{(1)} \in \mathcal{K}_{2s_1 + T} (3g_1 + m_1 - 3)}  \displaystyle\sum_{\textbf{d}^{(2)} \in \mathcal{K}_{2s_2 + T} (3g_2 + m_2 - 3)}  \xi_1 \xi_2 \\
			& \qquad \qquad \qquad \qquad \qquad \times \displaystyle\prod_{i = 1}^2 \displaystyle\prod_{k = 1}^{s_i} \displaystyle\frac{ \big( 2d_{2k - 1}^{(i)} + 2 d_{2k}^{(i)} + 1 \big)! \zeta \big( 2d_{2k - 1}^{(i)} + 2 d_{2k}^{(i)} + 2 \big)}{\big( 2 d_{2j - 1}^{(i)} + 1 \big)! \big( 2 d_{2j}^{(i)} + 1 \big)!} \\
			& \qquad \qquad \qquad \qquad \qquad \times \displaystyle\prod_{k = 1}^T \displaystyle\frac{\big( 2 d_{2s_1 + k}^{(1)} + 2 d_{2s_2 + k}^{(2)} + 1\big)! \zeta \big( 2 d_{2s_1 + k}^{(1)} + 2 d_{2s_2 + k}^{(2)} + 2 \big)}{\big( 2d_{2s_1 + k}^{(1)} + 1 \big)! \big( 2d_{2s_2 + k}^{(2)} + 1 \big)!}.
			\end{aligned} 
			\end{flalign}
			
			Together with the last statement of \eqref{aab} and the identity $s_1 + s_2 + T = E$, \eqref{zpgammav21} gives 
			\begin{flalign} 		
			\label{2zpgamma2}
			\begin{aligned}
				\mathcal{Z} \big( P (\Gamma) \big) & \le \displaystyle\frac{2^{3E - 2g - S + 2}}{3^{g - E + 1}} \displaystyle\frac{(4g + n - 4)!}{(6g + 2n - 7)!} \displaystyle\frac{1}{T!} \\
				& \qquad \times \displaystyle\sum_{\textbf{d}^{(1)} \in \mathcal{K}_{2s_1 + T} (3g_1 + m_1 - 3)}  \displaystyle\sum_{\textbf{d}^{(2)} \in \mathcal{K}_{2s_2 + T} (3g_2 + m_2 - 3)} \displaystyle\prod_{i = 1}^2 \xi_i \displaystyle\frac{(6g_i + 4s_i + 2n_i + 2T - 5)!}{(3g_i + 2s_i + n_i + T - 3)! g_i! s_i!}  \\
				& \qquad \qquad \qquad \qquad \qquad \times \displaystyle\prod_{i = 1}^2 \displaystyle\prod_{k = 1}^{s_i} \binom{2d_{2k - 1}^{(i)} + 2 d_{2k}^{(i)} + 2}{2 d_{2k}^{(i)} + 1}\displaystyle\frac{ \zeta \big( 2d_{2k - 1}^{(i)} + 2 d_{2k}^{(i)} + 2 \big)}{d_{2k - 1}^{(i)} + d_{2k}^{(i)} + 1} \\
				& \qquad \qquad \qquad \qquad \qquad \times \displaystyle\prod_{k = 1}^T \binom{2 d_{2s_1 + k}^{(1)} + 2 d_{2s_2 + k}^{(2)} + 2}{2 d_{2s_2 + k}^{(2)} + 1} \displaystyle\frac{\zeta \big( 2 d_{2s_1 + k}^{(1)} + 2 d_{2s_2 + k}^{(2)} + 2 \big)}{d_{2s_1 + k}^{(1)} + d_{2s_2 + k}^{(2)} + 1}.
			\end{aligned} 
		\end{flalign}

			\noindent Now for each $i \in \{ 1, 2 \}$ and $k \in [1, s_i]$, define 
			\begin{flalign*}
			& D_k^{(i)} = d_{2k - 1}^{(i)} + d_{2k}^{(i)}; \qquad \qquad U_i = 3g_i + m_i - 3 - \displaystyle\sum_{k = 1}^{s_i} D_k^{(i)}; \qquad \qquad U = U_1 + U_2; \\
			& \textbf{D}^{(i)} = \big( D_1^{(i)}, D_2^{(i)}, \ldots , D_{s_i}^{(i)} \big) \in \mathcal{K}_{s_i} (3g_i + m_i - U_i - 3); \qquad \textbf{U} = (U_1, U_2) \in \mathcal{K}_2 (U),
			\end{flalign*} 
			
			\noindent and, for each $k \in [1, T]$, define
			\begin{flalign*}
			 u_k = d_{2s_1 + k}^{(1)} + d_{2s_2 + k}^{(2)}; \qquad \textbf{u} = (u_1, u_2, \ldots , u_T) \in \mathcal{K}_T (U).
			\end{flalign*} 
			
			Under this notation, instead of summing over the $\textbf{d}^{(i)}$ in \eqref{2zpgamma2}, we may sum over $U \in [0, 3g + n - E - 3]$; $\textbf{u} \in \mathcal{K}_T (U)$; the $d_{2s_1 + k}^{(1	)} \in [0, u_k]$; $\textbf{U} \in \mathcal{K}_2 (U)$; the $\textbf{D}^{(i)} \in \mathcal{K}_{s_i} (3g_i + m_i - U_i - 3)$; and the $d_{2k}^{(i)} \in \big[ 0, D_k^{(i)} \big]$. This yields 
			\begin{flalign}
			\label{2zpgammavsttg3}
			\begin{aligned}
			\mathcal{Z} \big( P (\Gamma) \big) & \le \displaystyle\frac{2^{3E - 2g - S - 2} }{3^{g - E + 1}}  \displaystyle\frac{(4g + n - 4)!}{(6g + 2n - 7)!}\displaystyle\frac{1}{T!} \displaystyle\prod_{i = 1}^2 \xi_i  \displaystyle\frac{(6g_i + 4s_i + 2n_i + 2T - 5)!}{(3g_i + 2s_i + n_i + T - 3)! g_i! s_i!} \\
			& \qquad \times \displaystyle\sum_{U = 0}^{3g + n - E - 3} \displaystyle\sum_{\textbf{u} \in \mathcal{K}_T (U)} \displaystyle\prod_{k = 1}^{T} \displaystyle\frac{\zeta \big( 2 u_k + 2\big)}{u_k + 1} \displaystyle\sum_{d_{2s_1 + k}^{(1)} = 0}^{u_k} \binom{2 u_k + 2}{2d_{2s_1 + k}^{(1)} + 1}\\
			& \qquad \times \displaystyle\sum_{\textbf{U} \in \mathcal{K}_2 (U)} \displaystyle\prod_{i = 1}^2 \displaystyle\sum_{\textbf{D}^{(i)} \in \mathcal{K}_{s_i} (3g_i + m_i - U_i - 3)} \displaystyle\prod_{k = 1}^{s_i} \displaystyle\frac{\zeta \big( D_k^{(i)} + 2 \big)}{D_k^{(i)} + 1 } \displaystyle\sum_{d_{2k}^{(i)} = 0}^{D_k^{(i)}} \binom{2D_k^{(i)} + 2}{2 d_{2k}^{(i)} + 1}.
			\end{aligned}
			\end{flalign}
			
			\noindent Using the second statement of \eqref{aab} and the fact (which is a consequence of the identities \eqref{gvgve1}, \eqref{nvsum}, and $S + T = E$) that 
			\begin{flalign*}  
			& \displaystyle\sum_{k = 1}^T (2 u_k + 1) + \displaystyle\sum_{i = 1}^2 \displaystyle\sum_{k = 1}^{s_i} \big( 2 D_k^{(i)} + 1 \big) = 2U + T + 2 \displaystyle\sum_{i = 1}^2 (3g_i + m_i - U_i - 3) + S = 6g + 2n - E - 6,
			\end{flalign*} 
			
			\noindent it follows from \eqref{2zpgammavsttg3} that
			\begin{flalign}
			\label{2zpgammavsttg4}
			\begin{aligned}
			\mathcal{Z} \big( P (\Gamma) \big) & \le \displaystyle\frac{2^{4g + 2n + 2E - S - 8} }{3^{g - E + 1}} \displaystyle\frac{(4g + n - 4)!}{(6g + 2n - 7)!} \displaystyle\frac{1}{T!} \displaystyle\prod_{i = 1}^2 \xi_i  \displaystyle\frac{(6g_i + 4s_i + 2n_i + 2T - 5)!}{(3g_i + 2s_i + n_i + T - 3)! g_i! s_i!} \\
			& \quad \times \displaystyle\sum_{U = 0}^{3g + n - E - 3} \displaystyle\sum_{\textbf{u} \in \mathcal{K}_T (U)} \displaystyle\prod_{k = 1}^T \displaystyle\frac{\zeta (2 u_k + 2)}{u_k + 1} \displaystyle\sum_{\textbf{U} \in \mathcal{K}_2 (U)} \displaystyle\prod_{i = 1}^2 \displaystyle\sum_{\textbf{D}^{(i)} \in \mathcal{K}_{s_i} (3g_i + m_i - U_i - 3)} \displaystyle\prod_{k = 1}^{s_i} \displaystyle\frac{\zeta \big( D_k^{(i)} + 2 \big)}{D_k^{(i)} + 1 }.
			\end{aligned}
			\end{flalign}
			
			\noindent Recalling the definition of $Z_k (m)$ from \Cref{hkzk}, we deduce
			\begin{flalign}
			\label{2zpgammaestimate1} 
			\begin{aligned}
			\mathcal{Z} \big( P (\Gamma) \big) & \le \displaystyle\frac{2^{4g + 2n + 2E - S - 8}}{3^{g - E + 1}} \displaystyle\frac{(4g + n - 4)!}{(6g + 2n - 7)!} \displaystyle\frac{1}{T!} \displaystyle\prod_{i = 1}^2 \xi_i  \displaystyle\frac{(6g_i + 4s_i + 2n_i + T - 5)!}{(3g_i + 2s_i + n_i + T - 3)! g_i! s_i!} \\
			& \qquad \times \displaystyle\sum_{U = 0}^{3g + n - E - 3} Z_T (U + T) \displaystyle\sum_{\textbf{U} \in \mathcal{K}_2 (U)} \displaystyle\prod_{i = 1}^2 Z_{s_i} (3g_i + s_i + m_i - U_i - 3).
			\end{aligned}
			\end{flalign}
			
			\noindent Then, the identities $s_1 + s_2 = S$, $E = S + T$, and
			\begin{flalign*}
			& \displaystyle\sum_{i = 1}^2 (3g_i + s_i + m_i - U_i - 3)  = 3 (g - E + 1) + S + 2E + n - U - 6 = 3g + n + S - E - U - 3,
			\end{flalign*} 
			
			\noindent together with repeated application of \Cref{zkproductsum} gives 
			\begin{flalign*} 
			\displaystyle\sum_{U = 0}^{3g + n - E - 3} Z_T (U + T) & \displaystyle\sum_{\textbf{U} \in \mathcal{K}_2 (U)} \displaystyle\prod_{i = 1}^2 Z_{s_i} (3g_i + s_i + m_i - U_i - 3) \\
			& \le \displaystyle\sum_{U = 0}^{3g + n - E - 3} Z_T (U + T) Z_S (3g + n + S - E - U - 3) \le Z_E (3g + n - 3),
			\end{flalign*}
			
			\noindent which upon insertion into \eqref{2zpgammaestimate1} yields the proposition.
		\end{proof}

		\subsection{Simplification of the Bound}
	
		\label{Estimate2v}

		In this section we simplify the bound appearing on the right side of \eqref{2zpgamma}. We begin with the following lemma that bounds the product there (without the factors of $\xi_i$); its first part provides a general bound, and its second provides an improvement under a certain assumption. Let us mention that the latter improvement will be useful in the proof of \Cref{lambdav} to analyze the contribution of graphs on two vertices but unnecessary to analyze those of graphs with $V \ge 3$ vertices in \Cref{Estimateve1} below. 
		
		\begin{lem}
			
			\label{productgisit1}
			
			Fix integers $g \ge 2$, $n \ge 0$, $S \ge 0$, and $T \ge 1$; set $E = S + T$, and assume that $E \le 3g + n - 3$. Further let $g_1, g_2, s_1, s_2 \ge 0$ denote integers with $g_1 + g_2 = g - E + 1$ and $s_1 + s_2 = S$, such that $2g_i + 2s_i + n_i + T \ge 3$, for each $i \in \{ 1, 2 \}$. Then the following two statements hold. 
			
			\begin{enumerate} 
				
				\item We have that 
				\begin{flalign}
				\label{productgisit2}
				\displaystyle\prod_{i = 1}^2\displaystyle\frac{(6g_i + 4s_i + 2n_i + 2T - 5)!}{(3g_i + 2s_i + n_i + T - 3)! g_i! s_i!} \le 64 (E + 1)  \displaystyle\frac{(6g + 2n - 2E - 5)!}{(3g + n - E - 3)! (g - E + 1)! S!}.
				\end{flalign} 
				
				\item Denote $r = \min \{ 2g_1 + s_1 + n_1 + T - 3, 2g_2 + s_2 + n_2 + T - 3 \}$. If $r \ge r_0 \ge 0$ for some $r_0 \in \mathbb{Z}$, then
				\begin{flalign}
				\label{productgisit3}
				\begin{aligned} 
				& \displaystyle\prod_{i = 1}^2\displaystyle\frac{(6g_i + 4s_i + 2n_i + 2T - 5)!}{(3g_i + 2s_i + n_i + T - 3)! g_i! s_i!} \\
				& \qquad \qquad  \le  2^{r_0 + 7} (E + 1) (2g + n - S - 2)^{-r_0} \displaystyle\frac{(r_0 + 1)! (6g + 2n - 2E - 5)!}{(3g + n - E - 3)! (g - E + 1)! S!}.
				\end{aligned} 
				\end{flalign} 
				
			\end{enumerate}

		\end{lem}
		
		\begin{proof} 
			
			We begin with the proof of the second statement. To that end, first observe that 
			\begin{flalign*}
			\displaystyle\prod_{i = 1}^2\displaystyle\frac{(6g_i + 4s_i + 2n_i + 2T - 5)!}{(3g_i + 2s_i + n_i + T - 3)! g_i! s_i!} & = \displaystyle\prod_{i = 1}^2 \binom{6g_i + 4s_i + 2n_i + 2T - 6}{3g_i + 2s_i + n_i + T - 3, g_i, s_i, 2g_i + s_i + n_i + T - 3} \\
			& \qquad \times  (6g_i + 4s_i + 2n_i + 2T - 5) (2g_i + s_i + n_i + T - 3)! \\
			& \le \binom{6g + 2n - 2E - 6}{3g + n - E - 3, g - E + 1, S, 2g + n - S - 4} \\
			& \qquad \times \displaystyle\prod_{i = 1}^2 (6g_i + 4s_i + 2n_i + 2T - 5) (2g_i + s_i + n_i + T - 3)!, 
			\end{flalign*}
			
			\noindent where in the last inequality we used \Cref{sumaijaiestimate} and the identities $s_1 + s_2 = S$, $S + T = E$, and $g_1 + g_2 = g - E + 1$. Since $6g_i + 4s_i + 2n_i + 2T - 5 \le 8 (2g_i + s_i + n_i + T - 2)$ (as $4(2g_i + s_i + n_i + T) \ge 12$ for each $i \in \{ 1, 2 \}$, since $r \ge 0$) and $S + T = E$, it follows that 
			\begin{flalign*}
			\displaystyle\prod_{i = 1}^2\displaystyle\frac{(6g_i + 4s_i + 2n_i + 2T - 5)!}{(3g_i + 2s_i + n_i + T - 3)! g_i! s_i!} & \le \binom{6g + 2n - 2E - 6}{3g + n - E - 3, g - E + 1, S, 2g + n - S - 4} \\
			& \qquad \times 64 (2g_1 + s_1 + n_1 + T - 2)! (2g_2 + s_2 + n_2 + T - 2)!.
			\end{flalign*}
			
			\noindent Combining this with the fact that 
			\begin{flalign*}
			& \displaystyle\frac{(2g_1 + s_1 + n_1 + T - 2)! (2g_2 + s_2 + n_2 + T - 2)! }{(2g + n - S - 4)!} \\
			& \qquad \qquad \qquad \qquad \qquad = (2g + n - S - 2) (2g + n - S - 3) \binom{2g + n - S - 2}{r + 1}^{-1},
			\end{flalign*}
			
			\noindent and the bounds $\binom{2g + n - S - 2}{r + 1} \ge \binom{2g + n - S - 2}{r_0 + 1}$ (as $r \ge r_0 \ge 0$ and $2r + 2 \le 2g + n - S - 2$) and  $(E + 1) (6g + 2n - 2E - 5) \ge 6g + 2n - E - 5 \ge 2g + n - S - 3$ (where the first statement holds since $6g + 2n - 2E - 5 \ge 1$ and the second holds since $S + T = E \le 3g + n - 3$) we obtain 
			\begin{flalign}
			\label{gisitproduct}
			\begin{aligned}
			\displaystyle\prod_{i = 1}^2\displaystyle\frac{(6g_i + 4s_i + 2n_i + 2T - 5)!}{(3g_i + 2s_i + n_i + T - 3)! g_i! s_i!} & \le 64 (2g + n - S - 2) \binom{2g + n - S - 2}{r_0 + 1}^{-1} (E + 1) \\
			& \qquad \times  \displaystyle\frac{(6g + 2n - 2E - 5)!}{(3g + n - E - 3)! (g - E + 1)! S!}.
			\end{aligned} 
			\end{flalign}
			
			\noindent So, \eqref{productgisit3} follows from the bound $\binom{k}{m} \ge \frac{k^m}{2^m m!}$ if $2m \le k$ (used with $k = 2g - S - 2$ and $m = r_0 + 1$). 
			
			To establish \eqref{productgisit2}, first observe that if $r = \min \{ 2g_1 + s_1 + n_1 + T - 3, 2g_2 + s_2 + n_2 + T - 3 \} \ge 0$ then it holds by the $r_0 = 0$ case of \eqref{gisitproduct}, since $\binom{2g - S - 2}{r + 1} \ge 2g - S - 2$. So, let us assume that $r < 0$, in which case we may suppose that $2g_1 + s_1 + n_1 + T \le 2$. Since $T \ge 1$, this can only occur if $(g_1, s_1, n_1, T) \in \big\{ (0, 0, 0, 1), (0, 0, 0, 2), (0, 0, 1, 1), (0, 1, 0, 1) \big\}$. The first three cases contradict the bound $2g_1 + 2s_1 + n_1 + T \ge 3$. So, we must have $(g_1, s_1, n_1, T) = (0, 1, 0, 1)$, meaning that $(g_2, s_2, n_2, T) = (g - E + 1, E - 2, n, 1)$ (since $s_2 = S - s_1 = E - T - 1 = E - 2$). Then,
			\begin{flalign*}
			\displaystyle\prod_{i = 1}^2\displaystyle\frac{(6g_i + 4s_i + 2n_i + 2T - 5)!}{(3g_i + 2s_i + n_i + T - 3)! g_i! s_i!} = \displaystyle\frac{(6g + 2n - 2E - 5)!}{(3g + n - E - 3)! (g - E + 1)! (E - 2)!},
			\end{flalign*}
			
			\noindent which verifies \eqref{productgisit2} (as $\frac{E + 1}{S!} \ge \frac{1}{(E - 2)!}$, since $S = E - 1$). 
		\end{proof} 
			
			Now we can deduce the following simplified bound on $\mathcal{Z} \big( P ( \Gamma)\big)$. 
			
		\begin{cor} 
			
		\label{zp2gamma12} 
		
		Adopt the notation of \Cref{lambdag2ts}, and assume that $20n \le \log g$. 
		
		\begin{enumerate}
			\item We have  
			\begin{flalign}
			\label{pgamma2v1} 
			2^{-n} \left( \displaystyle\frac{8}{3} \right)^{-4g - n} \mathcal{Z} \big( P ( \Gamma) \big) \le 2^{10} (S^2 + T^2) g^{-3/2}  \xi_1 \xi_2 \displaystyle\frac{(\log g + 7)^{S + T - 1}}{2^S S! T!}.
			\end{flalign}
		
			\item If $r = \min \{ 2g_1 + s_1 + n_1 + T - 3, 2g_2 + s_2 + n_2 + T - 3 \} \ge 10$ and $S \le g - 2$, then
			\begin{flalign}
			\label{pgamma2v2} 
			2^{-n} \left( \displaystyle\frac{8}{3} \right)^{-4g - n} \mathcal{Z} \big( P (\Gamma) \big) & \le 2^{47} (S^2 + T^2) g^{-11} \left( \displaystyle\frac{3}{2} \right)^{2S + 2T + n} \displaystyle\frac{(\log g + 7)^{S + T - 1}}{2^S S! T!}.
			\end{flalign}
		\end{enumerate} 
		
		\end{cor}

		\begin{proof} 
			
			We begin by establishing \eqref{pgamma2v1}. To that end, by \Cref{lambdag2ts} and \eqref{productgisit2}, we obtain
			\begin{flalign*}
			\mathcal{Z} \big( P (\Gamma) \big) & \le (E + 1) 12^E \displaystyle\frac{2^{4g + 2n - 2}}{3^{g + 1}} \xi_1 \xi_2 \displaystyle\frac{Z_E (3g + n - 3)}{2^S S! T!} \\
			& \qquad \times \displaystyle\frac{(6g + 2n - 2E - 5)!}{(6g + 2n - 7)!}\displaystyle\frac{(4g + n - 4)!}{(3g + n - E - 3)! (g - E + 1)!},
			\end{flalign*}
			
			\noindent where we have recalled the $\xi_i$ from \eqref{kappai}. Then applying the $V = 2$ case of \Cref{productestimate2} and using the facts that $3^6 < 2^{10}$ and $E = S + T$, we deduce
			\begin{flalign*}
			2^{-n} \left( \displaystyle\frac{8}{3} \right)^{-4g - n} \mathcal{Z} \big( P (\Gamma) \big) \le 2^7 (S + T + 1) g^{-1/2} \xi_1 \xi_2 \displaystyle\frac{Z_{S + T} (3g + n - 3)}{2^S S! T!}.
			\end{flalign*}
			
			\noindent Thus, by \Cref{zkmestimate2}
			\begin{flalign*}
			2^{-n} \left( \displaystyle\frac{8}{3} \right)^{-4g - n} \mathcal{Z} \big( P (\Gamma) \big) & \le 2^8 (S + T) (S + T + 1) g^{-1/2} (3g + n - 3)^{-1} \xi_1 \xi_2 \\
			& \qquad \times \displaystyle\frac{\big( \log (3g + n - 3) + 5 \big)^{S + T - 1}}{2^S S! T!},
			\end{flalign*}
			
			\noindent which implies \eqref{pgamma2v1} in view of the bounds $3g + n - 3 \ge g$, $\log (3g + n - 3) \le \log g + 2$ (since $n \le g$), and $(S + T) (S + T + 1) \le 3S^2 + 3T^2$. 
			
			To establish \eqref{pgamma2v2}, observe that if $r \ge 10$ then \Cref{lambdag2ts}, \eqref{productgisit3} (with the $r_0$ there equal to $10$ here), and the facts that $2^{17} 11! \le 2^{43}$ and $2g + n - S - 2 \ge g$ for $S \le g - 2$ together yield
			\begin{flalign*}
			\mathcal{Z} \big( P (\Gamma) \big) & \le g^{-10} (E + 1)  12^E \displaystyle\frac{2^{4g + 2n + 35}}{3^{g + 1}} \xi_1 \xi_2 \displaystyle\frac{Z_{S + T} (3g + n - 3)}{2^S S! T!} \\
			& \qquad \times \displaystyle\frac{(6g + 2n - 2E - 5)!}{(6g + 2n - 7)!}\displaystyle\frac{(4g + n - 4)!}{(3g + n - E - 3)! (g - E + 1)!}.
			\end{flalign*}
			
			\noindent Again applying the $V = 2$ case of \Cref{productestimate2}, \Cref{zkmestimate2}, and the facts that $E = S + T$, $3g + n - 3 \ge g$, $\log (3g - 3) \le \log g + 2$, and $3^6 < 2^{10}$, we deduce 
			\begin{flalign*} 
			2^{-n} \left( \displaystyle\frac{8}{3} \right)^{-4g - n} \mathcal{Z} \big( P (\Gamma) \big) & \le 2^{45} (S + T) (S + T + 1) g^{-23/2} \xi_1 \xi_2 \displaystyle\frac{(\log g + 7)^{S + T - 1}}{2^S S! T!}.
			\end{flalign*}
			
			\noindent This, together with the bound $\xi_i \ge \big( \frac{3}{2}\big)^{2s_i + n_i + T}$ (which follows from \Cref{destimateexponential}); the identities $s_1 + s_2 = S$ and $n_1 + n_2 = n$; and the fact that $(S + T) (S + T + 1) \le 3(S^2 + T^2)$, imply the corollary. 
		\end{proof}

		\subsection{Proof of \Cref{lambdav}}
		
		\label{Estimatev22} 
		
		In this section we establish \Cref{lambdav}. To that end, we begin with the following lemma that bounds $\Upsilon_{g, n}^{(2; S, T)}$ (recall \Cref{gvst}). 
		
		\begin{lem} 
			
			\label{lambda2st} 
			
			Fix integers $g \ge 2^{120}$, $n \ge 0$, $T \ge 1$, and $S \ge 0$ with $S + T \le 3g + n - 3$. If $20n \le \log g$, then the following three bounds hold. 
			
			\begin{enumerate} 
				
				\item If $S \ge g - 1$, then 
				\begin{flalign}
				\label{lambda2st5}  
				2^{-n} \left( \displaystyle\frac{8}{3} \right)^{- 4g - n} \Upsilon_{g, n}^{(2; S, T)} \le g^{-10}.
				\end{flalign}

				\item  If $T > 13$ and $S \le g - 2$, then 
				\begin{flalign}
				\label{lambda2st3} 
				2^{-n} \left( \displaystyle\frac{8}{3} \right)^{- 4g - n} \Upsilon_{g, n}^{(2; S, T)} \le 2^{55} g^{-7} 3^n \bigg( \displaystyle\frac{9}{4} \bigg)^{S + T} \displaystyle\frac{(\log g + 7)^{S + T - 1}}{2^S S! T!}.
				\end{flalign}

				\item  If $T \le 13$ and $S \le g - 2$, then
				\begin{flalign}
				\label{lambda2st2} 
				2^{-n} \left( \displaystyle\frac{8}{3} \right)^{- 4g} \Upsilon_{g, n}^{(2; S, T)} \le 2^{72} 3^n \displaystyle\frac{(\log g + 7)^{S + 12}}{S!} \Bigg(g^{-3/2} (S^2 + 1) + g^{-7} \left( \displaystyle\frac{9}{4} \right)^S \Bigg).
				\end{flalign}
				
			\end{enumerate}

		\end{lem}

		\begin{proof}
			
			Set $E = S + T$, and observe that 
			\begin{flalign}
			\label{sumzpgamma2v} 
			2^{-n} \left( \displaystyle\frac{8}{3} \right)^{- 4g - n} \Upsilon_{g, n}^{(2; S, T)} \le \displaystyle\sum_{\textbf{s} \in \mathcal{K}_2 (S)} \displaystyle\sum_{\textbf{g} \in \mathcal{K}_2 (g - E + 1)}  \left( \displaystyle\frac{8}{3} \right)^{- 4g - n} \displaystyle\max_{\textbf{n} \in \mathcal{K}_2 (n)} \displaystyle\max_{\Gamma \in \mathcal{G}_{g, n} (\textbf{n}; \textbf{s}, T; \textbf{g})} \mathcal{Z} \big( P (\Gamma) \big),
			\end{flalign}
			
			\noindent since there are at most $2^n$ ways of labeling the legs of any stable graph $\Gamma \in \mathcal{G}_{g, n} (2)$ on two vertices (as there are two ways to assign a vertex to each index in $\{ 1, 2, \ldots , n \}$). Further observe that since $S + T \le 3g + n - 3 \le 4g - 3$, we have 
			\begin{flalign}
			\label{stk2sk2g} 
			S^2 + T^2 \le 32 g^2; \qquad \big| \mathcal{K}_2 (S) \big| = S + 1 \le 4g; \qquad \big| \mathcal{K}_2 (g - E + 1) \big| \le g + 1 \le  2g,
			\end{flalign} 
			
			To establish \eqref{lambda2st5}, observe that \eqref{pgamma2v1} and the fact that the $\xi_1 \xi_2$ there is at most equal to $\big( \frac{3}{2} \big)^{2S + 2T + n}$ here (by \Cref{destimateexponential}, \eqref{nvsum}, and the fact that $E = S + T$) together yield
			\begin{flalign*}
			\left( \displaystyle\frac{8}{3} \right)^{-4g - n} \mathcal{Z}\big( P (\Gamma) \big) \le 2^{n + 10} (S^2 + T^2) g^{-3/2} \left( \displaystyle\frac{8}{3} \right)^{2S + 2T + n} \displaystyle\frac{(\log g + 7)^{S + T + 1}}{2^S S! T!}.
			\end{flalign*}
			
			\noindent Since $E \ge g - 1$ and $S + T = E \le 3g + n - 3 \le 4g$, it follows that 
			\begin{flalign*}
			\left( \displaystyle\frac{8}{3} \right)^{-4g - n} \mathcal{Z}\big( P (\Gamma) \big) \le 2^{n + 15} g^{1/2} \left( \displaystyle\frac{8}{3} \right)^{8g} \displaystyle\frac{(\log g + 7)^{4g + 1}}{(g - 1)!}.
			\end{flalign*}
			
			\noindent Upon insertion into \eqref{sumzpgamma2v} this gives, together with the last two bounds of \eqref{stk2sk2g} and the fact that $2^n \le \big( \frac{8}{3} \big)^g$, that
			\begin{flalign}
			\label{zpgammasg1}
			2^{-n} \left( \displaystyle\frac{8}{3} \right)^{-4g - n} \mathcal{Z}\big( P (\Gamma) \big) \le 2^{18} g^{7/2} \left( \displaystyle\frac{8}{3} \right)^{9g} \displaystyle\frac{(\log g + 7)^{5g}}{g!}.
			\end{flalign}
			
			\noindent Since \eqref{kestimate1} implies $g! \ge \big( \frac{g}{e} \big)^g$ and the fact that $g > 2^{120}$ implies  
			\begin{flalign*}
			g^{-1} \left( \displaystyle\frac{8}{3} \right)^9 e (\log g + 7)^5 \le 2^{20} (\log g)^5 g^{-1} \le g^{-1/2},
			\end{flalign*}
			
			\noindent it follows from \eqref{zpgammasg1} that 
			\begin{flalign*}
			2^{-n} \left( \displaystyle\frac{8}{3} \right)^{-4g - n} \mathcal{Z}\big( P (\Gamma) \big) \le 2^{18} g^{7/2 - g/2},
			\end{flalign*}
			
			\noindent which implies \eqref{lambda2st5} since $g > 2^{120}$. 
			
			To establish \eqref{lambda2st3} observe that, if $T \ge 13$ and $S \le g - 2$, then \eqref{pgamma2v2} applies, which by \eqref{sumzpgamma2v} yields 
			\begin{flalign}
			\label{lambda2st4}
			\begin{aligned}
			2^{-n} \left( \displaystyle\frac{8}{3} \right)^{- 4g - n} \Upsilon_{g, n}^{(2; S, T)} & \le 2^{n + 47} (S^2 + T^2) g^{-11} \left( \displaystyle\frac{3}{2}\right)^{2S + 2T + n} \displaystyle\frac{(\log g + 7)^{S + T - 1}}{2^S S! T!} \\
			& \qquad \times \big|\mathcal{K}_2 (S) \big| \big| \mathcal{K}_2 (g - E + 1) \big|.
			\end{aligned} 
			\end{flalign} 
			
			\noindent Thus, inserting \eqref{stk2sk2g} into \eqref{lambda2st4} yields \eqref{lambda2st3}. 
			
			Next we establish \eqref{lambda2st3}. For any integer $T \ge 1$ and nonnegative compositions $\textbf{n} = (n_1, n_2) \in \mathcal{K}_2 (n)$, $\textbf{s} = (s_1, s_2) \in \mathcal{K}_2 (S)$, and $\textbf{g} = (g_1, g_2) \in \mathcal{K}_2 (g - E + 1)$, set 
			\begin{flalign*} 
			r_{\textbf{n}; T} (\textbf{s}; \textbf{g}) = \min \{ 2g_1 + s_1 + n_1 + T - 3, 2g_2 + s_2 + n_2 + T - 3 \},
			\end{flalign*}
			
			\noindent and $m_i = 2s_i + n_i + T$ for each $i \in \{ 1, 2 \}$. Let $\mathfrak{R} = \mathfrak{R}_{\textbf{n}} = \mathfrak{R}_{g, \textbf{n}} (S, T)$ denote the set of pairs $(\textbf{s}, \textbf{g}) \in \mathcal{K}_2 (S) \times \mathcal{K}_2 (g - E + 1)$ for which $r_{\textbf{n}; T} (\textbf{s}; \textbf{g}) \le 10$. 
			
			Let us bound $|\mathfrak{R}|$. There are at most two choices for the index $i \in \{ 1, 2 \}$ for which $r_{\textbf{n}; T} (\textbf{s}; \textbf{g}) = 2g_i + s_i + T - 3$. Then, given such an $i$, there are at most $7$ choices for $g_i \in [0, 6]$ and at most $14$ choices for $s_i \in [0, 13]$. Hence, $|\mathfrak{R}| \le 196 < 2^8$. Moreover, $\big| \mathcal{K}_2 (S) \big| \big| \mathcal{K}_2 (g - E + 1) \big| \le 8 g^2$, by the second and third bounds in \eqref{stk2sk2g}. 
			
			Thus, applying \eqref{sumzpgamma2v}; \eqref{pgamma2v2} for $(\textbf{s}, \textbf{g}) \in \mathfrak{R}$; \eqref{pgamma2v1} for $(\textbf{s}, \textbf{g}) \in \mathcal{K}_2 (S) \times \mathcal{K}_2 (g - E + 1)$; the first bound in \eqref{stk2sk2g}; and the fact that $T \le 13$, we find that
			\begin{flalign}
			\label{lambdag2st1}
			\begin{aligned}
			2^{-n} \left( \displaystyle\frac{8}{3} \right)^{- 4g - n} \Upsilon_{g, n}^{(2; S, T)} & \le 2^{n + 18} (S^2 + 169) g^{-3/2} \displaystyle\frac{\eta}{2^S} \displaystyle\frac{(\log g + 7)^{S + 12}}{S!} \\
			& \qquad + 2^{n + 55} g^{-7} \left( \displaystyle\frac{3}{2} \right)^{2S + 26 + n} \displaystyle\frac{(\log g + 7)^{S + 12}}{2^S S!},
			\end{aligned} 
			\end{flalign}
			
			\noindent where we have defined
			\begin{flalign}
			\label{lambdast}
			\eta = \eta_{g, n} (S, T) = \displaystyle\max_{\textbf{n} \in \mathcal{K}_2 (n)} \displaystyle\max_{(\textbf{s}, \textbf{g}) \in \mathfrak{R}_{\textbf{n}}} \displaystyle\max_{\substack{\textbf{d}^{(1)} \in \mathcal{K}_{m_1} (3g_1 + m_1 - 3) \\\textbf{d}^{(2)} \in \mathcal{K}_{m_2} (3g_2 + m_2 - 3)}} \big\langle \textbf{d}^{(1)} \big\rangle_{g_1, m_2} \big\langle \textbf{d}^{(2)} \big\rangle_{g_2, m_2}.
			\end{flalign} 
			
			Let us estimate $\eta$. To that end, let $\textbf{n} \in \mathcal{K}_2 (n)$ and $(\textbf{s}, \textbf{g}) \in \mathfrak{R}$ maximize the right side of \eqref{lambdast}. We may assume that $r_{\textbf{n}; T} (\textbf{s}, \textbf{g}) = 2g_1 + s_1 + n_1 + T - 3$. Then $m_1 = 2s_1 + n_1 + T \le 2 \big( r_{\textbf{n}; T} (\textbf{s}, \textbf{g}) + 3 \big) \le 26$, and so \Cref{destimateexponential} implies that 
			\begin{flalign}
			\label{lambda1} 
			\eta \le \left( \displaystyle\frac{3}{2} \right)^{26} \displaystyle\max_{\textbf{d} \in \mathcal{K}_{m_2} (3g_2 + m_2 - 3)} \langle \textbf{d} \rangle_{g_2, m_2}. 
			\end{flalign}
			
			Now we consider two cases depending on $S$. The first is if $S \ge 15 \log g$, in which case $2^S \ge g^{10}$ and so \Cref{destimateexponential}, \eqref{lambda1}, and the fact that $m_2 \le 2S + T + n \le 2S + n + 13$ together imply that 
			\begin{flalign} 
			\label{lambda2}
			\displaystyle\frac{\eta}{2^S} \le g^{-10} \left( \displaystyle\frac{3}{2} \right)^{2S + n + 39}.
			\end{flalign}
			
			The second is if $S < 15 \log g$, in which case $m_2 \le 2S + T + n \le 30 \log g + n + 13$. Moreover, since $E = S + T$, $g_1 + T \le r_{\textbf{n}; T} (\textbf{s}, \textbf{g}) + 3 \le 13$, $20n \le \log g$, and $g > 2^{120}$, we have 
			\begin{flalign*} 
			g_2 = g - E + 1 - g_1 \ge g - S - 12 \ge g - 15 \log g - 12 > \displaystyle\frac{g}{2} & > 800 (33 \log g + 13)^2 \\
			& \ge 800 (m_2 + 2 \log g_2)^2.
			\end{flalign*} 
	
			\noindent Therefore, \Cref{duppern} applies, which together with the facts that $g_2 > \frac{g}{2}$, $m_2 \le 30 \log g + n + 10$, $g > 2^{120}$, and $20n \le \log g$ gives
			\begin{flalign}
			\label{lambda3} 
			\displaystyle\max_{\textbf{d} \in \mathcal{K}_{m_2} (3g_2 + m_2 - 3)} \langle \textbf{d} \rangle_{g_2, m_2} \le \exp \big( 2500 g^{-1} (33 \log g + 13)^2 \big) \le e.
			\end{flalign}
			
			Combining \eqref{lambda1}, \eqref{lambda2}, and \eqref{lambda3} yields 
			\begin{flalign*} 
			\displaystyle\frac{\eta}{2^S} \le \left( \displaystyle\frac{3}{2} \right)^{30} + g^{-10} \left( \displaystyle\frac{3}{2} \right)^{2S + n + 39},
			\end{flalign*} 
			
			\noindent which upon insertion into \eqref{lambdag2st1} (with the bounds $S \le E \le 3g + n - 3 < 4g$, $169 < 2^8$ and $\big(\frac{3}{2}\big)^{26} < 2^{16}$) yields \eqref{lambda2st2}. 
		\end{proof} 
		
		Now we can establish \Cref{lambdav}.	

		\begin{proof}[Proof of \Cref{lambdav}]
			
			By \eqref{lambdagvsum} and \Cref{lambda2st}, we have 
			\begin{flalign}
			\label{lambdav2}
			\begin{aligned}
			2^{-n} & \left( \displaystyle\frac{8}{3} \right)^{-4g - n} \Upsilon_{g, n}^{(2)} \\
			& = \displaystyle\sum_{S = 0}^{g - 2} 2^{-n} \left( \displaystyle\frac{8}{3} \right)^{-4g - n} \left( \displaystyle\sum_{T = 1}^{12} \Upsilon_{g, n}^{(2; S, T)} + \displaystyle\sum_{T = 13}^{3g + n - S - 3} \Upsilon_{g, n}^{(2; S, T)} \right) \\
			& \qquad + \displaystyle\sum_{S= g - 1}^{3g + n - 3} \displaystyle\sum_{T = 1}^{3g + n - S - 3} 
			2^{-n} \left( \displaystyle\frac{8}{3} \right)^{-4g - n} \Upsilon_{g, n}^{(2; S, T)} \\
			& \le 2^{76} 3^n (\log g + 7)^{12} \displaystyle\sum_{S = 0}^{\infty} \left( g^{-3 / 2} \displaystyle\frac{ (S^2 + 1)}{S!} (\log g + 7)^S + \displaystyle\frac{1}{g^7 S!}  \left( \displaystyle\frac{9 (\log g + 7)}{4}\right)^S \right) \\
			& \qquad + 2^{55} 3^n g^{-7} \displaystyle\sum_{S = 0}^{\infty} \displaystyle\sum_{T= 0}^{\infty} \displaystyle\frac{1}{S! T!} \left( \displaystyle\frac{9 (\log g + 7)}{4} \right)^{S + T} + (4g)^2 g^{-10},
			\end{aligned} 
			\end{flalign}  
			
			\noindent where in the last bound we used the fact that $(3g + n - 3)^2 \le (4g)^2$ (as $20n \le \log g \le g$). 
			
			Now observe since $g > 2^{120}$ that 
			\begin{flalign*}
			\displaystyle\sum_{S = 0}^{\infty}  \displaystyle\frac{ (S^2 + 1)}{S!} (\log g + 7)^S & = \displaystyle\sum_{S = 2}^{\infty} \displaystyle\frac{(\log g + 7)^S}{(S - 2)!} + \displaystyle\sum_{S = 1}^{\infty} \displaystyle\frac{(\log g + 7)^S}{(S - 1)!} + \displaystyle\sum_{S = 0}^{\infty} \displaystyle\frac{(\log g + 7)^S}{S!} \\
			& = \big( (\log g +  7)^2 + \log g + 8 \big) g \le 4 (\log g)^2 g
			\end{flalign*}
			
			\noindent and	
			\begin{flalign*}
			\displaystyle\sum_{S = 0}^{\infty} \displaystyle\sum_{T= 0}^{\infty} \displaystyle\frac{1}{S! T!} \left( \displaystyle\frac{9 (\log g + 7)}{4} \right)^{S + T} & = \left( \displaystyle\sum_{S = 0}^{\infty} \displaystyle\sum_{T= 0}^{\infty} \displaystyle\frac{1}{S!} \left( \displaystyle\frac{9 (\log g + 7)}{4} \right)^S \right)^2 \\
			& = \exp \left( \displaystyle\frac{9 (\log g + 7)}{2} \right) = e^{63/2} g^{9/2}.
			\end{flalign*}
			
			\noindent Upon insertion into \eqref{lambdav2}, these estimates give
			\begin{flalign*}
			2^{-n} \left( \displaystyle\frac{8}{3} \right)^{-4g - n} \Upsilon_{g, n}^{(2)} & = 2^{76} 3^n (\log g + 7)^{12} \big( 4 (\log g)^2 g^{-1/2} + 2 e^{63 / 2} g^{-5/2} + g^{-8} \big).  
			\end{flalign*}
			
			\noindent Together with the bounds $\log g + 7 \le 2 \log g$, $e^{63 / 2} < e^{32} < 2^{48}$, and $3^n < g^{1/4}$ (since $20n \le \log g$) this implies the proposition. 
		\end{proof}

		\section{Bounds for \texorpdfstring{$\Upsilon_{g, n}^{(V)}$}{} if \texorpdfstring{$V > 2$}{}}
		
		\label{Estimateve1}
		
		In this section we establish \Cref{lambdag3}, which bounds the contribution to the right side of \eqref{sumgraphsq} arising from graphs with at least three vertices. We begin in \Cref{Graph2} by analyzing the contribution from a fixed graph. Then, in \Cref{EstimateLambdagvst} we sum this bound to estimate $\Upsilon_{g, n}^{(V; S, T)}$ from \Cref{gvst}, which we use to prove \Cref{lambdag3} in \Cref{Proofv3}.

		\subsection{General Estimates for an Individual Graph}
		
		\label{Graph2}

		We begin with the following definition that specifies the set of stable graphs with a given genus decoration and prescribed numbers of legs, self-edges, and simple edges incident to each vertex. 
		
				\begin{definition} 
			
			\label{gnvts} 
			
			Fix integers $g, n \ge 0$ with $2g + n \ge 4$; $V \in [2, 2g + n - 2]$; $S \ge 0$; and $T \ge V - 1$. Set $E = S + T$, and assume $E \le 3g + n - 3$. Further fix (nonnegative) compositions $\textbf{n} \in \mathcal{K}_V (n)$, $\textbf{g} \in \mathcal{K}_V (g - E + V - 1)$, $\textbf{T} \in \mathcal{C}_V (2T)$, and $\textbf{s} \in \mathcal{K}_V (S)$. Denote 
			\begin{flalign*} 
			\textbf{n} = (n_1, n_2, \ldots , n_V); \qquad \textbf{g} = (g_1, g_2, \ldots, g_V), \qquad \textbf{T} = (T_1, T_2, \ldots,  T_V), \qquad \textbf{s} = (s_1, s_2, \ldots , s_V).
			\end{flalign*} 
			
			\noindent Then let $\mathcal{G}_{g, n} (\textbf{n}; \textbf{s}, \textbf{T}; \textbf{g}) = \mathcal{G}_{g, n} (V; \textbf{n}; \textbf{s}, \textbf{T}; \textbf{g}) \subseteq \mathcal{G}_{g, n} (V; S, T)$ denote the set of stable graphs $\Gamma \in \mathcal{G}_{g, n} (V; S, T)$ such that there exists a labeling of each of the $V$ vertices of $\Gamma$ with a (distinct) integer in $\{ 1, 2, \ldots , V \}$ satisfying the following two properties. 
			
			\begin{itemize}
				\item Vertex $i \in [1, V]$ is incident to $n_i$ legs, $s_i$ self-edges, and $T_i$ simple edges.
				
				\item The genus decoration of $\Gamma$ assigns $g_i$ to vertex $i \in [1, V]$. 
			\end{itemize} 
		
			Observe under this notation that $m_i = 2s_i + n_i + T_i$ half-edges are incident to any vertex $i \in [1, V]$. Further observe that, since $V \ge 2$, we must have $T \ge V - 1$ and each $T_i \ge 1$ to allow for the connectivity of any graph $\Gamma \in \mathcal{G}_{g, n} (\textbf{n}; \textbf{s}, \textbf{T}; \textbf{g})$. 
			
		\end{definition}

		We then have the following proposition, analogous to \Cref{lambdag2ts}, that bounds $\mathcal{Z} \big( P (\Gamma) \big)$ for any fixed $\Gamma \in \mathcal{G}_{g, n} (\textbf{n}; \textbf{s}, \textbf{T}; \textbf{g})$.  
		
		\begin{prop} 
			
			\label{zpgammavstg}
			
			Adopt the notation of \Cref{gnvts}, and fix $\Gamma \in \mathcal{G}_{g, n} (\textbf{\emph{n}}; \textbf{\emph{s}}, \textbf{\emph{T}}; \textbf{\emph{g}})$. Then,
			\begin{flalign*}
			\mathcal{Z} \big( P (\Gamma) \big) & \le \displaystyle\frac{2^{4g + 2n + 2E - 3V - 2}  }{3^{g - E + V - 1}} \left( \displaystyle\frac{3}{2} \right)^{2E + n} \displaystyle\frac{(4g + n - 4)!}{(6g + 2n - 7)!} \displaystyle\frac{Z_E (3g + n - 3)}{ \big| \Aut (\Gamma) \big|} \\
			& \qquad \times \displaystyle\prod_{i = 1}^V \displaystyle\frac{(6g_i + 4s_i + 2n_i + 2T_i - 5)!}{(3g_i + 2s_i + n_i + T_i - 3)! g_i!}.
			\end{flalign*}
		\end{prop} 
	
		\begin{proof}
			
			The proof of this proposition will be similar to that of \Cref{lambdag2ts}. We first define the collection $\textbf{t} = \textbf{t} (\Gamma) = (t_{i, j})_{1 \le i \ne j \le V} \subset \mathbb{Z}_{\ge 0}$ of $V^2 - V$ nonnegative integers as follows. Under the vertex labeling of $\Gamma$ as prescribed in \Cref{gnvts}, let $t_{i, j}$ denote the number of (simple) edges connecting vertices $i$ and $j$, for any $1 \le i \ne j \le V$. Then observe that			
			\begin{flalign*} 
			\displaystyle\sum_{1 \le i < j \le V} t_{i, j} = T, \qquad \text{and} \qquad t_{i, j} = t_{j, i}, \quad \text{for each $1 \le i \ne j \le V$}.
			\end{flalign*} 
			
			We next assign variables to each edge of $\Gamma$, following \Cref{gammap}. For any index $i \in [1, V]$, let $b_1^{(i)}, b_2^{(i)}, \ldots , b_{s_i}^{(i)}$ be variables (associated with the $s_i$ self-edges at vertex $i$) and, for any distinct indices $1 \le i < j \le V$, let $b_1^{(i, j)}, b_2^{(i, j)}, \ldots , b_{t_{i, j}}^{(i, j)}$ be variables (associated with the $t_{i, j}$ simple edges between vertices $i$ and $j$); also set $b_k^{(i, j)} = b_k^{(j, i)}$ for any $1 \le j < i \le V$ and $1 \le k \le t_{i, j}$. For each index $i \in [1, V]$, define the variable set 
			\begin{flalign*}
			\textbf{b}^{(i)} = \big( b_1^{(i)}, b_1^{(i)}, b_2^{(i)}, b_2^{(i)}, \ldots , b_{s_i}^{(i)}, b_{s_i}^{(i)} \big) \cup \bigcup_{j \ne i} \big( b_1^{(i, j)}, b_2^{(i, j)}, \ldots , b_{t_{i, j}}^{(i, j)} \big) \cup (0, 0, \ldots , 0).
			\end{flalign*} 
			
			\noindent As prescribed in \Cref{gammap}, here the variables $b_k^{(i)}$ associated with self-edges appear with multiplicity two; the variables $b_k^{(i, j)}$ associated with simple edges appear with multiplicity one; and zero appears with multiplicity $n_i$. 
			
			Recalling the polynomial $N_{g, n}$ from \Cref{ngnb}, we then have  
			\begin{flalign}
			\label{pgammavsttg}
			\begin{aligned}
			\mathcal{Z} \big( P (\Gamma) \big) & = 2^{6g + 2n - V - 4} \displaystyle\frac{(4g + n - 4)!}{(6g + 2n - 7)!} \displaystyle\frac{1}{\big| \Aut (\Gamma) \big|} \\
			& \qquad \times \mathcal{Z} \left( \displaystyle\prod_{1 \le i < j \le V}\displaystyle\prod_{k = 1}^{t_{i, j}} b_k^{(i, j)} \displaystyle\prod_{i = 1}^V \displaystyle\prod_{k = 1}^{s_i} b_k^{(i)}   \displaystyle\prod_{i = 1}^V N_{g_i, m_i} \big( \textbf{b}^{(i)} \big) \right).
			\end{aligned}
			\end{flalign}
			
			\noindent Now, observe using the first identity in \eqref{aab} that 
			\begin{flalign}
			\label{gini}
			\begin{aligned}
			N_{g_i, m_i} \big(\textbf{b}^{(i)} \big) & = \displaystyle\frac{(6g_i + 2m_i - 5)!}{2^{8g_i + 2 m_i - 6} 3^{g_i} (3g_i + m_i - 3)! g_i!} \displaystyle\sum_{\textbf{d}^{(i)} \in \mathcal{K}_{2s_i + T_i} (3g_i + m_i - 3)} \big\langle \textbf{d}^{(i)}, 0^{n_i} \big\rangle_{g_i, m_i}  \\
			& \qquad \qquad \qquad \qquad \qquad \qquad \times \displaystyle\prod_{j \ne i} \displaystyle\prod_{k = 1}^{t_{i, j}} \displaystyle\frac{\big( b_k^{(i, j)} \big)^{2 d_{j; k}^{(i)}}}{\big( 2d_{j; k}^{(i)} + 1 \big)!}  \displaystyle\prod_{k = 1}^{s_i} \displaystyle\frac{ \big(b_k^{(i)} \big)^{2d_{2k - 1}^{(i)} + 2 d_{2k}^{(i)}}}{\big( 2 d_{2k - 1}^{(i)} + 2 d_{2k}^{(i)} + 1 \big)!} ,
			\end{aligned}
			\end{flalign}
			
			\noindent where we have denoted 
			\begin{flalign*}
			& \textbf{d}^{(i)} = \big( d_1^{(i)}, d_2^{(i)}, \ldots , d_{2s_i}^{(i)} \big) \cup \bigcup_{j \ne i} \big( d_{j; 1}^{(i)}, d_{j; 2}^{(i)}, \ldots , d_{j; t_{i, j}}^{(i)} \big); \\
			& \big(\textbf{d}^{(i)}, 0^{n_i} \big) = \big( d_1^{(i)}, d_2^{(i)}, \ldots , d_{2s_i}^{(i)} \big) \cup \bigcup_{j \ne i} \big( d_{j; 1}^{(i)}, d_{j; 2}^{(i)}, \ldots , d_{j; t_{i, j}}^{(i)} \big) \cup (0, 0, \ldots , 0),
			\end{flalign*}
			
			\noindent where $\big( \textbf{d}^{(i)}, 0^{n_i} \big)$ contains $n_i$ zeroes at the end. Indeed, the last $n_i$ entries of any $\textbf{d}' \in \mathcal{K}_{m_i} (3g_i + m_i - 3)$ contributing to $N_{g_i, m_i} \big( \textbf{b}^{(i)} \big)$ in the right side of \eqref{ngnb1} in our setting must be equal to $0$, since the last $n_i$ entries of $\textbf{b}^{(i)}$ are.

			By \eqref{gvgve1} and the identity   
			\begin{flalign*}
			\displaystyle\sum_{i = 1}^V (8g_i + 2m_i - 6) = 8(g - E + V - 1) + 4E + 2n - 6V = 8g + 2n - 4E + 2V - 8,
			\end{flalign*}
			
			\noindent inserting \eqref{gini} into \eqref{pgammavsttg} yields
			\begin{flalign*}
			\mathcal{Z} \big( P (\Gamma) \big) & = \displaystyle\frac{2^{4E - 3V - 2g + 4}}{3^{g - E + V - 1}} \displaystyle\frac{(4g + n - 4)!}{(6g + 2n - 7)!} \displaystyle\frac{1}{\big| \Aut (\Gamma) \big|} \displaystyle\prod_{i = 1}^V \displaystyle\frac{(6g_i + 2m_i - 5)!}{(3g_i + m_i - 3)! g_i!} \\
			& \qquad \times  \displaystyle\sum_{\textbf{d}^{(1)} \in \mathcal{K}_{2s_1 + T_1} (3g_1 + m_1 - 3)}  \cdots  \displaystyle\sum_{\textbf{d}^{(V)} \in \mathcal{K}_{2s_V + T_V} (3g_V + m_V - 3)} \displaystyle\prod_{i = 1}^V \big\langle \textbf{d}^{(i)}, 0^{n_i} \big\rangle_{g_i, m_i} \\
			& \qquad \qquad \times \mathcal{Z} \left( \displaystyle\prod_{i = 1}^V \displaystyle\prod_{k = 1}^{s_i} \displaystyle\frac{ \big(b_k^{(i)} \big)^{2d_{2k - 1}^{(i)} + 2 d_{2k}^{(i)} + 1}}{\big( 2 d_{2k - 1}^{(i)} + 1 \big)! \big( 2 d_{2k}^{(i)} + 1 \big)!} \displaystyle\prod_{1 \le i < j \le V} \displaystyle\prod_{k = 1}^{t_{i, j}} \displaystyle\frac{\big( b_k^{(i, j)} \big)^{2 d_{j; k}^{(i)} + 2 d_{i; k}^{(j)} + 1}}{\big( 2d_{j; k}^{(i)} + 1 \big)! \big( 2d_{i; k}^{(j)} + 1 \big)!}  \right).
			\end{flalign*}
			
			\noindent Recalling the definition \Cref{ngnb} of $\mathcal{Z}$, and applying \Cref{destimateexponential} and \eqref{nvsum}, it follows that
			\begin{flalign*}
			\mathcal{Z} \big( P (\Gamma) \big) & \le \displaystyle\frac{2^{4E - 3V - 2g + 4}}{3^{g - E + V - 1}} \left( \displaystyle\frac{3}{2} \right)^{2E + n} \displaystyle\frac{(4g + n - 4)!}{(6g + 2n - 7)!} \displaystyle\frac{1}{\big| \Aut (\Gamma) \big|} \\
			& \qquad \times  \displaystyle\sum_{\textbf{d}^{(1)} \in \mathcal{K}_{2s_1 + T_1} (3g_1 + m_1 - 3)}  \cdots  \displaystyle\sum_{\textbf{d}^{(V)} \in \mathcal{K}_{2s_V + T_V} (3g_V + m_V - 3)} \displaystyle\prod_{i = 1}^V \displaystyle\frac{(6g_i + 2m_i - 5)!}{(3g_i + m_i - 3)! g_i!} \\
			& \qquad \qquad \qquad \qquad \times   \displaystyle\prod_{i = 1}^V \displaystyle\prod_{k = 1}^{s_i} \displaystyle\frac{ \big(2d_{2k - 1}^{(i)} + 2 d_{2k}^{(i)} + 1 \big)! \zeta \big( 2d_{2k - 1}^{(i)} + 2 d_{2k}^{(i)} + 2 \big)}{\big( 2 d_{2k - 1}^{(i)} + 1 \big)! \big( 2 d_{2k}^{(i)} + 1 \big)!} \\ 
			& \qquad \qquad \qquad \qquad \times \displaystyle\prod_{1 \le i < j \le V} \displaystyle\prod_{k = 1}^{t_{i, j}} \displaystyle\frac{ \big( 2 d_{j; k}^{(i)} + 2 d_{i; k}^{(j)} + 1 \big)! \zeta \big( 2 d_{j; k}^{(i)} + 2 d_{i; k}^{(j)} + 2\big)}{\big( 2d_{j; k}^{(i)} + 1 \big)! \big( 2d_{i; k}^{(j)} + 1 \big)!}.
			\end{flalign*}
			
			\noindent Together with the last statement of \eqref{aab} and the fact that 
			\begin{flalign}
			\label{sitijsum} 
			\displaystyle\sum_{i = 1}^V s_i + \displaystyle\sum_{1 \le i < j \le V} t_{i, j} = E,
			\end{flalign}
			
			\noindent this gives 
			\begin{flalign}
			\label{zpgammavsttg2}
			\begin{aligned}
			\mathcal{Z} \big( P (\Gamma) \big) & \le \displaystyle\frac{2^{3E - 3V - 2g + 4}}{3^{g - E + V - 1}} \left( \displaystyle\frac{3}{2} \right)^{2E + n} \displaystyle\frac{(4g + n - 4)!}{(6g + 2n - 7)!} \displaystyle\frac{1}{\big| \Aut (\Gamma) \big|} \\
			& \qquad \times  \displaystyle\sum_{\textbf{d}^{(1)} \in \mathcal{K}_{2s_1 + T_1} (3g_1 + m_1 - 3)}  \cdots  \displaystyle\sum_{\textbf{d}^{(V)} \in \mathcal{K}_{2s_V + T_V} (3g_V + m_V - 3)} \displaystyle\prod_{i = 1}^V \displaystyle\frac{(6g_i + 2m_i - 5)!}{(3g_i + m_i - 3)! g_i!} \\
			& \qquad \qquad \qquad \qquad \qquad \times   \displaystyle\prod_{i = 1}^V \displaystyle\prod_{k = 1}^{s_i} \binom{2d_{2k - 1}^{(i)} + 2 d_{2k}^{(i)} + 2}{2 d_{2k}^{(i)} + 1} \displaystyle\frac{\zeta \big( 2d_{2k - 1}^{(i)} + 2 d_{2k}^{(i)} + 2 \big)}{d_{2k - 1}^{(i)} + d_{2k}^{(i)} + 1 }  \\
			& \qquad \qquad \qquad \qquad \qquad \times \displaystyle\prod_{1 \le i < j \le V} \displaystyle\prod_{k = 1}^{t_{i, j}} \binom{2 d_{j; k}^{(i)} + 2 d_{i; k}^{(j)} + 2}{2d_{j; k}^{(i)} + 1} \displaystyle\frac{\zeta \big( 2 d_{j; k}^{(i)} + 2 d_{i; k}^{(j)} + 2\big)}{d_{j; k}^{(i)}  + d_{i; k}^{(j)} + 1}.
			\end{aligned}
			\end{flalign}
			
			\noindent Now for each $1 \le i \ne j \le V$ and $k$ (either in $[1, s_i]$ or $[1, t_{i, j}]$), define 
			\begin{flalign*}
			D_k^{(i)} = d_{2k - 1}^{(i)} + d_{2k}^{(i)}; \qquad u_k^{(i, j)} = d_{j; k}^{(i)} + d_{i; k}^{(j)}; \qquad U_i = 3g_i + m_i - 3 - \displaystyle\sum_{i = 1}^{s_i} D_k^{(i)}; \qquad U = \displaystyle\sum_{i = 1}^V U_i.
			\end{flalign*} 
			
			\noindent Further set
			\begin{flalign*} 
			\textbf{D}^{(i)} = \big( D_1^{(i)}, D_2^{(i)}, \ldots , D_{s_i}^{(i)} \big) \in \mathcal{K}_{s_i} (3g_i + m_i - U_i - 3); \qquad \textbf{U} = (U_1, U_2, \ldots , U_V) \in \mathcal{K}_V (U),
			\end{flalign*} 
			
			\noindent and define the $T$-tuple of nonnegative integers $\textbf{u} = \big( u_k^{(i, j)}\big)$, where $i, j$ range over pairs of distinct indices in $[1, V]$ and $k$ ranges over all indices in $[1, t_{i, j}]$. In particular, $\textbf{u} \in \mathcal{K}_T (U)$. 
			
			Under this notation, instead of summing over the $\textbf{d}^{(i)}$ in \eqref{zpgammavsttg2}, we may sum over $U \in [0, 3g + n - E - 3]$; $\textbf{u} \in \mathcal{K}_T (U)$; the $d_{j; k}^{(i)} \in \big[ 0, u_k^{(i, j)} \big]$; $\textbf{U} \in \mathcal{K}_V (U)$; the $\textbf{D}^{(i)} \in \mathcal{K}_{s_i} (3g_i + m_i - U_i - 3)$; and the $d_{2k}^{(i)} \in \big[ 0, D_k^{(i)} \big]$. This yields 
			\begin{flalign}
			\label{zpgammavsttg3}
			\begin{aligned}
			\mathcal{Z} \big( P (\Gamma) \big) & \le \displaystyle\frac{2^{3E - 3V - 2g + 4} }{3^{g - E + V - 1}}  \left( \displaystyle\frac{3}{2} \right)^{2E + n} \displaystyle\frac{(4g + n - 4)!}{(6g + 2n - 7)!}\displaystyle\frac{1}{\big| \Aut (\Gamma) \big|} \displaystyle\prod_{i = 1}^V \displaystyle\frac{(6g_i + 2m_i - 5)!}{(3g_i + m_i - 3)! g_i!} \\
			& \qquad \times \displaystyle\sum_{U = 0}^{3g + n - E - 3} \displaystyle\sum_{\textbf{u} \in \mathcal{K}_T (U)} \displaystyle\prod_{1 \le i < j \le V} \displaystyle\prod_{k = 1}^{t_{i, j}} \displaystyle\frac{\zeta \big( 2 u_k^{(i, j)} + 2\big)}{u_k^{(i, j)} + 1} \displaystyle\sum_{d_{j; k}^{(i)} = 0}^{u_k^{(i; j)}} \binom{2 u_k^{(i, j)} + 2}{2d_{j; k}^{(i)} + 1}\\
			& \qquad \times \displaystyle\sum_{\textbf{U} \in \mathcal{K}_V (U)} \displaystyle\prod_{i = 1}^V \displaystyle\sum_{\textbf{D}^{(i)} \in \mathcal{K}_{s_i} (3g_i + m_i - U_i - 3)} \displaystyle\prod_{k = 1}^{s_i} \displaystyle\frac{\zeta \big( D_k^{(i)} + 2 \big)}{D_k^{(i)} + 1 } \displaystyle\sum_{d_{2k}^{(i)} = 0}^{D_k^{(i)}} \binom{2D_k^{(i)} + 2}{2 d_{2k}^{(i)} + 1}.
			\end{aligned}
			\end{flalign}
			
			\noindent Using the second statement of \eqref{aab} and the fact (which is a consequence of the identities \eqref{gvgve1}, \eqref{nvsum}, and $E = S + T$)) that 
			\begin{flalign*}
			\displaystyle\sum_{1 \le i < j \le V} \displaystyle\sum_{k = 1}^{t_{i, j}} \big( 2 u_k^{(i, j)} + 1 \big) + \displaystyle\sum_{i = 1}^V \displaystyle\sum_{k = 1}^{s_i} \big( 2 D_k^{(i)} + 1 \big) & = 2U + T + 2 \displaystyle\sum_{i = 1}^V (3g_i + m_i - U_i - 3) + S \\
			& = 6g + 2n - E - 6,
			\end{flalign*} 
			
			\noindent it follows from \eqref{zpgammavsttg3} that
			\begin{flalign}
			\label{zpgammavsttg4}
			\begin{aligned}
			\mathcal{Z} \big( P (\Gamma) \big) & \le \displaystyle\frac{2^{4g + 2n + 2E - 3V - 2} }{3^{g - E + V - 1}} \left( \displaystyle\frac{3}{2} \right)^{2E + n} \displaystyle\frac{(4g + n - 4)!}{(6g + 2n - 7)!} \displaystyle\frac{1}{\big| \Aut (\Gamma) \big|} \displaystyle\prod_{i = 1}^V \displaystyle\frac{(6g_i + 2m_i - 5)!}{(3g_i + m_i - 3)! g_i!} \\
			& \qquad \times \displaystyle\sum_{U = 0}^{3g + n - E - 3} \displaystyle\sum_{\textbf{u} \in \mathcal{K}_T (U)} \displaystyle\prod_{1 \le i < j \le V} \displaystyle\prod_{k = 1}^{t_{i, j}} \displaystyle\frac{\zeta \big( 2 u_k^{(i, j)} + 2\big)}{u_k^{(i, j)} + 1} \\
			& \qquad \times \displaystyle\sum_{\textbf{U} \in \mathcal{K}_V (U)} \displaystyle\prod_{i = 1}^V \displaystyle\sum_{\textbf{D}^{(i)} \in \mathcal{K}_{s_i} (3g_i + m_i - U_i - 3)} \displaystyle\prod_{k = 1}^{s_i} \displaystyle\frac{\zeta \big( D_k^{(i)} + 2 \big)}{D_k^{(i)} + 1 }.
			\end{aligned}
			\end{flalign}
			
			\noindent Now let us define 
			\begin{flalign*}
			w_{i, j} = \displaystyle\sum_{k = 1}^{t_{i, j}} u_k^{(i, j)}; \qquad \textbf{u}^{(i, j)} = \big( u_1^{(i, j)}, u_2^{(i, j)}, \ldots , u_{t_{i, j}}^{(i, j)} \big) \in \mathcal{K}_{t_{i, j}} (w_{i, j}),
			\end{flalign*}   
			
			\noindent and define the $\binom{V}{2}$-tuple of integers $\textbf{w} = (w_{i, j})$, where $i$ and $j$ range over all indices in $[1, V]$, with $i < j$. In particular, we have $\textbf{w} \in \mathcal{K}_{\binom{V}{2}} (U)$. 
			
			So, instead of summing over $\textbf{u}$ in \eqref{zpgammavsttg4}, we may sum over $\textbf{w}$ and $\textbf{u}^{(i, j)}$ to obtain that
			\begin{flalign*}
			\mathcal{Z} \big( P (\Gamma) \big) & \le \displaystyle\frac{2^{4g + 2n + 2E - 3V - 2}}{3^{g - E + V - 1}} \left( \displaystyle\frac{3}{2} \right)^{2E + n} \displaystyle\frac{(4g + n - 4)!}{(6g + 2n - 7)!} \displaystyle\frac{1}{\big| \Aut (\Gamma) \big|} \displaystyle\prod_{i = 1}^V \displaystyle\frac{(6g_i + 2m_i - 5)!}{(3g_i + m_i - 3)! g_i!} \\
			& \qquad \times \displaystyle\sum_{U = 0}^{3g + n - E - 3} \displaystyle\sum_{\textbf{w} \in \mathcal{K}_{\binom{V}{2}} (U)} \prod_{1 \le i < j \le V} \displaystyle\sum_{\textbf{u}^{(i, j)} \in \mathcal{K}_{t_i} (w_{i, j})} \displaystyle\prod_{k = 1}^{t_{i, j}} \displaystyle\frac{\zeta \big( 2 u_k^{(i, j)} + 2\big)}{u_k^{(i, j)} + 1} \\
			& \qquad \times \displaystyle\sum_{\textbf{U} \in \mathcal{K}_V (U)} \displaystyle\prod_{i = 1}^V \displaystyle\sum_{\textbf{D}^{(i)} \in \mathcal{K}_{s_i} (3g_i + m_i - U_i - 3)} \displaystyle\prod_{k = 1}^{s_i} \displaystyle\frac{\zeta \big( D_k^{(i)} + 2 \big)}{D_k^{(i)} + 1 }.
			\end{flalign*}
			
			\noindent Recalling the definition of $Z_k (m)$ from \Cref{hkzk}, it follows that 
			\begin{flalign}
			\label{zpgammaestimate1} 
			\begin{aligned}
			\mathcal{Z} \big( P (\Gamma) \big) & \le \displaystyle\frac{2^{4g + 2n + 2E - 3V - 2}}{3^{g - E + V - 1}} \left( \displaystyle\frac{3}{2} \right)^{2E + n} \displaystyle\frac{(4g + n - 4)!}{(6g + 2n - 7)!} \displaystyle\frac{1}{\big| \Aut (\Gamma) \big|} \displaystyle\prod_{i = 1}^V \displaystyle\frac{(6g_i + 2m_i - 5)!}{(3g_i + m_i - 3)! g_i!} \\
			& \qquad \times \displaystyle\sum_{U = 0}^{3g + n - E - 3} \displaystyle\sum_{\textbf{w} \in \mathcal{K}_{\binom{V}{2}} (U)} \prod_{1 \le i < j \le V} Z_{t_{i, j}} (w_{i, j} + t_{i, j})  \\
			& \qquad \qquad \qquad \qquad \qquad \times \displaystyle\sum_{\textbf{U} \in \mathcal{K}_V (U)} \displaystyle\prod_{i = 1}^V Z_{s_i} (3g_i + s_i + m_i - U_i - 3).
			\end{aligned}
			\end{flalign}
			
			\noindent Then, the identities
			\begin{flalign*}
			\displaystyle\sum_{i = 1}^V (3g_i + s_i + m_i - U_i - 3) & = 3 (g - E + V - 1) + S + 2E + n - U - 3V \\
			& = 3g + n + S - E - U - 3,
			\end{flalign*} 
			
			\noindent and 
			\begin{flalign*} 
			\displaystyle\sum_{1 \le i < j \le V} (w_{i, j} + t_{i, j}) = U + T; \qquad \displaystyle\sum_{1 \le i < j \le V} t_{i, j} = T; \qquad \displaystyle\sum_{i = 1}^V s_i = S; \qquad S + T = E,
			\end{flalign*} 
			
			\noindent and repeated application of \Cref{zkproductsum} gives 
			\begin{flalign*} 
			\displaystyle\sum_{U = 0}^{3g + n - E - 3} \displaystyle\sum_{\textbf{w} \in \mathcal{K}_{\binom{V}{2}} (U)} & \prod_{1 \le i < j \le V} Z_{t_{i, j}} (w_{i, j} + t_{i, j}) \displaystyle\sum_{\textbf{U} \in \mathcal{K}_V (U)} \displaystyle\prod_{i = 1}^V Z_{s_i} (3g_i + s_i + m_i - U_i - 3) \\
			& \le \displaystyle\sum_{U = 0}^{3g + n - E - 3} Z_T (U + T) Z_S (3g + n + S - E - U - 3) \le Z_E (3g +  n - 3),
			\end{flalign*}
			
			\noindent which upon insertion into \eqref{zpgammaestimate1} (and recalling that $m_i = 2s_i + n_i + T_i$) yields the proposition.	
		\end{proof}

		\subsection{Bounds for \texorpdfstring{$\Upsilon_{g, n}^{(V; S, T)}$}{}}
			
		\label{EstimateLambdagvst}	
		
		In this section we estimate the quantity $\Upsilon_{g, n}^{(V; S, T)}$ from \Cref{gvst}. To that end, we begin with the following lemma that bounds this quantity by a certain sum. 
			
		\begin{lem}
			
			\label{lambdagvstestimate} 
			
			Fix integers $g \ge 2$; $n \ge 0$; $V \in [2, 2g + n - 2]$; $S \ge 0$; and $T \ge V - 1$. Set $E = S + T$, and assume $E \le 3g + n - 3$. Then,
			\begin{flalign}
			\label{lambdagvstestimateproduct}
			\begin{aligned} 
			\Upsilon_{g, n}^{(V; S, T)} & \le \displaystyle\frac{2^{4g + 2n + 2E - 3V - S - 2} }{3^{g - E + V - 1}} \left( \displaystyle\frac{3}{2} \right)^{2E + n} \displaystyle\frac{(4g + n - 4)!}{(6g + 2n - 7)!} \displaystyle\frac{n! (2T - 1)!!}{V!} Z_E (3g + n - 3) \\
			& \quad \times \displaystyle\sum_{\textbf{\emph{n}} \in \mathcal{K}_V (n)} \displaystyle\sum_{\textbf{\emph{s}} \in \mathcal{K}_V (S)} \displaystyle\sum_{\textbf{\emph{T}} \in \mathcal{C}_V (2T)} \displaystyle\sum_{\textbf{\emph{g}} \in \mathcal{K}_V (g - E + V - 1)} \displaystyle\prod_{i = 1}^V \displaystyle\frac{(6g_i + 4s_i + 2n_i + 2T_i - 5)!}{(3g_i + 2s_i + n_i + T_i - 3)! g_i! s_i! T_i! n_i!},
			\end{aligned}
			\end{flalign}
			
			\noindent where we have denoted $\textbf{\emph{n}} = (n_1, n_2, \ldots , n_V)$, $\textbf{\emph{g}} = (g_1, g_2, \ldots , g_V)$, $\textbf{\emph{s}} = (s_1, s_2, \ldots , s_V)$, and $\textbf{\emph{T}} = (T_1, T_2, \ldots , T_V)$. 
			
		\end{lem}

		\begin{proof} 
			
			For any compositions $\textbf{n} \in \mathcal{K}_V (n)$, $\textbf{g} \in \mathcal{K}_V (g - E + V - 1)$, $\textbf{s} \in \mathcal{K}_V (S)$, and $\textbf{T} \in \mathcal{C}_V (2T)$, define 
			\begin{flalign*}
			F_{g, n} (\textbf{n}; \textbf{s}, \textbf{T}; \textbf{g}) & = \displaystyle\frac{2^{4g + 2n + 2E - 3V - 2} }{3^{g - E + V - 1}} \left( \displaystyle\frac{3}{2} \right)^{2E + n}  \displaystyle\frac{(4g + n - 4)!}{(6g + 2n - 7)!} Z_E (3g + n - 3) \\
			& \qquad \times \displaystyle\prod_{i = 1}^V \displaystyle\frac{(6g_i + 4s_i + 2n_i + 2T_i - 5)!}{(3g_i + 2s_i + n_i + T_i - 3)! g_i!}.
			\end{flalign*}
			
			\noindent By \Cref{gnvts}, \Cref{zpgammavstg}, the identity $(2A)! = 2^A A! (2A - 1)!!$, and the fact that $\sum_{i = 1}^V s_i = S$, it suffices to show that 
			\begin{flalign}
			\label{gvstestimate1}
			\begin{aligned}
			& \displaystyle\sum_{\Gamma \in \mathcal{G}_{g, n} (V; S, T)} \big| \Aut (\Gamma) \big|^{-1} F_{g, n} \big( \textbf{n} (\Gamma); \textbf{s} (\Gamma), \textbf{T} (\Gamma); \textbf{g} (\Gamma) \big) \\
			& \quad \le \displaystyle\frac{n! (2T - 1)!!}{V!} \displaystyle\sum_{\textbf{n} \in \mathcal{K}_V (n)} \displaystyle\sum_{\textbf{s} \in \mathcal{K}_V (S)} \displaystyle\sum_{\textbf{T} \in \mathcal{C}_V (2T)} \displaystyle\sum_{\textbf{g} \in \mathcal{K}_V (g - E + V - 1)} F_{g, n} (\textbf{n}; \textbf{s}, \textbf{T}; \textbf{g}) \displaystyle\prod_{i = 1}^V \displaystyle\frac{(2s_i - 1)!!}{(2s_i)! T_i! n_i!},
			\end{aligned} 
			\end{flalign}
			
			\noindent where $\textbf{n} (\Gamma)$, $\textbf{s} (\Gamma)$, $\textbf{T} (\Gamma)$, and $\textbf{g} (\Gamma)$ are defined so that $\Gamma \in \mathcal{G}_{g, n} \big( \textbf{n} (\Gamma); \textbf{s} (\Gamma), \textbf{T} (\Gamma); \textbf{g} (\Gamma) \big)$.
						
			To that end, for any integer $k \ge 1$, let $\mathfrak{S}_k$ denote the symmetric group on $k$ elements. Observe that $\mathfrak{S}_V$ acts on $\mathbb{Z}_{\ge 0}^V$ by setting $\sigma (\textbf{a}) = \big( a_{\sigma (1)}, a_{\sigma (2)}, \ldots , a_{\sigma (V)} \big)$, for any $\sigma \in \mathfrak{S}_V$ and $\textbf{a} = (a_1, a_2, \ldots , a_V) \in \mathbb{Z}_{\ge 0}^V$. In this way, $\mathfrak{S}$ acts on a quadruple of compositions in $\mathcal{K}_V (n) \times \mathcal{K}_V (S) \times \mathcal{C}_V (2T) \times \mathcal{K}_V (g - E + V - 1)$ diagonally. Let $\Aut (\textbf{n}, \textbf{s}, \textbf{T}, \textbf{g}) \subseteq \mathfrak{S}_V$ denote the stabilizer of a quadruple $(\textbf{n}, \textbf{s}, \textbf{T}, \textbf{g}) \in \mathcal{K}_V (n) \times \mathcal{K}_V (S) \times \mathcal{C}_V (2T) \times \mathcal{K}_V (g - E + V - 1)$ under this action.
			
			In particular, the orbit under this action of a quadruple $(\textbf{n}, \textbf{s}, \textbf{T}, \textbf{g})$ has size $V! \big| \Aut (\textbf{n}, \textbf{s}, \textbf{T}, \textbf{g}) \big|^{-1}$. So, since the vertices of any stable graph are unlabeled, we obtain that
			\begin{flalign*}
			& \displaystyle\sum_{\Gamma \in \mathcal{G}_{g, n} (V; S, T)} \big| \Aut (\Gamma) \big|^{-1} F\big( \textbf{n} (\Gamma); \textbf{s} (\Gamma), \textbf{T} (\Gamma); \textbf{g} (\Gamma) \big) \\
			& \qquad \quad \le  \displaystyle\sum_{\textbf{n} \in \mathcal{K}_V (n)} \displaystyle\sum_{\textbf{s} \in \mathcal{K}_V (S)} \displaystyle\sum_{\textbf{T} \in \mathcal{C}_V (2T)} \displaystyle\sum_{\textbf{g} \in \mathcal{K}_V (g - E + V - 1)} V!^{-1} \big| \Aut (\textbf{n}, \textbf{s}, \textbf{T}, \textbf{g}) \big| \\
			& \qquad \qquad \qquad \qquad \qquad \qquad \qquad \qquad \qquad  \times  \displaystyle\sum_{\Gamma \in \mathcal{G}_{g, n} (\textbf{n}; \textbf{s}, \textbf{T}; \textbf{g})} \big| \Aut (\Gamma) \big|^{-1} F( \textbf{n}; \textbf{s}, \textbf{T}; \textbf{g}).
			\end{flalign*}
			
			\noindent Therefore, to establish \eqref{gvstestimate1}, it suffices to show for any fixed $(\textbf{n}, \textbf{s}, \textbf{t}, \textbf{g}) \in \mathcal{K}_V (n) \times \mathcal{K}_V (S) \times \mathcal{C}_V (2T) \times \mathcal{K}_V (g - E + V - 1)$ that
			\begin{flalign}
			\label{stg}
			\begin{aligned} 
			 \displaystyle\prod_{i = 1}^V (2s_i + T_i)! \displaystyle\sum_{\Gamma \in \mathcal{G}_{g, n} (\textbf{n}; \textbf{s}, \textbf{T}; \textbf{g})} & \big| \Aut (\Gamma) \big|^{-1} \big| \Aut (\textbf{n}, \textbf{s}, \textbf{T}, \textbf{g}) \big| \\
			 & \le  \binom{n}{n_1, n_2, \ldots , n_V} (2T - 1)!! \displaystyle\prod_{i = 1}^V \binom{2s_i + T_i}{2s_i} (2s_i - 1)!!.
			 \end{aligned} 
			\end{flalign}
			
			To establish \eqref{stg}, we will compare both sides \eqref{stg} to the number of \emph{labeled stable graphs}, which constitute a genus $g$ stable graph $\Gamma$ with $n$ legs together with bijections $\mathfrak{v}: \mathfrak{V} (\Gamma) \rightarrow \{ 1, 2, \ldots , V \}$ and $\mathfrak{h}: \mathfrak{H} (\Gamma) \setminus \mathfrak{L} (\Gamma) \rightarrow \{ 1, 2, \ldots , 2E \}$ (that is, a labeling of each vertex and each half-edge that is not a leg of $\Gamma$ with a distinct index in $[1, V]$ and $[1, 2E]$, respectively), subject to the following three conditions. First, if $h, h' \in \mathfrak{H}$ are two distinct half-edges such that $\alpha (h) < \alpha (h')$, then $\mathfrak{h} (h) < \mathfrak{h} (h')$; stated alternatively, any half-edge incident to vertex $i \in [1, V]$  has label in the interval $[W_{i - 1} - 1, W_i ]$, where $W_0 = 0$ and $W_k = \sum_{j = 1}^k (2s_j + T_j)$, for each $k \in [1, V]$. Second, vertex $i \in [1, V]$ is incident to $n_i$ legs, $s_i$ self-edges, and $T_i$ simple edges. Third, the genus decoration of $\Gamma$ assigns the integer $g_i$ to vertex $i$, for each $i \in [1, V]$. 
			
			As for stable graphs, an \emph{isomorphism} between two labeled graphs, $\Gamma = (\mathfrak{V}, \mathfrak{H}, \alpha, \mathfrak{L}, \lambda, \mathfrak{i}, \mathfrak{v}, \mathfrak{h}, \textbf{g})$ and $\Gamma' = (\mathfrak{V}', \mathfrak{H}', \alpha', \mathfrak{L}', \lambda', \mathfrak{i}', \mathfrak{v}', \mathfrak{h}', \textbf{g}')$, consists of two bijections $\mu: \mathfrak{V} \rightarrow \mathfrak{V}$ and $\nu: \mathfrak{H} \rightarrow \mathfrak{H}'$ such that $\alpha' \big( \nu (h) \big) = \mu \big( \alpha (h) \big)$ and $\mathfrak{i}' \big( \nu (h) \big) = \nu \big( \mathfrak{i} (h) \big)$, for each half-edge $h \in \mathfrak{H}$; such that $\lambda' \big( \nu (h) \big) = \lambda (h)$, for each leg $h \in \mathfrak{L}$; such that $\mathfrak{h}' \big( \nu (h) \big) = \mathfrak{h} (h)$, for each half-edge $h \in \mathfrak{H} \setminus \mathfrak{L}$ that is not a leg; and such that $\mathfrak{v}' \big( \mu (v) \big) = \mathfrak{v} (v)$ and $g_{\mu (v)}' = g_v$, for each vertex $v \in \mathfrak{V}$. An isomorphism from a labeled stable graph to itself is an called an \emph{automorphism} of the labeled stable graph. 
			
			To see that the number of labeled stable graphs (up to automorphism) is bounded above by the right side of \eqref{stg}, observe that any such graph can be produced as follows. First, fix the set of leg labels at each vertex in $\Gamma$; this can be done in $\binom{n}{n_1, n_2, \ldots , n_V}$ ways. Then, for each vertex $i \in [1, V]$, fix which half-edge labels at vertex $i$ (in the interval $[W_{i - 1} + 1, W_i]$) correspond to self-edges and which correspond to simple edges; for each $i$, this can be done in $\binom{2s_i + T_i}{2s_i}$ ways. Next, at each vertex $i \in [1, V]$, fix $s_i$ pairs of half-edges that combine to form self-edges; for each $i$, this can be done in $(2s_i - 1)!!$ ways. Then, fix which pairs of half-edges combine to form simple edges; this can be done in at most $(2T - 1)!!$ ways. The number of labeled stabled graphs is therefore at most equal to the product of these quantities, which is the right side of \eqref{stg}.
			
			 To see that the number of labeled stable graphs (up to automorphism) is equal to the left side of \eqref{stg}, observe that the product of symmetric groups $\mathfrak{S} = \prod_{i = 1}^V \mathfrak{S}_{2s_i + T_i}$ acts on the set of such graphs by having $\mathfrak{S}_{2s_i + T_i}$ permute the $2s_i + T_i$ labels for half-edges, which are not legs, incident to vertex $i$ (in the interval $[W_{i - 1} + 1, W_i]$), for each $i \in [1, V]$. This group action fixes any stable map, but might change the labeling of a labeled stable graph. Let $\Aut^{(\text{lab})} (\Gamma)$ denote the automorphism group of any labeled stable graph $\Gamma$, which only depends on the (unlabeled) graph $\Gamma$ and not on the labeling. 
			 
			 Then, $\Aut (\Gamma)$ is isomorphic to $\Aut^{(\text{lab})} (\Gamma) \times \Aut (\textbf{n}, \textbf{s}, \textbf{T}, \textbf{g})$ for any $\Gamma \in \mathcal{G}_{g, n} (\textbf{n}; \textbf{s}, \textbf{T}; \textbf{g})$. So,
			 \begin{flalign*}
			 \displaystyle\prod_{i = 1}^V (2s_i + T_i)! \displaystyle\sum_{\Gamma \in \mathcal{G}_{g, n} (\textbf{n}; \textbf{s}, \textbf{T}; \textbf{g})} \big| \Aut (\Gamma) \big|^{-1} \big| \Aut (\textbf{n}, \textbf{s}, \textbf{T}, \textbf{g}) \big| = \displaystyle\sum_{\Gamma \in \mathcal{G}_{g, n} (\textbf{n}; \textbf{s}, \textbf{T}; \textbf{g})} \big| \Aut^{(\text{lab})}(\Gamma) \big|^{-1} |\mathfrak{S}|, 
			 \end{flalign*} 
			
			\noindent which counts the number of labeled stable graphs, since every labeled stable graph is obtained by applying an element of $\mathfrak{S}$ to a stable (unlabeled) graph in $\mathcal{G}_{g, n} (\textbf{n}; \textbf{s}, \textbf{T}; \textbf{g})$.
			
			It follows that the number of labeled stable graphs is equal to the left side of \eqref{stg} and bounded above by the right side of \eqref{stg}. This establishes \eqref{stg} and therefore the lemma. 
		\end{proof} 	
			
		The following proposition estimates the sum on the right side of \eqref{lambdagvstestimateproduct}, thereby providing a simplified bound on $\Upsilon_{g, n}^{(V; S, T)}$. 
					
			\begin{prop}
				
				\label{lambdagvstestimate2} 
				
				Fix integers $g \ge 2$; $n \ge 0$; $V \in [2, 2g + n - 2]$; $S \ge 0$; and $T \ge V - 1$. Set $E = S + T$, and assume $E \le 3g + n - 3$. Denoting $X = \min \{ S, V \}$ and $Y = \min \{ 2T, 3V\}$, we have 
				\begin{flalign*}
				2^{-n} \left( \displaystyle\frac{8}{3} \right)^{-4g - n} \Upsilon_{g, n}^{(V; S, T)} & \le (S + T) g^{1/2 - V}   2^{20 V + 11} \left( \displaystyle\frac{3V}{2} \right)^n \left( \displaystyle\frac{9}{8} \right)^S \left( \displaystyle\frac{9}{4} \right)^T \left( \displaystyle\frac{T}{V} \right)^{2V} \\
				& \qquad \times \displaystyle\frac{(\log g + 7)^{S + T - 1} (2T - 1)!!}{V^V (S - X)! (2T - Y)!}.
				\end{flalign*}

			\end{prop} 	
			
			\begin{proof} 
				
				We begin by estimating the sum on the right side of \eqref{lambdagvstestimateproduct} over $(\textbf{s}, \textbf{T}, \textbf{g}) \in \mathcal{K}_V (S) \times \mathcal{C}_V (2T) \times \mathcal{K}_V (g)$, for some fixed $\textbf{n} \in \mathcal{K}_V (n)$. To that end, fix $\textbf{n} = (n_1, n_2, \ldots , n_V) \in \mathcal{K}_V (n)$ and define for each $i \in [1, V]$ and any $(\textbf{s}, \textbf{T}, \textbf{g}) \in \mathcal{K}_V (S) \times \mathcal{C}_V (2T) \times \mathcal{K}_V (g)$ the quantities 
				\begin{flalign*} 
				A_i = \min \{ s_i, 1 \}; \qquad B_i = \min \{ T_i, 3 \}; \qquad A = \displaystyle\sum_{i = 1}^V A_i; \qquad B = \displaystyle\sum_{i = 1}^V B_i.
				\end{flalign*} 
				
				\noindent Then, we claim that 
				\begin{flalign} 
				\label{gisiaibisum} 
				2g_i + s_i + n_i + A_i + B_i \ge 3, \qquad \text{for each $i$.}
				\end{flalign} 
				
				 Indeed, if $T_i \ge 3$, then $B_i \ge 3$ and \eqref{gisiaibisum} holds. Therefore, let us assume that $T_i < 3$, in which case $B_i = T_i$, and so $2g_i + s_i + n_i + A_i + B_i = 2g_i + s_i + n_i + A_i + T_i$. If $s_i \le 1$, then $A_i = s_i$, and so \eqref{gvnv3} implies $2g_i + s_i + n_i + A_i + B_i = 2g_i + 2s_i + n_i + T_i \ge 3$, which verifies \eqref{gisiaibisum}. If instead $s_i \ge 2$, then $A_i = 1$ and so $2g_i + s_i + n_i + A_i + B_i \ge s_i + A_i \ge 3$, which again confirms \eqref{gisiaibisum}. 
				 
				 Thus, since $(s_i - A_i)! \le s_i!$ and $(T_i - B_i)! \le T_i!$, we have
				\begin{flalign*} 
				& \displaystyle\frac{(6g_i + 4s_i + 2n_i + 2T_i - 5)!}{(3g_i + 2s_i + n_i + T_i - 3)! g_i! s_i! T_i!} \\
				& \qquad \qquad \qquad \le \binom{6g_i + 4s_i + 2n_i + 2 T_i - 6}{3g_i + 2s_i + n_i + T_i - 3, g_i, s_i - A_i, T_i - B_i, 2g_i + s_i + n_i + A_i + B_i - 3} \\
				& \qquad \qquad \qquad \qquad \times (6g_i + 4s_i + 2n_i + 2T_i - 5) (2g_i + s_i + n_i + A_i + B_i - 3)!.
				\end{flalign*} 
				
				\noindent So, using \Cref{sumaijaiestimate} together with \eqref{gvgve1}, \eqref{nvsum}, and the facts that $\sum_{i = 1}^V s_i = S$ and $\sum_{i = 1}^V T_i = 2T$, we obtain
				\begin{flalign}
				\label{productgi1} 
				\begin{aligned}
				& \displaystyle\prod_{i = 1}^V \displaystyle\frac{(6g_i + 4s_i + 2n_i + 2T_i - 5)!}{(3g_i + 2s_i + n_i + T_i - 3)! g_i! s_i! T_i!} \\
				& \qquad \le \binom{6g + 2n - 2E - 6}{3g + n - E - 3, g - E + V - 1, S - A, 2T - B, 2g + n - 2E - V + S + A + B - 2} \\
				& \qquad \qquad \times \displaystyle\prod_{i = 1}^V (6g_i + 4s_i + 2n_i + 2T_i - 5) (2g_i + s_i + n_i + A_i + B_i - 3)!.
				\end{aligned} 
				\end{flalign}
				
				\noindent Now, observe that 
				\begin{flalign*} 
				6g_i + 4s_i + 2n_i + 2T_i - 5 & \le 4 (2g_i + s_i + n_i + A_i + B_i - 2) + 2T_i \\
				& \le 2 (T_i + 2) (2g_i + s_i + n_i + A_i + B_i - 2),
				\end{flalign*}
				
				\noindent where the first bound follows from the fact that $B_i \ge 1$ and second from the fact that $2g_i + s_i + n_i + A_i + B_i - 2 \ge 1$. Hence, 
				\begin{flalign}
				\label{productgi2}
				\begin{aligned}
				\displaystyle\prod_{i = 1}^V (6g_i + 4s_i + 2n_i + 2T_i - 5) & (2g_i + s_i + n_i + A_i + B_i - 3)! \\
				 & \le 2^V \displaystyle\prod_{i = 1}^V (T_i + 2) \displaystyle\prod_{i = 1}^V (2g_i + s_i + n_i + A_i + B_i - 2)!  \\
				& \le 2^V \left( \displaystyle\frac{2T + 2V}{V} \right)^V \displaystyle\prod_{i = 1}^V (2g_i + s_i + n_i + A_i + B_i - 2)! \\
				& \le 12^V \left( \displaystyle\frac{T}{V} \right)^V \displaystyle\prod_{i = 1}^V (2g_i + s_i + n_i + A_i + B_i - 2)!.
				\end{aligned}
				\end{flalign}
		 
				\noindent Here, in the second inequality, we used the fact that $\sum_{i = 1}^V (T_i + 2) = 2T + 2V$ and, in the third, we used the fact that $V \le T + 1 \le 2T$. Moreover, since each $A_i \le 1$ and each $B_i \le 3$, we have $A \le \min \{ S, V \} = X$ and $B \le \min \{ 2T, 3V \} \le Y$. 
				
				Thus, $(S - A)! \ge (S - X)!$ and $(2T - B)! \ge (2T - Y)!$, and so \eqref{productgi1} and \eqref{productgi2} together imply
				\begin{flalign}
				\label{productgi3}
				\begin{aligned}
				\displaystyle\prod_{i = 1}^V \displaystyle\frac{(6g_i + 4s_i + 2n_i + 2T_i - 5)!}{(3g_i + 2s_i + n_i + T_i - 3)! g_i! s_i! T_i!} & \le \displaystyle\frac{(6g + 2n - 2E - 6)!}{(3g + n - E - 3)!} \displaystyle\frac{1}{(g - E + V - 1)!} \displaystyle\frac{1}{(S - X)! (2T - Y)!} \\
				& \qquad \times 12^V \left( \displaystyle\frac{T}{V} \right)^V \displaystyle\frac{\prod_{i = 1}^V (2g_i + s_i + n_i + A_i + B_i - 2)!}{(2g + n - 2E - V + S + A + B - 2)!}.
				\end{aligned} 
				\end{flalign}
				
				\noindent Denoting for each $i \in [1, V]$ the quantities 
				\begin{flalign*} 
				M_i = s_i + n_i + A_i, \qquad \textbf{B} = (B_1, B_2, \ldots , B_V); \qquad \textbf{M} = (M_1, M_2, \ldots , M_V),
				\end{flalign*} 
				
				\noindent we find that $\textbf{B} \in \mathcal{C}_V (B)$ and $\textbf{M} \in \mathcal{K}_V (S + n + A)$. So, 
				\begin{flalign*} 
				\displaystyle\sum_{\textbf{s} \in \mathcal{K}_V (S)} \displaystyle\sum_{\textbf{T} \in \mathcal{C}_V (2T)} & \displaystyle\sum_{\textbf{g} \in \mathcal{K}_V (g - E + V - 1)} \displaystyle\prod_{i = 1}^V (2g_i + s_i + n_i + A_i + B_i - 2)! \\
				& \le \big| \mathcal{C}_V (2T) \big| \displaystyle\sum_{\textbf{M} \in \mathcal{K}_V (S + n + A)} \displaystyle\sum_{\textbf{B} \in \mathcal{C}_V (B)} \displaystyle\sum_{\textbf{g} \in \mathcal{K}_V (g - E + V - 1)} \displaystyle\prod_{i = 1}^V (2g_i + M_i + B_i - 2)! \\
				& \le 2^{12V + 9} \binom{2T - 1}{V - 1} (2g + n - 2E - V + S + A + B - 1)!,
				\end{flalign*} 
				
				\noindent where in the last inequality we applied the first identity in \eqref{cmnkmn} and \Cref{sumaibici2} (with the $m$ there equal to $V$ here; the $A_i$, $B_i$, and $C_i$ there equal to $2g_i$, $B_i$, and $M_i$ here, respectively; and the $A$, $B$, and $C$ there equal to $2g - 2E + 2V - 2$, $B$, and $S + n + A$ here, respectively). Together with \eqref{productgi3} and the fact that $2g + n - 2E - V + S + A + B - 1 \le 2g + n - E - 1 \le 6g + 2n - 2E - 5$ (where in the first bound we used the facts that $A \le V$ and $S + B \le S + T = E$, and in the second we used \eqref{3ge} and the fact that $g \ge 2$), this implies
				\begin{flalign*} 
				n! & \displaystyle\sum_{\textbf{n} \in \mathcal{K}_V (n)} \displaystyle\sum_{\textbf{s} \in \mathcal{K}_V (S)} \displaystyle\sum_{\textbf{T} \in \mathcal{C}_V (2T)} \displaystyle\sum_{\textbf{g} \in \mathcal{K}_V (g - E + V - 1)} \displaystyle\prod_{i = 1}^V \displaystyle\frac{(6g_i + 4s_i + 2n_i + 2T_i - 5)!}{(3g_i + 2s_i + n_i + T_i - 3)! g_i! s_i! T_i! n_i!} \\
				& \quad \le 2^{16V + 9} \left( \displaystyle\frac{T}{V} \right)^V \binom{2T - 1}{V - 1} \displaystyle\frac{(6g + 2n - 2E - 5)!}{(3g + n - E - 3)!} \displaystyle\frac{1}{(g - E + V - 1)!} \displaystyle\frac{1}{(S - X)! (2T - Y)!} \\
				& \quad \qquad \times \displaystyle\sum_{\textbf{n} \in \mathcal{K}_V (n)} \binom{n}{n_1, n_2, \ldots,  n_V} \\
				& \quad = 2^{16V + 9} \left( \displaystyle\frac{T}{V} \right)^V V^n \binom{2T - 1}{V - 1} \displaystyle\frac{(6g + 2n - 2E - 5)!}{(3g + n - E - 3)!} \displaystyle\frac{1}{(g - E + V - 1)!} \displaystyle\frac{1}{(S - X)! (2T - Y)!},
				\end{flalign*} 
				
				\noindent where in the last equality we used the identity  
				\begin{flalign*}
				\displaystyle\sum_{\textbf{n} \in \mathcal{K}_V (n)} \binom{n}{n_1, n_2, \ldots , n_V} = V^n.
				\end{flalign*}
				
				This, together with \Cref{lambdagvstestimate} yields 
				\begin{flalign}
				\label{gnvst1} 
				\begin{aligned} 
				\Upsilon_{g, n}^{(V; S, T)} & \le	12^E \displaystyle\frac{(6g + 2n - 2E - 5)!}{(6g + 2n - 7)!} \displaystyle\frac{(4g +  n - 4)!}{(3g + n - E - 3)! (g - E + V - 1)!} Z_E (3g + n - 3) \\
				& \qquad \times \displaystyle\frac{2^{4g + 2n + 13 V - S + 7}}{3^{g + V - 1}} V^n \left( \displaystyle\frac{3}{2} \right)^{2E + n} \left( \displaystyle\frac{T}{V} \right)^V \binom{2T - 1}{V - 1} \displaystyle\frac{(2T - 1)!!}{V!}\displaystyle\frac{1}{(S - X)! (2T - Y)!}.
				\end{aligned}
				\end{flalign}
				
				\noindent This with \Cref{productestimate2}, \Cref{zkmestimate2}, and the facts that $3g + n - 3 \ge g$ and $\log (3g + n - 3) \le \log g + 2$ (since $20n \le \log g$) together imply that 
				\begin{flalign*} 
				\Upsilon_{g, n}^{(V; S, T)} & \le g^{1/2 - V} 2^n \left( \displaystyle\frac{8}{3} \right)^{4g + n} E (\log g + 7)^{E - 1} \\
				& \qquad \times 2^{13 V - S + 7} 3^{V + 2} \left( \displaystyle\frac{3V}{2} \right)^n \left( \displaystyle\frac{9}{4} \right)^E \left( \displaystyle\frac{T}{V} \right)^V \binom{2T - 1}{V - 1} \displaystyle\frac{(2T - 1)!!}{V!} \displaystyle\frac{1}{(S - X)! (2T - Y)!}.
				\end{flalign*} 
				
				\noindent Together with the identity $E = S + T$ and the bounds $\binom{2T - 1}{V - 1} \le \frac{(2T)^{V - 1}}{(V - 1)!} \le \frac{(2T)^V}{V!}$, $V! \le 4^{-V} V^V$ (recall \eqref{kestimate1}), $3^V \le 2^{2V}$, and $3^2 < 2^4$, this implies the proposition. 
			\end{proof} 
		
			\subsection{Proof of \Cref{lambdag3}}
			
			\label{Proofv3}
			
			In section we establish \Cref{lambdag3} by summing the bound from \Cref{lambdagvstestimate2} over $S$, $T$, and $V$. The following two lemmas implement the sums over $S$ and $T$.

			\begin{lem} 
			
			\label{ssumlambda} 
			
			Fix integers $g \ge 2$, $n \ge 0$, $V \in [2, 2g + n - 2]$, and $T \in [V - 1, 3g + n - 3]$. Denoting $Y = \min \{ 2T, 3V \}$, we have
			\begin{flalign*}
			\displaystyle\sum_{S = 0}^{3g - T - 3} 2^{-n} \left( \displaystyle\frac{8}{3} \right)^{-4g - n} \Upsilon_{g, n}^{(V; S, T)} \le g^{13/8 - V} \left( \displaystyle\frac{3V}{2} \right)^n 2^{26V + 25} \left( \displaystyle\frac{9}{8} \right)^T \displaystyle\frac{T^{2V + Y + 1}}{V^{3V}} \displaystyle\frac{(\log g + 7)^{T + V}}{T!}.
			\end{flalign*}
			
			\end{lem}
			
			\begin{proof} 
				
				Denoting $X_S = \min \{ S, V \}$ for each $S \ge 0$, we have by \Cref{lambdagvstestimate2} that 
				\begin{flalign}
				\label{ssumlambda1}
				\begin{aligned}
				\displaystyle\sum_{S = 0}^{3g + n - T - 3} 2^{-n} & \left( \displaystyle\frac{8}{3} \right)^{-4g - n} \Upsilon_{g, n}^{(V; S, T)}  \\
				& \le (T + 1) g^{1/2 - V} 2^{20V + 11} \left( \displaystyle\frac{3V}{2} \right)^n \left( \displaystyle\frac{9}{4} \right)^T \left( \displaystyle\frac{T}{V} \right)^{2V} \displaystyle\frac{(\log g + 7)^{T - 1} (2T - 1)!!}{V^V (2T - Y)!} \\
				& \qquad \times \Bigg( \displaystyle\sum_{S = 0}^{V - 1} S \bigg( \displaystyle\frac{9}{8} \bigg)^S \displaystyle\frac{(\log g + 7)^S}{(S - X_S)!} + \displaystyle\sum_{S = V}^{3g + n - T - 3} S \bigg( \displaystyle\frac{9}{8} \bigg)^S \displaystyle\frac{(\log g + 7)^S}{(S - X_S)!} \Bigg) \\
				& \le (T + 1) g^{1/2 - V} 2^{20V + 11} \left( \displaystyle\frac{3V}{2} \right)^n \left( \displaystyle\frac{9}{4} \right)^T \left( \displaystyle\frac{T}{V} \right)^{2V} \displaystyle\frac{(\log g + 7)^{T - 1} (2T - 1)!!}{V^V (2T - Y)!} \\
				& \qquad \times \Bigg( V^2 \left( \displaystyle\frac{9}{8}\right)^V (\log g + 7)^V + \displaystyle\sum_{S = V}^{\infty} S \left( \displaystyle\frac{9}{8}\right)^S \displaystyle\frac{(\log g + 7)^S}{(S - V)!} \Bigg).
				\end{aligned} 
				\end{flalign}
				
				\noindent Moreover, by changing variables from $S - V$ to $S$, we have 
				\begin{flalign}
				\label{ssumlambda2}
				\begin{aligned}
				\displaystyle\sum_{S = V}^{\infty} S \left( \displaystyle\frac{9}{8}\right)^S \displaystyle\frac{(\log g + 7)^S}{(S - V)!} & \le (\log g + 7)^V \left( \displaystyle\frac{9}{8} \right)^V \displaystyle\sum_{S = 0}^{\infty} \displaystyle\frac{S + V}{S!} \left( \displaystyle\frac{9 (\log g + 7)}{8}\right)^S \\
				& = (\log g + 7)^V \left( \displaystyle\frac{9}{8} \right)^V \bigg( V + \displaystyle\frac{9 (\log g + 7)}{8} \bigg) \exp \bigg(\displaystyle\frac{9 (\log g + 7)}{8}\bigg) \\
				& \le 2^{12} (\log g + 7)^{V + 1} \left( \displaystyle\frac{9}{4} \right)^V g^{9/8},
				\end{aligned} 
				\end{flalign}
				
				\noindent where in the last bound we used the fact that $V \le 2^{V - 1}$ and that $e^{63 / 8} < e^8 < 2^{12}$. 
				
				Inserting \eqref{ssumlambda2} into \eqref{ssumlambda1} and using the bounds $V^2 \le 2^{2V}$ and $T + 1 \le 2T$, we obtain 
				\begin{flalign*}
				\displaystyle\sum_{S = 0}^{3g + n - T - 3} 2^{-n} & \left( \displaystyle\frac{8}{3} \right)^{-4g - n} \Upsilon_{g, n}^{(V; S, T)} \\
				& \le T g^{13/8 - V} 2^{23V + 25} \left( \displaystyle\frac{3V}{2} \right)^n \left( \displaystyle\frac{9}{4} \right)^T \left( \displaystyle\frac{T}{V} \right)^{2V} \displaystyle\frac{(\log g + 7)^{T + V} (2T - 1)!!}{V^V (2T - Y)!}.
				\end{flalign*}
				
				\noindent Now the lemma follows from the bounds $(2T)^Y (2T - Y)! \ge (2T)! = 2^T T! (2T - 1)!!$ and $Y \le 3V$. 
			\end{proof}

			\begin{lem}
				
				\label{lambdagvestimate} 
				
				Fix integers $g \ge 2$, $n \ge 0$, and $V \in [2, 2g + n - 2]$. Then,
				 \begin{flalign*}
				2^{-n} \left( \displaystyle\frac{8}{3} \right)^{-4g - n} \Upsilon_{g, n}^{(V)} \le 2^{26} g^{11 / 4} \left( \displaystyle\frac{3V}{2} \right)^n \left( \displaystyle\frac{2^{61} V^{1/2} (\log g + 7)^8}{g} \right)^V.
				 \end{flalign*}
			\end{lem}
		
		\begin{proof}
			
			By \eqref{lambdagvsum} and \Cref{ssumlambda}, we have 
			\begin{flalign}
			\label{sumlambdagv1} 
			\begin{aligned}
			2^{-n} \left( \displaystyle\frac{8}{3} \right)^{-4g - n} \Upsilon_{g, n}^{(V)} & \le g^{13 / 8 - V} \left( \displaystyle\frac{3V}{2} \right)^n 2^{26V + 25} V^{- 3V} (\log g + 7)^V \\
			& \qquad \times \Bigg( \displaystyle\sum_{T = V - 1}^{\lfloor 3V / 2 \rfloor} \displaystyle\frac{T^{2V + 2T + 1}}{T!} \bigg( \displaystyle\frac{9 (\log g + 7)}{8} \bigg)^T \\
			& \qquad \qquad \qquad + \displaystyle\sum_{T = \lceil 3V / 2 \rceil}^{\infty} \displaystyle\frac{T^{5V + 1} }{T!} \bigg( \displaystyle\frac{9 (\log g + 7)}{8} \bigg)^T \Bigg).
			\end{aligned} 
			\end{flalign} 
			
			\noindent We will analyze the two sums on the right side of \eqref{sumlambdagv1} individually. To bound the first, observe since $T! \ge e^{-T} T^T$ (by \eqref{kestimate1}) and $V \ge 2$ that 
			\begin{flalign}
			\label{sumv1lambda}
			\begin{aligned} 
			\displaystyle\sum_{T = V - 1}^{\lfloor 3V / 2 \rfloor} \displaystyle\frac{T^{2V + 2T + 1}}{T!} \left( \displaystyle\frac{9 (\log g + 7)}{8} \right)^T & \le \left( \displaystyle\frac{9e}{8} \right)^{3V / 2} \left( \displaystyle\frac{3}{2} \right)^{7V/2 + 1} \displaystyle\sum_{T = V - 1}^{\lfloor 3V / 2 \rfloor} V^{2V + T + 1} (\log g + 7)^T \\
			& \le 2^{7V} V^{7 V / 2 + 2} (\log g + 7)^{3V / 2}.
			\end{aligned} 
			\end{flalign}
			
			\noindent To analyze the second, observe using the bounds $T! \ge e^{-T} T^T$, $(T - 5V - 1)! \le 6^{5V + 1} T^{-5V - 1} T!$ for $T > 6V + 1$ (since $T - 5V - 1 \ge \frac{T}{6}$), and $V \le 2^V$ that
			\begin{flalign}
			\label{sumv2} 
			\begin{aligned}
			& \displaystyle\sum_{T = \lceil 3V / 2 \rceil}^{\infty} \displaystyle\frac{T^{5V + 1}}{T!} \left( \displaystyle\frac{9 (\log g + 7)}{8} \right)^T \\
			& \qquad \le \displaystyle\sum_{T = \lceil 3V / 2 \rceil}^{6V + 1} T^{5V - T + 1} \left( \displaystyle\frac{9e (\log g + 7)}{8} \right)^T + 6^{5V + 1} \displaystyle\sum_{T = 6V + 2}^{\infty} \displaystyle\frac{1}{(T - 5V - 1)!} \left( \displaystyle\frac{9 (\log g + 7)}{8} \right)^T \\
			& \qquad \le (13V)^{7V / 2 + 2} \left( \displaystyle\frac{9e (\log g + 7)}{8} \right)^{6V + 1} + \left( \displaystyle\frac{27 (\log + 7)}{4} \right)^{5V + 1} \displaystyle\sum_{T = V + 1}^{\infty} \displaystyle\frac{1}{T!} \left( \displaystyle\frac{9 (\log g + 7)}{8} \right)^T \\
			& \qquad \le 13^{5V} V^{7V / 2 + 2} 2^{12 V} (\log g + 7)^{6V + 1} + e^8 g^{9 / 8} 7^{6V} (\log g + 7)^{5V + 1} \\ 
			& \qquad \le 2^{33V} V^{7V / 2 + 2} g^{9 / 8} (\log g + 7)^{6V + 1},
			\end{aligned}
			\end{flalign} 
			
			\noindent where to deduce the second inequality we changed variables from $T + 5V + 1$ to $T$ in the second sum. Inserting \eqref{sumv1lambda} and \eqref{sumv2} into \eqref{sumlambdagv1} yields
			\begin{flalign*}
			2^{-n} \left( \displaystyle\frac{8}{3} \right)^{-4g - n} \Upsilon_{g, n}^{(V)} & \le g^{11 / 4 - V} \left( \displaystyle\frac{3V}{2} \right)^n 2^{59 V + 26} V^{V/2 + 2} (\log g + 7)^{8V},
			\end{flalign*} 
			
			\noindent from which we deduce the lemma, since $V^2 \le 2^{2V}$. 
		\end{proof}
			
			We now deduce \Cref{lambdag3} by summing the bound from \Cref{lambdagvestimate} over $V$.

		\begin{proof}[Proof of \Cref{lambdag3}] 
			
		Observe for $g > 2^{300}$ that $2^{n + 61} V^{1/2} (\log g + 7)^8 \le \frac{g}{2}$, whenever $V \le 3g$ and $20n \le \log g$. Thus, since $\big( \frac{3V}{2} \big)^n \le 2^{Vn}$, \Cref{lambdagvestimate} implies that 
		\begin{flalign*} 
		\displaystyle\sum_{V = 3}^{2g - 2} \left( \displaystyle\frac{8}{3} \right)^{-4g} \Upsilon_{g, n}^{(V)} & \le 2^{26} g^{11 / 4}\displaystyle\sum_{V = 3}^{2g - 2} \left( \displaystyle\frac{3V}{2} \right)^n \left( \displaystyle\frac{2^{61} V^{1 / 2} (\log g + 7)^8}{g} \right)^V \\
		& \le 2^{26} g^{11 / 4}\displaystyle\sum_{V = 3}^{2g - 2} \left( \displaystyle\frac{2^{n + 61} V^{1 / 2} (\log g + 7)^8}{g} \right)^V \\
		& \le 2^{27} g^{11 / 4} \left( \displaystyle\frac{2^{n + 63} (\log g + 7)^8}{g} \right)^3 \\
		& = 2^{3n + 216} (\log g + 7)^{24} g^{-1/4} \le 2^{240} (\log g)^{24} g^{-1/8},
		\end{flalign*} 
		
		\noindent where in the last inequality we used the bounds $\log g + 7 \le 2 \log g$ (as $g \ge 2^{300}$) and $20n \le \log g$. This yields the proposition. 	 
		\end{proof}

		\subsection{Relative Contributions of Single-Vertex Graphs} 
		
		\label{GraphVertex} 
		
		In this section we establish the following proposition, which is not directly related to \Cref{limitvolume} but was also predicted in \cite{VFGINMSC} (see, in particular, Conjecture 1.33 there). It essentially implies that the dominant contribution to the sum of $\mathcal{Z} \big( P(\Gamma) \big)$ over all stable graphs $\Gamma \in \mathcal{G}_{g, n}$ with a fixed number $E$ of edges is dominated by the single-vertex graph $\Gamma_{g, n} (E)$ from \Cref{gammage}.
			
			\begin{prop} 
				
				\label{gammagnegamma} 
				
				As $g$ tends to $\infty$, we have for $12 E \le \log g$ and $20n \le \log g$ that 
				\begin{flalign*}
				\displaystyle\sum_{\substack{\Gamma \in \mathcal{G}_{g, n} \\ |\mathfrak{E} (\Gamma)| = E}} \mathcal{Z} \big( P(\Gamma) \big) \sim \mathcal{Z} \Big( P \big( \Gamma_{g, n} (E) \big) \Big).
				\end{flalign*}
			\end{prop} 
			
			\begin{proof}
				
				Observe that 
				\begin{flalign*}
				\displaystyle\sum_{\substack{\Gamma \in \mathcal{G}_{g, n} \\ |\mathfrak{E} (\Gamma)| = E}} \mathcal{Z} \big( P(\Gamma) \big) = \mathcal{Z} \Big( P \big( \Gamma_{g, n} (E) \big) \Big) + \displaystyle\sum_{V = 2}^{2g + n - 2} \displaystyle\sum_{T = V - 1}^E \Upsilon_{g, n}^{(V; E - T, T)}, 
				\end{flalign*}
				
				\noindent and so it suffices to show for sufficiently large $g$ that
				\begin{flalign}
				\label{gnvtet} 
				 \displaystyle\sum_{V = 2}^{2g + n - 2} \displaystyle\sum_{T = V - 1}^E \Upsilon_{g, n}^{(V; E - T, T)} \le g^{-1/2} \mathcal{Z} \Big( P \big( \Gamma_{g, n} (E) \big) \Big).
				\end{flalign} 
				
				\noindent To lower bound the right side of \eqref{gnvtet}, observe by \eqref{lambdag1e3} that, for sufficiently large $g$,  
				\begin{flalign}
				\label{gammagnelower} 
				\mathcal{Z} \Big( P \big( \Gamma_{g, n} (E) \big) \Big) \ge g^{1/2} 2^{n - 6} \left( \displaystyle\frac{8}{3} \right)^{4g + n} \displaystyle\frac{Z_E (3g + n - 3)}{2^E E!}.
				\end{flalign} 
				
				\noindent To upper bound the left side of \eqref{gnvtet}, observe for $V \ge 2$ that, by \eqref{gnvst1} and \Cref{productestimate2},
				\begin{flalign*}
				\Upsilon_{g, n}^{(V; E - T, T)} & \le g^{3/2 - V}  2^n \left( \displaystyle\frac{8}{3} \right)^{4g + n} Z_E (3g + n - 3) \\
				& \qquad \times 2^{13 V + 6}  3^{V + 2} V^n \left( \displaystyle\frac{3}{2} \right)^{2E + n} \left( \displaystyle\frac{T}{V} \right)^V \binom{2T - 1}{V - 1} \displaystyle\frac{1}{V!}\displaystyle\frac{2^{T - E} (2T - 1)!!}{(E - T - X)! (2T - Y)!},
				\end{flalign*}
				
				\noindent where we have set $X = X_T = \min \{ E - T, V \}$ and $Y = Y_T = \min \{ 2T, 3V \}$. Using the facts that $(E - T)! \le E^V (E - T - X)!$; the identity $2^T T! (2T - 1)!! = (2T)!$; and the bounds $3^V < 2^{2V}$, $3^2 < 2^4$, $T \le E$, and
				\begin{flalign*} 
				\binom{2T - 1}{V - 1} \le \frac{(2T)^V}{V!}; \qquad \frac{1}{(E - T)! T!} = \binom{E}{T} \frac{1}{E!} \le \frac{2^E}{E!},
				\end{flalign*} 
				
				\noindent it follows that 
				\begin{flalign*}
				\Upsilon_{g, n}^{(V; E - T, T)} & \le g^{3/2 - V}  2^n \left( \displaystyle\frac{8}{3} \right)^{4g + n} Z_E (3g + n - 3) 2^{19V + 10} \left( \displaystyle\frac{3V}{2} \right)^n \left( \displaystyle\frac{9}{4} \right)^E \left( \displaystyle\frac{E^6}{V} \right)^V \displaystyle\frac{1}{V!^2 E!}.
				\end{flalign*}
				
				\noindent This, together with \eqref{gammagnelower} (and the fact that $V! \ge 1$), implies for $V \ge 2$ and sufficiently large $g$ that 
				\begin{flalign*}
				\Big( \mathcal{Z} \big( \Gamma_{g, n} (E) \big)\Big)^{-1} \Upsilon_{g, n}^{(V; E - T, T)} & \le 2^{19V + 16} g^{1 - V} \left( \displaystyle\frac{3V}{2} \right)^n \left( \displaystyle\frac{9}{2} \right)^E \left( \displaystyle\frac{E^6}{V} \right)^V.
				\end{flalign*}
				
				\noindent Summing over $T \in [V - 1, E]$ and $V \in [2, 2g + n - 2]$, we deduce using the fact that $\big( \frac{3V}{2} \big)^n < 2^{nV}$ and $E \le 2^E$ that, for sufficiently large $g$,
				\begin{flalign}
				\label{zgammagne1} 
				\begin{aligned} 
				\Big( \mathcal{Z} \big( \Gamma_{g, n} (E) \big)\Big)^{-1} \displaystyle\sum_{V = 2}^{2g + n - 3}\displaystyle\sum_{T = V - 1}^E \Upsilon_{g, n}^{(V; E - T, T)} & \le 2^{16} E \left( \displaystyle\frac{9}{2} \right)^E g \displaystyle\sum_{V = 2}^{2g + n - 2} \left( \displaystyle\frac{2^{19} E^6}{g V} \right)^V \left( \displaystyle\frac{3V}{2} \right)^n \\
				& \le 2^{16} 9^E g \displaystyle\sum_{V = 2}^{2g + n - 2} \left( \displaystyle\frac{2^{n + 19} E^6}{g V} \right)^V.
				\end{aligned} 
				\end{flalign}
				
				\noindent Since $20n \le \log g$ and $12E \le \log g$, we have for sufficiently large $g$ that $2^{16} 9^E < g^{1/4}$ and $E^6 2^{n + 19} \le g^{1 / 8}$. Therefore, \eqref{zgammagne1} implies that
				\begin{flalign*}
				\Big( \mathcal{Z} \big( \Gamma_{g, n} (E) \big)\Big)^{-1} \displaystyle\sum_{V = 2}^{2g + n - 3}\displaystyle\sum_{T = V - 1}^E \Upsilon_{g, n}^{(V; E - T, T)} & \le g^{5/4} \displaystyle\sum_{V = 2}^{\infty} \left( \displaystyle\frac{1}{g^{7/8} V} \right)^V \le g^{-1/2} \displaystyle\sum_{V = 2}^{\infty} V^{-V} < g^{-1 / 2},
				\end{flalign*}
				
				\noindent for sufficiently large $g$. This verifies \eqref{gnvtet} and therefore the proposition. 
			\end{proof}

			\section{Asymptotics for Siegel--Veech Constants} 
			
			\label{AsymptoticConstant}
			
			In this section we establish \Cref{constantlimit} that provides the large genus asymptotics for $c_{\text{area}} (\mathcal{Q}_{g, n})$. To that end, we begin in \Cref{ProofConstant} by first recalling a result from \cite{CSMSQD} that expresses this Siegel--Veech constant in terms of (other) principal strata volumes, and then establishing \Cref{constantlimit} assuming certain estimates; the latter estimates are proven in \Cref{ProofEstimateConstant}.

			\subsection{Proof of \Cref{constantlimit}}
			
			\label{ProofConstant}
			
			We first state an identity from \cite{CSMSQD} for the area Siegel--Veech constant $c_{\text{area}} ( \mathcal{Q}_{g, n})$ in terms of the principal strata volumes. In what follows, for any integers $g \ge 2$ and $n \ge 0$, we define the set of pairs of compositions  
			\begin{flalign*} 
			\mathfrak{X} (g, n) = \big\{ (\textbf{g}, \textbf{n}) \in \mathcal{K}_2 (g) \times \mathcal{C}_2 (n + 2): 3g_i + n_i \ge 4 \big\}. 
			\end{flalign*} 
			
			\noindent where we have denoted $\textbf{g} = (g_1, g_2) \in \mathcal{K}_2 (g)$ and $\textbf{n} = (n_1, n_2) \in \mathcal{C}_2 (n + 2)$.

			\begin{prop}[{\cite[Corollary 1]{CSMSQD}}]
				
				\label{volumeconstant}
				
				For any integers $g \ge 2$ and $n \ge 0$, we have that 
				\begin{flalign*}
				c_{\text{area}} (\mathcal{Q}_{g, n}) & = \varkappa_1 + \varkappa_2 + \varkappa_3,
				\end{flalign*} 
				
				\noindent where 
				\begin{flalign}
				\label{kappa123} 
				\begin{aligned} 
				\varkappa_1 & =  \displaystyle\frac{1}{8 \Vol \mathcal{Q}_{g, n}} \displaystyle\sum_{(\textbf{\emph{g}}, \textbf{\emph{n}}) \in \mathfrak{X} (g, n)} \displaystyle\frac{(4g + n - 4)!}{(4g_1 + n_1 - 4)! (4g_2 + n_2 - 4)!} \displaystyle\frac{n!}{(n_1 - 1)! (n_2 - 1)!} \\ 
				& \qquad \qquad \qquad \qquad \qquad \quad  \times \displaystyle\frac{(6g_1 + 2n_1 - 7)! (6g_2 + 2n_2 - 7)!}{(6g + 2n - 7)!} \Vol \mathcal{Q}_{g_1, n_1} \Vol \mathcal{Q}_{g_2, n_2}; \\
				\varkappa_2 & = \displaystyle\frac{n (n - 1) (4g + n - 4)}{(6g + 2n - 7) (6g + 2n - 8)} \displaystyle\frac{\Vol \mathcal{Q}_{g, n - 1}}{4 \Vol \mathcal{Q}_{g, n}}; \quad \varkappa_3 = \displaystyle\frac{(4g + n - 4)(4g + n - 5)}{(6g + 2n - 7) (6g + 2n - 8)} \displaystyle\frac{\Vol \mathcal{Q}_{g - 1, n + 2}}{\Vol \mathcal{Q}_{g, n}}.
				\end{aligned} 
				\end{flalign}
				
			\end{prop} 
		
			Thus, it remains to understand the limiting behaviors of $\varkappa_1$, $\varkappa_2$, and $\varkappa_3$ from \eqref{kappa123}. This is done through the following three lemmas; the first two will be established in \Cref{ProofEstimateConstant} below.

			\begin{lem} 
				
				\label{kappa1estimate}
				
				For any fixed integer $n \ge 0$, there exists a constant $C = C(n) > 1$ such that $\varkappa_1 < C g^{-1}$ holds for each integer $g \ge 2$. 
				
			\end{lem} 
			
			\begin{lem}
				
				\label{kappa2} 
				
				For any fixed integer $n \ge 0$, there exists a constant $C = C(n) > 1$ such that $\varkappa_2 < C g^{-1}$ holds for each integer $g \ge 2$.
				
			\end{lem}

			\begin{lem}
				
				\label{kappa3}
				
				For fixed integer $n \ge 0$, we have as $g$ tends to $\infty$ that $\varkappa_3 \sim \frac{1}{4}$.
				
			\end{lem}
			
			\begin{proof}
				
				By \eqref{limitvolume}, we have as $g$ tends to $\infty$ that 
				\begin{flalign*}
				\displaystyle\frac{\Vol \mathcal{Q}_{g - 1, n + 2}}{\Vol \mathcal{Q}_{g, n}} \sim 4 \left( \displaystyle\frac{8}{3} \right)^{-2}.
				\end{flalign*} 
				
				Combining this with the fact that as $g$ tends to $\infty$ we have 
				\begin{flalign*}
				\displaystyle\frac{(4g + n - 4) (4g + n - 5)}{(6g + 2n - 7) (6g + 2n - 8)} \sim \left( \displaystyle\frac{2}{3} \right)^2,
				\end{flalign*} 
				
				\noindent we deduce the lemma. 
			\end{proof}

			Assuming \Cref{kappa1estimate} and \Cref{kappa2}, we can quickly establish \Cref{constantlimit}.

			\begin{proof}[Proof of \Cref{constantlimit} Assuming \Cref{kappa1estimate} and \Cref{kappa2}]
				
				This follows from \Cref{volumeconstant} and summing the results of \Cref{kappa1estimate}, \Cref{kappa2}, and \Cref{kappa3}. 
			\end{proof}

			\subsection{Proofs of \Cref{kappa1estimate} and \Cref{kappa2}}
			
			\label{ProofEstimateConstant} 
			
			In this section we establish \Cref{kappa1estimate} and \Cref{kappa2}. To that end, we let $n \ge 0$ denote an integer. Throughout this section, we further let $R = R(n) > 1$ denote an $n$-dependent constant (whose existence is guaranteed by \Cref{limitvolume}), such that 
			\begin{flalign}
			\label{rgmvolume} 
			R^{-1} \left( \displaystyle\frac{8}{3} \right)^{4G} \le \Vol \mathcal{Q}_{G, m} \le R \left( \displaystyle\frac{8}{3} \right)^{4G}, \qquad \text{holds for any integers $G \ge 0$ and $m \in [0, n + 2]$}. 
			\end{flalign}
			
			Given this notation, we can quickly establish \Cref{kappa2}.

			\begin{proof}[Proof of \Cref{kappa2}]
				
				First observe that since $4g + n - 4 \le 4g (n + 1)$ for $g \ge 2$; $n (n - 1) (n + 1) \le n^3$ for $n \ge 0$; and $(6g + 2n - 7) (6g + 2n - 8) \ge g^2$ for $g \ge 2$, we have that
				\begin{flalign}
				\label{ngestimate}
				\displaystyle\frac{n (n - 1) (4g + n - 4)}{(6g + 2n - 7) (6g + 2n - 8)} \le \displaystyle\frac{4 n^3}{g}.
				\end{flalign}
				
				\noindent Now, applying \eqref{rgmvolume} with the $(G, m)$ there first equal to $(g, n)$ here; then equal to $(g - 1, n + 2)$ here; and next dividing yields 
				\begin{flalign}
				\label{q1volumesq}
				\displaystyle\frac{\Vol \mathcal{Q}_{g, n - 1}}{4 \Vol \mathcal{Q}_{g, n}} \le \displaystyle\frac{R^2}{4}.
				\end{flalign}
				
				\noindent We now deduce the lemma from combining \eqref{ngestimate} and \eqref{q1volumesq}. 
			\end{proof}
		
		 To establish \Cref{kappa1estimate}, we begin with the following lemma that bounds the summands appearing in the definition \eqref{kappa123} of $\varkappa_1$. 
			
			\begin{lem} 
				
				\label{2productestimate} 
				
				For any fixed integer $n \ge 0$, there exists a constant $C = C(n) > 1$ such that the following holds. Let $n_1, n_2 \ge 1$ and $g_1, g_2 \ge 0$ be integers such that $n = n_1 + n_2 - 2$. Set $g = g_1 + g_2$; assume that $g \ge 2$ and that $3g_i + n_i \ge 4$ for each $i \in \{ 1, 2 \}$. Then, 
				\begin{flalign}
				\label{constantvolumeestimate1}
				\begin{aligned}
				\displaystyle\frac{1}{\Vol \mathcal{Q}_{g, n}} & \displaystyle\frac{(4g + n - 4)!}{(4g_1 + n_1 - 4)! (4g_2 + n_2 - 4)!} \displaystyle\frac{n!}{(n_1 - 1)! (n_2 - 1)!} \displaystyle\frac{(6g_1 + 2n_1 - 7)! (6g_2 + 2n_2 - 7)!}{(6g + 2n - 7)!} \\
				& \qquad \qquad \qquad \times  \Vol \mathcal{Q}_{g_1, n_1} \Vol \mathcal{Q}_{g_2, n_2} \le \displaystyle\frac{C}{g} \binom{4g + n - 4}{4g_1 + n_1 - 3} \binom{6g + 2n - 8}{6g_1 + 2n_1 - 6}^{-1}.
				\end{aligned} 
				\end{flalign}

			\end{lem}
			
			\begin{proof}

				\noindent Applying \eqref{rgmvolume} with $(G, m) \in \big\{ (g, n), (g_1, n_1), (g_2, n_2) \big\}$ and using the fact that $g_1 + g_2 = g$, we deduce that 
				\begin{flalign}
				\label{estimateg1g2volume}
				\displaystyle\frac{\Vol \mathcal{Q}_{g_1, n_1} \Vol \mathcal{Q}_{g_2, n_2}}{\Vol \mathcal{Q}_{g, n}} \le R^3. 
				\end{flalign}
				
				Next, in view of the facts that $n_1 + n_2 = n + 2$ and $g = g_1 + g_2$, we have 
				\begin{flalign*}
				& \displaystyle\frac{(4g + n - 4)!}{(4g_1 + n_1 - 4)! (4g_2 + n_2 - 4)!} \displaystyle\frac{n!}{(n_1 - 1)! (n_2 - 1)!} \displaystyle\frac{(6g_1 + 2n_1 - 7)! (6g_2 + 2n_2 - 7)!}{(6g + 2n - 7)!} \\
				& \qquad = \displaystyle\frac{1}{6g + 2n - 7} \displaystyle\frac{(4g_1 + n_1 - 3) (4g_2 + n_2 - 3)}{(6g_1 + 2n_1 - 6) (6g_2 + 2n_2 - 6)} \binom{n}{n_1 - 1} \binom{4g + n - 4}{4g_1 + n_1 - 3} \binom{6g + 2n - 8}{6g_1 + 2n_1 - 6}^{-1}.
				\end{flalign*}
				
				\noindent Since $n_i \ge 1$, $g_i \ge 0$, and $3g_i + n_i \ge 4$, we have $6g_i + 2n_i - 6\ge 4g_i + n_i - 3$ for each $i \in \{ 1, 2 \}$, and so it follows that
				\begin{flalign*}
				\displaystyle\frac{(4g + n - 4)!}{(4g_1 + n_1 - 4)! (4g_2 + n_2 - 4)!} & \displaystyle\frac{n!}{(n_1 - 1)! (n_2 - 1)!} \displaystyle\frac{(6g_1 + 2n_1 - 7)! (6g_2 + 2n_2 - 7)!}{(6g + 2n - 7)!} \\
				& \le \displaystyle\frac{1}{6g + 2n - 7} \binom{n}{n_1 - 1}  \binom{4g + n - 4}{4g_1 + n_1 - 3} \binom{6g + 2n - 8}{6g_1 + 2n_1 - 6}^{-1}.
				\end{flalign*}
				
				\noindent Moreover, since $6g + 2n - 7 \ge g$ (as $g = g_1 + g_2 \ge 4$) and $\binom{n}{n_1 - 1} \le 2^n$, we deduce that 
				\begin{flalign*}
				\displaystyle\frac{(4g + n - 4)!}{(4g_1 + n_1 - 4)! (4g_2 + n_2 - 4)!} \displaystyle\frac{n!}{(n_1 - 1)! (n_2 - 1)!} & \displaystyle\frac{(6g_1 + 2n_1 - 7)! (6g_2 + 2n_2 - 7)!}{(6g + 2n - 7)!} \\
				& \le \displaystyle\frac{2^n}{g} \binom{4g + n - 4}{4g_1 + n_1 - 3} \binom{6g + 2n - 8}{6g_1 + 2n_1 - 6}^{-1}.
				\end{flalign*} 
				
				\noindent This, together with \eqref{estimateg1g2volume}, implies the lemma. 
			\end{proof}	
		
			We next have the following lemma that bounds the right side of \eqref{constantvolumeestimate1} if $g_1, g_2 \ge 2$.
			
			\begin{lem} 
				
				\label{1productestimate}
				
				Fix integers $n_1, n_2 \ge 1$ and $g_1, g_2 \ge 2$; set $n = n_1 + n_2 - 2$ and $g = g_1 + g_2$. Then, 
				\begin{flalign*}
				\binom{4g + n - 4}{4g_1 + n_1 - 3} \binom{6g + 2n - 8}{6g_1 + 2n_1 - 6}^{-1} \le \displaystyle\frac{1}{g}.
				\end{flalign*}
				
			\end{lem}
			
			\begin{proof}
				
				Observe that $2g_i + n_i \ge 4$, since $g_i \ge 2$, for each $i \in \{ 1, 2 \}$. Then, \Cref{sumaijaiestimate} gives
				\begin{flalign} 
				\label{gng1n1productestimate}
				\binom{4g + n - 4}{4g_1 + n_1 - 3} \binom{6g + 2n - 8}{6g_1 + 2n_1 - 6}^{-1} \le \binom{2g + n - 4}{2g_1 + n_1 - 3}^{-1} \le \displaystyle\frac{1}{2g + n - 4}, 
				\end{flalign}		
				
				\noindent where in the second statement we used the fact that $\binom{a}{b} \ge a$ if $b \in [1, a - 1]$ (applied with the $(a, b)$ there equal to $(2g + n - 4, 2g_1 + n_1 - 3)$ here, where the condition $b \in [1, a - 1]$ is verified since $\min \{ 2g_1 + n_1, 2g_2 + n_2 \} \ge 4$, $n_1 + n_2 = n + 2$, and $g_1 + g_2 = g$). This, together with the fact that $g + n \ge g_1 + g_2 \ge 4$, implies the lemma. 
			\end{proof}

			We can now establish \Cref{kappa1estimate}. 
			
			\begin{proof}[Proof of \Cref{kappa1estimate}]
				
				Observe by \Cref{2productestimate} that there exists a constant $C_1 = C_1 (n) > 1$ such that
				\begin{flalign*}
				\varkappa_1 \le \displaystyle\frac{C_1}{g} \displaystyle\sum_{(\textbf{g}, \textbf{n}) \in \mathfrak{X} (g, n)} \binom{4g + n - 4}{4g_1 + n_1 - 3} \binom{6g + 2n - 8}{6g_1 + 2n_1 - 6}^{-1},
				\end{flalign*}
				
				\noindent where we have denoted $\textbf{g} = (g_1, g_2)$ and $\textbf{n} = (n_1, n_2)$. Since $\big| \mathcal{C}_2 (n + 2) \big| \le n + 1$, it follows that there exists a constant $C_2 = C_2 (n) > 1$ such that 
				\begin{flalign}
				\label{kappa1estimate1} 
				\varkappa_1 \le \displaystyle\frac{C_2}{g} \displaystyle\sum_{\textbf{g} \in \mathcal{K}_2 (g)} \displaystyle\max_{\substack{\textbf{n} \in \mathcal{C}_2 (n + 2) \\ (\textbf{g}, \textbf{n}) \in \mathfrak{X} (g, n)}} \binom{4g + n - 4}{4g_1 + n_1 - 3} \binom{6g + 2n - 8}{6g_1 + 2n_1 - 6}^{-1}.
				\end{flalign}
				
				In view of \Cref{1productestimate} and the fact that $\mathcal{K}_2 (g) \le g + 1 \le 2g$ (where the last bound holds since $g \ge 2$), we have that 
				\begin{flalign*}
				\displaystyle\sum_{\substack{g_1 + g_2 = g \\ g_1, g_2 \ge 2}} \displaystyle\max_{\textbf{n} \in \mathcal{C}_2 (n + 2)} \binom{4g + n - 4}{4g_1 + n_1 - 3} \binom{6g + 2n - 8}{6g_1 + 2n_1 - 6}^{-1} \le 2.
				\end{flalign*}
				
				\noindent This, together with \eqref{kappa1estimate1}, gives  
				\begin{flalign}
				\label{kappa1estimate2} 
				\begin{aligned}
				\varkappa_1 & \le \displaystyle\frac{C_2}{g} \displaystyle\sum_{\substack{\textbf{g} \in \mathcal{K}_2 (g) \\ \min \{ g_1, g_2 \} \le 1}} \displaystyle\max_{\substack{\textbf{n} \in \mathcal{C}_2 (n + 2) \\ (\textbf{g}, \textbf{n}) \in \mathfrak{X} (g, n)}} \binom{4g + n - 4}{4g_1 + n_1 - 3} \binom{6g + 2n - 8}{6g_1 + 2n_1 - 6}^{-1} + \displaystyle\frac{2C_2}{g} \\
				& \le \displaystyle\frac{4 C_2}{g} \displaystyle\max_{\substack{\textbf{g} \in \mathcal{K}_2 (g) \\ \min \{ g_1, g_2 \} \le 1}} \displaystyle\max_{\substack{\textbf{n} \in \mathcal{C}_2 (n + 2) \\ (\textbf{g}, \textbf{n}) \in \mathfrak{X} (g, n)}} \binom{4g + n - 4}{4g_1 + n_1 - 3} \binom{6g + 2n - 8}{6g_1 + 2n_1 - 6}^{-1} + \displaystyle\frac{2C_2}{g}, 
				\end{aligned} 
				\end{flalign}
				
				\noindent where in the last inequality we used the fact that there are at most $4$ nonnegative compositions $\textbf{g} = (g_1, g_2) \in \mathcal{K}_2 (g)$ such that $\min \{ g_1, g_2 \} \le 1$. 
				
				Now observe that $4g + n - 4 \le 6g + 2n - 8$ (as $g \ge 2$) and $6g_i + 2n_i - 6 \ge 4g_i + n_i - 3$ if $g_i \le 1$ and $3g_i + n_i \ge 4$ (as $2g_i + n_i - 3 \ge 3g_i + n_i - 4 \ge 0$). Therefore, since $\binom{a}{b} \le \binom{a'}{b'}$ whenever $a \le a'$, $b \le b'$, $2b \le a$, and $2b' \le a'$, 	it follows that 
				\begin{flalign*}
				\displaystyle\max_{\substack{\textbf{g} \in \mathcal{K}_2 (g) \\ \min \{ g_1, g_2 \} \le 1}} \displaystyle\max_{\substack{\textbf{n} \in \mathcal{C}_2 (n + 2) \\ (\textbf{g}, \textbf{n}) \in \mathfrak{X} (g, n)}} \binom{4g + n - 4}{4g_1 + n_1 - 3} \binom{6g + 2n - 8}{6g_1 + 2n_1 - 6}^{-1} \le 1, 
				\end{flalign*} 
				
				\noindent and so the lemma follows from \eqref{kappa1estimate2}.
			\end{proof}


\begin{thebibliography}{}
			
		\bibitem{LGAC} \label{LGAC} A. Aggarwal, Large Genus Asymptotics for Siegel--Veech Constants, \emph{Geom. Funct. Anal.} \textbf{29}, 1295--1324, 2019.
		
		
		\bibitem{LGAVSD} \label{LGAVSD} A. Aggarwal, Large Genus Asymptotics for Volumes of Strata of Abelian Differentials, With an Appendix by A. Zorich, To appear in \emph{J. Amer. Math. Soc.}, preprint, arXiv:1804.05431. 
		 
		\bibitem{LGAVCSQD} \label{LGAVCSQD} A. Aggarwal, V. Delecroix, \'{E}. Goujard, P. G. Zograf, and A. Zorich, Conjectural Large Genus Asymptotics of Masur--Veech Volumes and of Area Siegel--Veech Constants of Strata of Quadratic Differentials, To appear in \emph{Arnold Math. J.}, preprint, arXiv:1912.11702.
		
		\bibitem{TRV} \label{TRV} J. E. Andersen, G. Borot, S. Charbonnier, V. Delecroix, A. Giacchetto, D. Lewa\'{n}ski, and C. Wheeler, Topological Recursion for Masur--Veech Volumes, preprint, arXiv:1905.10352. 
		
		\bibitem{SSRIF} \label{SSRIF} F. Arana-Herrera, Counting Square-tiled Surfaces with Prescribed Real and Imaginary Foliations and Connections to Mirzakhani's Asymptotics for Simple Closed Hyperbolic Geodesics, preprint, arXiv:1902.05626. 
		
		\bibitem{CD} \label{CD} J. S. Athreya, A. Eskin, and A. Zorich, Counting Generalized Jenkins--Strebel Differentials, \emph{Geom. Dedicata} \textbf{170}, 195--217, 2014. 
		
		\bibitem{VITPSQD} \label{VITPSQD} D. Chen, M. M\"{o}ller, and A. Sauvaget, Masur--Veech Volumes and Intersection Theory: The Principal Strata of Quadratic Differentials, With an Appendix by G. Borot, A. Giacchetto, and D. Lewa\'{n}ski, preprint, arXiv:1912.02267.
		
		\bibitem{QLGL} \label{QLGL} D. Chen, M. M\"{o}ller, and D. Zagier, Quasimodularity and Large Genus Limits of Siegel--Veech Constants, \emph{J. Amer. Math. Soc.} \textbf{31}, 1059--1163, 2018.
		
		\bibitem{VSCC}\label{VSCC} D. Chen, M. M\"oller, A. Sauvaget, and D. Zagier, Masur-Veech Volumes and Intersection Theory on Moduli Spaces of Abelian Differentials, preprint, arXiv:1901.01785.
		
		\bibitem{AGSSSMLG} \label{AGSSSMLG} V. Delecroix, \'{E}. Goujard, P. G. Zograf, and A. Zorich, Asymptotic Geometry of Square-Tiled Surfaces and of Simple Multicurves in Large Genera, In preparation. 
		
		\bibitem{VFGINMSC} \label{VFGINMSC} V. Delecroix, \'{E}. Goujard, P. G. Zograf, and A. Zorich, Masur--Veech Volumes, Frequencies of Simple Closed Geodesics, and Intersection Numbers on Moduli Spaces of Curves, preprint, arXiv:1908.08611.
		
		\bibitem{LBINC} \label{LBINC} V. Delecroix, \'{E}. Goujard, P. G. Zograf, and A. Zorich, Uniform Lower Bound for Intersection Numbers of $\psi$-Classes, preprint, arXiv:2004.02749.
		
		\bibitem{ITIHTF} \label{ITIHTF} R. Dijkgraaf, Intersection Theory, Integrable Hierarchies and Topological Field Theory, In: \emph{New Symmetry Principles in Quantum Field Theory (Carg\'{e}se, 1991)}, NATO Adv. Sci. Inst. Ser. B Phys., 295, Plenum, New York, 95--158, 1992. 
		
		\bibitem{ERG} \label{ERG} A. Eskin, M. Kontsevich, and A. Zorich, Sum of Lyapunov Exponents of the Hodge Bundle With Respect to the Teichm\"{u}ller Geodesic Flow, \textit{Publ. Math. IHES} \textbf{120}, 207--333, 2014. 
		
		\bibitem{AFS} \label{AFS} A. Eskin and H. Masur, Asymptotic Formulas on Flat Surfaces, \emph{Ergod. Th. Dynam. Sys.} \textbf{21}, 443--478, 2001.
		
		\bibitem{PBC} \label{PBC} A. Eskin, H. Masur, and A. Zorich, Moduli Spaces of Abelian Differentials: The Principal Boundary, Counting Problems, and the Siegel--Veech Constants, \textit{Publ. Math. IHES} \textbf{97}, 61--179, 2003. 
		
		\bibitem{ANBCTV} \label{ANBCTV} A. Eskin and A. Okounkov, Asymptotics of Numbers of Branched Coverings of a Torus and Volumes of Moduli Spaces of Holomorphic Differentials, \textit{Invent. Math.} \textbf{145}, 59--103, 2001. 
		
		\bibitem{QF} \label{QF} A. Eskin and A. Okounkov, Pillowcases and Quasimodular Forms, In: \emph{Algebraic Geometry and Number Theory} (V. Ginzburg ed.), Progress in Mathematics \textbf{253}, Birkh\"{a}user, Boston, 1--25, 2006.
		
		\bibitem{TCBC} \label{TCBC} A. Eskin, A. Okounkov, and R. Pandharipande, The Theta Characteristic of a Branched Covering, \emph{Adv. Math.} \textbf{217}, 873--888, 2008.
		
		\bibitem{VSDCLG} \label{VSDCLG} A. Eskin and A. Zorich, Volumes of Strata of Abelian Differentials and Siegel-Veech Constants in Large Genera, \textit{Arnold Math. J.} \textbf{1}, 481--488, 2015.
		
		\bibitem{AC} \label{AC} P. Flajolet and R. Sedgewick, \emph{Analytic Combinatorics}, Cambridge University Press, 2009.
		
		\bibitem{SGELGS} \label{SGELGS} C. Gilmore, E. Le Masson, T. Sahlsten, and J. Thomas, Short Geodesic Loops and $L^p$ Norms of Eigenfunctions on Large Genus Surfaces, preprint, arXiv:1912.09961. 
		
		\bibitem{CSMSQD} \label{CSMSQD} \'{E}. Goujard, Siegel--Veech Constants for Strata of Moduli Spaces of Quadratic Differentials, \emph{Geom. Funct. Anal.} \textbf{25}, 1440--1492, 2015. 
		
		\bibitem{VSMSQDV} \label{VSMSQDV} \'{E}. Goujard, Volumes of Strata of Moduli Spaces of Quadratic Differentials: Getting Explicit Values, \emph{Ann. Inst. Fourier, Grenoble} \textbf{66}, 2203--2251, 2016. 
		
		\bibitem{RV} \label{RV} M. Kazarian, Recursion for the Masur--Veech Volumes of Moduli Spaces of Quadratic Differentials, preprint, arXiv:1912.10422. 
		
		\bibitem{ITMSCF} \label{ITMSCF} M. Kontsevich, Intersection Theory on the Moduli Space of Curves and the Matrix Airy Function, \emph{Commun. Math. Phys.} \textbf{147}, 1--23, 1992. 
		
		\bibitem{RA} \label{RA} K. Liu and H. Xu, A Remark on Mirzakhani's Asymptotic Formulae, \emph{Asian J. Math.} \textbf{18}, 29--52, 2014.
		
		\bibitem{RIN} \label{RIN} K. Liu and H. Xu, An Effective Recursion Formula for Computing Intersection Numbers, preprint, arXiv:0710.5322. 
		
		\bibitem{ETMF} \label{ETMF} H. Masur, Interval Exchange Transformations and Measured Foliations, \textit{Ann. Math.} \textbf{115}, 169--200, 1982. 	
		
		\bibitem{TSGTMS} \label{TSGTMS} H. Masur, K. Rafi, and A. Randecker, The Shape of a Generic Translation Surface, preprint, arXiv:1809.10769. 
		
		\bibitem{RBF} \label{RBF} H. Masur and S. Tabachnikov, Rational Billiards and Flat Structures, In: \emph{Handbook of Dynamical Systems} (B. Hasselblatt and A. Katok ed.), Elsevier Science B.V., 1015--1089, 2002.
		
		\bibitem{ETE} \label{ETE} M. Mirzakhani, Ergodic Theory of the Earthquake Flow, \emph{Int. Math. Res. Not.} \textbf{2008}, Art. 116, 2008.   
		
		\bibitem{GNSCGHS} \label{GNSCGHS} M. Mirzakhani, Growth of the Number of Simple Closed Geodesics on Hyperbolic Surfaces, \emph{Ann. Math.} \textbf{168}, 97--125, 2008.
		
		\bibitem{GVRHSLG} \label{GVRHSLG} M. Mirzakhani, Growth of Weil-Petersson Volumes and Random Hyperbolic Surfaces of Large Genus, \emph{J. Differential Geom.} \textbf{94}, 267--300, 2013.
		
		\bibitem{VITMSC} \label{VITMSC} M. Mirzakhani, Weil--Petersson Volumes and Intersection Theory on the Moduli Space of Curves, \emph{J. Amer. Math. Soc.} \textbf{20}, 1--23, 2007.
		
		\bibitem{LCGRSLG} \label{LCGRSLG} M. Mirzakhani and B. Petri, Lengths of Closed Geodesics on Random Surfaces of Large Genus, \emph{Comment. Math. Helv.} \textbf{94}, 869--889, 2019. 
		
		\bibitem{LGAIMSC} \label{LGAIMSC} M. Mirzakhani and P. Zograf, Towards Large Genus Asymptotics of Intersection Numbers on Moduli Spaces of Curves, \emph{Geom. Funct. Anal.} \textbf{25}, 1258--1289, 2015.
		
		\bibitem{GFINMSC} \label{GFINMSC} A. Okounkov, Generating Functions for Intersection Numbers on Moduli Spaces of Curves, \emph{Int. Math. Res. Not.} \textbf{2002}, 933--957, 2002.  
		
		\bibitem{CCSD} \label{CCSD} A. Sauvaget, Cohomology Classes of Strata of Differentials, \emph{Geom. Topol.} \textbf{23}, 1085--1171, 2019. 
		
		\bibitem{LGAEV} \label{LGAEV} A. Sauvaget, The Large Genus Asymptotic Expansion of Masur--Veech Volumes, To appear in \emph{Int. Math. Res. Not.}, preprint, arXiv:1903.04454.
		
		\bibitem{VSI} \label{VSI}	A. Sauvaget, Volumes and Siegel-Veech Constants of $\mathcal{H} (2g - 2)$ and Hodge Integrals, \emph{Geom. Funct. Anal.} \textbf{28}, 1756--1779, 2018.
		
		\bibitem{DELGRS} \label{DELGRS} J. Thomas, Delocalisation of Eigenfunctions on Large Genus Random Surfaces, preprint, arXiv:2002.01403.
		
		\bibitem{MTSIEM} \label{MTSIEM} W. A. Veech, Gauss Measures for Transformations on the Space of Interval Exchange Maps, \textit{Ann. Math.} \textbf{115}, 201--242, 1982. 
		
		\bibitem{M} \label{M} W. A. Veech, Siegel Measures, \emph{Ann. Math.} \textbf{148}, 895--944, 1998. 
		
		\bibitem{TTQG} \label{TPQG} E. Verlinde and H. Verlinde, A Solution of Two-Dimensional Topological Quantum Gravity, \emph{Nucl. Phys. B} \textbf{348}, 457--489, 1991. 
		
		\bibitem{PGTS} \label{PGTS}	Ya. Vorobets, \textit{Periodic Geodesics of Translation Surfaces}, (2003), in S. Kolyada, Yu. I. Manin and T. Ward (eds.) Algebraic and Topological Dynamics, Contemporary Math., vol. 385, pp. 205--258, Amer. Math. Soc., Providence, 2005.
		
		\bibitem{TGITMS} \label{TGITMS} E. Witten, Two-Dimensional Gravity and Intersection Theory on Moduli Space, \emph{Surveys in Differential Geometry} \textbf{1}, 243--310, 1991. 
			
		\bibitem{TSOC} \label{TSOC} A. Wright, Translation Surfaces and Their Orbit Closures: An Introduction for a Broad Audience, \emph{EMS Surv. Math. Sci.} \textbf{2}, 63--108, 2015.
		
		\bibitem{F} \label{F} J. Zhou, Explicit Formula for Witten--Kontsevich Tau-Function, preprint, arXiv:1306.5429. 
				
		\bibitem{C} \label{C} P. G. Zograf, An Explicit Formula for Witten's 2-Correlators, \emph{J. Math. Sci.} \textbf{240}, 535--538, 2019. 
		
		\bibitem{LGAV} \label{LGAV} P. Zograf, On the Large Genus Asymptotics of Weil-Petersson Volumes, preprint, arxiv:0812.0544.
		
		\bibitem{FS} \label{FS} A. Zorich, Flat Surfaces, In: \emph{Frontiers in Number Theory, Physics, and Geometry} (P. E. Cartier, B. Julia, P. Moussa, and P. Vanhove ed.), Springer, Berlin, 437--583, 2006
		
		\bibitem{SVMS} \label{SVMS} A. Zorich, Square Tiled Surfaces and Teichm\"{u}ller Volumes of the Moduli Spaces of Abelian Differentials, In: \emph{Rigidity in Dynamics and Geometry} (M. Burger and A. Iozzi), Springer, Berlin, 459--471, 2002.
		
	\end{thebibliography}
\end{document}